\documentclass[11pt]{article} 
\usepackage{amsmath,amssymb,amsfonts,amsthm,amscd,graphicx,psfrag,epsfig}
\usepackage{color}
\usepackage{xcolor}
\definecolor{blue}{rgb}{0,0,0.7}
\definecolor{red}{rgb}{0.75, 0, 0}
\definecolor{midnight}{rgb}{0.0,0.2,0.4}
\usepackage{stmaryrd}
\usepackage{shuffle}

\usepackage[titletoc,toc]{appendix}
\usepackage[
colorlinks=true, 
linkcolor=blue, 
citecolor=blue, 
filecolor=blue, 
menucolor=blue, 
urlcolor=blue]{hyperref}
 \usepackage{hyperref}
\hypersetup{colorlinks=true, citecolor=blue,linkcolor=midnight}
\usepackage[all]{xy}           
\usepackage{marginnote}  
\usepackage{float}
\usepackage{enumitem}

 \usepackage{xcolor}
\usepackage{hyperref}
\usepackage{amsmath,amssymb,amsfonts,amsthm,amscd,graphicx,psfrag,epsfig,mathabx}
\usepackage{color}
\definecolor{blue}{rgb}{0,0,0.7}
\definecolor{red}{rgb}{0.75, 0, 0}
\usepackage[titletoc,toc]{appendix}
\usepackage{pgf,tikz,pgfplots}
\usepackage{mathrsfs}

\newtheorem{theorem}{Theorem}[section]

\newtheorem{theorem-definition}[theorem]{Theorem-Definition}
\newtheorem{theorem-construction}[theorem]{Theorem-Construction}
\newtheorem{lemma-definition}[theorem]{Lemma--Definition}
\newtheorem{lemma-construction}[theorem]{Lemma--Construction}
\newtheorem{lemma}[theorem]{Lemma}
\newtheorem{proposition}[theorem]{Proposition}
\newtheorem{corollary}[theorem]{Corollary}
\newtheorem{conjecture}[theorem]{Conjecture}
\newtheorem{definition}[theorem]{Definition}
\newtheorem{example}[theorem]{Example}

\newenvironment{remark}[1][Remark.]{\begin{trivlist}
\item[\hskip \labelsep {\bfseries #1}]}{\end{trivlist}}

\newcommand{\old}[1]{}

\newcommand{\Z}{{\mathbb Z}}
\newcommand{\R}{{\mathbb R}}
\newcommand{\Q}{{\mathbb Q}}
\newcommand{\C}{{\mathbb C}}

\newcommand{\A}{{\rm A}}
\newcommand{\F}{{\rm F}}
\newcommand{\B}{{\rm B}}
\newcommand{\G}{{\rm G}}

\renewcommand{\P}{{\mathbb P}}

\newcommand{\lms}{\longmapsto}
\newcommand{\lra}{\longrightarrow}
\newcommand{\hra}{\hookrightarrow}

\newcommand{\be}{\begin{equation}}
\newcommand{\ee}{\end{equation}}
\newcommand{\bt}{\begin{theorem}}
\newcommand{\et}{\end{theorem}}
\newcommand{\bd}{\begin{definition}}
\newcommand{\ed}{\end{definition}}
\newcommand{\bp}{\begin{proposition}}
\newcommand{\ep}{\end{proposition}}

\newcommand{\bl}{\begin{lemma}}
\newcommand{\el}{\end{lemma}}
\newcommand{\bc}{\begin{corollary}}
\newcommand{\ec}{\end{corollary}}
\newcommand{\bcon}{\begin{conjecture}}
\newcommand{\econ}{\end{conjecture}}
\newcommand{\la}{\label}
\renewcommand{\L}{{\mathcal L}}
\newcommand{\alt}{\rm Alt}

\newcommand{\bi}{\begin{itemize}}
\newcommand{\ei}{\end{itemize}}
\newcommand{\bs}{\begin{split}}
\newcommand{\es}{\end{split}}
\begin{document}

\date{}
 
\title{Motivic correlators, cluster   varieties   and Zagier's conjecture on $\zeta_\F(4)$}

\author{Alexander  Goncharov, Daniil Rudenko}

\maketitle

\tableofcontents

\begin{abstract}
We prove Zagier's conjecture 
 on  the  value at $s=4$ of  the  Dedekind $\zeta$-function  of  a number field $\F$:
\[ 
\zeta_\F(4) = \pi^{4(r_1+r_2)}|d_\F|^{-1/2}\cdot {\rm det} \Bigl({\cal L}_4(\sigma_i(y_j))\Bigr), ~~~~1\leq i,j \leq r_2.
\]
Here $\{\sigma_i\}$ is the set of all complex embeddings of $\F$, up to the conjugation. Namely,  up to a standard factor,  $\zeta_\F(4)$ is equal to a   
  $r_2\times r_2$ determinant, whose entries are $\Q$-linear combinations of  the values of a single-valued version ${\cal L}_4(z)$ of the classical 4-logarithm  at   some 
numbers in  $\sigma_j(\F)$. These linear combinations are subject to a specific condition discovered by Zagier.

For an infinite field $\F$, we define a  map  of 
  $K$-groups $K_{8-i}(\F)$, $i=1, \dots ,4$,  to the $i$-th cohomology  of the 
weight 4 polylogarithmic motivic complex ${\cal B}^\bullet(\F, 4)$. When $\F$ is the function  field  of    a  complex variety, 
composing the map with the regulator map on the polylogarithmic complex to the Deligne cohomology, we get a $\Q^\times$-multiple of Beilinson's regulator. 
This implies that the composition 
$K_7(\C) \to H^1{\cal B}^\bullet(\C, 4) \to \R$, where the second map is given by   ${\cal L}_4(z)$,  
is   a $\Q^\times$-multiple of   Borel's regulator. This plus Borel's theorem  on the ranks of algebraic $K$-groups of number fields implies Zagier's conjecture. 


We get  a strong evidence  for  the  part of Freeness Conjecture
 describing the weight four part ${\cal L}_4(\F)$ of the motivic Lie coalgebra  
of   $\F$ via  higher Bloch groups ${\cal B}_j(\F)$ as an extension:  
\[
0 \lra {\cal B}_4(\F) \lra {\cal L}_4(\F)\lra \Lambda^2{\cal B}_2(\F)\lra 0.
\]

The main tools are  motivic correlators  and  a new link  of  cluster varieties to polylogarithms. 
\end{abstract}

\section{Introduction and the architecture of the proof}

\subsection{The classical polylogarithms and algebraic $K$-theory} \la{SS1.1}
\paragraph{1. The classical $n$-logarithm.} The classical $n$-logarithm function  ${\rm Li}_n(z)$ on the unit disc $|z|<1$ is given by the  absolutely convergent power series: 
$$
{\rm Li}_n(z) = \sum_{k=1}^\infty\frac{z^k}{k^n}, ~~~~|z|<1.
$$
It is  continued analytically  to a multivalued analytic function on  $\C - \{0, 1\}$  by induction, setting
$$
{\rm Li}_n(z) = \int_0^z{\rm Li}_{n-1}(z)d\log z, ~~~~n >1.
$$
Here  we integrate over a path  in $\C-\{0,1\}$ from $0$ to $z$.  
The obtained  multivalued analytic function  has a single valued cousin  \cite{Zag90}.  
 Namely, consider the projection given by
\be \la{pin}
\pi_n: \C \lra  (2\pi i)^{n-1}\R, ~~~~z\in \C \lms \begin{cases} 
{\rm Re}(z)& n=2k+1\\
 {\rm Im}(z) & n =2k. 
\end{cases}.
\ee
Then  one   shows that the following expression  is a single-valued function on ${\Bbb C}{\Bbb P}^1-\{0, 1, \infty\}$:
$$
{\cal L}_n(z):= \pi_n\left(\sum_{k=0}^{n-1}\frac{2^kB_k}{k!} {\rm Li}_{n-k}(z) \log^k |z|\right), ~~n>1.
$$ 
Here  $B_k$ are the  Bernoulli numbers, and ${\rm Li}_{n-k}(z)$ are defined via the same integration path. For example,    ${\cal L}_2(z)$ is the Bloch-Wigner dilogarithm.  
The function ${\cal L}_n(z)$  is continuous on $\C{\Bbb P}^1$. For a  
 Hodge-theoretic interpretation of the  functions ${\cal L}_n(z)$ see   \cite{BD94}. 

\paragraph{2. Zagier's conjecture.}  Let   $\zeta_\F(s)$ be the Dedekind $\zeta$-function  of  a number field $\F$. Zagier's conjecture  \cite{Zag90} predicts that 
the classical regulator formula 
$$
{\rm Res}_{s=1}\zeta_\F(s) = \frac{2^{r_1+r_2}\pi^{r_2}R_\F h_\F}{w_\F\sqrt{|d_\F|}}
$$
for the residue of the Dedekind $\zeta$-function of a number field $\F$ at $s=1$ has analogs   for $\zeta_\F(n)$ for any positive integer $n$.
Theorem \ref{ZCZ4t} proves Zagier's conjecture for $n=4$.
Let us state it first  in a preliminary form. 
Denote  by $\Q[X]$ the $\Q$-vector space generated by a set $X$. 

\bt \la{ZCZ4ta}
Let $\F$ be a number field,  $[\F:\Q]=r_1+2r_2$, and  the set $\{\sigma_j\}$ of all  embeddings $\F \to \C$ is numbered so that 
$\overline \sigma_{r_1+i} = \sigma_{r_1+r_2+i}$. Let $d_\F$ be the discriminant of $\F$. Then there exist elements
$y_1, \dots , y_{r_2} \in \Q[\F]$  such that 
\be \la{ZCZ4a}
\zeta_\F(4) = \pi^{4(r_1+r_2)}|d_\F|^{-1/2}\cdot {\rm det} \Bigl({\cal L}_4(\sigma_{r_1+i}(y_j))\Bigr), ~~~~1\leq i,j \leq r_2.
\ee
\et


Similar results about $\zeta_\F(2)$ and $\zeta_\F(3)$ were proved in \cite{Zag86} and \cite{Gon91}, \cite{Gon95}, respectively. Our next goal is to formulate Theorem \ref{ZCZ4t}, which is a more precise version of Theorem \ref{ZCZ4ta}.



 \paragraph{3. Functional equations for polylogarithms.} 
It has been known for a long time that polylogarithms satisfy various algebraic functional equations, i.e., polynomial relations with rational coefficients between expressions ${\rm Li}_n(P)$ for $P\in \mathbb{Q}(z_1,\dots,z_m)$ which hold in some open ball. One of the reasons for appearance of the modified functions ${\cal L}_n(z)$ is that the functional relations satisfied by them   are ``clean,'' that is,  no products of polylogarithms of smaller weights are involved.  For example, the dilogarithm  ${\cal L}_2(z)$ satisfies   the five-term relation. Namely,  recall the  cross-ratio  of four points on ${\Bbb P}^1(\F)$:


%

\be \la{CR1}
[x_1, x_2, x_3, x_4] := \frac{(x_1- x_2) (x_3- x_4) }{(x_1- x_4) (x_3- x_2) }, ~~~~[\infty, 1, 0, x] = x.
\ee
Then for any five distinct points $x_1, \ldots, x_5$ on $\C{\Bbb P}^1$ we have:\footnote{Note that $[x_2, x_3, x_4, x_1] = [x_1, x_2, x_3, x_4]^{-1}$. So using ${\cal L}_2(z) = - {\cal L}_2(z^{-1})$, relation (\ref{AFED}) can be expressed in the usual form $\sum_{i=1}^5(-1)^i{\cal L}_2([x_1, \ldots , \widehat x_i, \ldots, x_5]) =0$. }
\be \la{AFED}
\sum_{i=1}^5{\cal L}_2([x_i, x_{i+1}, x_{i+2},  x_{i+3}])=0, ~~~~i \in \Z/5\Z.
\ee
From \cite[Corollary 5.6]{Sus90} one can deduce that any linear relation with rational coefficients between functions ${\cal L}_2(P)$ for $P\in \mathbb{Q}(z_1,\dots,z_m)$ follows from the five-term relation.


Although we do not know explicitly functional equations for $n$-logarithms for large $n$ except a trivial one ${\cal L}_n(z) + (-1)^n{\cal L}_n(z^{-1}) =0$ and the distribution relations,
one can define a certain group of functional equations of the type above.  This leads to higher Bloch groups, whose definition we  recall now.

\paragraph{Conventions.} A few remarks about conventions are in order. First, we work everywhere with infinite fields.  Second, we work modulo torsion, i.e., with $\Q$-vector spaces. In particular, our definition of ${\cal B}_n(\F)$ differs from the one in \cite{Gon95}: our group  ${\cal B}_n(\F)$ is  the rationalization of the corresponding abelian group defined in \cite{Gon95}. For an abelian group $A$, consider its rationalization $A:= A \otimes \Q$. We mostly omit the rationalization sign: for instance, we denote the space $\F^{\times}\otimes_{\Z} \Q$ simply by $\F^{\times}.$ 

Finally, the name {\it  Bloch group} was coined by  Suslin \cite{Sus90} for the kernel of the map  $\delta: \mathfrak{p}(\F) \lra \Lambda^2\F^\times,$ where  $\frak{p}(\F)$ is an abelian group whose rationalization is isomorphic to ${\cal B}_2(\F)$. We use the terminology of \cite{Gon95}, where the group ${\cal B}_2(\F)$ was called the Bloch group, and its higher analogs ${\cal B}_n(\F)$ were defined and called ``higher Bloch groups.'' Our point is that the groups ${\cal B}_n(\F)$ are the primary objects, and thus deserve the name ``higher Bloch groups.''

\paragraph{4. Polylogarithmic motivic complexes \cite{Gon95}.} 
For a set $X$ we denote by $\{x\}$ the generator  of $\Q[X]$  assigned to an   $x\in X$.  
Given an infinite field $\F$, one defines  inductively for each $n\geq 1$ a subspace ${\cal R}_n(\F) \subset \Q[\F]$ reflecting  functional equations for the classical n-logarithm 
function,  and define {\it the higher Bloch group}
$$
{\cal B}_n(\F):= \frac{\Q[\F]}{{\cal R}_n(\F)}.
$$
We denote by $\{z\}_n$ the projection of the generator $\{z\}$ to the quotient ${\cal B}_n(\F)$.

The subgroup ${\cal R}_1(\F)$ is generated by the elements $\{xy\}-\{x\}-\{y\}$ where $x,y\in \F^\times$, and $\{0\}$. So $\Q[\F]/{\cal R}_1(\F) = \F^\times_\Q$, which we denote simply by $\F^\times.$ We define by  induction  a map 
$$
\Q[\F] \stackrel{\delta_n}{\lra} \begin{cases} 
{\cal B}_{n-1}(\F)\otimes \F^\times & n>2,\\
 \Lambda^2 \F^\times& n =2,
\end{cases}
 ~~~~\{x\} \lms \begin{cases} 
 \{x\}_{n-1}\otimes x & n>2,\\
 (1-x)\wedge x & n =2, 
\end{cases}~~\delta_2\{1\}=\delta_2\{0\}=0.
$$
It is handy to add a generator $\{\infty\}$ together with the relation $\{\infty\}=0$.

 Let us define a subspace ${\cal R}_n(\F) \subset {\rm Ker}~\delta_n$. 
Any expression $\sum_in_i\{f_i(t)\}$ which lies in the kernel of $\delta_n$ for the field $\F(t)$ gives rise to  
an element $\sum_in_i\bigl(\{f_i(1)\} - \{f_i(0)\}\bigr)$.   The subgroup 
${\cal R}_n(\F)$ is generated by all elements obtained this way,   and $\{0\}$. 

  One proves \cite[Theorem 1.5]{Gon94a} that 
there is a map of abelian groups
$$
{\cal L}_n: {\cal B}_n(\C) \lra \R, ~~~~\{z\}_n \lms {\cal L}_n(z), ~~n>1. 
$$
For  $n=1$ we have   a map 
${\cal B}_1(\C)\to \R, \{z\}_1 \lms \log|z|$. This means that the subgroup ${\cal R}_n(\C)$ is indeed a subgroup of functional equations for the polylogarithm function 
${\cal L}_n(z)$. One can show that it contains all functional equations  of the type mentioned above.

One proves that  the  map $\delta_n$ induces    a group homomorphism
\be \la{BB}
\delta_n: {\cal B}_n(\F) \lra \begin{cases} 
{\cal B}_{n-1}(\F)\otimes \F^\times& \text{for}\ n>2,\\
\Lambda^2 \F & \text{for}\ n =2
\end{cases}. 
\ee
Thus we get a complex
\be \la{BCOM}
{\cal B}^\bullet(\F;n) : ~~~~{\cal B}_n(\F) \stackrel{}{\to} {\cal B}_{n-1}(\F)\otimes \F^\times \stackrel{}{\to}  {\cal B}_{n-2}(\F)\otimes \Lambda^2 \F^\times \to \ldots \to  {\cal B}_{2}(\F)\otimes \Lambda^{n-2} \F^\times \to \Lambda^{n} \F^\times
\ee
in the degrees $[1,n]$, where the differential has degree $+1$ and sends $\{x\}_{k}\wedge x_1 \wedge \dots \wedge x_{n-k}$ to  $\delta_k(\{x\}_{k})\wedge x_1 \wedge \dots \wedge x_k.$
This complex is called the {\it weight $n$ polylogarithmic motivic complex}; the first four polylogarithmic motivic   complexes are
\[
\begin{split}
&{\cal B}^\bullet(\F;1): ~~~~{\cal B}_1(\F) = \F^\times, \\
&{\cal B}^\bullet(\F;2): ~~~~{\cal B}_2(\F) \stackrel{}{\lra}   \Lambda^2 \F^\times, \\
&{\cal B}^\bullet(\F;3): ~~~~{\cal B}_3(\F) \stackrel{}{\lra} {\cal B}_{2}(\F)\otimes \F^\times \stackrel{}{\lra}    \Lambda^{3} \F^\times, \\
&{\cal B}^\bullet(\F;4): ~~~~{\cal B}_4(\F) \stackrel{}{\lra} {\cal B}_{3}(\F)\otimes \F^\times\stackrel{}{\lra}  {\cal B}_{2}(\F)\otimes \Lambda^2 \F^\times \lra  \Lambda^{4} \F^\times. 
\end{split}
\]
Complexes ${\cal B}^\bullet(\F;n)$ are rationalizations of complexes $\Gamma_\F(n)$ defined in \cite{Gon95}.


\paragraph{6. Main results.} 

Let us state now   our first main result precisely. 

  \bt \la{ZCZ4t}
Let $\F$ be a number field,  $[\F:\Q]=r_1+2r_2$, and  the set $\{\sigma_j\}$ of all  embeddings $\F \to \C$ is numbered so that 
$\overline \sigma_{r_1+i} = \sigma_{r_1+r_2+i}$. Let $d_\F$ be the discriminant of $\F$. Then there exist elements $y_1, \ldots, y_{r_2} \in {\rm Ker} \delta_4 \subset {\cal B}_4(\F)$  such that 
\be \la{ZCZ4}
\zeta_\F(4) = \pi^{4(r_1+r_2)}|d_\F|^{-1/2}\cdot {\rm det} \Bigl({\cal L}_4(\sigma_{r_1+i}(y_j))\Bigr), ~~~~1\leq i,j \leq r_2.
\ee
For any    $y_1, \ldots, y_{r_2} \in {\rm Ker}\delta_4$, the right-hand side of  (\ref{ZCZ4}) equals $q\cdot\zeta_\F(4)$ for some $q\in \Q$. 
\et

Zagier's conjecture concerns number fields. For an arbitrary infinite field $\F$, it was conjectured in \cite{Gon95} that the weight $n$ polylogarithmic motivic complexes calculate the weight $n$ pieces of the Quillen's $K$-groups of the field $\F$ modulo torsion. 
Precisely, let  $\gamma$ be the Adams $\gamma$-filtration on Quillen's algebraic $K$-theory. The  conjecture states   that one expects the following isomorphisms: 
\[
{\rm gr}_\gamma^nK_{2n-i}(\F)_\Q \stackrel{?}{=} H^i{\cal B}^\bullet(\F;n), ~~~~i>0.
\]

 This conjecture has a variant that is formulated more elementary, 
without reference to Quillen's definition of algebraic $K$-theory and the $\gamma$-filtration, which we recall now. 

Denote by $\mathrm{GL}$ the infinite general linear group, defined as the inductive limit of the groups ${\rm GL}_m$, sitting one in the other in a natural way. There is a canonical map 
\be \la{HUR}
K_n(\F) \lra \pi_n({\rm BGL}(\F)^+) \xrightarrow{\rm Hurewicz}  H_n({\rm BGL}(\F)^+, \Z) = H_n({\rm GL}(\F), \Z). 
\ee
Here $+$ stands for the Quillen plus construction and ${\rm B}$ for the classifying space. 

Let $G$ be any group, and $\Delta: G  \to  G \times G$   the diagonal map. Then  the primitive part of the $n$th rational homology of any group $G$ is defined by
 $$
{\rm Prim}H_n(G, \Q):= \{X \in H_n(G, \Q)~|~ \Delta_*(X) = X \otimes 1 + 1 \otimes X\}.
$$ 
It is well known that  the map (\ref{HUR}) induces an isomorphism
\be \la{RANK}
K_n(\F)_\Q \stackrel{\sim}{\lra} {\rm Prim}H_n(\textup{GL}(\F), \Q).
\ee

 The natural filtration of the group ${\rm GL}(\F)$ by the subgroups ${\rm GL}_m(\F)$ induces an increasing  filtration on (\ref{RANK}), known as the {\it rank filtration}:
 $$
 {\cal F}_m^{\rm rk}K_n(\F)_\Q:= {\rm Im}\Bigl({\rm Prim}H_n(\textup{GL}_m(\F), \Q) \lra {\rm Prim}H_n({\rm GL}(\F), \Q)\Bigr).
  $$
  The stabilization theorem of Suslin \cite{Sus84} implies that 
  $$
  {\cal F}_n^{\rm rk}K_n(\F)_\Q =   K_n(\F)_\Q.
  $$
Suslin  conjectured cf. \cite[Conjecture 3.2]{Gon94a} that for an infinite field $\F$  the rank and the Adams filtrations have isomorphic associate graded pieces: 
 $$
 {\rm gr}^{\rm rk}_{i}K_{n}(\F)_\Q  \stackrel{?}{\cong } {\rm gr}_{\gamma}^{i}K_{n}(\F)_\Q.
 $$
 
 Here is the main result of this paper. Denote by $\R_{\cal D}(n)$ the weight $n$ real Deligne complex \cite{Bei84}. Recall the canonical map of complexes \cite{Gon00a}
 \be \la{RM}
 H^i{\cal B}^\bullet(\C(X);4) \lra H^i({\rm Spec}(\C(X)), \R_{\cal D}(4)).
  \ee

\bt \la{ZCZ4t2}\begin{enumerate}[label=(\roman*)]
\item Let $\F$ be an infinite field.  Then there are canonical homomorphisms
$$
K_{8-i}(\F)_\Q \lra H^i{\cal B}^\bullet(\F;4), ~~~~i=1,2,3,4.
$$
Their  restrictions   to ${\cal F}_3^{\rm rk}K_{8-i}(\F)_\Q$ are zero.

\item For any complex variety $X$, the  composition  
\be \la{BRM}
K_{8-i}(\C(X))_\Q \lra H^i{\cal B}^\bullet(\C(X);4) \stackrel{(\ref{RM})}{\lra} H^i({\rm Spec}(\C(X)), \R_{\cal D}(4))
\ee
is a non-zero rational multiple of Beilinson's regulator map \cite{Bei84}.

\item For any regular   curve  $X$ over a number field,  the map (\ref{BRM}) gives rise to a non-zero rational multiple of Beilinson's regulator map 
$$
K_{8-i}(X)_\Q \lra  H^i(X, \R_{\cal D}(4)).
$$

\item The following composition is a non-zero rational multiple of the Borel regulator map \cite{Bor77}:
 $$
K_{7}(\C)_\Q \lra H^1{\cal B}^\bullet(\C;4)  \stackrel{{\cal L}_4}{\lra} \R.
$$
\end{enumerate}
\et

The proof of Theorem \ref{ZCZ4t2} is completed in Section \ref{Sec8.3}. 

The key ingredients of the proof of Theorem \ref{ZCZ4t2} 
are Theorem \ref{THRN}, whose proof is finished in Section \ref{SSEECC10}, Theorem \ref{MapBtoP}, whose proof is completed  in Section \ref{Sec4}, and Theorem \ref{TH1.7}.

Theorem \ref{ZCZ4t} follows  from the part (iv) of Theorem \ref{ZCZ4t2} and Borel's theorem \cite{Bor77}. The part (i) of Theorem \ref{ZCZ4t2} implies that we get maps 
\be \la{QUIL1}
{\rm gr}^{\rm rk}_{4}
K_{8-i}(\F)_\Q \stackrel{}{\lra} H^i{\cal B}^\bullet(\F;4), ~~~~i=1,2,3,4.
\ee
The map (\ref{QUIL1}) for $i=4$ is an isomorphism due to a theorem of Suslin \cite{Sus84} relating  Quillen's and   Milnor's $K$-groups.

\bcon
The maps (\ref{QUIL1}) are  isomorphisms.
\econ
  
  Theorem \ref{ZCZ4t2}  gives a strong evidence to the weight $4$ parts of several ``beyond-the-standard'' conjectures about mixed Tate motives, which we are going to discuss now. 
  

 \paragraph{7. The motivic Tate Lie algebra of a field $\F$.} According to Beilinson, one should have a  Tannakian  category  ${\cal M}(\F)$ of mixed motives over a field $\F$. Let ${\cal M}_T(\F)$ be its minimal exact subcategory containing pure Tate motives $\mathbb{Q}(n), n\in \mathbb{Z}.$ Beilinson conjectured the following.
 
 \bcon \la{MTM=DM} \cite{Bei87} One should have
$$
{\rm Ext}^\bullet_{{\cal M}_T(\F)}(\Q(0), \Q(n)) = {\rm Ext}^\bullet_{{\cal M}(\F)}(\Q(0), \Q(n)).
 $$
\econ

The abelian category of mixed Tate motives   ${\cal M}_T(\F)$ is available  with all the expected properties when $\F$ is a number field \cite{Lev93}, \cite{DG03}.  

Let us assume that the category ${\cal M}_T(\F)$ of mixed Tate motives over a field $\F$ does exist. 
Then every object of this category carries a canonical increasing weight filtration $W_\bullet$, and  there is a canonical fiber functor to the category of  $\Q$-vector spaces:
\be \la{FF1}
\omega: {\cal M}_T(\F) \lra {\rm Vect}_\Q, ~~~~\omega(M):= \oplus_{n \in \Z}{\rm Hom}(\Q(-n), {\rm gr}^W_{2n}M).
\ee  
The Lie algebra ${\rm L}_\bullet(\F):= {\rm Der}^{\otimes}(\omega)$ of derivations with respect to the tensor product in ${\cal M}_T(\F)$  of the functor $\omega$, of degree $<0$,  
is a  Lie algebra in the category of projective limits of  $\Q$-vector spaces, graded by negative integers.  
It is called the motivic Tate Lie algebra. 

Recall the   Chevalley-Eilenberg  cochain complex of a Lie coalgebra ${\cal L}$:
\be \la{CEil}
 {\rm CE}^\ast({\cal L}): =~~~~ {\cal L} \stackrel{\delta}{\lra} \Lambda^2{\cal L}    \stackrel{}{\lra} \Lambda^3{\cal L} \lra \ldots 
\ee
The first map is the Lie cobracket. We extend it to a  differential   on ${\rm Sym}^\ast({\cal L} [-1]),$  the symmetric algebra of ${\cal L}$ with shifted grading,  by the Leibniz rule. 
If the Lie coalgebra ${\cal L}$   is graded, we get a complex of graded vector spaces. 

Conjecture \ref{MTM=DM} implies the isomorphism 
\[
{\rm Ext}^i_{{\cal M}(\F)}(\Q(0), \Q(n))\cong \left[H^i({\rm L}_\bullet(\F),\Q)\right]_n
\]
where the group on the right is the degree $n$ part of the cohomology group of the Lie 
algebra ${\rm L}_\bullet(\F)$. Let  ${\cal L}_\bullet(\F)$ be the graded dual to the Lie algebra ${\rm L}_\bullet(\F)$. It is  a  Lie coalgebra  called the motivic Tate Lie coalgebra of $\F$. 
The weight $4$ part of  Conjecture \ref{MTM=DM} implies that the cohomology groups of the complex 
\be \la{MTM=DM1}
{\cal L}_4(\F) \lra \bigl({\cal L}_3(\F)\otimes  {\cal L}_1(\F)\bigr) \oplus \Lambda^2 {\cal L}_2(\F) \lra {\cal L}_2(\F)\otimes \Lambda^2{\cal L}_1(\F) \lra \Lambda^4{\cal L}_1(\F)
\ee
are isomorphic to ${\rm Ext}^i_{{\cal M}(\F)}(\Q(0), \Q(4)).$
 
The space ${\cal B}_n(\F)$ is a vector space. Denote by ${\cal B}_n(\F)^\vee$ its linear dual. 
Conjecture \ref{FRC} below  is a part of  the Freeness Conjecture  \cite[Conjecture 1.20]{Gon94a}. It describes the relation between 
the classical polylogarithms and all mixed Tate motives. Consider the graded ideal of the Lie algebra ${\rm L}_{\bullet}(\F)$:
$$
{\rm I}_{\bullet}(\F) = \oplus_{n>1}{\rm L}_{-n}(\F).
$$
Denote by $H^i_{(n)}{\rm I}_{\bullet}(\F)$ the degree $n$ part of the  $i-$th cohomology of the graded Lie algebra ${\rm I}_{\bullet}(\F)$. 
\bcon \la{FRC} Assume the existence of the category of mixed Tate motives over the field $F$.
The ideal ${\rm I}_{\bullet}(\F) $   has  the cohomology   given by 
$$
H^1_{(n)}{\rm I}_{\bullet}(\F) = {\cal B}_n(\F), ~~~~H^{i}_{(n)}{\rm I}_{\bullet}(\F)= 0, ~~\forall i,n>1.
 $$
So the graded Lie algebra ${\rm I}_{\bullet}(\F)$ is isomorphic to the free graded Lie algebra with the  generators   in the degree $-n$,  $n\geq 2$,  
given by the spaces ${\cal B}_n(\F)^\vee$. However, this isomorphism is non-canonical, that is there is no canonical, that is functorial in $\F$, map realizing it. 
\econ

 Here is a  sharper version of  Freeness Conjecture \ref{FRC}.  
Observe that the lower central series in the  Lie algebra ${\rm I}_{\bullet}(\F)$  gives rise to an increasing filtration ${\cal C}_\ast$ of the 
  Lie coalgebra 
 ${\cal L}_{\geq 2}(\F)$, indexed by positive integers.  
 
   \bcon \la{FRCS}  Assume the existence of the category of mixed Tate motives over the field $F$. The associate graded  Lie coalgebra 
 ${\rm gr}^{\cal C}_\ast{\L}_{\geq 2}(\F)$  
  is canonically isomorphic to the cofree Lie coalgebra cogenerated by  the $\Q-$vector spaces 
 ${\cal B}_m(\F)$, $m \geq 2$,   which have  the ${\cal C}-$degree $1$ and the weight $m$. 
 \econ

Conjectures  \ref{MTM=DM}, \ref{FRC} and \ref{FRCS} predicts 
 the following canonical isomorphisms:
  \be \la{MTLC1}
{\cal L}_1(\F) = \F^\times, ~~~~{\cal L}_2(\F) = {\cal B}_2(\F), ~~~~{\cal L}_3(\F) = {\cal B}_3(\F).
\ee
 Furthermore, it predicts that the  weight $4$ component ${\cal L}_4(\F)$    
is described by an extension: 
\be \la{MTLC2}
\begin{split}
&0 \lra {\cal B}_4(\F) \lra {\cal L}_4(\F)\stackrel{p}{\lra} \Lambda^2{\cal B}_2(\F)\lra 0.\\
\end{split}
\ee
The map $p$ is identified with the component  ${\cal L}_4(\F) \to \Lambda^2{\cal L}_2(\F)$ of the cobracket via the conjectural isomorphism ${\cal L}_2(\F) = {\cal B}_2(\F)$. 
Therefore  complex (\ref{MTM=DM1}) looks now as follows: 

\be \la{MTM=DM12}
{\cal L}_4(\F) \lra {\cal B}_3(\F)\otimes \F^\times \bigoplus  \Lambda^2 {\cal B}_2(\F) \lra {\cal B}_2(\F)\otimes \Lambda^2\F^\times \lra \Lambda^4\F^\times.
\ee

  Extension (\ref{MTLC2}) does not have a    functorial in $\F$ splitting  $s\colon \Lambda^2{\cal B}_2(\F) \to {\cal L}_4(\F)$. Indeed,  
  composing the map $s$  with the ${\cal B}_3(\F)\otimes \F^\times$-component of the cobracket, we get  a map  
$$
\Lambda^2{\cal B}_2(\F) \lra {\cal B}_3(\F)\otimes \F^\times.
$$
However, there is no such a map  given  by rational functions on  generators \cite[Theorem~4.7]{Gon94a}. 
Thus Conjecture \ref{FRCS} does not   describe the  functor  $ \F \lra {\L}_\bullet(\F)$, since the Lie coalgebra $ {\L}_{\geq 2}(\F)$ does not have canonical,   depending functorially on   $\F$, cogenerators.  In the next Section, we will define explicitly a functor   $\F \lra \mathbb{L}_n(\F)$ for $n\leq 4$ and prove for it a version of (\ref{MTLC2}). Generalizing this construction to  $n>4$ is an important open problem.

 



\subsection{Functional equations for polylogarithms and motivic Tate Lie coalgebras}  \label{SectionFunctionalEquations}
We start with recalling the construction of the higher Bloch groups $\B_2(\F)$ and $\B_3(\F)$ from \cite{Gon95}. We will work with the  $\mathbb{Q}$-vector spaces $\B_2(\F)\otimes_\mathbb{Z}\mathbb{Q}$ and $\B_3(\F)\otimes_\mathbb{Z}\mathbb{Q}$ but keep the original notation.

Consider a subspace ${\rm R}_2(\F)$  of $\mathbb{Q}[\F]$ generated by elements 
\[
\sum_{i=1}^5(-1)^i \{[x_1, \ldots , \widehat x_i, \ldots,  x_5]\}
\]
for distinct points $x_1,\dots, x_5\in \mathbb{P}^1_\F.$ We put  ${\B}_2(\F)=\dfrac{\mathbb{Q}[\F]}{{\rm R}_2(\F)}.$ A deep result of Suslin (\cite[Corollary 5.6]{Sus90}) implies that the natural map ${\B}_2(\F)\lra \mathcal{B}_2(\F)$ is an isomorphism.\footnote{The map  $\B_2(\F) \to {\cal B}_2(\F)$, induced by the identity map $\{x\}\to \{x\}$ on the generators, gives rise to a map from the complex $\B_2(\F) \lra \Lambda^2\F^\times$ to the complex ${\cal B}_2(\F)\lra \Lambda^2\F^\times$,  identical on  $\Lambda^2\F^\times$. \cite[Corollary 5.6]{Sus90} implies that the first homology  of these complexes coincide, so it is a quasiisomorphism.  Indeed,
 the former first homology  is the group $\B(\F)$ in loc. cit.,  the latter is  $\B(\F(t))$, and $\B(\F) = \B(\F(t))$ by loc. cit.  The claim follows. }

Let us denote by   ${\rm R}_3(\F)$  the subspace  of $\mathbb{Q}[\F]$ generated by  the elements 
\be \la{EATRI}
\begin{split}
\left \{x\right\}_{3}-\left \{x^{-1}\right\}_{3}, ~~~~\left \{x\right\}_{3}+\left \{1-x\right\}_{3}+\left \{1-x^{-1}\right\}_{3}- \{1\}_3 \ \ \ \ \ \ \forall x\in \F^\times\\
\end{split}
\ee 
and the 22-term relation \cite[Formula 1.10, Theorem 1.3]{Gon95} for the trilogarithm. Let us set 
$$
{\B}_3(\F)=\dfrac{\mathbb{Q}[\F]}{{\rm R}_3(\F)}.
$$
 It was conjectured in \cite{Gon95} that the natural map  ${\B}_3(\F)\lra {\cal B}_3(\F)$ is an isomorphism.  
 \vskip 1mm
 
 We use both  groups ${\B}_k(\F)$ and ${\cal B}_k(\F)$, for $k=2,3$.   The ${\cal B}-$groups are handy in general considerations. The groups $\B_k(\F)$  appear naturally in 
 explicit constructions  in Sections \ref{SEC2N}-\ref{SEC2}, \ref{SEC8.4n}, often combined with 
 the  corollary (\ref{7term})  of the 22-term relation. Let us    elaborate it. 
 \vskip 1mm
 
Let $V_3$ be a three-dimensional $\F-$vector space. Given a generic configuration $(x_1, \ldots, x_6)$ of 6 points  in ${\Bbb P}^2(\F)$, we lift it to a configuration of  vectors $(l_1, \ldots, l_6)$ in $V_3$, 
pick a volume form   $\omega$ in  $V_3$, and consider following the ``triple ratio'", where   ${\rm Alt}_6$ is the  antisymmetrization:
\be \la{TR}
r_3(x_1, x_2, x_3, x_4, x_5, x_6):=   {\rm Alt}_6\left\{\frac{\omega(l_1,l_2,l_4)\omega(l_2,l_3,l_5)\omega(l_1,l_3,l_6)}{\omega(l_1,l_2,l_5)\omega(l_2,l_3,l_6)\omega(l_1,l_3,l_4)}\right\} \in \Q[\F].
\ee
 The right-hand side is independent of the choices. 

\bt  \la{R37term} For any generic configuration $(x_1, \ldots, x_7)$ of 7 points  in ${\Bbb P}^2(\F)$ we have  
\be \la{7term}
\sum_{i=1}^7(-1)^i r_3(x_1, \ldots , \widehat x_i, \ldots,  x_7) \in {\rm R}_3(\F).
\ee
\et

Theorem \ref{R37term} is deduced  in  Section \ref{intermezzo} from  Theorem A  in \cite{Gon95}.  
So for generic configurations $(x_1, \ldots , x_7)$ of   points in ${\Bbb P}^2(\C)$ we get a   functional relation for the trilogarithm \cite{Gon94}:
\be \la{TRIF}
\sum_{i=1}^7(-1)^i {\cal L}_3(r_3(x_1, \ldots , \widehat x_i, \ldots,  x_7))=0.
\ee

 

The 22-term relation for the trilogarithm and relation (\ref{TRIF})  played a central role in connecting  the
   trilogarithmic motivic complex to the algebraic $K$-theory for an arbitrary infinite field  $\F$, and proving Zagier's conjecture for $\zeta_\F(3)$.  
Yet it is unclear how to generalize  relations  (\ref{AFED})  and (\ref{TRIF}) to the classical n-logarithms to implement a similar strategy for $\zeta_\F(n)$ for $n>3$. 

\vskip 2mm

Our strategy is different.  Given a field $\F$, 
we look for an explicit construction of the weight $\leq n$ part   of the whole motivic Tate Lie coalgebra ${\cal L}_\bullet(\F)$ rather than 
just ${\cal B}_n(\F)$.  The  Freeness Conjecture predicts that for $n\leq 3$ we should have ${\cal L}_n(\F)={\cal B}_n(\F)$, so the new phenomena start to appear at the weight four.
\vskip 2mm
 
 The  existence of the  Lie coalgebra ${\cal L}_\bullet(\F)$ for a general field $\F$ is   unknown yet. So  
 we  aim   at a functorial in $\F$   explicit construction 
 of a Lie coalgebra $\mathbb{L}_\bullet(\F)$, which is conjecturally   isomorphic to    ${\cal L}_\bullet(\F)$, and its subspace $\B_\bullet(\F)$, which is conjecturally   isomorphic to    ${\cal B}_\bullet(\F).$
We define   a Lie coalgebra   $\mathbb{L}_{\leq 4}(\F)$   by using   simple and uniform in $n$ 
  relations ${\bf Q}_n$ between   weight $n$  iterated integrals,  $n\leq 4$. 
  The relation  ${\bf Q}_2$ is just the Abel five-term relation.  The  relation   ${\bf Q}_3$ not only implies the 22-term relation, and hence the 
 relation (\ref{7term})  for the   trilogarithm, but also  allows   to express any weight $3$ iterated integral  via the classical trilogarithm. So it is a new way to present the trilogarithm story. The  relation   ${\bf Q}_4$ is new.

 \paragraph{1. Constructing the Lie coalgebra $\mathbb{L}_{\leq 4}(\F)$.} Recall the   moduli space   $ {\mathcal{M}_{0,n}}$ parametrizing configurations of 
distinct points $(x_1, \ldots, x_n)$ on ${\Bbb P}^1$ modulo the  diagonal action of the group ${\rm PGL}_2$. 
 Recall the cross-ratio $[x_1, x_2, x_3, x_4]$, see (\ref{CR1}). Consider the following  regular function  on ${\mathcal{M}_{0,6}}$:\footnote{Mind the minus sign.}
\be \la{CR0}
[x_1,x_2, \ldots, x_{6}]:=-\frac{(x_1-x_2)(x_3-x_4)  (x_{5}-x_{6})}{(x_2-x_3)(x_4-x_5)  (x_{6}-x_{1})}.
\ee  
We use the cyclic   summation notation, where the indices are modulo $n$:
\be \la{CYC}
\begin{split}
&{\rm Cyc}_n \varphi(x_1, \ldots , x_n):= \sum_{k\in \Z/n\Z}  \varphi(x_{k+1}, x_{k+2}, \ldots , x_{k+n}). \\
\end{split}
\ee

Before we proceed with the definitions, let us make some comments applicable to  all of them. 
In Definitions \ref{DEFL2a}-\ref{DEF4La} we assume that  $\F$ is an infinite field. Lie coalgebra $\mathbb{L}_{\leq 4}(\F)$ is defined by generators and relations. Generators are symbols 
\be \la{symb}
\begin{split}
&x\in \mathbb{L}_{1}(\F)\cong \F^\times, \ \ \{x\}_2\in \mathbb{L}_{2}(\F),\ \  \{x\}_3 \in \mathbb{L}_{3}(\F),\ \  \{x\}_4 \in \mathbb{L}_{4}(\F) \text{\ \ for } x \in \F^{\times};\\
&\{x,y\}_{2,1}\in \mathbb{L}_{3}(\F),\ \  \{x,y\}_{3,1}\in \mathbb{L}_{4}(\F) \text{\ \  for } x, y \in \F^{\times}.\\		
\end{split}
\ee
Here elements $\{x\}_n$ are related to the corresponding elements of ${\cal B}_n(\F)$ which explains our notation.
 
 Relations consist of a ``generic'' relation  ${\bf Q}_n$ and  all its specializations. Precisely, consider any curve $\alpha(t): {\Bbb A}^1 - \{0\} \lra {\cal M}_{0,n+3}$, where $t$ 
 be a local parameter near $0$. Let us define the specialization ${\rm Sp}_{t = 0}({\bf Q}_n)$ of the relation ${\bf Q}_n$ at $t=0$ along the curve $\alpha(t)$. 
 Since the relation ${\bf Q}_n$ is a linear combination of   generators  (\ref{symb}),  it is sufficient 
 to define the specializations of  generators (\ref{symb}). Consider any functions $x(t)=t^{a} u(t)$ and $y(t)=t^b v(t)$  with $a,b\in \Z$ and $u(0), \ v(0)\neq 0.$  
 Then, we define specialization in the following way:\footnote{Formulas (\ref{DegenerationRules}) are essentially nailed by the requirement that specializations commute with the cobracket (\ref{37}).} 
\be \label{DegenerationRules}
\begin{split}
&{\rm Sp}_{t = 0}(x(t))=u(0)\in \F^\times;\\
&{\rm Sp}_{t = 0}(\{x(t)\}_n)=
\begin{cases}
\{u(0)\}_n & \text{ if } a= 0, \\	
0 & \text{ if } a\neq 0 \\
\end{cases} \text{ \ \ \ for $n=2,3,4$;}\\
&{\rm Sp}_{t = 0}\{x(t), y(t)\}_{2,1}=
\begin{cases}
\{u(0), v(0)\}_{2,1} & \text{ if } a=0,\ b=0,\\
-\{v(0)\}_{3} & \text{ if } a>0,\ b=0,\\
-\{u(0)\}_{3} & \text{ if }a=0,\ b>0,\\
2\left\{\dfrac{u(0)}{v(0)}\right\}_{3} & \text{ if }  a=b< 0,\\
0 & \text{ otherwise;} 
\end{cases}\\
&{\rm Sp}_{t = 0}\{x(t), y(t)\}_{3,1}=
\begin{cases}
\{u(0), v(0)\}_{3,1} & \text{ if } a=0,\ b=0,\\
\{v(0)\}_{4} & \text{ if } a>0,\ b=0,\\
-\{u(0)\}_{4} & \text{ if }a=0,\ b>0,\\
3 \left\{\dfrac{u(0)}{v(0)}\right\}_{4}  & \text{ if }  a=b< 0,\\
0 & \text{ otherwise.} 
\end{cases}\\
\end{split}
\ee
Now, we can get new relations from ${\bf Q}_n$ by looking at various specializations.

\begin{definition} \la{DEFL2a}
The  $\Q$-vector space $\mathbb{L}_2(\F)$ is  
generated by  symbols
$\left\{x \right \}_2$ 
for $x \in \F^\times$ satisfying  the following pentagon relation:
\begin{itemize}

\item ${\bf Q}_2$: for any   configuration  $(x_1, \ldots, x_5)\in {\mathcal{M}_{0,5}}(\F)$  the following cyclic sum is zero:
\be \la{121}
\begin{split}
&
\left\{[x_1,x_2,x_3,x_4]\right \}_2+
\left\{[x_2,x_3,x_4,x_5]\right \}_2+\\
&\left\{[x_3,x_4,x_5,x_1]\right \}_2+
\left\{[x_4,x_5,x_1,x_2]\right \}_2+
\left\{[x_5,x_1,x_2,x_3]\right \}_2=0.\\
\end{split}
\ee
\end{itemize}
\end{definition}

Denote $\overline{\mathcal{M}}_{0,n}$ the Deligne-Mumford compactification of $\mathcal{M}_{0,n}$ defined in \cite{DM69}, see also \cite{Man99}. Specializing the   relation ${\bf Q_2}$  to  the boundary divisors $D_{ij}$   of 
$\overline {\mathcal{M}}_{0,5}$, where the points $x_i$ and $x_j$ collide, we get antisymmetry relations:   for any permutation $\sigma \in \mathrm{S}_4$, 
$$
\left\{[x_{\sigma(1)}, x_{\sigma(2)},x_{\sigma(3)},x_{\sigma(4)}]\right \}_2 = (-1)^{|\sigma|}\left\{[x_1,x_2,x_3,x_4]\right \}_2.
$$ 
By the very definition, we have
${\Bbb L}_2(\F) =  \B_2(\F)$. 
\begin{figure}[ht]
\centerline{\epsfbox{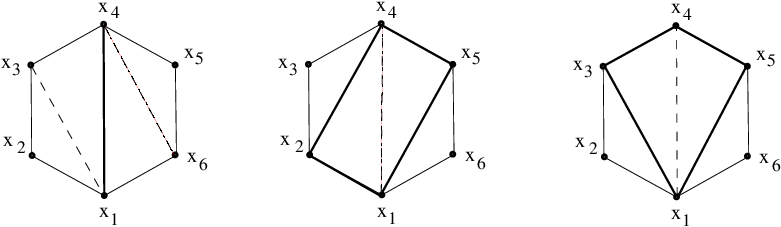}}
\caption{Geometry of   relation ${\bf Q}_3$. The diagonal $x_1x_4$   divides  the hexagon  into  quadrangles.    
 The arguments of   $\left \{[x_1, x_2,x_3,x_4],[x_4, x_5,  x_6,   x_1] \right\}_{2,1}$ are the cross-ratios assigned to  
  them, and  $[x_1, x_2,x_3,x_4, x_5, x_6]$ is  their ratio. Elements $\{[x_1,x_2,x_4,x_5]\}_3$ and $\{[x_1,x_3,x_4,x_5]\}_3$
     correspond to the  quadrangles  in the two hexagons on the right.}
\label{zc8}
\end{figure}

\begin{definition} \la{DEFQ3}
The $\Q$-vector space $\mathbb{L}_3(\F)$ is  
generated by symbols
$\left\{x \right \}_3$ where $x \in \F^{\times}$ and $\left \{x,y \right\}_{2,1}$
where $x,y \in  \F^\times$, satisfying the relation $\{x,y\}_{2,1}=\{y,x\}_{2,1}$ and the following relation: 
 \begin{itemize}
  \item ${\bf Q}_3:$ For any   $6$ distinct points $(x_1, x_2,\ldots,x_6)$ on ${\Bbb P}^1(\F)$ the following cyclic sum is zero:
 \be
 \begin{split}
&{\rm Cyc}_6\left( \left \{[x_1, x_2,x_3,x_4],[x_4, x_5,x_6,x_1]\right\}_{2,1} - \{[x_1,x_2,x_4,x_5] \}_3  + 2  \left\{[x_1,x_3,x_4,x_5]\right \}_3 \right)   \\
&  
 -4  \left\{[x_1,x_2,x_3,x_4,x_5,x_6]\right \}_3=0. \\
 \end{split}
 \ee
\end{itemize}
\end{definition}

  Specializing the relation ${\bf Q}_3$  to the divisor $D_{13}$ in $\overline {\mathcal{M}}_{0,6}$,  we get the following key relation: 
\be \la{144a}
\begin{split}
&\left \{x,y \right\}_{2,1}=
  \left\{1-x^{-1}\right \}_3+ \left\{1-y^{-1}\right \}_3+\left\{\frac{y}{x}\right \}_3 +\left\{\frac{1-y}{1-x}\right \}_3-\left\{\frac{1-y^{-1}}{1-x^{-1}}\right \}_3  - \{1\}_3.\\
\end{split}
\ee
  Relation  (\ref{144a})
 has a geometric interpretation. Namely, take five points $(\infty; 0,x,1,y)$ on ${\Bbb P}^1$, where  the last four points are ordered cyclically. Then\footnote{Note that 
$\left\{[z_1, z_2, z_3, z_4]\right \}_3$ is cyclically invariant.} 
\be \la{144ac}
\begin{split}
&\left \{x,y \right\}_{2,1}=\\
&\left\{[\infty, 0, x,1]\right \}_3+\left\{[\infty, 1,y,0]\right \}_3+\left\{[\infty, y,0, x]\right \}_3 + \left\{[\infty, x,1,y\right \}_3-\left\{[0,x,1,y]\right \}_3 -  \{1\}_3.\\
\end{split}
\ee

Substituting (\ref{144a}) to  the  relation ${\bf Q}_3$,     we get  the $22$-term relation for trilogarithm from \cite{Gon94}. The map $\{x\}_3 \lms \{x\}_3$ induces a map
\be \la{ISOBL}
  \B_3(\F) \stackrel{\sim}{\lra} {\Bbb L}_3(\F)
\ee
which is an isomorphism. The only nontrivial step in the proof  is contained in  \cite[\S 5]{Gon95}. 

Specializing the  relation $(\ref{144a})$ even further,  we obtain  relations (\ref{EATRI}).

\begin{figure}[ht]
\centerline{\epsfbox{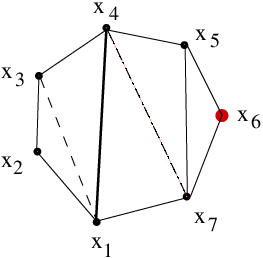}}
\caption{Geometry of   relation ${\bf Q}_4$. The diagonal $x_1x_4$   divides the heptagon   into a quadrangle and a pentagon. 
Deleting  one of three  vertices $\{x_5, x_6, x_7\}$  of the pentagon  we get another quadrangle: on the picture we deleted $x_6$. The arguments of  $\left \{[x_1, x_2,x_3,x_4],[x_4, x_5, 
 \widehat x_6,  x_7,  x_1] \right\}_{3,1}$ are  cross-ratios assigned to  these quadrangles.}
\label{zc2}
\end{figure} 

\begin{definition} \la{DEF4La}
The  $\Q$-vector space  $\mathbb{L}_4(\F)$ is generated by elements 
$\left\{x \right \}_4$, where $x \in \F^\times$, and $\left \{x,y \right\}_{3,1}$
where $x,y \in  \F^\times$,  obeying the following relation:
\begin{itemize}
 \item ${\bf Q}_4:$ For any configuration  $(x_1, x_2,\ldots, x_7)\in  {\mathcal{M}_{0,7}}(\F)$ the following cyclic sum is zero:
 \be
 \begin{split}
{\rm Cyc}_7& \Bigl( -\left \{[x_1,x_2,x_3,x_4],[x_4,x_6,x_7,x_1] \right\}_{3,1}\\
&+ \left \{[x_1, x_2,x_3,x_4],[x_4,x_5,x_7,x_1] \right\}_{3,1}\\
&- \left \{[x_1, x_2,x_3,x_4],[x_4,x_5,x_6,x_1 ] \right\}_{3,1}\\
&- \left \{[x_1, x_2,x_4,x_6]\right\}_{4}+3 \left \{[x_1,x_2,x_3, x_4,x_5, x_6]\right\}_{4}\Bigr)=0.\\
\end{split}
\ee
\end{itemize}
\end{definition} 

We denote by $\B_4(\F)$ the subspace of $\mathbb{L}_4(\F)$ spanned by symbols $\{x\}_4$ for $x\in \F^\times.$ Conjecturally, the natural map $\B_4(\F)\lra {\cal B}_4(\F)$ is an isomorphism.
 Specializing relation ${\bf Q}_4$, we obtain
 \be \la{202}
\left \{x, 1 \right\}_{3,1}=-\left \{1-x^{-1} \right\}_4
-\left \{1-x \right\}_4 + \left \{x \right\}_4.
\ee



\paragraph{2. An alternative  form  of relation ${\bf Q}_4$. }   We use a signed cyclic sum notation
\be \la{SIGNS}
\begin{split}
&{\rm Cyc}^-_n \varphi(x_1, \ldots , x_n):= \sum_{k\in \Z/n\Z} (-1)^k \varphi(x_{k+1}, x_{k+2}, \ldots , x_{k+n}). \end{split}
\ee
 For any   six distinct points $(x_1, x_2,\ldots,x_6)$ on ${\Bbb P}^1(\F)$ let us set
 \be \la{HHHH}
 \begin{split}
 &{{\rm L}^1_4}(x_1, \ldots , x_6) := \\
&{\rm Cyc}^-_6\Bigl( \left \{[x_1, x_2,x_3,x_4],[x_4, x_5,x_6,x_1] \right\}_{3,1} -
 \left\{[x_1,x_3,x_4,x_5]\right \}_4   
 -\left\{[x_1, x_2, x_3, x_4,x_5,x_6]\right \}_4\Bigr). \\
 \end{split}
 \ee
Then relation ${\bf Q}_4$ becomes
 \be \la{HHHH1}
 \begin{split}
&\sum_{k\in \Z/7\Z}    {{\rm L}^1_4}(x_{k+1}, \ldots ,   x_{k+6}) = 0. 
 \end{split}
 \ee  
From this perspective, the relations ${\bf Q}_2$ and ${\bf Q}_4$  look remarkably similar: 
 
${\bf Q}_2$: We start from   elements  $\{[x_1, x_2, x_3, x_4]\}_2$ assigned to 
 configurations of four distinct  points on ${\Bbb P}^1$. Then  for any   five distinct  points on ${\Bbb P}^1(\F)$ 
we have the five-term relation (\ref{121}).

${\bf Q}_4$: We start   from   elements  (\ref{HHHH})  assigned to 
  configurations of six distinct  points on ${\Bbb P}^1$. Then for   any  seven distinct  points on ${\Bbb P}^1(\F)$ we have   relation (\ref{HHHH1}).

\paragraph{3. The cobracket maps.} Recall that, according to our convention,  $\F^\times$ stands for $\F^\times\otimes \Q$. We define the cobracket maps:  
 \be \la{37}
 \begin{split}
  &\delta   \colon \mathbb{L}_2(\F) \longrightarrow  \Lambda^2\F^\times,\\
  &\delta   \colon \mathbb{L}_3(\F) \longrightarrow \mathbb{L}_2(\F) \otimes \F^{\times},\\
 &\delta  \colon \mathbb{L}_4(\F) \longrightarrow (\mathbb{L}_3(\F) \otimes \F^{\times}) \oplus  \Lambda^2\mathbb{L}_2(\F). \\
\end{split}
\ee 
First, we   define them on the generators: the cobracket $\delta \left \{x\right\}_{k} $ is given by formula (\ref{BB}), and
\be \la{26}
 \begin{split}
 \delta   \left \{x,y \right\}_{2,1}=&
\left \{\frac{1-y}{1-x}\right\}_2 \otimes \frac{y}{x}+\left \{\frac{y}{x}\right \}_2 \otimes \frac{1-y}{1-x} + \left \{x\right\}_2 \otimes (1-y^{-1}) + \left \{y\right\}_2 \otimes (1-x^{-1}),\\
\delta \left \{x,y \right\}_{3,1}=&
~ \left \{x,y\right\}_{2,1} \otimes \frac{x}{y}~+
\left \{\frac{x}{y}\right\}_3 \otimes \frac{1-x}{1-y}+
\left \{x\right\}_3 \otimes (1-y^{-1})-
\left \{y\right\}_3 \otimes (1-x^{-1})\\
& +\{x\}_2 \wedge \{y\}_2.\\
 \end{split}
 \ee
 Note that the symmetry of $\{x,y\}_{2,1}$,  implies that $\delta \left \{x,y \right\}_{3,1} = -\delta \left \{y,x \right\}_{3,1}$. \vskip 1mm
  
 We need to show that cobracket maps $\delta$ are well defined, i.e., that they vanish on the  relations ${\bf Q}_n$.  For $n=2$ the statement is well-known, for $n=3$ we prove it in Proposition \ref{PP1} and for $n=4$ we prove it in Proposition \ref{FRWT4}. 
 
It is easy to see that $\delta$  commutes with specialization. This can be checked case-by-case; we do ${\rm Sp}_{t = 0}\{t^a u, t^a v\}_{3,1}=\left\{\frac{u}{v}\right\}_{4}$ for $a<0$ as an example. In this case we have

\begin{align*}
& {\rm Sp}_{t = 0}\delta\{t^a u, t^b v\}_{3,1}=\\
& {\rm Sp}_{t = 0}\Bigl(\left \{t^a u,t^a v\right\}_{2,1} \otimes \frac{u}{v}~+
\left \{\frac{u}{v}\right\}_3 \otimes \frac{1-t^a u}{1-t^a v}+\\
&\left \{t^a u\right\}_3 \otimes (1-(t^a v)^{-1})-
\left \{t^a v\right\}_3 \otimes (1-(t^a u)^{-1})+\{t^a u\}_2\wedge\{t^a v\}_2 \Bigr)= \\
&2\left\{\frac{u}{v}\right\}_3\otimes \frac{u}{v}+\left\{\frac{u}{v}\right\}_3\otimes \frac{u}{v}=\delta \left(3\left\{\frac{u}{v}\right\}_4\right).\\
\end{align*}
The other cases are similar.\vskip 2mm

Next, we give two motivic interpretations of  elements $\{x,y\}_{m-1,1}$ and   cobracket formula (\ref{26}). 
 
 \paragraph{4. Elements $\{x,y\}_{m-1,1}$  and   multiple polylogarithms.}  Motivic multiple polylogarithms  \cite{Gon94} are elements of the motivic Hopf algebra.\footnote{We use the  summation convention 
 ${\rm Li}_{p,q}(x,y) = \sum_{0 < m < n}\frac{x^my^n}{m^p n^q}$ following the original definition \cite{Gon94}.} 
 Although do not use  them in this paper,  we present  the following identities in the motivic Lie coalgebra. We denote the equality in the Lie coalgebra (i.e. modulo products)  by $\stackrel{\shuffle}{=}$.
\be \la{36}
 \begin{split}
 &\left \{x,y \right \}^{\cal M}_{2,1} \stackrel{\shuffle}{=}  -{\rm Li}^{\cal M}_{2,1}\left ( \frac{x}{y},y \right )- {\rm Li}^{\cal M}_3(x)-{\rm Li}^{\cal M}_3(y),\\
&\left \{x,y \right \}^{\cal M}_{3,1}\stackrel{\shuffle}{=} {\rm Li}^{\cal M}_{1,3}\left (\frac{1}{x}, \frac{x}{y} \right )-3{\rm Li}^{\cal M}_4 \left ( \frac{x}{y} \right )-{\rm Li}^{\cal M}_4(y).\\
 \end{split}
 \ee
There is a more natural motivic interpretation of the elements $\{x,y\}_{m-1,1}$ via motivic correlators, which we will   explain now.

\paragraph{5. Motivic correlators and the Lie coalgebra ${\Bbb L}_{\leq 4}(\F)$.}  {\it Motivic correlators}  are  elements of the motivic Lie coalgebra  introduced in \cite{Gon08}. They are defined whenever we have the abelian  category of mixed Tate motives, e.g., when $\F$ is a number field or in realizations (Betti, de Rham, l-adic, crystalline, etc.) For concreteness, we will use the motivic notation.
 
Weight $m$ motivic correlators are determined by an arbitrary  cyclically ordered collection of   points 
 $b_1, \ldots, b_{m+1}$ on ${\Bbb P}^1(\F)$ - some or all of them may coincide, a  point $a\in {\Bbb P}^1(\F)$ different from $ b_i$, and a non-zero tangent vector $t$ at $a$, see Figure \ref{zg10}:
 $$ 
 {\rm Cor}_{a,t}^{\cal M}(b_1, \ldots , b_{m+1}) \in {\cal L}_m(\F). 
 $$
   If $m>1$, they do not depend on the tangent vector $t$. 
 
    \begin{figure}[ht]
\centerline{\epsfbox{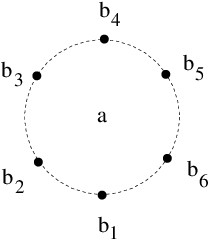}}
\caption{A data defining a weight 5 motivic correlator.}
\label{zg10}
\end{figure}

  To state the properties of motivic correlators, 
let us recall that a {\it framed mixed Tate motive} is  a mixed Tate motive  $M$ plus 
  non-zero linear maps $v: \Q(-m) \lra {\rm gr}^W_{2m} M$ and $f: {\rm gr}^W_{2n} M    \lra \Q(-n)$ for $m>n$. It provides an element 
  $(M; f,v) \in {\cal L}_{m-n}(\F)$. 
  
    Let $X:= {\Bbb P}^1_\F - \{0, b_1, \ldots , b_{m+1}\}$.  The following  properties explain ubiquity of motivic correlators: 
  
  \begin{enumerate}
  
  \item Motivic correlators describe the structure of the motivic fundamental group $\pi_1^{\cal M}(X, t_0)$ discussed in \S \ref{SectionMotivicFundamentalGroup}. Here $t_0$ is the   tangent vector $\partial/\partial z$ at $0$, where $z$ is the standard coordinate on ${\Bbb P}^1$. 
  
  In particular, 
  any element $(M; f,v) \in {\cal L}_m(\F)$ given by a framed  subquotient  $M$ of $\pi_1^{\cal M}(X, t_0)$ is a linear combination of motivic correlators.   
  
  \item The cobracket of  motivic correlators is given by a simple  formula, illustrated in Figure \ref{zg9}.
  
  \item The Hodge realization of motivic correlators is described by a Feynman type integral. 
  \end{enumerate} 
  
  Motivic correlators with arguments   in a field $\F$ are closed under the cobracket. So they form a graded Lie coalgebra,   denoted by ${\cal L}_\bullet'(\F)$. 
  It is a Lie subcoalgebra of ${\cal L}_\bullet(\F)$. Then Universality Conjecture \cite[Conjecture 17a]{Gon94} can be restated as follows.
  
  \bcon The canonical embedding ${\cal L}_\bullet'(\F) \hookrightarrow {\cal L}_\bullet(\F)$  is an isomorphism. 
  \econ
    
  In   Section \ref{SEC8} we show that the elements $\{x,y\}_{m-1,1}$ can be materialized by  
 {motivic correlators}:
  \bt \la{JMCi} a) There is a map of Lie coalgebras 
  \be \la{MCOA1} 
  {\Bbb L}_{\leq 4}(\F) \lra {\cal L}_{\leq 4} (\F),
  \ee
  defined on  generators as 
\[ 
\begin{split}
&\{x,y\}_{2, 1} \lms     {\rm Cor}^{\cal M}_\infty( {0} ,x,1,y).\\
&\{x,y\}_{3, 1} \lms     -{\rm Cor}^{\cal M}_\infty( {0, 0} ,x,1,y).\\
&\{y\}_{m} \lms     -{\rm Cor}^{\cal M}_\infty(\underbrace{0, \ldots , 0}_{m-1 ~{\rm zeros}},1,y).\\\end{split}
\]

b) The map (\ref{MCOA1}) provides a surjective homomorphism
  \be \la{MCOA2} 
  {\Bbb L}_{\leq 4}(\F) \lra {\cal L}'_{\leq 4} (\F).
  \ee
  
  \et 
 
 We prove Theorem \ref{JMCi} in Section \ref{SEC8}, restated there as Theorem \ref{JMC}.

Let us now continue to discuss the combinatorially defined Lie coalgebra $\mathbb{L}_{\leq 4}(\F)$.

\paragraph{6. The structure of $\mathbb{L}_{\leq 4}(\F)$.} It is described by  Theorem \ref{THRN}, proved in Section \ref{SSEECC10}. Its part c)  tells us that the $\Q$-vector space $\mathbb{L}_4(\F)$ is an extension,  predicted by the Freeness Conjecture.

\bt \la{THRN} a) The  cobracket maps  (\ref{26}) descend to the quotient spaces ${\Bbb L}_n(\F)$, $n \leq 4$, providing a graded Lie coalgebra structure on ${\Bbb L}_{\leq 4}(\F)$. 

b) For any    five distinct points $(x_1, \ldots, x_5)$ on ${\Bbb P}^1(\F)$ and any $y \in \F$ we have:
\be \la{5TERMRzuu}
\begin{split}
&\sum_{i=1}^5(-1)^i\left\{[x_1, \ldots, \widehat x_i, \ldots,  x_5], y\right \}_{3,1}\in  {\rm B}_4(\F).\\
\end{split}
\ee

c) There is a short exact sequence of $\Q$-vector space, functorial in $\F$: 
\be \la{Ext4}
0 \lra {\rm B}_4(\F) \stackrel{i}{\lra} \mathbb{L}_4(\F) \stackrel{p}{\lra}  \Lambda^2{\Bbb L}_2(\F) \lra 0.
 \ee
 
 Here the map $i$ is the natural embedding from Definition \ref{DEF4La}. 
 The projection $p$ is   the $(2,2)$-component of the cobracket $\delta: \mathbb{L}_4(\F) \lra \Lambda^2\mathbb{L}_2(\F)$. 
 \et
 
 We  prove   relation (\ref{5TERMRzuu}) by  specializing    relation ${\bf Q}_4$ to certain strata in $\overline {\cal M}_{0,7}$. 

The  natural isomorphisms 
${\rm B}_{2}(\F) =  {\Bbb L}_2(\F)$ and $ {\rm B}_{3}(\F)  \stackrel{\sim}{\lra}  {\Bbb L}_3(\F)$ are compatible with the cobrackets, and therefore 
 give rise to a map of complexes:
\be \la{M112} 
\begin{gathered}
    \xymatrix{
        {\rm B}_4(\F) \ar[r]^{}  \ar[d]^{ }   &{\rm B}_{3}(\F)\otimes \F^\times \ar[d]^{ } \ar[r]^{}  & {\B}_{2}(\F)\otimes \Lambda^2 \F^\times \ar[r]^{}  \ar[d]^{\sim}& 
        \Lambda^{4} \F^\times \ar[d]^{\sim}\\
             {\Bbb L}_4(\F)   \ar[r]^{}  & {\Bbb L}_3(\F)  \otimes {\Bbb L}_1(\F) \bigoplus \Lambda^2{\Bbb L}_2(\F)  \ar[r]^{}&   {\Bbb L}_2(\F ) \otimes  \Lambda^2{\Bbb L}_1(\F)  
               \ar[r]^{ }& \Lambda^4{\Bbb L}_1(\F)}  \\
                                             \end{gathered}
 \ee
 
 \bc \la{COR112}
 The map of complexes (\ref{M112})  is a    quasi-isomorphism.
  \ec
 \begin{proof}
  The map is injective by the very definition. By Theorem \ref{THRN} its cokernel   is identified with the complex  $\Lambda^2{\Bbb L}_2(\F)   \stackrel{\sim}{\lra}  \Lambda^2\B_2(\F)$, evidently quasi-isomorphic to zero. 
 \end{proof}
     
Note that the Bloch complex ${\rm B}_2(\F) \to  \Lambda^2\F^\times$ is quasi-isomorphic to   ${\rm R}_2(\F) \to \Q[\F] \to \Lambda^2\F^\times$. 

 \bc There  exists an  antisymmetric  map of complexes
\be \la{Mmapc} 
 \begin{split}
  &\Bigl( {\rm R}_2(\F) \to \Q[\F] \to \Lambda^2\F^\times \Bigr)\wedge \Bigl( {\rm R}_2(\F) \to \Q[\F] \to \Lambda^2\F^\times \Bigr)~ \lra \\
 &{\rm B}_4(\F) \stackrel{}{\lra} {\B}_{3}(\F)\otimes \F^\times \stackrel{}{\lra}  {\B}_{2}(\F)\otimes \Lambda^2 \F^\times \lra  \Lambda^{4} \F^\times. \\
 \end{split}
 \ee
 given on the generators as follows
 \be \la{Mmapc1} 
 \begin{split}
&( x_1\wedge x_2) \wedge ( y_1\wedge y_2) \lra  x_1\wedge x_2 \wedge y_1\wedge y_2 \in \Lambda^{4} \F^\times,\\
&\{x\}  \wedge (y_1\wedge y_2) \lms  \{x\}_2 \otimes y_1\wedge y_2 \in  {\B}_{2}(\F)\otimes \Lambda^2 \F^\times,\\
&\{x\}  \wedge \{y\}  \lms    \delta \{x,y\}_{3,1} - \{x\}_2 \wedge \{y\}_2\in  {\B}_{3}(\F)\otimes \F^\times.     \\
 \end{split}
 \ee
 \ec

 \paragraph{7. Gangl's formula.} Combining formula (\ref{5TERMRzuu})  with    formula (\ref{36}) expressing $\{x,y\}_{3,1}$ via   the weight four motivic double polylogarithm 
${\rm Li}_{1,3}^{\cal M}(x,y)$, we get a  ``five-term relation'' expressing
\be \la{5TERMRx}
\sum_{i=1}^5 (-1)^i{\rm Li}_{1,3}^{\cal M}\left(\frac{1}{[x_1, \ldots, \widehat x_i, \ldots,  x_5]}, \frac{[x_1, \ldots, \widehat x_i, \ldots,  x_5]}{y}\right)\\
\ee
as a linear combination of elements  ${\rm Li}_4(x)$ for $x\in \F^{\times}.$ 
Such a five-term relation for   ${\rm Li}_{1,3}^{\cal M}(x,y)$ was conjectured by Goncharov \cite[Section 4.8, p. 82]{Gon94a}, and discovered  
  by an amazing  tour de force computer calculations  by Herbert Gangl in \cite{Gan16}. Gangl's formula  contains 122 terms ${\rm Li}_{4}^{\cal M}(\ast)$ with   arguments of a  rather complicated and   obscure nature. So although it is given explicitly, it does not seem feasible to prove it directly, without using a computer. Our approach gives a conceptual proof of a variant of Gangl's  formula, explaining   the arguments of ${\rm Li}_{4}^{\cal M}$. 

Note also   that relation (\ref{5TERMRx}) is sufficient to define a functorial  extension of $\Lambda^2\B_2(\F)$ by ${\rm B}_4(\F)$.

\paragraph{8. The role of  elements $\{x,y\}_{3,1}$.}   

\begin{enumerate}

\item 
The element $\{x,y\}_{3,1}\in {\Bbb L}_4(\F)$ shows that the projection ${\Bbb L}_4(\F) \lra \Lambda^2{\Bbb L}_2(\F)$ is surjective. Indeed, the 
$\Lambda^2{\Bbb L}_2(\F)$-component of its cobracket is $\{x\}_2 \wedge \{y\}_2$. 

\item 
The element  $\{x,y\}_{3,1}$  is the crucial part  of  the   product  map (\ref{Mmapc}).

\item 
The element $\{x,y\}_{3,1}$ 
is   a weight four motivic correlator. Therefore it is related to 
 the motivic fundamental group $\pi_1^{\cal M}({\Bbb P}^1- \{0, 1, \infty, x, y\}, t_0)$.

\end{enumerate}

\paragraph{9. Functional equations for $4$-logarithm.}  According to (\ref{5TERMRzuu}), 
substituting into one of the arguments of $\{x,y\}_{3,1}$ the five-term relation for the dilogarithm we 
get a sum of elements $\{z\}_{4}$. Now substituting a relation between the five-term relations, we get 
a functional equation for the  tetralogarithm. Let us elaborate on this. 

Denote by  ${\rm C}_n({\Bbb P}^1(\F))$   the $\mathbb{Q}$-vector space generated by  configurations $(x_1, \ldots , x_n)$ 
of $n$ distinct points on ${\Bbb P}^1(\F)$.  These vector spaces are organized into    the standard chain complex 
$$
\ldots  \stackrel{\partial}{\lra} {\rm C}_6({\Bbb P}^1(\F)) \stackrel{\partial}{\lra} {\rm C}_5({\Bbb P}^1(\F)) \stackrel{\partial}{\lra} {\rm C}_4({\Bbb P}^1(\F))  \stackrel{\partial}{\lra} \ldots 
$$

According to  (\ref{5TERMRzuu}), there is a map 
\[
 \kappa_4: {\rm C}_5({\Bbb P}^1(\F)) \otimes \Q[\F] \lra {\rm B}_4(\F).
\]
 Let us pick its lift to $\Q[\F]$:
 \[
\widetilde \kappa_4: {\rm C}_5({\Bbb P}^1(\F)) \otimes \Q[\F] \lra \Q[\F].
\]
 It leads to functional relations for the tetralogarithm as follows.  

Let    ${\rm S}(\F) \subset {\rm C}_5({\Bbb P}^1(\F))$  be the subgroup of antiinvariants of the order reversing  involution, acting    by 
$(x_1, x_2, x_3, x_4, x_5) \lms (x_5, x_4, x_3, x_2, x_1)$.\footnote{It can be identified with  the one in \cite[Appendix]{Gon95}. }
Then $\partial ({\rm S}(\F)) =0$. Set\footnote{Informally,  ${\rm A}(\F)$ is the group generated by  relations between the five-term relations.}  
$$
{\rm A}(\F) := {\rm C}_6({\Bbb P}^1(\F)) \oplus {\rm S}(\F). 
$$
Consider the following complex:
\[
{\rm A}(\F) \stackrel{\alpha}{\lra}  {\rm C}_5({\Bbb P}^1(\F)) \stackrel{\beta}{\lra} \Q[\F^\times] \stackrel{\delta}{\lra}  \Lambda^2\F^\times.
\]
Note that the cross-ratio provides an isomorphism ${\rm C}_4({\Bbb P}^1(\F)) = \Q[\F^\times\backslash \{1\}]$. Then
 $\beta$ is just the map $\partial: {\rm C}_5({\Bbb P}^1(\F)) {\lra} {\rm C}_4({\Bbb P}^1(\F))$. Its 
  image is the subgroup ${\rm R}_2(\F)$  of five-term relations.  
The restriction of   $\alpha$ to ${\rm C}_6({\Bbb P}^1(\F))$ is the differential  $\partial$.  Its restriction   to ${\rm S}(\F)$ is   the tautological inclusion. 
We get a complex: $\beta \circ \alpha =0$. 
A.A. Suslin conjectured (unpublished, cf. \cite[page 314]{Gon95}) that 
$$
{\rm Ker}(\beta)/{\rm Im}(\alpha)={\rm gr}_\gamma^2K_4(\F).
$$ 
 Restricting the map $\widetilde \kappa_4$ to the subgroup $\alpha({\rm A}(\F)) \subset  {\rm C}_5({\Bbb P}^1(\F))$ we get  relations for ${\rm Li}_{4}$:
$$
\widetilde \kappa_4: \alpha({\rm A}(\F)) \otimes \Q[\F] \lra {\cal R}_4(\F).
$$
This follows from the   commutative diagram below: 
\[
\begin{gathered}
    \xymatrix{
        {\rm A}(\F)\otimes \Q[\F] \ar[r]^{\alpha\otimes{\rm Id}~~}  \ar[d]^{ }   &{\rm C}_5({\Bbb P}^1(\F))\otimes \Q[\F] \ar[d]^{ \kappa_4} \ar[r]^{~~~~\beta\otimes{\rm Id}}  & \Q[\F^{\times}]\otimes \Q[\F]  \ar[d]^{ (\ref{Mmapc1})} \\
             0  \ar[r]^{}  & {\rm B}_4(\F)     \ar[r]^{}&   {\B}_3(\F)\otimes \F^\times 
               \\} 
                                             \end{gathered}
\]

 \paragraph{10. Conclusion.} If $n>3$, the problem of writing explicitly functional equations for the classical $n-$logarithm  might  not  be the ``right'' problem. It seems that 
 when  $n$ is growing, the functional equations  become so complicated  that one can not   write them down  on a piece of paper. 
 
Contrary to this,  relations ${\bf Q}_n$ for $n=2,3,4$  are simple, geometric,   
   surprisingly uniform in $n$, and  describe functorially  the motivic Lie algebra in the weights $\leq 4$, e.g.  define  extension (\ref{Ext4}). 
    They live on   moduli spaces 
 ${\cal M}_{0,n+3}$. We suggest that there is a similar  description for all $n$, and   
  that a  combinatorial construction of the   functor  $ \F \lra {\L}_\bullet(\F)$ is    the way to  
 approach   functional equations for   classical $n-$logarithms.  The   Lie coalgebra ${\Bbb L}_{\leq 4}(\F)$ is the first step in this direction. 
  
    \vskip 3mm
The next question is where the relations ${\bf Q}_n$ come from, and why they are   uniform in $n$. Our answer relates them to the cluster structure of the moduli spaces ${\cal M}_{0,n+3}$ and   ${\rm Conf}_n(V_2)$.

\subsection{Cluster polylogarithm maps  for ${\rm Conf}_n(V_2)$} \la{sec8.2}

Recall that the Milnor ring of a field $\F$  is the quotient of the exterior algebra   ${\Lambda}^\bullet\F^{\times}$ of the multiplicative group $\F^\times$   by the  ideal generated by the elements $ x\wedge (1-x)$ where $x \in \F^\times - \{ 1 \}$:
$$
{\rm K}_\bullet^M(\F)=\frac{{\Lambda}^\bullet(\F^{\times})  }{\langle x\wedge (1-x) \rangle}.
$$ 
There is a canonical homomorphism, written using the notation $d\log (f):= f^{-1}df$. 
\be \la{MK2}
\begin{split}
&d\log:  {\rm K}^{M}_n(\F) \lra \Omega^n_{\F/\Q},\\
&f_1 \wedge \ldots \wedge f_n \lms d\log (f_1) \wedge \ldots \wedge d\log (f_n).\\
\end{split}
\ee

Let $V_2$ be a $2$-dimensional symplectic space over a field $\F$ with a symplectic 2-form $\omega$.  Denote by ${\rm Aut}(V_2, \omega)$ the automorphism group of the symplectic vector space $(V_2, \omega)$. For $n\geq 2$ consider the space ${\rm Conf}_n(V_2)$ parametrizing ordered tuples of pairwise non-collinear vectors $(l_1, \ldots, l_n)$ for $l_i\in V_2$ modulo the action of the group ${\rm Aut}(V_2, \omega).$   Let $\mathbb{F}_n$ be the function field $\Q({\rm Conf}_n(V_2)).$ Finally, denote by $x_i$ the projections of $l_i$ to $\mathbb{P}(V_2).$

Let us label vertices of a convex $n$-gon ${\rm P}_n$ by the vectors $l_1, \ldots, l_n$ following the  cyclic order of the vertices. Each triangulation ${\cal T}$ of ${\rm P}_n$ provides the following data on the space ${\rm Conf}_n(V_2)$:

\begin{enumerate}
\item 
 A coordinate system $\{\Delta_E\}$ on ${\rm Conf}_n(V_2)$, whose coordinates are the following invertible functions parametrized by the set 
$\{E\}$   of the edges of the triangulation ${\cal T}$, including the sides of the polygon, defined as follows:
\be
\begin{split}
&\Delta_E 
:=  \omega(l_i,  l_j), ~~~~E = (i,j), ~ i<j.\\
&\Delta_E \in \mathbb{F}_n^{\times}.\\
\end{split}
\ee

 \item 

  Motivic avatar of the ``Weil-Petersson 2-form'' on the space ${\rm Conf}_n(V_2)$ \cite[Section 6]{FG03}:
\be \la{CLASSW2}
W_{\cal T}:= \frac{1}{2}\sum_{E, F} \varepsilon_{EF} \Delta_E \wedge \Delta_F \in \Lambda^2\mathbb{F}_n^{\times}.
\ee
 Here   $\varepsilon_{EF} \in \{1, -1, 0\}$, $\varepsilon_{EF} = -\varepsilon_{FE}$, and  
\be \la{EPS}
\varepsilon_{EF} :=  \begin{cases} 
+1 & \mbox{if $E$ and $F$ share a vertex, and $F$ is just after $E$ in the clockwise order,}\\
-1 &\mbox{if $E$ and $F$ share a vertex, and $E$ is just after $F$ in the clockwise order,}\\
0 &\mbox{otherwise.}
\end{cases} 
\ee 
 
\end{enumerate}

   \begin{figure}[ht]
\centerline{\epsfbox{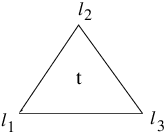}}
\caption{The element $W_t$ is assigned to a triangle $t$ decorated by a configuration of vectors $(l_1, l_2, l_3)$. The elements $\omega(l_i, l_j)$ are assigned to the sides of the triangle.}
\label{zc4}
\end{figure} 

Here is the simplest example. Let $(l_1, l_2, l_3)$ be a generic configuration of three vectors in $V_2$. We assign them  to the vertices of a 
triangle $t$, see Figure \ref{zc4}. Then the element  (\ref{CLASSW2})   is
\be \la{MAPT}
W_t:= \omega(l_1, l_2) \wedge  \omega(l_2, l_3) + \omega(l_2, l_3) \wedge  \omega(l_1, l_3) +\omega(l_1, l_3) \wedge  \omega(l_1, l_2).
\ee
Using the elements $W_t$ we can rewrite the definition of $W_{\cal T}$ as follows: 
\be \la{MAPWt}
W_{\cal T} = \sum_{t \in T} W_t. 
\ee
Here the sum is over all triangles $t$ of the triangulation ${\cal T}$.

\begin{figure}[ht]
\centerline{\epsfbox{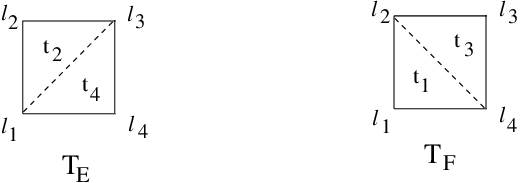}}
\caption{The two triangulations of the 4-gon. They are related by a flip of an edge.}
\label{zc3}
\end{figure}

 If we change a triangulation ${\cal T}$ by   flipping   its edge, the   element $W_{\cal T}$ does change. Precisely, 
 let ${\cal T}$ and ${\cal T}'$ be the two triangulations related by a flip at an edge $E$, see Figure \ref{zc3}. Let $X_E:= -[x_1, x_2, x_3, x_4]$ be the negative cross-ratio of the four vectors   at the vertices of the  $4$-gon assigned to the flip, so that $E = (l_1, l_3)$. 
 \be \la{FLIUPEF}
 W_{{\cal T}} -  W_{{\cal T}'}  = (1+X_E) \wedge X_E= \delta \{-X_E\}_2.
\ee
Indeed, using the shorthand $|ij|:= \omega(l_i,l_j)$, one has thanks to the Pl\"ucker relations:
   \be \la{MCda} 
\begin{split}
        & \left( 1 + \frac{|12| |34|}{|23||14|}\right)\wedge \frac{|12| |34|}{|23||14|}\  =  \frac{|13| |24|}{|23||14|} \wedge \frac{|12| |34|}{|23||14|}\ 
       = W_{t_2} + W_{t_4} - W_{t_1} - W_{t_3}  \in  \Lambda^{2}{\F}^\times.
\end{split}
\ee
 
   Any two triangulations of a polygon are related by a sequence of flips. 
  Thus the class $[W_{\cal T}]$ in $K_2$  does not depend on the triangulation.   
In particular, it gives rise to a well defined 2-form  
$$
\Omega_{\cal T}:= (d\log \wedge d\log) (W_{\cal T}) = \frac{1}{2}\sum_{E, F} \varepsilon_{E F} d\log( \Delta_E) \wedge d\log (\Delta_F).
 $$

Our key point is to  assign to each triangulation ${\cal T}$  of the polygon ${\rm P}_n$  and $k\geq 1$ the $k$-th power of $W_{\cal T}$:
 $$
 \left( {\cal T}, k\right) \lms W_{\cal T}^k:= W_{\cal T} \wedge \cdots \wedge W_{\cal T} \in \Lambda^{2k}\mathbb{F}_n^\times.
 $$
Now the  Stasheff polytope ${\rm K}_{n-3}$ enters into the picture.
 
  \begin{figure}[ht]
\centerline{\epsfbox{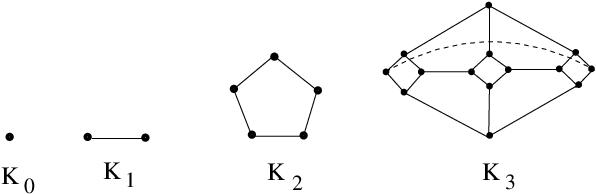}}
\caption{The Stasheff polytopes ${\rm K}_0, {\rm K}_1, {\rm K}_2, {\rm K}_3$.  }
\label{zc6}
\end{figure}  

 Recall that the Stasheff polytope ${\rm K}_m$ is a convex polytope whose $p$-dimensional faces are in bijection with the collections of 
 $m-p$ non-intersecting  
diagonals of ${\rm P}_{m+3}$. 
So the vertices of the Stasheff polytope ${\rm K}_m$ are parametrized by the triangulations of ${\rm P}_{m+3}$.  
The  $1$-dimensional faces of 
${\rm K}_m$ correspond to pairs of triangulations ${\cal T}\stackrel{}{\longleftrightarrow} {\cal T}'$ related by a flip, etc.

\paragraph{Examples.} 
Figure \ref{zc6} shows   Stasheff polytopes of small dimensions: $K_0$ is a point, $K_1$ is a segment, $K_2$ is a pentagon.  The polytope $K_3$ is obtained by gluing two tetrahedra 
  into a pyramid,  and cutting off 
 the three glued vertices in the middle. It has 14 vertices, 21 edges, and 9 faces: 
 3 squares  and 6 pentagons.  
 \vskip 3mm

Denote by   ${\rm C}_p({\rm K}_m)$ the free abelian group  generated by the   oriented $p$-dimensional faces of the polytope ${\rm K}_m$  - altering the orientation we flip the sign of the generator.  
Consider the chain complex of the Stasheff polytope ${\rm K}_m$:

$$
{\rm C}_*({\rm K}_m):= ~~~~~~{\rm C}_{m}({\rm K}_m) \stackrel{d}{\lra} {\rm C}_{m-1}({\rm K}_m)\stackrel{d}{\lra}  \ldots \stackrel{d}{\lra} {\rm C}_0({\rm K}_m).
 $$


 The discussion above shows that for the pentagon ${\rm K}_2$ there is  a map of complexes.
 \be \la{MCDImmmytz} 
  \begin{gathered}
    \xymatrix{
   &  {\rm C}_{2}({\rm K}_{2})  \ar[d]^{{\Bbb L}^0_2}   \ar[r]^{d}    & {\rm C}_{1}({\rm K}_{2})  \ar[d]^{{\Bbb L}^1_2}  \ar[r]^{d} &        
      {\rm C}_{0}({\rm K}_{2})  \ar[d]^{{\Bbb L}^{2}_2}\\
  & 0    \ar[r]^{}      &  \B_2({\Bbb F}_{5})   \ar[r]^{\delta}    
               &\Lambda^{2}{\Bbb F}^\times_{5}.}
                          \end{gathered}
 \ee 
The map ${\Bbb L}^{2}_2$ sends   each vertex $v_{\cal T}$ of the Stasheff pentagon ${\rm K}_2$, corresponding to a triangulation ${\cal T}$ of the pentagon ${\rm P}_5$, to the element $W_{\cal T}$. 
The map ${\Bbb L}^{1}_2$  sends each edge of the  pentagon ${\rm K}_2$,  
corresponding to a flip of   triangulations of the pentagon ${\rm P}_5$ assigned to an internal edge $E$, to the element $\{-X_E\}_2 \in \B_2({\Bbb F}_{5})$.
Finally, the claim that $ {\Bbb L}^{1}_2 \circ d=0$ is just  the pentagon relation ${\bf Q}_2$. 

\vskip 3mm
 Theorem \ref{MC4} provides a similar interpretation of the relation ${\bf Q}_4$,   starting by assigning the element $W_{\cal T} \wedge   W_{\cal T}$ to each triangulation ${\cal T}$ of the polygon ${\rm P}_7$.   

\bt \la{MC4}Let    $m\geq 0$ be an integer. Then,  using  shorthands $\B_n=\B_n({\Bbb F}_{m+3})$, 
 there exists a  map of complexes, called   the {\rm cluster polylogarithm map}: 
 \be \la{MCDImmmxz} 
\begin{gathered}
    \xymatrix{
    {\rm C}_{4}({\rm K}_{m})  \ar[d]^{{\Bbb L}^0_{4}}  \ar[r]^{d}    & {\rm C}_{3}({\rm K}_{m})  \ar[d]^{{\Bbb L}^1_{4}}  \ar[r]^{~~~ d}&{\rm C}_{2}({\rm K}_{m}) \ar[d]^{{\Bbb L}^2_{4}}  \ar[r]^{d} &   {\rm C}_{1}({\rm K}_{m})  \ar[d]^{{\Bbb L}^{3}_{4}} \ar[r]^{d}& 
       {\rm C}_{0}({\rm K}_{m})  \ar[d]^{{\Bbb L}^{4}_{4}}\\
            0  \ar[r]  &  {\Bbb L}_{4}({\Bbb F}_{m+3} )    \ar[r]^{\delta~~~~~~}&  \B_3  \otimes  {\Bbb F}_{m+3}^\times  \oplus \Lambda^2\B_2  \ar[r]^{\delta}& \B_2  \otimes \Lambda^{2} {\Bbb F}_{m+3}^{\times}  \ar[r]^{\delta}  
               &\Lambda^{4}{\Bbb F}_{m+3}^{\times}
               }
\end{gathered}
\ee
such that the vertex $v_{\cal T}$ of the  polytope ${\rm K}_m$ assigned to a triangulation ${\cal T}$ is mapped to $W_{\cal T}^{2}$:
$$
{\Bbb L}^{4}_{4}: v_{\cal T} \in {\rm C}_0({\rm K}_m)\lms W_{\cal T} \wedge   W_{\cal T} \in \Lambda^{4}{\Bbb F}^{\times}_{m+3}.
$$
\et

We   define   map (\ref{MCDImmmxz}) explicitly in  Section \ref{SEC10.2}. In fact, it is sufficient to 
define it for $m=3, 4$: the general case reduces to this. The key point is that for $m=4$ 
the left commutative square in (\ref{MCDImmmxz}) is nothing else but the relation ${\bf Q}_4$. It lives on the space ${\rm Conf}_7({\Bbb P}^1)$ rather than ${\rm Conf}_7(V_2)$.

Precisely, there is a canonical isomorphism $\Z = {\rm C}_{m}({\rm K}_{m})$. It sends   $1 \in \Z$ to  the oriented polytope itself, viewed as the 
$m$-chain $[{\rm K}_{m}]$  in the chain complex. For $m=3$ its image   under the map ${\Bbb L}^1_{4}$ is an element of  the group  ${\Bbb L}_{4}({\rm F})$ for  ${\rm F}={\Bbb F}_{6}$: 
$$
{\Bbb L}^1_{4}([{\rm K}_{3}]) \in  {\Bbb L}_{4}({\Bbb F}_{6}).   
$$
For an arbitrary infinite field $\F$, it can be interpreted as a function on the $\F$-points of the space ${\rm Conf}_{6}(V_2)$ with values in ${\Bbb L}_{4}(\F)$:
$$
{\rm L}^1_{4}(l_1, \ldots , l_{6}) \in {\Bbb L}_{4}(\F). 
$$
Commutativity of the left square is equivalent to  the $7$-term relation
\be \la{MOTREL}
\sum_{i=1}^{7}(-1)^i {\rm L}^1_{4}(l_1, \ldots , \widehat l_i, \ldots , l_{7}) =0
\ee
We will see that this  is just the relation ${\bf Q}_4$.

 \subsection{Cluster nature of functional equations for polylogarithms} \la{Sec1.4a}
 
 In Section \ref{Sec1.4a} we collect several observations about the mysterious link between cluster Poisson varieties \cite{FG03a} and polylogarithms. 
 The main feature of cluster Poisson varieties is the fact that they admit non-perturbative in the Planck constant $\hbar$ 
 quantization, with intrinsically built modular duality\footnote{Note that the involution $\hbar \to 1/\hbar$ differs by the sign from the standard modular transformation $z \to -1/z$.}  $\hbar \leftrightarrow 1/\hbar$, see  \cite{FG07}. The modular quantum dilogarithm $\Phi^\hbar(z)$ is the main working horse 
 carrying out the cluster quantization. 
 
 Although none of the observations below   is used in the proof of our main results, we feel  that they indicate strongly   
 the future development of the subject    towards  quantum polylogarithms. 

\paragraph{1. Cluster Poisson structure of the space $ {\mathcal{M}_{0,n+3}}$ and relations ${\bf Q}_n$.} The moduli space $ {\mathcal{M}_{0,n}}$ has a {\it cluster Poisson} structure  
  invariant under the cyclic shift  $(x_1, \ldots, x_n)\lra (x_2, \ldots, x_n, x_1)$ \cite{FG03a}.   
It comes with a finite collection of cluster Poisson coordinate systems, parametrized by  triangulations  of the oriented 
 convex polygon   ${\rm P_n}$. Namely, let us label  the vertices of the polygon  by the points 
$(x_1, \ldots, x_n)$  cyclically, following the polygon orientation. 
Then any   diagonal $E = (i,k)$ in a rectangle $(i, j, k, l)$  of a triangulation  gives rise to a cluster Poisson coordinate, given by the {\it negative} of the cross-ratio (\ref{CR1}):  
\be \label{ClusterCoordinates}
X_E:=-[x_{i}, x_{j}, x_{k}, x_{l}].
\ee
The Poisson structure on the space $ {\mathcal{M}_{0,n}}$ is described as follows. Given a triangulation ${\cal T}$ of the polygon, the Poisson bracket between the functions $X_E$ and $X_F$ assigned to   internal 
edges $E$ and $F$ 
of the triangulation ${\cal T}$ is given by the following formula, where   $\varepsilon_{EF}$ was defined in (\ref{EPS}):
$$
\{X_E, X_F\} = \varepsilon_{EF}X_EX_F.
$$


The  space $ {\mathcal{M}_{0,2n+1}}$ is symplectic. The  center of the  Poisson algebra of functions on  $ {\mathcal{M}_{0,2n}}$ is generated by 
the  {\it Casimir element}, encompassing both the cross-ratio (\ref{CR1}) and the function (\ref{CR0}): 
$$
[x_1,x_2, \ldots, x_{2n}]:= (-1)^n\frac{(x_1-x_2)(x_3-x_4) \ldots  (x_{2n-1}-x_{2n})}{(x_2-x_3)(x_4-x_5)   \ldots  (x_{2n}-x_{1})\ \ \ \  }.
$$

We observe that for $n=2,3,4$ we have the following.

\vskip 2mm
{\it Barring  $\{[x_1, \ldots, x_6]\}_n$, all arguments of  elements $\{\ast\}_n$ and $\{\ast, \ast\}_{n-1, 1}$ in   relation ${\bf Q}_n$  are  
negatives of  cluster Poisson coordinates. For $n=3, 4$ arguments of $\{\ast, \ast\}_{n-1, 1}$ Poisson commute (have a zero Poisson bracket). The argument of $\{[x_1, \ldots, x_6]\}_n$ is a product of two Poisson-commuting cross-ratios. }  
 
\vskip 3mm

Our next remark is   that  the functional equation (\ref{TRIF}) for the trilogarithm is related to the cluster Poisson structure of finite type $D_4$.

\paragraph{2. Cluster nature of the triple ratio.}  Although the seven-term  relation (\ref{TRIF}) for the trilogarithm 
generalizes  the five-term relation for the dilogarithm,  formula (\ref{TR}) looks mysterious. It became clear only  a decade later after relation (\ref{TRIF}) 
was discovered 
that the triple ratio (\ref{TR}) has a cluster origin: it is related to the cluster Poisson  structure of the  
space ${\rm Conf}_6({\Bbb P}^2)$ of configurations of six  points  in ${\Bbb P}^2$, which is of  type $D_4$. 
The triple ratio (\ref{TR})  is the simplest cluster Poisson coordinate which is not   a cross-ratio. 

Furthermore, the proof of  relation (\ref{TRIF}) in \cite[Appendix]{Gon95a} is based on the following identity:

\be
\begin{split}
&\omega(l_1,l_2,l_5)\omega(l_2,l_3,l_6)\omega(l_1,l_3,l_4) -  \omega(l_1,l_2,l_4)\omega(l_2,l_3,l_5)\omega(l_1,l_3,l_6)  \\
&= \omega(l_1,l_2,l_3) \cdot \omega^*(l_1 \times l_4, l_2 \times l_5, l_3 \times l_6).\\
\end{split}
\ee
Here $\omega^*$ is the dual volume form in $V_3^*$, and $a\times b$ is the cross-product. 
This identity  is just an exchange relation in the cluster algebra of type $D_4$ with six frozen variables, describing the algebra of   functions on the space ${\rm Conf}_6(V_3)$ of configurations of 
six  vectors in $V_3$. 

It is also interesting to note a 40-term functional equation for the trilogarithm \cite{GGSVV13} whose arguments are   cluster Poisson coordinates on the moduli space 
of 6 cyclically ordered points in ${\Bbb P}^2$, which is also of finite type $D_4$. 
 
\paragraph{3. Pentagon relation for the quantum dilogarithm.} 
A  pentagon relation for the   
 quantum dilogarithm power series  $\Psi_q(x)$ was proved in \cite{FK93}.  However, the series converges only if $|q|<1$. 
 
The {\it modular quantum dilogarithm 
function} is  defined   by the following convergent integral:
$$
\Phi^\hbar(z) := {\rm exp}\Bigl(- \int_{\Omega}\frac{e^{-ipz}}{ (e^{\pi p} - e^{-\pi p}) (e^{\pi \hbar p} - e^{-\pi \hbar p}) 
} \frac{dp}{p} \Bigr), ~~~~\hbar>0.
$$
Here $\Omega$ is a path from $-\infty$ to $+\infty$, making a little
half circle going over the zero.
The name is justified by the asymptotic expansion when $\hbar \to 0$  \cite[Section 4]{FG03a}: 
$$
\Phi^\hbar(z) \sim {\rm exp}\Bigl(\frac{{\rm L}_2(e^{z})}{2\pi i \hbar}\Bigr),
\quad \mbox{where } {\rm L}_2(x):= \int_0^x\log(1+t)\frac{dt}{t}.
$$
The function $\Phi^\hbar(z)$ goes back to  Barnes \cite{Bar99}, and reappeared at the modern time in \cite{Bax82},  \cite{Fad95} and other works. It plays a key role in quantization of cluster Poisson varieties \cite{FG07}.

The cluster Poisson structure of the space $ \mathcal{M}_{0,n+3}$ is of finite type $A_n$. 
 Therefore there are only finitely many cluster Poisson coordinates (\ref{ClusterCoordinates}).

The function $\Phi^\hbar(z)$ provides an operator $K: L^2(\R) \to L^2(\R)$,
defined as a rescaled
Fourier transform followed by
the operator of multiplication by   $\Phi^\hbar(x)$.
$$
Kf(z):= \int_{-\infty}^{\infty}f(x)\Phi^\hbar(x) {\rm exp}(\frac{-xz}{2\pi i \hbar})dx.
$$
Since $|\Phi^\hbar(x)|=1$ on the real line, $2\pi \sqrt {\hbar}K$ is unitary.  A proof of the following result can be found in \cite{Gon07}.\footnote{Unfortunately, the argument suggested in \cite{FC99}  was wrong.}
Its quasiclassical limit recovers Abel's five-term
relation for the dilogarithm.

\begin{theorem} \label{qp4*} 
$(2\pi  \sqrt {\hbar} K)^5 = \lambda \cdot {\rm Id}$, where
$|\lambda|= 1$.
\end{theorem}
 
{ We suggest that the relation ${\bf Q}_4$ should have a  quantum analog related to the quantized   cluster symplectic variety $ {\mathcal{M}_{0,7}}$}.

 \subsection{Organization of the paper and other applications}

 \paragraph{1. The structure of the paper.} In Section \ref{SEC8} we recall motivic correlators and  introduce  motivic correlators $\{x,y\}_{m-1, 1}$. Using them, 
we define  a map of Lie coalgebras ${\Bbb L}_{\leq 4}(\F) \lra {\cal L}_{\leq 4}(\F)$. 
 The key fact that this map  is well defined  is established in    Sections \ref{SEC2N} - \ref{SEC9}. 
 
 In Section \ref{SEC2N} we recall  cluster varieties and  introduce cluster polylogarithm maps of complexes. 
 
 In  Section \ref{sec8} we consider our main example:  the  cluster variety ${\rm Conf}_{n+3}(V_2)$  
  parametrizing generic configurations of $n+3$ vectors in a 2-dimensional symplectic vector space.  It is a cluster variety of finite type $A_{n}$ with $n+3$ frozen variables. 
 Its cluster modular complex is the n-dimensional Stasheff polytope.   
We work out the cluster polylogarithm map  
  in the weights $2$ and  $4$, and  prove  relation ${\bf Q}_4$ modulo relation ${\bf Q}_3$. 

In Section \ref{SEC9} we prove   relation ${\bf Q}_3$, and thus complete the proof of relation ${\bf Q}_4$.

In Sections \ref{SSEECC10} we deduce   Theorem \ref{THRN} from relation ${\bf Q}_4$.  
 
In Section \ref{SEC2} we recall weight $p$ Bigrassmannian complexes and the map from the  decorated $(n-3)-$flag complex for ${\rm GL}_n$ to the weight $4$ Bigrassmannian complex constructed in \cite{Gon93}, and   define a map from the weight $4$ Bigrassmannian complex  to   weight $4$ motivic complex (\ref{MTM=DM12}).   

In Section \ref{Sec4}  we prove that it is a map of complexes.

Combining these results,   we get   maps from   appropriate parts of the algebraic $K$-theory of $\F$ to the cohomology of the weight four polylogarithmic  motivic complex  ${\cal B}^\bullet(\F;4)$, see  (\ref{BCOM}). However, we still need to prove that   these maps are non-trivial, even in the number field case. 

To do this, we construct   in Section \ref{SEC8.4n} the  weight $4$ Beilinson regulator,  given by a  
   class    in the weight $4$ real Deligne cohomology of the classifying space 
${\rm BGL}_{{n}  }$ of the algebraic group ${\rm GL}_{n} $:
\be \la{CHHO}
{c}_{4, {n} }^{\cal H}\in H^8({\rm BGL}_{{n}  }(\C), \underline{\R}_{\cal D}(4)), \ \ \ \ \forall {n} \geq 4.
\ee

To explain  the strategy, let us recall that one should have 
 the universal motivic Chern class 
$$
{c}_{p, {n} }^{\cal M}\in H^{2p}_{\rm Zar}({\rm BGL}_{{n}  }, \underline{\Z}_{\cal M}(p)) 
$$
 in the weight $n$ motivic cohomology of ${\rm BGL}_{{n}  }$. Here $H_{\rm Zar}$ is Zariski cohomology, and $\underline{\Z}_{\cal M}(p)$   the weight $p$  motivic complex  \cite{Voe00}.
These classes provide   Chern classes in various realizations.

  In Section \ref{Sec7} we define for a  regular  variety $X$ over a field $\F$  a complex $\Gamma(X;4)$ which should calculate  the rational weight $4$ motivic cohomology of $X$. When $X = {\rm Spec}(\F)$, it is  the   complex ${\cal B}^\bullet(F; 4)$.  Complexes $\Gamma(X;n)$  for  $n \leq 3$ were defined in \cite{Gon95}. The construction of   complex $\Gamma(X;4)$ uses   the  reciprocity law conjectured in \cite{Gon02a}, and proved in full generality in \cite{Bol21}. 
 
  When $\F = \C$, we define an explicit  regulator map   
  $$
  r_{\cal D}: \Gamma(X;4) \lra \R_{\cal D}(X(\C);4)
  $$
    to  the complex calculating the weight $4$ real Deligne cohomology of $X$, see Theorem \ref{MTHRCC}.

The map  from Section \ref{SEC2} provides a cocycle $C_{4, {n} ; \circ}^{\cal M}$ for the motivic Chern class  of the generic point  of Milnor's simplicial realization ${\rm BGL}_{{n} \bullet}$ of   ${\rm BGL}_{{n}  }$ with values in    the weight $4$ polylogarithmic motivic complex. 
 In Section \ref{SEC8.4n} we extend  it     to a  cocycle    
   $C_{4, {n} }^{\cal M}\in \Gamma({\rm BGL}_{{n} \bullet}; 4)$.  
   Applying   the    regulator map $r_{\cal D}$, we get a cocycle  
$$
C_{4, {n}  }^{\cal H}:= r_{\cal D}(C_{4, {n} }^{\cal M})  \in \R_{\cal D}({\rm BGL}_{{n}  \bullet}(\C);4).
$$
We prove that its  cohomology class   is  equal to  Beilinson's class. This proves  the  injectivity of the map 
   $K_7(\F) \to H^1{\cal B}^\bullet(\F;4)$ for number fields, and   Zagier's conjecture on $\zeta_\F(4)$.

In fact the construction of the cocycle $C_{4, {n}  }^{\cal H}$  uses only simple features of the complex $\Gamma(X;4)$. So the reader may skip    details in Section \ref{Sec7}.

 \paragraph{2. Other applications.} 1. Our construction  implies that for a number field $\F$, the  group $K_7(\F)_\Q$ can be  realized in the motivic fundamental group of the punctured ${\Bbb P}^1$.  
This plus the Rigidity Conjecture on ${\rm gr}^\gamma_4K_7(\F)_\Q$ would imply   Conjecture \ref{MTM=DM} in the weight $4$.  

2. Let $E$ be an  elliptic curve over $\Q$. Part 3) of Theorem \ref{ZCZ4t2} implies a formula expressing  the $L$-value $L(E, 4)$ via 
 the generalizing Eisenstein-Kronecker series.  This proves Theorems 0.9 and 0.11 from \cite{Gon95} for $n=4$. 
 The $n=3$ case was proved in \cite{Gon95a}. 
 
3.  The explicit construction of the regulator map on the weight 4 
polylogarithmic motivic complex with values in the exponential Deligne complex \cite[Section 3]{Gon15} allows to refine 
 the construction of the class (\ref{CHHO})   to a class   with values in the integral Deligne cohomology $H^{8}({\rm BGL}_{{n} \bullet}(\C), \Z_{\cal D}(4))$. 
Using this, we get   
 a local combinatorial formula for the classical Chern class $c_4(E)\in H^8(X, \Z(4))$ of a complex vector bundle $E$ on a complex manifold $X$.   
 
 4. The cluster structure plays an important role in combinatorial formulas for the universal motivic Chern classes $c_n$ for $n\leq 4$. We expect this will happen for all $n$. 

5. A real variant of this story should give a local combinatorial formula for the second Pontryagin class,  generalizing  celebrated  Gabrielov-Gelfand-Losik's formula for the first 
Pontryagin class \cite{GGL75}. 


 \paragraph{Acknowledgments.}
This work was
supported by the  NSF grants  DMS-1564385,  DMS-1900743 and DMS-2153059. A part of the work was written when A.G. enjoyed the hospitality and support of IHES at Bures sur Yvette and IAS at Princeton. We are grateful to A. Levin, who pointed an issue in the definition of the complex $\Gamma(X;4)$. We are also grateful to S. Charlton for numerically verifying some of our computations. The Mathematica notebook with these computations is attached to the ArXiv version of the text. We would like to thank the referee for the extraordinary job. The referee's comments and suggestions  saved the text from a number of misprints and  errors and improved the exposition.

\section{Motivic correlators} \la{SEC8}

Motivic correlators were defined in \cite{Gon08}. 
 Below we recall  their definition and key features and then use them to give a natural definition of  
  motivic elements $\{x,y\}_{m-1,1}$ for all $m\geq 2$.
 
 Note that by the very definition elements $\{x,y\}_{m-1,1}$  live in the motivic Lie coalgebra, rather than in the motivic Hopf algebra.  Although one can define elements $\{x,y\}_{m-1,1}$  via   motivic multiple polylogarithms, the latter live in the motivic Hopf algebra. 
 
\paragraph{1. The specialization functor.} 
       Let $Z$ be a regular subvariety of a
regular variety $X$ over a field. Let $N^\circ_ZX$ be the normal bundle to $Z$ in $X$, with the zero
section removed. Set $U := X - Z$. Recall that the Verdier specialization functor is an exact functor between the bounded derived categories of constructible sheaf complexes: 
$$
{\rm Sp}_Z: D^b_{\rm Sh}(U) \lra D^b_{\rm Sh}(N^\circ_ZX).
 $$
 Since $Z$ is a regular subvariety of a regular variety $X$, it
  provides a Tate functor between the corresponding mixed Tate categories of unipotent variations of Hodge-Tate structures or lisse Tate
l-adic sheaves. Therefore it provides a map of Tannakian Tate Hodge / l-adic  Lie coalgebras. \vskip 2mm

{\it Example}. Assume  that $X={\Bbb A}^1, Z=\{0\}$ with a local parameter $\varepsilon$. There is a  tangent vector $  \partial / \partial \varepsilon\in N^\circ_{\{0\}}X$ at 
$\varepsilon =0$. The specialization   ${\rm Sp}_{\varepsilon =0}$    means the fiber of the specialization functor at  $  \partial / \partial \varepsilon$.  
We calculate  it on 
${\cal L}_{1}(\Q(\varepsilon)) = \Q(\varepsilon)^\times$  as follows. Let 
$F(\varepsilon) \in    \Q(\varepsilon)^\times$. 
Write $ F(\varepsilon) = \varepsilon^n F_\circ(\varepsilon)$ where 
$F_\circ(0)\not = 0$. Then ${\rm Sp}_{\varepsilon =0}F( \varepsilon ) =   F_\circ(0)$.

\paragraph{2. Motivic fundamental groups.} \label{SectionMotivicFundamentalGroup}
Let $\{s_0, s_1, \ldots, s_n\}$ be a collection of distinct points on ${\Bbb P}^1(\F)$, and $t_0$ a non-zero tangent vector at $s_0$, defined over $\F$. 
Then the motivic fundamental group
\[
\pi_1^{\cal M}(X, t_0), ~~~~X:= {\Bbb P}^1(\F) - \{s_1, \ldots, s_n\}
\]
is   defined  whenever we have motivic formalism available, see \cite{DG03}. So for now we must assume that  either  $\F$ is a number field, or 
 work in the  Hodge or l-adic realizations. 
 
 If $\F$ is a number field,  there is a mixed Tate category ${\cal M}_T(\F)$ of mixed Tate motives over $\F$. 
If $\F$ is not a number field, we have the related mixed Tate categories 
 in realizations. From now on, we work with the category ${\cal M}_T(\F)$. The same construction works in any mixed Tate category ${\cal C}$ in which we have the ${\cal C}-$motivic fundamental group available. Translation to other realizations is straightforward. 
 
Recall the canonical fiber functor $\omega$ on the category ${\cal M}_T(\F)$, see (\ref{FF1}), and the fundamental Lie algebra 
${\rm L}_\bullet(\F):= {\rm Der}^{\otimes}(\omega)$. Lie coalgebra ${\cal L}_\bullet(\F)$ is the graded dual of the Lie algebra ${\rm L}_\bullet(\F).$ Then there is a canonical equivalence of tensor categories
\[
\omega: {\cal M}_T(\F) \stackrel{\sim}{\lra}  \mbox{finite dimensional graded ${\rm L}_\bullet(\F)-$modules}.
\] 
The motivic fundamental group $\pi_1^{\cal M}(X, t_0) $ is a pro-Lie algebra object in the  category ${\cal M}_T(\F)$.   Therefore there is a canonical 
 homomorphism of Lie algebras
$$
{\rm L}_\bullet(\F) \lra {\rm Der}\Bigl(\omega(\pi_1^{\cal M}(X, t_0)\Bigr).
$$ 
  
  The tangential base point $t_0$ provides a canonical map $\varphi_{t_0}: \Q(1) \to \pi_1^{\cal M}(X, t_0)$. This implies that 
  the object $X_0:= \omega\circ \varphi_{t_0}(\Q(1))$ is killed by the action of ${\rm L}_\bullet(\F)$.   Furthermore, there are objects 
  $X_0, X_{s_1}, \ldots,   X_{s_n} \in  \omega(\pi_1^{\cal M}(X, t_0))$ such that 
$X_0 + X_{s_1} +  \ldots +  X_{s_n}  =0$ and the conjugacy classes of  the objects $X_{s_i}$, $i=1, \ldots, n$ are preserved  by the action of ${\rm L}_\bullet(\F)$. 
  
Denote by ${\rm Der}^S\Bigl(\omega ( \pi_1^{\cal M}(X, t_0))\Bigr)$ the Lie subalgebra of all {\it special} derivations of the graded 
  Lie algebra $\omega\left(\pi_1^{\cal M}(X, t_0)\right)$, defined as the derivations which have the properties as above. Then there is a canonical 
 homomorphism of Lie algebras
\be \la{HLA1}
 {\rm L}_\bullet(\F) \lra {\rm Der}^S\Bigl(\omega(\pi_1^{\cal M}(X, t_0)\Bigr).
\ee   
The target Lie algebra   has a simple combinatorial description, which we recall now. 
    
 Given a  vector space $V$ over a field, denote by ${\cal C}_{V}$ the cyclic tensor envelope of 
$V$:
$$
{\cal C}_{V}:= \bigoplus_{m=1}^{\infty}\left(\otimes^mV\right)_{\Z/m\Z}
$$
where the subscript $\Z/m\Z$ denotes the coinvariants of the cyclic shift.  We denote  the projection of the tensor 
$v_1 \otimes   \ldots \otimes v_m$ to the coinvariants of the cyclic shift by $(v_1 \otimes   \ldots \otimes v_m)_{\Z/m\Z}$. 

The subspace of {\it shuffle relations} in 
${\cal C}_{V}$  generated by the 
elements
$$
\sum_{\sigma \in \Sigma_{p,q}}  (v_0 \otimes v_{\sigma(1)} 
\otimes \ldots \otimes v_{\sigma(p+q)})_{\Z/(p+q+1)\Z}, 
\qquad p, q \geq 1, 
$$ 
where the sum is over all $(p,q)$-shuffles. 
 
Set
\[ \la{HHH}
{\cal C}{{\cal L}ie}_X:=  
\frac{ {\cal C}_{H^1(X)}}{\mbox{Shuffle relations}};
\qquad 
{\rm CLie}_{X}:= \mbox{the graded dual of 
${\cal C}{{\cal L}ie}_X$}.
\]
 The space ${\cal C}{{\cal L}ie}_X$ is equipped with a Lie coalgebra structure, which is reflected in Figure \ref{zg9} below.  So   ${\rm CLie}_X$ is   a Lie algebra. The following theorem was proven in \cite[Proposition 8.4]{Gon08}.

\bt \la{COR11} Let $\{s_0, s_1, \ldots, s_n\}$ be a collection of distinct points on ${\Bbb P}^1(\F)$, and $t_0$ a non-zero tangent vector at $s_0$, defined over $\F$. For $X=\mathbb{P}^1(\F)-\{s_1,\dots,s_n\}$ there is a canonical isomorphism of Lie algebras:
\be \la{KAP1}
\kappa:  {\rm CLie}_X \stackrel{\sim}{\lra} {\rm Der}^S\Bigl(\omega(\pi_1^{\cal M}(X, t_0))\Bigr).
\ee
\et

Composing the map (\ref{HLA1}) with the isomorphism (\ref{KAP1}),  and dualizing the composition, 
we arrive at a canonical map of graded Lie coalgebras, called the {\it motivic correlator map}:
$$
{\rm Cor}_{\cal M}: {\cal C}{\cal L}ie_X \lra  {\cal L}_\bullet(\F). 
$$

There is a collection of elements in ${\cal C}{\cal L}ie_X$ which span it as a vector space. Namely, these are the 
tensor products of the basis elements $X_{s_{i}}$ of $H_1(X)$ corresponding to the punctures $s_i$. In the Betti realization $X_s$ is the homology class of
 a little loop around the puncture $s$. Their images under the motivic correlator map are called motivic correlators:
$$
{\rm Cor}^{\cal M}_{t_0}(s_{i_1}, \ldots , s_{i_m}):= {\rm Cor}_{\cal M}\Bigl(X_{s_{i_1}} \otimes X_{s_{i_2}} \otimes \ldots \otimes X_{s_{i_m}} \Bigr)_{\Z/m\Z}
$$

Then ${\rm gr}^W_{2k}\pi_1^{\cal M}(X, t_0)$ is a finite direct sum  of copies of $\Q(-k)$. There is an isomorphism $f: {\rm gr}^W_{0}\pi_1^{\cal M}(X, t_0) \lra \Q(0) $. 
So any non zero map  $v: \Q(-k) \lra {\rm gr}^W_{2k}\pi_1^{\cal M}(X, t_0)$ for $k\geq 0$ provides a framed mixed Tate motive 
$\left(\pi_1^{\cal M}(X, t_0), v, f\right)$. The following theorem follows immediately from Theorem \ref{COR11} and the basic definitions.

\bt The element of the motivic Lie coalgebra ${\cal L}_\bullet(\F)$ induced by the framed mixed Tate motive  $\bigl(\pi_1^{\cal M}(X, t_0), v, f\bigr)$ is a linear combination of the weight $k$ correlators.
\et

\paragraph{3. Motivic correlators: a description.} 
 Given a collection of points $(a; b_0, \ldots, b_m)$ of ${\Bbb P}^1(\F)$, such that $a \not = b_i$, although some of the   points $b_i$  may coincide, and  
 a non-zero tangent vector $v$ at the point $a$ defined over $\F$, there is  a   motivic correlator:
 \be \la{MCR}
 {\rm Cor}^{\cal M}_{a, v}(b_0, \ldots , b_m) \in {\cal L}_m(\F).
\ee
 By the very definition, motivic correlator (\ref{MCR}) is invariant under  cyclic shifts of   points 
 $b_0, \ldots , b_m$:
  \be \la{MCR1}
 {\rm Cor}^{\cal M}_{a, v}(b_0, \ldots , b_m)  = {\rm Cor}^{\cal M}_{a, v}(b_1, \ldots , b_m, b_0).
\ee 

Motivic correlators of  weights $m>1$ do not depend on $v$. So we skip $v$ in the notation. 
 
 Motivic correlators with the base point at $\infty$ are invariant under the shift of the arguments: 
$$
{\rm Cor}^{\cal M}_{\infty, v}(b_0, \ldots , b_m) = {\rm Cor}^{\cal M}_{\infty, v}(b_0-c, \ldots , b_m-c).
$$

Motivic correlators of weight $m>1$ with the base point at $\infty$ are invariant under rescaling: 
$$
{\rm Cor}^{\cal M}_{\infty, v}(\lambda \cdot b_0, \ldots , \lambda \cdot b_m) = {\rm Cor}^{\cal M}_{\infty, v}(b_0, \ldots , b_m), ~~~~m>1.
$$

The key properties  of motivic correlators are the following:
 
 \begin{itemize}
 
 \item There is a simple formula for the cobracket:
 \be \la{wt2}
  \delta: {\rm Cor}^{\cal M}_a(b_0, \ldots , b_m)\lms  \sum_{i<j} {\rm Cor}^{\cal M}_a(b_i, \ldots , b_j) \wedge{\rm Cor}^{\cal M}_a(b_{j}, \ldots , b_{i-1}).
\ee 
The notation $i<j$ means that   the order of the points $b_i, \ldots , b_j$ in the first factor, and hence in the second, is compatible with the cyclic order of the points $b_k$. 

\item  Given an embedding $\F \hra \C$, the canonical real period of the Hodge realization of the 
 motivic correlator is calculated as the Hodge correlator integral \cite{Gon08}. 
 
\end{itemize}
    \begin{figure}[ht]
\centerline{\epsfbox{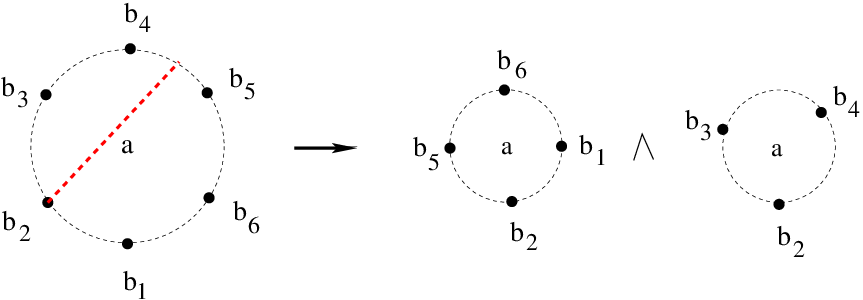}}
\caption{The cobracket of the weight $5$ motivic correlator  ${\rm Cor}^{\cal M}_{a}(b_1, \ldots , b_6)$. We cut the circle at a marked point and inside of an arc, 
making two new circles with marked points, and taking the wedge product of the related motivic correlators. Then we sum over all cuts. }
\label{zg9}
\end{figure}

These two properties determine  motivic correlators (\ref{MCR}) uniquely when $\F$ is a number field. 

 \paragraph{4. Simplest motivic correlators.} Here are  first examples:
\be \la{wt1}
\begin{split}
& {\rm Cor}^{\cal M}_{\infty, v_\infty}(b_0, b_1) = b_0-b_1\in \F^\times = {\cal L}_1(\F).\\
&{\rm Cor}^{\cal M}_{b_0}(b_1,b_2,b_3)= -{\rm Li}^{\cal M}_2([b_0,b_1,b_2,b_3]).
\end{split}
\ee

In the first example,  the tangent vector $v_\infty$  determines  the standard coordinate on ${\Bbb A}^1$. The formula follows since 
by the very definition the corresponding Hodge correlator is just $\log|b_0-b_1|$. 
To prove the second formula, note that  cobracket formula (\ref{wt2}) plus the first formula in (\ref{wt1}) imply
$$
\delta {\rm Cor}^{\cal M}_\infty(b_1, b_2, b_3) = (b_1-b_2) \wedge (b_2-b_3) + (b_2-b_3) \wedge (b_3-b_1) +(b_3-b_1) \wedge (b_1-b_2). 
$$
Therefore one has, where  ${\rm Li}^{\cal M}_2(*)$ denotes the projection of the motivic dilogarithm \cite{Gon00} to the motivic Lie coalgebra ${\cal L}_\bullet(\F)$: 
$$
{\rm Cor}^{\cal M}_\infty(b_1, b_2, b_3) + {\rm Li}^{\cal M}_2([\infty,b_1,b_2,b_3]) \in {\rm Ext}_B^1(\Q(0), \Q(2)).
$$
The Ext group is over the base $B$ parametrizing triples of points  $(b_1, b_2, b_3)$. 
Thanks to the rigidity of ${\rm Ext}^1(\Q(0), \Q(2))$ in the Hodge realization (see \cite[Lemma 8.2]{Gon02b}), it is constant. 
Since  ${\rm Ext}_{{\cal M}_T(\Q)}^1(\Q(0), \Q(2))=0$, 
specializing to a rational point of $B$ we get the second formula in (\ref{wt1}).

 \paragraph{5. The Lie coalgebra ${\Bbb L}_{\leq 4}(\F)$  and motivic correlators.}  Elements  $\{x\}_m$ and 
 $\{x,y\}_{m-1,1}$ which appear in the definition of   spaces ${\Bbb L}_m(\F)$ are   most naturally
 expressed  as motivic correlators. Below we use motivic correlators to define   these   elements.
 
 \bd \la{DEFMC} Motivic correlators $\{x\}^{\cal M}_m$ and 
 $\{x,y\}^{\cal M}_{m-1,1}$ are defined as follows:
   \be \la{144ads1}
\begin{split}
&\left \{x,y \right\}^{\cal M}_{2,1}:={\rm Cor}^{\cal M}_\infty(0,x,1,y), \\
& \left \{x,y\right\}^{\cal M}_{3,1} := -{\rm Cor}^{\cal M}_\infty(0,0,x,1,y),\\
&\{y\}^{\cal M}_2:= -{\rm Cor}^{\cal M}_\infty(0,1,y),\\
&\{y\}^{\cal M}_3:= -{\rm Cor}^{\cal M}_\infty(0,0,1,y),\\
&\{y\}^{\cal M}_4:= -{\rm Cor}^{\cal M}_\infty(0,0,0,1,y).\\
\end{split}
\ee    \ed 

 One can introduce  similar weight $m$ motivic correlators   as follows:
\[
\begin{split}
&\{[\infty,  1, 0, x], [\infty,  1, 0, y ]\}^{\cal M}_{m-1,1} = (-1)^{m-1}{\rm Cor}^{\cal M}_\infty(0, \ldots , 0, x, 1, y).\\ 
& \left\{[a, b , c, d]\right \}^{\cal M}_m = - {\rm Cor}^{\cal M}_a(\underbrace{b, \ldots , b}_{m-1}, c, d).\\
\end{split}
\] 

 Using motivic correlators,  relation (\ref{144ac})  acquires a neat form: 
\[
\begin{split}
&{\rm Cor}^{\cal M}_a(b_1, b_2, b_3, b_4) + 
 \sum_{i \in \Z/4\Z} {\rm Cor}^{\cal M}_a(b_i, b_i, b_{i+1}, b_{i+2})  = - \{[b_1, b_2, b_3, b_4]\}^{\cal M}_3 - \{1\}^{\cal M}_3.\\
\end{split}
\]

 \bt \la{JMC} Motivic correlators $\left \{x \right \}^{\cal M}_{2}$,$\left \{x \right \}^{\cal M}_{3}$, $\left \{x \right \}^{\cal M}_{4}$,  $\left \{x,y \right \}^{\cal M}_{2,1}$, $\left \{x,y \right \}^{\cal M}_{3,1}$ 
  give rise to a map  
 of Lie coalgebras 
\[
\begin{split}
& {\Bbb L}_{\leq 4}(\F) \lra {\cal L}_{\leq 4} (\F),\\
&\{x\}_m \lms \{x\}^{\cal M}_m, ~~\{x,y\}_{m-1, 1} \lms \{x,y\}^{\cal M}_{m-1, 1}.\\
\end{split}
\]
  \et
 
 \paragraph{Remark.} Below we skip superscript ${\cal M}$ from the notation   $\{x,y\}^{\cal M}_{m-1, 1}$ and $\{x\}^{\cal M}_{m}$. 
 So for example $\{x,y\}_{m-1, 1}$ may mean either a generator of the combinatorially defined space ${\Bbb L}_{m}(\F)$, or its image in 
 $ {\cal L}_{m} (\F)$. Referring to $\{x,y\}_{m-1, 1}$ as a motivic correlator   we mean the latter; referring to it as an element    we mean the former. 
 We hope the reader will not be confused.  We also skip the superscript ${\cal M}$ from the notation of motivic correlators. 
 
 \vskip 3mm
 
 The proof of Theorem \ref{JMC} consists of the two steps.  First, we need to prove that the maps in Theorem \ref{JMC} are compatible with the cobrackets. 
 This is what  Proposition \ref{LL1} does,  justifying   Definition \ref{DEFMC} of  the motivic correlators $\{x,y\}_{m-1,1}$.  
 Second, we need to prove that relations ${\bf Q}_m$ map to zero, as well as the specialization relations for $m=4$. 
The claim for the relations between the elements $\{x\}_m$ is similar and well known. 
  The   second claim is proved at the end of Section \ref{SEC9}, after we prove that the cobracket in ${\Bbb L}_{\leq 4}(\F)$ kills relations   ${\bf Q}_m$.

 \bp \la{LL1} The cobracket of   motivic correlators $\{x,y\}_{2,1}$ and $\{x,y\}_{3,1}$ is given by the   four-term formulas (\ref{26}). 
\ep

\begin{proof}   i) The  cobracket for weight three motivic correlators looks as follows:
\be \la{26aA}
\delta {\rm Cor}_a(b_1, b_2, b_3,  b_4)  = \sum_{i\in \Z/4\Z}\left \{[a, b_{i-1}, b_{i+1}, b_{i+2}]\right \}_2 \otimes \frac{b_i-b_{i-1}}{b_i-b_{i+1}} \in \B_2(\F) \otimes\F^\times.
\ee
Therefore one easily recognizes in (\ref{26aA})  the formula for the cobracket $\delta\{x,y\}_{2,1}$:
$$
\delta {\rm Cor}_\infty(0,x,1,y)  =  \left \{y\right \}_2 \otimes (1-x^{-1})   +\left \{\frac{x}{y}\right \}_2 \otimes \frac{1-x}{1-y} + 
\left \{x\right \}_2 \otimes (1-y^{-1})  +\left \{\frac{x-1}{y-1}\right \}_2 \otimes \frac{ x}{ y} .
$$
Observe also that we get the boundary of (\ref{144a}) by adding to the right-hand side the expression 
$$
\Bigl(-\{x\}_2+\{y\}_2 - \{y/x\}_2 +  \{(1-y)/(1-x)\}_2 -\{(1-1/y)/(1-1/x)\}_2\Bigr)\otimes \frac{1-y^{-1}}{1-x^{-1}}.
$$
 Here the first factor is a five-term relation. 
 
 ii)  The cobracket formula for weight four motivic correlators looks as follows:
\[
\bs
\delta {\rm Cor}_a(b_1, b_2, b_3,  b_4, b_5)  = &\sum_{i\in \Z/5\Z}{\rm Cor}_a(b_i, b_{i+1}, b_{i+2},  b_{i+3})\wedge \frac{b_{i+4}-b_{i+3}}{b_{i+4}- b_{i+5}} ~~+\\
&  \sum_{i\in \Z/5\Z}{\rm Cor}_a(b_i, b_{i+1}, b_{i+2})\wedge  {\rm Cor}_a( b_{i+2}, b_{i+3}, b_{i+4}).
\\
\end{split}
\]
Specializing this formula, we get:
\[
\bs
&\delta {\rm Cor}_\infty(0,0,x,1,y)  =   \\
&{\rm Cor}_\infty(0, x, 1, y)\wedge  {y} - {\rm Cor}_\infty(x, 1, y, 0)\wedge  {x} + 
{\rm Cor}_\infty(1, y, 0, 0)\wedge \frac{x}{x-1}  \\
& + {\rm Cor}_\infty(y, 0, 0, x)\wedge \frac{x-1}{y-1} + {\rm Cor}_\infty( 0, 0,x, 1)\wedge \frac{y-1}{y} + \{x\}_2\wedge  \{y\}_2.\\
\end{split}
\]
The $\B_3\otimes \F^\times$ component is given by 
\[
\bs
&\{x,y\}_{2,1}\wedge  \frac{y}{x}  + \left\{\frac{y}{x}\right\}_3 \wedge \frac{y-1}{x-1} + \{y\}_3  \wedge (1-x^{-1}) - \{x\}_3  \wedge (1-y^{-1}).\\
\end{split}
\] 
The claim follows by comparing with the four term formula (\ref{26}) for the cobracket $\delta\{x,y\}_{3,1}$.

iii) Specializing  $x = 0$ in $\{x,y\}_{2,1}$,   and using (\ref{144ac}) we get the trilogarithm: 
\[
\begin{split}
&{\rm Cor}_\infty(0,0,1,y)= \left \{0,y \right\}_{2,1} \stackrel{(\ref{EATRI})}{=}   -\left\{ {y} \right \}_3.\\
\end{split}
\]

iv) Specializing  $x = 0$ in $\{x,y\}_{3,1}$,   and using (\ref{202}) we get the tetralogarithm: 
\[
\begin{split}
&{\rm Cor}_\infty(0,0,0,1,y)= \left \{0,y \right\}_{3,1} \stackrel{(\ref{202})}{=}   -\left\{ {y} \right \}_4.\\
\end{split}
\]
 \end{proof} 

\paragraph{6. The Hodge correlator construction \cite{Gon08}.} For convenience of the reader we briefly recall the definition of  numbers ${\rm Cor}^{\cal H}_{a, v_a}(b_1, \ldots ,  b_m)$.
  First, recall  
 the Green function $G_{v_a}(x,y)$ on $X\times X$, where $X=\C{\Bbb P}^1$, which  is determined up to a constant by the differential equation 
$$
\frac{1}{2\pi i}\overline \partial \partial G_{v_a}(x,y) = \delta_\Delta - \{a\}\times X - X \times \{a\}. 
$$
Here $\Delta \subset X \times X$ is the diagonal. The constant is determined by the tangent vector $v_a$. 
For example, if $a = \infty, v_a=\partial / \partial t$, then $G_{v_a} (x,y)= \log|x-y|$.

Take a plane trivalent tree $T$ whose ends are decorated by the points $b_1, \ldots, b_m$.  
Let us define a differential form $\kappa_T$ on ${\C{\Bbb P}}^{{\cal V}(T)}$, where ${\cal V}(T)$ is the set of 
internal vertices of $T$. Given an edge $E$ of the tree $T$, external or internal, with the vertices $v_i, v_j$, we assign to the edge $E$ the Green function $G_E:= G_{v_a}(x_i,x_j)$
 on $X\times X$.

    \begin{figure}[ht]
\centerline{\epsfbox{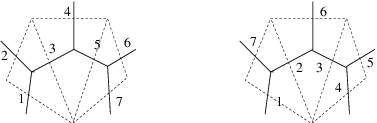}}
\caption{The canonical orientation of a plane trivalent tree for two different plane orientations. }
\label{zg13}
\end{figure}
Recall that there is a canonical orientation of a plane trivalent tree $T$, provided by the plane orientation. This means that there is a class of  
enumerations of the edges $E_1, \ldots, E_{2m-3}$, well defined up to an even permutation. Recall $d^\C:= \partial - \overline \partial$. 
Then we set 
$$
\kappa^a_T(a, v_a; b_1, \ldots ,      b_m):= {\rm Alt}_{2m-3}\Bigl(G_{E_{2m-3}}  d^\C G_{E_1} \wedge \ldots \wedge d^\C G_{E_{2m-4}}\Bigr).
$$
This differential form is integrable. We define the Hodge correlator by integrating it over ${\C{\Bbb P}}^{{\cal V}(T)}$ and taking the sum over all plane trivalent trees $T$ decorated by the points $b_1, \ldots, b_m$:
$$
{\rm Cor}^{\cal H}_{a, v_a}(b_1, \ldots ,     b_m):= \sum_T\frac{1}{{\rm Aut}(T)}\int_{{\C{\Bbb P}}^{{\cal V}(T)}} \kappa^a_T(a, v_a; b_1, \ldots ,     ,b_m).
$$
The collection of these numbers allows to define a real mixed Hodge structure on the pro nilpotent completion of $\pi_1(\C{\Bbb P}^1-\{b_1, \ldots, b_m\}, v_a)$  (see \cite[\S 1.10]{Gon08}). 

\paragraph{7. Reduction of motivic correlators to lower depth.} 
\bt \la{CHPI1} Let $a, b, x_i \in {\Bbb P}^1$. Suppose that $a \not = b$. Then for any $m>2$ one has 
\be \la{FOR1}
\begin{split}
&{\rm Cor}_a(x_1, \ldots x_m) =  
 \sum_{k=0}^m\sum_{1 \leq i_1 < \ldots < i_k  \leq m} (-1)^k {\rm Cor}_b(x_1, \ldots ,    x_{i_1-1}, a,x_{i_1+1}, \ldots ,x_{i_2-1}, a, x_{i_2+1}, \ldots ,x_m).
  \end{split}
\ee
Here the $k$-th term is $(-1)^k$ times a sum over $1 \leq i_1 < \ldots < i_k  \leq m$ of  motivic correlators obtained from the correlator ${\rm Cor}_b(x_1, \ldots x_m)$ by inserting  
  the point $a$    instead of  points $x_{i_1}, \ldots , x_{i_k}$.  
\et

 For example, the first few terms are 
\[
\begin{split}
&{\rm Cor}_a(x_1, \ldots, x_m) = {\rm Cor}_b(x_1, \ldots, x_m)  - 
 \sum_{1\leq i_1\leq m} {\rm Cor}_b(x_1,  \ldots, x_{i_1-1}, a, x_{i_1+1}, \ldots ,x_m)  \\
 &+\sum_{1 \leq i_1<i_2\leq m} {\rm Cor}_b(x_1,  \ldots, x_{i_1-1}, a, x_{i_1+1}, \ldots, x_{i_2-1}, a, x_{i_2+1},  \ldots, x_m) - \ldots \\
 \end{split}
\]

When $m=3$ we get   nothing else but the five-term relation for the motivic dilogarithm:
\[
{\rm Cor}_a(x_1, x_2, x_3) = {\rm Cor}_b(x_1, x_2, x_3) - {\rm Cor}_b(a, x_2, x_3) - {\rm Cor}_b(x_1, a, x_3) - {\rm Cor}_b(x_1, x_2, a).
\]

\begin{proof} We start with an analytic argument. 
The relation between the Green functions assigned to  base points $a$ and $b$ is reflected in the following  identity: 
\be \la{EQa}
G_{v_a}(x,y) =  G_{v_b}(x,y) - G_{v_b}(a,y) -G_{v_b}(x,a) +C.
\ee

 Let us calculate the difference 
 \be \la{DIFF1}
 \int_{X^{m-2}}(\kappa_T^a - \kappa_T^b).
\ee 
Take an internal edge $e$ of the tree $T$.   We claim that the contribution of the summands $C$, $G_{v_b} (a,y) $ and $G_{v_b} (x,a) $ from (\ref{EQa}) are zero. Indeed, cutting the edge $E$,  the tree $T$ splits into a disjoint union   $T= T_1 \cup T_2$. 
Therefore we can split $X^{m-2}= X^{p_1}\times X^{p_2}$, where the $i$-th 
  factor corresponds to the copies of $X$ parametrized by the internal vertices of the tree $T_i$. 
 Since one of the arguments of $G_{v_b} (a,y) $, as well as $G_{v_b} (x,a) $, is a constant, the integral  splits into a product of two integrals; it is easy to see that at least one of theses integrals vanishes. Therefore for an internal edge $E$ of $T$, only the summand $G_{v_b}(x,y)$  contributes to (\ref{DIFF1}). 
  
 It remains to calculate the contribution of the external edges $E$ of the tree $T$. The same argument shows that the   contribution  of $C$, $G_{v_b}(a,y)$ and $G_{v_b} (x,a)$, where 
 $a$ is at an internal vertex, is zero. This  implies the claim. 
 
 This implies that  formula (\ref{FOR1}) is valid for the real Hodge realization of motivic correlators. The claim for  motivic correlators 
 with arguments in a number field follows from this by using the injectivity of the regulator map, see \cite{Bor77}. The   identity for the $\Q$-Hodge realization can be deduced  from the 
 motivic one for arbitrary number fields.  Alternatively, one checks by the induction on the weight that the cobracket of the both sides of the identity is the same. Then the rigidity, see Lemma \ref{RL} below,  plus  the injectivity of the regulator imply the claim in the $\Q$-Hodge and l-adic realizations.  Indeed, a framed variation with zero cobracket is equivalent to an extension from Lemma  \ref{RL}. 
 
 Here by the rigidity we mean the following well known Lemma (see \cite[Lemma 8.2]{Gon02b}):
 \bl \la{RL} Let $ \underline{{\cal L}}$ be a two dimensional variation of mixed Hodge structures, or lisse l-adic sheaf over a base  ${\cal Z}$, which is an extension
\[
0\lra \underline{\Q}(m) \lra   \underline{{\cal L}} \lra \underline{\Q}(0) \lra 0, ~~~~~~ m\geq 2.  
\]
 where $\underline{\Q}(i)$ is the constant sheaf $\Q(i)$  on ${\cal Z}$ in either of these two realizations. Then $ \underline{{\cal L}}$ is a constant sheaf, that is isomorphic to a pull back from the point.
 \el
 
 \begin{proof} The proof in the Hodge realization, which is the only one we need in the paper,  follows immediately from the Griffiths transversality. 
 \end{proof}
\end{proof}

\bl Consider a collection of points $a, x_2,\dots, x_m \in \mathbb{P}^1$ such that  $x_i\neq a$ for some $2\leq i\leq m.$ The following specialization is zero:
\be \la{EX10}
{\rm Cor}_a(a, x_2, \ldots,  x_m) := {\rm Sp}_{x_1 = a}{\rm Cor}_a(x_1, x_2, \ldots,  x_m)  =0.
\ee
\el

\begin{proof} Without loss of generality, we can assume that $a=\infty.$ 

Let us  prove by the induction on $m$ that the cobracket  is zero:
$$
\delta {\rm Sp}_{x_1 = a}{\rm Cor}_a(x_1, x_2, \ldots,  x_m) =0.
$$
We have ${\rm Sp}_{x_1 = \infty}{\rm Cor}_{\infty, v_\infty}(x_1, x_2) =0$. Indeed, ${\rm Cor}^{\cal M}_{\infty, v_\infty}(x_1, x_2) = x_1-x_2$.\footnote{Here is how this works in the Hodge realization: $\log|t-x| = \log |t| + \log |1-x/t|$. So when $t \to \infty$, the second term has limit zero, while the first is zero because of the special choice of the tangential base vector.} 
Using either this or the induction assumption, we see that the cobracket vanishes  for  $m\geq 3$. 

By the rigidity, the left-hand side in (\ref{EX10}) is a constant. 
The shuffle relation for the shuffle of $\{x_2\}$ and $\{x_3, \ldots, x_m\}$, which we refer to as the shuffle $\{2\}\ast \{3, \ldots, m\}$,  implies that this constant is zero. 
\end{proof}

\bp \la{MOTCOIN1}   If $m > 3$, any motivic correlator   is a linear combination of the motivic  correlators ${\rm Cor}_\infty(x_1, \ldots,  x_m)$ 
where at least two of the  points $x_i$ coincide.\footnote{ The statement of Proposition \ref{MOTCOIN1} is false for $m=3$. } 
\ep

\begin{proof} We say that a motivic correlator has {\it lower depth} if it has at least two coinciding arguments $x_i=x_j$. We denote equalities of motivic correlators 
 modulo lower depth by $\sim$. Then, using (\ref{EX10}),  Theorem \ref{CHPI1} implies that under the specialization  $x_1 = b$ we have:
\[
\begin{split}
&{\rm Cor}_a(b, x_2,, \ldots x_m) \sim  - 
   {\rm Cor}_b(a, x_2,   \ldots ,x_m).  \\
 \end{split}
\]

Thanks to the  cyclic symmetry (\ref{MCR1}) of  motivic correlators, ${\rm Cor}_{y_1}(y_2, \ldots , \ldots y_{m+1})$ is antisymmetric modulo lower depth under  transpositions 
$(1, k)$ for $k=2, \ldots, m+1$. These transpositions generate the symmetric group $S_{m+1}$. 

Let  $m=2k$. Then   cyclic shift  (\ref{MCR1}) is an odd permutation. 
Thus  
 ${\rm Cor}_{a}(x_1, \ldots , \ldots x_{m})\sim 0.$ Both arguments are motivic, that is work in realizations and for a number field. 

 Let $m=2k+1>3$. We apply the shuffle relation for the shuffle $\{1, 2\}\ast\{3, \ldots, 2k\}$ in the correlator ${\rm Cor}_{a}(x_0, x_1, \ldots ,  x_{2k})$.
 Let us elaborate the  $m=5$ example,   crucial for the paper:
\[
\begin{split}
& 0 = {\rm Cor}_{a}( x_0,  x_1,   x_2,  x_3, x_4 ) + {\rm Cor}_{a}( x_0,   x_1, x_3,  x_2, x_4 ) + {\rm Cor}_{a}( x_0,   x_1, x_3,x_4, x_2 )   +\\
& {\rm Cor}_{a}( x_0,  x_3, x_1, x_2,    x_4   )+ {\rm Cor}_{a}( x_0,   x_3, x_1, x_4, x_2   )   +  {\rm Cor}_{a}( x_0,   x_3,  x_4,   x_1,   x_2)\sim \\
&   {\rm Cor}_{a}( x_0,   x_1, x_3, x_4, x_2 )     +  {\rm Cor}_{a}( x_0,   x_3,  x_4,   x_1,   x_2) \sim 2 \cdot {\rm Cor}_{a}( x_0,   x_1,  x_3, x_4, x_2 ). \\
\end{split}
\]
Here we use the antisymmetry of correlators under transpositions modulo  $\sim$ equivalencies. For example,  
${\rm Cor}_{a}( x_0,  x_1,   x_2,  x_3, x_4 ) + {\rm Cor}_{a}( x_0,   x_1, x_3,  x_2, x_4 )\sim 0$. In general   we use  the following:
 
1) The shuffle  $\{2\}\ast \{3, \ldots, 2p\}$ produces a sum off an odd number of  correlators, and all but the last one are combined into pairs such that the sum of the terms in a pair   is $\sim 0$. 

2) Similarly   the shuffle $\{2\}\ast \{3, \ldots, 2p+1\}$ produces a sum which is  $\sim 0$. 

3) Correlators   appearing after step 1)   differ by  even permutations, and thus are $\sim$. 
  
So we proved the claim for   Hodge correlators, and hence for the real Hodge realization, and therefore for number fields. 
 The latter is sufficient to deduce the claim for $\Q$-Hodge and {'e}tale realizations. 
 \end{proof}

\bc \la{WTRD43} Any  weight $4$ motivic correlator is a sum of the ones $\{x,y\}_{3,1}$ and $\{x\}_4$.
\ec

\begin{proof}  Proposition \ref{MOTCOIN1} and (\ref{144ads1}) imply that any weight $4$ motivic correlator is a sum of lower depth correlators.  
\end{proof}

\paragraph{Remark.} It was known that any weight $n$ multiple polylogarithm can be reduced  to a linear combination of the ones of depth $\leq n-2$, considered modulo products, see  \cite{Bro09}, \cite{Dan08}, \cite{Dan11}, and \cite{Ch17}. Using  relation (\ref{36}),  Corollary \ref{WTRD43} follows from that.

 \paragraph{8. Reduction to motivic classical polylogarithms.} Below  $\stackrel{\shuffle}{=}$  means equivalent modulo products, that is  in the motivic Lie coalgebra. 

 \bl \la{LEM2.9}
Let $X$ be an algebraic variety over $\Q$. Let $\A$ be any $\Q$-linear combination of  motivic correlators $\{f\}_4$ and $\{g, h\}_{3,1}$ for $f,g,h \in \Q(X)$ such that the    $\Lambda^2\B_2$-component $\delta_{2,2}(\A)$ of the   cobracket  vanishes. Then in the motivic Lie coalgebra $\A$ is a linear combination of motivic classical $4$-logarithms $\{f\}_4$ for $f\in \Q(X).$
\el
\begin{proof}
For $F_1,\dots,F_5\in \Q(X)$ consider the five-term relation
\[
r_5(F_1, \dots , F_5)=\sum_{i=1}^5(-1)^5 \left\{[F_1^{(k)},  \ldots , \widehat F_i^{(k)}, \ldots F_5^{(k)}] \right\} \in \Q[\Q(X)].
\] 
 
Assume that 
\[
\A= \sum_i a_i\{f_i\}_4 + \sum_j b_j  \{g_j, h_j\}_{3,1}, ~~~~a_i, b_j\in \Q, ~~f_i, g_j, h_j \in \Q(X).  
\]
We have 
\[ 
\delta_{2,2}\Bigl(\sum_j b_j ( \{g_j, h_j\}_{3,1})\Bigr) = \sum_j b_j (\{g_j\}_2 \wedge \{h_j\}_{2}),
\]
so 
\[
\sum_j b_j \cdot \{g_j\} \wedge \{h_j\}=\sum c_k \left(r_5\Bigl(F_1^{(k)},   F_2^{(k)}, F_3^{(k)}, F_4^{(k)}, F_5^{(k)}\Bigr)\wedge G_k \right)
\] 
for $c_k\in \Q$ and $F_i^k, G_k\in \Q(X).$
Since this is an identity in $\Lambda^2(\Q[\Q(X)] )$, it   implies that 
$$
\sum_j b_j  \{g_j, h_j\}_{3,1}  - \sum_k c_k  \sum_{i=1}^5(-1)^5\left \{[F_1^{(k)},  \ldots , \widehat F_i^{(k)}, \ldots F_5^{(k)}], G_k \right \}_{3,1} =0.
 $$
Theorem \ref{L4TH} (or \ref{THRN}) implies that the second summand in this formula is a sum of  motivic classical $4$-logarithms. 
\end{proof}

\section{Cluster $K_2$-varieties and cluster   polylogarithm maps} \la{SEC2N}

Cluster algebras were invented in  \cite{FZ01}. Cluster $K_2$-varieties are a geometric variant of cluster algebras. As the name suggests, they come equipped with a canonical class in $K_2$  \cite{FG03a}. 

For simplicity of the exposition, we discuss cluster varieties only in the simply-laced case, which means that we assume 
 $\varepsilon_{ij}= - \varepsilon_{ji}\in \Z$ in Definition \ref{D3.1}. 

\subsection{$K_2$-morphisms and cluster transformations} \la{SEC7.1}

 The exposition  in Section \ref{SEC7.1}  is borrowed from  \cite[Section 6]{FG03a}.

\paragraph{$K_2$-tori and $K_2$-morphisms.} Below, we always consider $K_2-$groups modulo $2-$torsion. The Milnor $K_2$-group of $\F$ is given by 
$$
{\rm K}_2^M(\F)=  \frac{\Lambda^2\F^{\times}}{\langle x\wedge (1-x),\  x\in \F^{\times}\setminus \{1\} \rangle}.
$$

\begin{definition} \la{D3.1}
A $K_2$-torus of rank $n\geq 1$ is a pair $(T, W)$, where $T$  is a split torus over $\Z$, and  
$W\in \Lambda^2\Z[T]^{\times}\otimes\Q$, where $\varepsilon_{ij}= - \varepsilon_{ji} $ below are integers:
\[
T={\rm Spec}\bigl(\Z[A^{\pm 1}_1,\ldots, A^{\pm 1}_n]\bigr), ~~~~~~~W=\frac{1}{2}\sum_{i,j=1}^{n} \varepsilon_{ij}A_i \wedge A_j = \sum_{1\leq i<j\leq n}  \varepsilon_{ij}A_i \wedge A_j.
\]
\end{definition}

A $K_2$-torus is the same thing as  a pair $(\Lambda, \varepsilon)$, where   $\Lambda$  is a lattice and $\varepsilon\in \Lambda^2(\Lambda)$ - a antisymmetric bilinear form  
on the dual lattice. 
The torus is recovered by $T:= {\rm Hom}(\Lambda, {\Bbb G}_m)$, so that $\Z[T]=\Z[\Lambda]$. 
So $W$ is just a coordinate expression of   $\varepsilon$. 
The name comes from the fact that $W$ defines an element $[W]$ in the Milnor 
$K_2$-group of the function field $ \Q(T) $  of the torus. 
Applying   map (\ref{MK2})  we get a closed $2$-form with logarithmic singularities on the $K_2-$torus:
\[
\frac{1}{2}\sum_{i,j=1}^{n} \varepsilon_{ij} d \log (A_i)  \wedge d\log (A_j).
\]

\begin{definition}
A $K_2$-morphism  of $K_2$-tori $(T_1,W_1)\to (T_2,W_2)$  is a birational map $f \colon T_1 \to T_2$ such that $[f^*W_2]=[W_1]$ in  $K_2$, modulo $2$-torsion.
\end{definition}

A trivial example of a $K_2$-morphism is a map of  tori preserving the forms: $f^*W_2=W_1$.  The first examples of non-trivial $K_2$-morphisms, which turned out to be  the main source of  $K_2$-morphisms,    are given by {\it cluster transformations}, whose definition we recall shortly.  

\paragraph{Seeds and mutations.} Consider a $K_2-$torus $(T, W)$ and choose a basis $A_1, \ldots, A_n$ of the lattice of characters of $T$. Such a  triple $(T, W, \{A_i\})$, where the coordinates $A_i$ are parametrized by a set $I$, is called {\it a seed}. 
It is the same thing as an oriented quiver, whose vertices  are indexed by the elements of the set $I$, so that the vertices  $i$ and $j$ are connected by $|\varepsilon_{ij}|$ arrows going from $i$ to $j$ if  
 $\varepsilon_{ij}>0$ and in the opposite direction otherwise. 
 
 For each  
$1 \leq k \leq n$ we consider a new $K_2$-torus $(T', W')$ with basis $A_1', \ldots A_n'$ and the form
\[
\begin{split}
&W'=\frac{1}{2} \sum_{i,j=1}^{n} \varepsilon_{ij}'A_i' \wedge A_j',\\
 \end{split}
\]
where the new matrix $\varepsilon'_{ij}$ is given by the  Fomin-Zelevinsky formula \cite{FZ01}:
\[
\begin{split}
&\varepsilon_{ij}':=
\begin{cases} 
-\varepsilon_{ij}, & \mbox{if } k \in \{i,j\}, \\ 
\varepsilon_{ij}, & \mbox{if } \varepsilon_{ik}\varepsilon_{kj}\leq 0, \\
\varepsilon_{ij}+|\varepsilon_{ik}|\varepsilon_{kj}, & \mbox{if } \varepsilon_{ik}\varepsilon_{kj} > 0. \\
\end{cases}\\
\end{split}
\]
There is   a birational transformation    
$
\mu_k\colon T \lra T',
$ 
called a mutation at vertex $k$, defined by setting   \cite{FZ01}:
\[
\begin{split}
&\mu_k^* (A_i')=A_i  ~~\mbox{for $i \neq k$, and}\\
&\mu_k^* (A_k')=\frac{\prod_{j|\varepsilon_{kj}>0}A_j^{\varepsilon_{kj}}+\prod_{j|\varepsilon_{kj}<0}A_j^{-\varepsilon_{kj}}}{A_k}. \\
\end{split}
\]

 The crucial fact is that although the mutation $\mu_k$ does not preserve  $W$, 
it is   a $K_2$-morphism. 
Even better, for a seed given by a $K_2$-torus $(T,W)$ and a basis  $A_1, \ldots A_n$, consider the following regular function on the torus:
\be \la{FXk}
X_k=\prod_{1\leq j \leq n}A_j^{\varepsilon_{kj}} \in \Z[T].
\ee
Then according to   \cite[Proposition 6.2]{FG03a}
\be \la{WEQ}
W-W'=(1+X_k) \wedge X_k = \delta \{-X_k\}_{2}.
\ee
Motivated by relation (\ref{WEQ}),  we  assign to a mutation $\mu_k$   an element in the Bloch group 
\be \la{XELL}
\{-X_k\}_2\in \B_2(\Q(T)).
\ee

\paragraph{Cluster transformations and cluster $K_2$-varieties.}  
An  {\it isomorphism} of seeds $
(T, W, \{A_i\}) \to (T', W', \{A'_{i'}\})
$
   is  given by  an isomorphism of tori $\sigma: T \to T'$ together with    a bijection $\overline \sigma: I \to I'$  such that 
  $\sigma^*W'=W$ and $ \sigma^*A'_i = A_{\overline \sigma^{-1}(i)}$. 
  
Compositions of mutations and isomorphisms  are called {\it cluster transformations}.  Thanks to (\ref{WEQ}), every cluster transformation is a $K_2$-morphism.

\bd A cluster $K_2$-structure on an algebraic variety ${\cal A}$ over $\Q.$

A variety ${\cal A}$ has a cluster $K_2$-structure if there is a connected groupoid   ${\cal C}$ of seeds and cluster transformations,  so that    
for every seed ${\bf s} = (T, W, \{A_i\})$  in ${\cal C}$  there is a regular map, which is a birational isomorphism, 
$$
\psi_{\bf s}: {\cal A} \lra T,
$$
and for every mutation $\varphi$ between seeds $s=(T, W, \{A_i\})$  and  $s'=(T', W', \{A_{i'}'\})$ in   ${\cal C}$ there is a commutative diagram 
\[
\centerline{
  \xymatrix{
    T \ar[rr]^{\mathcal{ \varphi}}  & & T'\\
    &\ar[ul]^{\psi_{\bf s}}{\cal A}\ar[ru]_{\psi_{\bf s'}}&
  }
}
\]

An algebraic variety ${\cal A}$ over $\Q$  is called a cluster $K_2$-variety if it admits
a non-trivial cluster $K_2-$structure. In this case we call ${\cal A}$ a cluster
$K_2-$variety, or cluster variety for short.\ed

The algebra ${\cal O}({\cal A})$ of regular functions on a cluster
$K_2$-variety ${\cal A}$ consists of all rational functions $F \in \Q({\cal A})$ such that for each seed ${\bf s}$ the function $F$ is expressed as 
a Laurent polynomial 
of the variables  $\{\psi_{\bf s}^*A_i\}$.

The most basic   variety with a cluster 
$K_2-$structure is   the moduli space ${\rm Conf}_n(V_2)$ of generic configurations of $n$ vectors   in a two dimensional 
symplectic vector space, see Section \ref{sec8.2}. The objects of the groupoid ${\cal C}$ are given by the triangulations ${\cal T}$ of the 
convex $n-$gon ${\rm P}_n$. The morphisms are generated by the flips of triangulations at the diagonals, and by isomorphisms of triangulations. 

 \paragraph{The cluster origin of  relations for the dilogarithm.}  Given a cluster variety ${\cal A}$ and a function $F$ on the cluster torus 
of a seed ${\bf s}$ of the associated connected groupoid,   we  abuse notation by writing $F:= \psi_{\bf s}^*(F) \in {\cal O}({\cal A})$. 
For example, there are functions $X_k = \psi_{\bf s}^*(X_k)$ on ${\cal A}$ given by (\ref{FXk}). Similarly we write $A_i= \psi_{\bf s}^*(A_i)$.

Let ${\cal A}$ be a cluster variety. Recall that the field of rational functions $\Q(\mathcal{A})$ can be identified with the field of $\Q(T)$ for every seed ${\bf s} = (T, W, \{A_i\}).$  A  cluster transformation $\varphi: (T_1, W_1) \lra (T_2, W_2)$, presented as a composition of mutations $\varphi_i$ and isomorphisms,   gives rise to 
an element  of the Bloch group, given by the sum of the elements $\{-X_{\varphi_i}\}_2$ assigned  to the mutations $\varphi_i$ by (\ref{XELL}):
$$
X_\varphi := \sum_i\{-X_{\varphi_i}\}_2 \in  \B_2(\Q({\cal A})).
$$  
 Formula (\ref{WEQ})    implies that
\[
W_1-W_2=\delta X_\varphi.
\]
So if  a  cluster transformation is {\it trivial}, that is  results in the identity map  $\varphi: (T, W) \lra (T, W)$,  then $\delta X_\varphi=0$. 
Thanks to Suslin's theorem \cite[Theorem 5.2]{Sus90}, it defines an element of the 
 group $K_3^{\rm ind}(\Q(T))$ modulo $\Z/4\Z$. The group  $K_3^{\rm ind}(\Q(T))$ is isomorphic by \cite[Corollary 5.6]{Sus90} to $K_3^{\rm ind}(\Q)$,  and thus is torsion by Borel's theorem, or, better, by the 
 theorem of Lee and Szczarba \cite{LS} that $K_3(\Z) = \Z/48\Z$.

 Therefore we arrive at a functional relation for the dilogarithm
$
\sum_{i}{\cal L}_2(-X_{\varphi_i})=0. 
$
 

\paragraph{Basic examples of trivial cluster transformations.}  
Given a seed ${\bf s}$, denote by 
$\sigma_{ij}({\bf s})$  a new seed induced by the bijection  
$I \to I$  
interchanging $i$ and $j$.  The following relations (\ref{K10}) are affiliated with the rank two 
Dynkin diagrams, of the Dynkin types $A_1 \times A_1$ and $A_2$. 

\bl \la{FTFTFT} 
For any cluster variety ${\cal A}$, given a pair $i,j$ such that $\varepsilon_{ij}=-p$, where $p=0,1$, we have
\be \label{K10}
(\sigma_{ij}\circ \mu_i)^{p+4} = \mbox{a trivial cluster transformation}. 
\ee
\el

\paragraph{The Abel five-term relation for the dilogarithm.} 
We start with a $K_2$-torus of the Dynkin type $A_2$, with the coordinates $A_1, A_2$ and the form 
$
W=A_1 \wedge A_2.
$
Then  the   functions $X_k$, defined recursively as $X_k:= (\sigma \circ \mu_1)^*X_{k-1}$,  where $\sigma(1)=2$,  $\sigma(2)=1$, satisfy the pentagon recurrence:
\be \la{PENR2}
X_{k+2}=\frac{X_{k+1}+1}{X_k}.
\ee
Equation (\ref{K10}) for the  Dynkin type $A_2$ just means that the sequence $X_k$ is $5-$periodic. Therefore
\be \la{PENR}
\delta \left(\{-X_1\}_2+\{-X_2\}_2+\{-X_3\}_2+\{-X_4\}_2+\{-X_5\}_2\right)=0. 
\ee
One can show it directly:   recursion (\ref{PENR2})  just means that $X_k X_{k+2}=1+X_{k+1},$
so  
\[
 X_k \wedge X_{k+1}-X_{k+1} \wedge X_{k+2}= (1+X_{k+1})\wedge X_{k+1}. 
\]
So  
$\sum (1+X_{k+1}) \wedge X_{k+1}=0$
due to $5-$periodicity. 
   This  implies the five-term relation for the dilogarithm.

 \subsection{The cluster modular complex}
 
 The material below borrowed from \cite[Sections 2.4, 2.6]{FG03a}.
 
 Below we fix a cluster variety ${\cal A}$ together with 
the set  $I$   parametrizing the cluster ${\cal A}-$variables $A_i$ in a given seed and a subset $I_0\subset I$. One can opt to perform mutations only at the directions $k \in I-I_0$. Then  
the variables $A_j$, $j \in I_0$ do not change, and called the {\it frozen variables}. The non-frozen variables $A_i$ mutate, and produce 
new cluster variables $A_i'$, and so on. Cluster variables, aka cluster coordinates, are assigned to the vertices of the simplicial complex ${\Bbb S}$ defined below. 
 
\paragraph{The simplicial complex ${\Bbb S}$.} 
Let ${\mathbf s}$ be a seed for a given cluster ${\cal A}-$variety. Let $n:=|I|$, $m:=|I-I_0|$. 
A  decorated simplex $S$ is an $(n-1)$-dimensional
simplex equipped with a bijection  
\begin{equation} \label{17:00a}
\{\mbox{vertices of $S$}\} \stackrel{\sim}{\lra} I.
\end{equation} 
Note that the codimension one faces of $S$ are in a canonical bijection with the vertices of $S$: a vertex $k$ of $S$ is identified with the codimension one face of $S$ which does not contain the vertex $k$. 

Take a decorated simplex $S$, and glue to it $m$ other ones as follows. 
To each codimension one face of the  simplex $S$ 
 decorated by an element $k\in I-I_0$ we glue a new decorated 
simplex along its  codimension one face decorated by the same $k$. Repeating  this construction indefinitely,  
we get a simplicial complex ${\Bbb S}$. 
Let ${\mathbf S}$ be the set of all its simplices.


The cluster ${\cal A}-$coordinates are
assigned to the vertices of the simplicial complex ${\Bbb S}$. 
Cluster ${\cal X}-$coordinates are assigned
to  cooriented faces of ${\Bbb S}$.  
 Changing coorientation  
amounts to inversion of the   cluster ${\cal X}$-coordinate.
Mutations are parametrized by codimension one cooriented 
faces of ${\Bbb S}$ parametrised by  $k \in I-I_0$.

\paragraph{The cluster complex.} 
Let $ {\bf F}({\Bbb S})$ be the set of  pairs $(S,F)$ where $S$ 
is a simplex of ${\Bbb S}$ and $F$ is a codimension one face of $S$. 
Pairs of codimension one faces of the same simplex are parametrized 
by the possibly infinite set
\begin{equation} \label{5.20.03.11}
{\bf F}({\Bbb S})\times_{{\bf S}} {\bf F}({\Bbb S}).
\end{equation}
Recall that a given seed ${\bf s}$ include a matrix $\varepsilon_{ij}= \varepsilon_{ij}^{\bf s}$. 
The possibly infinite collection of  all functions $\{\varepsilon^{\bf s}_{ij}\}$ can be viewed as  a single 
function ${\cal E}$ on the set (\ref{5.20.03.11}). Note that the superscript ${\bf s}$ points out the simplex of ${\bf S}$ in (\ref{5.20.03.11}), since the set of  seeds obtained by mutations of a given one 
is canonically identified with the set ${\bf S}$.

Let  ${\rm Aut}({\Bbb S})$ be the automorphism group  of the 
simplicial complex ${\Bbb S}$. It contains the 
subgroup ${\rm Aut}_0({\Bbb S})$ respecting the decoration bijections (\ref{17:00a}). 
The group ${\rm Aut}({\Bbb S})$ is a semidirect product:
$$
0 \lra {\rm Aut}_0({\Bbb S}) \lra  {\rm Aut}({\Bbb S})
\stackrel{}{\lra}  {\rm Per}
\lra 0.
$$ 
Here ${\rm Per}$ is the group of automorphisms of the pair $(I, I_0)$. 
Given an ${\bf s} \in {\bf S}$, the group ${\rm Per}$ is realized as a subgroup ${\rm Aut}({\Bbb S})$ permuting the faces of the simplex $S^{{\bf s}}$, parametrised by the ${\bf s}$. The group ${\rm Aut}({\Bbb S})$ 
acts on the set (\ref{5.20.03.11}), and hence on the  functions ${\cal E}$. 

\begin{definition} \label{5.19.03.1}  Given a cluster $K_2-$variety ${\cal A}$ with the sets $I$ parametrizing the cluster variables in any given seed, 
and the subset $I_0$ parametrizing the frozen ones, 
the group $\Delta= \Delta({\cal A}, I, I_0)$  is the subgroup of ${\rm Aut}({\Bbb S})$  preserving ${\cal E}$ and  all cluster ${\cal A}$-coordinates:
 \begin{equation} \label{5.19.03.11}
\Delta:= \{\gamma \in 
{\rm Aut}({\Bbb S}) \quad | \quad \gamma^*A_{j} = A_j,\  \forall j\in I, ~~\gamma^* ({\cal E}) = 
{\cal E}\}.
\end{equation}

The {cluster complex} $C = C({\cal A}, I, I_0)$ is the quotient of 
${\Bbb S}$ by the action of the group $\Delta$:
$$
C: 
= {\Bbb S}/\Delta. 
$$
\end{definition}


A simplicial face of a simplicial complex  
is of {\it finite type} if it is contained in finitely many  simplices. 
We define the reduced cluster complex $C^*$ as the union of finite type faces of $C$.

Although the reduced cluster complex may not be   a simplicial complex since 
 certain faces of its simplices may not belong to it,   it   has a  
 topological realization.  

\begin{theorem}  \label{12.9.03.2}  \cite[Theorem 2.23]{FG03a}
The topological realization of the reduced cluster complex $C^*$ is a smooth real manifold.
\end{theorem}

 \paragraph{Cluster modular complexes and generalized Stasheff polytopes.}  
 Suppose  that we have a decomposition of a manifold into incomplete simplices, meaning that  
some faces of the simplices may be missing.\footnote{In other words, we have a simplicial complex minus a simplicial subcomplex.}   Then,  
the dual polyhedral decomposition of the  manifold is a locally finite {\it polyhedral complex}.  
So by Theorem \ref{12.9.03.2} the dual  complex $M$  for $C^*$ is a  polyhedral complex.   

\begin{definition} \label{12.9.03.4}
Under the same input as in Definition \ref{5.19.03.1}, the {\em cluster modular complex $M$} is the dual polyhedral complex 
for the reduced cluster complex 
$C^*$. 
\end{definition}

According to 
 \cite{FZ02} cluster algebras of {\it finite type}, i.e., the ones whose  
  cluster complexes  have  a finite number of simplices, 
are classified  by the Dynkin diagrams of type $A, D, E$. 
Their cluster modular complexes $M$ are the generalized Stasheff polytopes  \cite{FZ01b} of  types $A, D, E$ respectively. 
The cluster modular complex of type $A_n$ is the classical  Stasheff polytope ${\rm K}_n$.

\bt \label{11.2.03.1d} \cite[Section 2.6]{FG03a}
The faces of any cluster modular complex 
$M$ are  generalized  Stasheff polytopes  
of type $A,D,E$.
\et

 Lemma \ref{FTFTFT}  describes   a collection of 
 $2$-faces of the  cluster modular complex
$M$. Namely, take a  seed ${\bf s}$ with a pair of vertices $i,j$ with 
$\varepsilon_{ij}\in \{0,\pm 1\}$. Performing $h+2$ 
mutations  at the vertices $i, j, i, j, i,
\ldots$,  augmented    by the isomorphism $\sigma_{ij}^{h+2}$, we get an $(h+2)$-gon, where $h=2,3$ is the Coxeter number 
of the Dynkin diagram of type $A_1 \times A_1, A_2$ respectively.   
These  $(h+2)$-gons  are called the {\it standard 
$(h+2)$-gons} in  $M$. So they are either the {\it standard squares}, or the {\it standard pentagons}.  
Here is a more precise version of Theorem \ref{11.2.03.1d} describing the $2$-skeleton of $M$.

\bp  \label{11.2.03.1}
The $2$-skeleton of the cluster modular complex $M$ has the following structure: the vertices of  $M$ parametrize the isomorphism classes of seeds, the   oriented edges of  $M$ correspond to mutations, the $2$-faces of  $M$ are the  standard squares and pentagons.  
 
 \ep

\subsection{The cluster polylogarithm map of complexes: the setup}
Our key idea is to   associate with every $K_2$-torus $(T,W)$  the exponential of $W$:
\[
e^{ W}=\sum_{k=0}^{\infty}\frac{W^{k}}{k!}\in \Lambda^{*} (\Q(T)^{\times}).
\]
It provides   a differential form, as well as  a class in the Milnor   ring 
$K^M_{\bullet}(\Q(T))$:
\[
[e^W]=e^ {[W]} \in  K^M_{\bullet}(\Q(T)). 
\] 

\begin{color}{red}
\end{color}

The following observation is crucial. Since the class $[W]$ is invariant under $K_2-$morphisms, so is $[e^W].$ 
Next,
 given a cluster $K_2-$variety ${\cal A}$, the vertices $v_{\bf s}$ of the cluster modular polyhedral complex ${\rm M}_{\cal A}$ are  assigned to the seeds ${\bf s}$, considered modulo the 
 action of the group $\Delta$. The antisymmetric bilinear 2-form $W$ can be realized by any seed ${\bf s},$ so each seed ${\bf s}$ provides an element  
 $W_{\bf s}$.  We  assign to each vertex $v_{\bf s}$   a   representative of the same element in Milnor $K$-theory ring: 
\[
L[v_{\bf s}]=e^{W_{\bf s}} \in \Lambda^* \Q(T)^{\times}.
\]

Denote by   ${\rm C}_p({\rm M}_{\cal A})$ the free abelian group  generated by the   oriented $p$-dimensional faces of the cluster modular complex  ${\rm M}_{\cal A}$  - altering the orientation we flip the sign of the generator.  
Consider the chain complex of the modular complex ${\rm M}_{\cal A}$:
\[
{\rm C}_*({\rm M}_{\cal A}):= ~~~~~~{\rm C}_{m}({\rm M}_{\cal A}) \stackrel{d}{\lra} {\rm C}_{m-1}({\rm M}_{\cal A})\stackrel{d}{\lra}  \ldots \stackrel{d}{\lra} {\rm C}_0({\rm M}_{\cal A}).
\]
 
Now the Chevalley-Eilenberg complex of the motivic Tate Lie coalgebra ${{\cal L}}_{\bullet}$  comes into play.  
\[
{\rm CE}^{*}\left({{\cal L}}_{\bullet} \right) :=~~~~~~{{\cal L}_{\bullet} \stackrel{\delta}{\lra} \Lambda^2 {{\cal L}}_{\bullet} 
\stackrel{\delta}{\lra} \Lambda^3 {{\cal L}}}_{\bullet} \stackrel{\delta}{\lra} \ldots
\]
It is a commutative dg-algebra $\Lambda^*({\cal L}_\bullet[-1])$ with the differential 
$\delta$ induced  on the generators by the   map ${\cal L}_\bullet \lra \Lambda^2{\cal L}_\bullet$. The product  is denoted by $\ast$.    


\bcon \la{MTTHH1} For any cluster $K_2-$variety ${\cal A}$  there exists a  map of complexes, called the \underline{cluster polylogarithm map}
$$
\alpha: {\rm C}_*({\rm M}_{\cal A})\lra {\rm CE}^{*}\left({\cal L}_\bullet({\Bbb F}_{\cal A})\right)
$$
from the chain  complex of the cluster modular complex ${\rm M}_{\cal A}$ 
to  the Cartan-Eilenberg  complex  of the motivic Tate Lie coalgebra ${\cal L}_\bullet({\Bbb F}_{\cal A})$ of the field of functions ${\Bbb F}_{{\cal A}}$ on ${\cal A}$ such that 

\begin{itemize}

\item 
each vertex $v_{\bf s}$ of the  cluster modular complex ${\rm M}_{\cal A}$  is mapped to $e^{W_{\bf s} }$.

\item The map $\alpha$ is multiplicative: for each oriented face $M$ of ${\rm M}_{\cal A}$, decomposed into a product of oriented faces $M = M_1 \times \ldots \times M_n$, one has 
$$
\alpha(M_1 \times \ldots \times M_n) = \alpha(M_1) \ast \ldots \ast \alpha(M_n).
 $$
\end{itemize}
\econ

Consider the degree $2k$ part of the Chevalley-Eilenberg  complex  of the Lie coalgebra ${\cal L}_\bullet(\F)$: 
 $$
{\rm CE}_{(2k)}^{*}\left({\cal L}_\bullet({\F})\right):= ~~~~~~
 {\cal L}_{2k}(\F) \stackrel{\delta}{\lra} \ldots \stackrel{\delta}{\lra}  {\cal L}_{2}(\F) \otimes \Lambda^{2k-2}\F^\times \stackrel{\delta}{\lra} \Lambda^{2k} \F^\times.
 $$
 One has  
 $$
 H^{2k}({\rm CE}_{(2k)}^*\left({{\cal L}}_{\bullet}(\F) \right))=K^M_{2k}(\F)_\Q.
 $$ 
 
 Conjecture \ref{MTTHH1} has the following specification, describing the {\it weight $2k$ cluster polylogarithm map}.
  
  \bcon \la{MTTHH} For any cluster $K_2-$variety ${\cal A}$, and an integer  $k>0$, there exists the weight $2k$ cluster polylogarithm map, given by a  map from the chain  complex of the cluster modular complex ${\rm M}_{\cal A}$ 
to the weight $2k$ part of the  cochain complex of the motivic Tate Lie coalgebra ${\cal L}_\bullet({\Bbb F}_{\cal A})$ of the field of functions ${\Bbb F}_{{\cal A}}$ on ${\cal A}$ such that 

\begin{itemize}

\item 
each vertex $v_{\bf s}$ of the  cluster modular complex ${\rm M}_{\cal A}$  is mapped to $W_{\bf s}^{k}$:
\end{itemize}
\[
\begin{split}
&{\Bbb L}^*_{2k}: {\rm C}_{2k-*}({\rm M}_{\cal A}) \lra {\rm CE}_{(2k)}^{*}\left({\cal L}_\bullet({\Bbb F}_{\cal A})\right),\\
&{\Bbb L}^{2k}_{2k}: v_{\bf s} \in {\rm C}_0({\rm M}_{\cal A})\lms   W_{\bf s} \wedge \ldots \wedge W_{\bf s} \in \Lambda^{2k}{\Bbb F}^\times_{{\cal A}}.\\
\end{split}
\]
\econ

Elaborating this,  the weight $2k$ cluster polylogarithm map ${\Bbb L}^*_{2k}$ should look as follows:
\[ 
\begin{gathered}
    \xymatrix{
        & {\rm C}_{2k}({\rm M}_{\cal A})  \ar[d]^{{\Bbb L}^0_{2k}}  \ar[r]^{d}& {\rm C}_{2k-1}({\rm M}_{\cal A})  \ar[d]^{{\Bbb L}^1_{2k}}  \ar[r]^{~~~~~~d}&\ldots \ar[r]^{d~~~~} &   {\rm C}_{1}({\rm M}_{\cal A})  \ar[d]^{{\Bbb L}^{2k-1}_{2k}} \ar[r]^{d}& 
       {\rm C}_{0}({\rm M}_{\cal A})  \ar[d]^{{\Bbb L}^{2k}_{2k}}\\
             &  0   \ar[r]^{\delta}  &  {\cal L}_{2k}({\Bbb F}_{\cal A})    \ar[r]^{~~~~\delta}&  \ldots   \ar[r]^{\delta~~~~~~~~~~~~}& {\cal L}_2({\Bbb F}_{\cal A})  \otimes \Lambda^{2k -2} {\Bbb F}^\times_{\cal A}  \ar[r]^{~~~~~~~~\delta}  
               &\Lambda^{2k}{\Bbb F}^\times_{\cal A}.}
                          \end{gathered}
\]


In Section \ref{SEC7.4}, we define the map ${\Bbb L}^*_{2k}$, for any $k>0$,  for $1$ and $2-$dimensional faces of the cluster modular complex.  We  present it in the ``exponential''  form which works for all $k>0$.  So, we produce a map of  complexes 
\[
\begin{gathered}
    \xymatrix{
            C_{2}({\rm M}_{\cal A})  \ar[d]^{{\Bbb L}^{2k-2}_{2k}} \ar[r]^{d}&   C_{1}({\rm M}_{\cal A})  \ar[d]^{{\Bbb L}^{2k-1}_{2k}} \ar[r]^{d}& 
       C_{0}({\rm M}_{\cal A})  \ar[d]^{{\Bbb L}^{2k}_{2k}}\\ 
               \bigl(\B_3 \otimes  \Lambda^{2k-3} {\Bbb F}^\times_{\cal A}\bigr) \oplus    \bigl(\Lambda^2\B_2   \otimes \Lambda^{2k -4} {\Bbb F}^\times_{\cal A}\bigr) \ar[r]^{~~~~~~~~~~~~~\delta}& 
             \B_2  \otimes \Lambda^{2k -2} {\Bbb F}^\times_{\cal A}  \ar[r]^{~~~~~\delta}  
               &\Lambda^{2k}{\Bbb F}^\times_{\cal A}.}
                          \end{gathered}
\]

In Section \ref{sec8}, we define the map ${\Bbb L}^1_{4}$ explicitly for the cluster variety $\textup{Conf}_n(V_2)$.

\subsection{Edges, squares, pentagons} \la{SEC7.4}

Below we denote the cluster polylogarithm map of complexes  just by ${\Bbb L}$.

\paragraph{Edges.} Consider an oriented edge $e$ of the cluster modular complex. It corresponds to a mutation of seeds $\mu_e\colon (T_1, W_1) \to (T_2, W_2)$. 
Let $X_e$ be the corresponding $\mathcal{X}-$coordinate (\ref{FXk}). 
Since $[e^{W_1}]=\left[e^{W_2}\right]$ in the Milnor $K$-theory,
\[
e^{W_2}-e^{W_1}\in {\rm Im} \left (\B_2 \otimes \Lambda^* \Q(T)^{\times} \lra \Lambda^* \Q(T)^{\times}\right).
\]
For every oriented edge $e$ in the cluster modular complex we  define
\[
{\Bbb L}[e]: =-\{-X_e\}_2\otimes \frac{e^{W_1}+e^{W_2}}{2} \in \B_2 \otimes \Lambda^* \Q(T)^{\times}.
\]
If  $\overline{e}$ is the same edge with the opposite orientation then $ X_{\overline{e}}  =  X_{ {e}}^{-1}$. So     
$\{-X_{\overline{e}}\}_2 = -\{-X_{ {e}}\}_2$ and  
\[
{\Bbb L}[\overline{e}] 
=-{\Bbb L}[e].
\]

\bl \la{L6.5} One has 
\be
\delta {\Bbb L}[e] =e^{W_2}-e^{W_1}.
\ee
\el

\begin{proof}
Since  $W_2-W_1=X_e \wedge (1+X_e)$ by (\ref{WEQ}),  it is nilpotent: 
\[
(W_2-W_1)^2=(W_2-W_1)\wedge (W_2-W_1)=0.
\] 
So
\[
\begin{split}
&e^{W_2-W_1}=1+X_e \wedge (1+X_e),\\
&e^{W_1-W_2}=1-X_e \wedge (1+X_e).\\
\end{split}
\]
Using this, we get
\[
\begin{split}
& e^{W_2}-e^{W_1}=\frac{1}{2}\left(((1-e^{W_1-W_2})e^{W_2}+(e^{W_2-W_1}-1)e^{W_1} \right)=\\
&X_e \wedge (1+X_e)\wedge \frac{e^{W_1}+e^{W_2}}{2}=\delta {\Bbb L}[e].\\
\end{split}
\]
\end{proof}

Notice also that the  fact that $(W_2-W_1)^2=0$ has the following corollary:
\[
\frac{e^{W_1}+e^{W_2}}{2}=e^{\dfrac{W_1+W_2}{2}}
\]

\paragraph{Squares.}
By Proposition \ref{11.2.03.1} there are just two types  of 2-dimensional   faces in  a cluster modular complex:  squares and pentagons, 
  assigned to the    Stasheff polytopes $K_1\times K_1$ and $K_2$.

Let $f_{K_1\times K_1}$ be an oriented  square face in the cluster modular complex. Its vertices correspond to four cluster tori $(T_i, W_i)$, and the edges to the pair of commuting mutations. 
We denote by  $v_i$, $i \in \Z/4\Z$,  the corresponding vertices,   by $e_i$ the edges connecting vertices $v_i$ and $v_{i+1}$, and by $X_{i}$  
the  corresponding cluster $\mathcal{X}$-coordinates. Notice that $X_{1}X_{3}=X_{2}X_{4}=1.$ Since the boundary of the boundary of the square vanishes, we immediatelly get
\[
\delta({\Bbb L}[e_1]+{\Bbb L}[e_2]+{\Bbb L}[e_3]+{\Bbb L}[e_4])=0.
\]
We define 
\[
{\Bbb L}[f_{K_1\times K_1}]=-\{-X_{1}\}_2\wedge \{-X_{2}\}_2\otimes \frac{e^{W_1}+e^{W_2}+e^{W_3}+e^{W_4}}{4}.
\]
\bl
\[
\delta {\Bbb L}[f_{K_1\times K_1}]={\Bbb L}[e_1]+{\Bbb L}[e_2]+{\Bbb L}[e_3]+{\Bbb L}[e_4].
\]
\el
\begin{proof} Using $\{-X_{i+2}\}_2 = - \{-X_i\}_2$  we get: 
\[
\begin{split}
&{\Bbb L}[e_1]+{\Bbb L}[e_2]+{\Bbb L}[e_3]+{\Bbb L}[e_4]= -\sum_{i \in \Z/4\Z} 
 \{-X_i\}_2 \otimes \frac{e^{W_i}+e^{W_{i+1}}}{2} =  \\
&-\{-X_1\}_2\otimes  \frac{e^{W_1}+e^{W_2}-e^{W_3}-e^{W_4}}{2}-\{-X_2\}_2\otimes  \frac{e^{W_2}+e^{W_3}-e^{W_4}-e^{W_1}}{2}.
\end{split}
\]
From Lemma \ref{L6.5} we have  
\[
\begin{split}
&e^{W_{i+1}}-e^{W_i}= - \delta \{-X_i\}_2 \otimes \frac{e^{W_{i+1}}+e^{W_i}}{2}.\\
\end{split}
\]
This implies that 
\[
\begin{split}
& \frac{e^{W_1}+e^{W_2}-e^{W_3}-e^{W_4}}{2}  = \\
& \delta\{-X_2\}_2\otimes \frac{e^{W_2}+e^{W_3}}{2} - \delta\{-X_4\}_2\otimes \frac{e^{W_4}+e^{W_1}}{2} = \\
&  \delta\{-X_2\}_2\otimes \frac{e^{W_1}+e^{W_2}+ e^{W_3}+e^{W_4}}{4}. \\
\end{split}
\]
Using this identity and its cyclic shift by $1$,   we get  
\[
\begin{split}
&{\Bbb L}[e_1]+{\Bbb L}[e_2]+{\Bbb L}[e_3]+{\Bbb L}[e_4]=\\
&-(\{-X_1\}_2 \otimes \delta\{-X_2\}_2 + \{-X_2\}_2 \otimes \delta\{-X_1\}_2)\otimes \frac{e^{W_1}+e^{W_2}+e^{W_3}+e^{W_4}}{4}=\\
&\delta {\Bbb L}[f_{K_1\times K_1}].\\
\end{split}
\]
\end{proof}

\paragraph{Pentagons.}
Let $f_{K_2}$ be an oriented pentagon of the cluster modular complex. Its vertices $v_i$, where $i \in \Z/5\Z$, correspond to five cluster tori $(T_i, W_i)$. 
Denote by  $e_i = (v_i, v_{i+1})$ 
  the edge from the  verticex $v_i$ to $v_{i+1}$, and by $X_{i}$  the corresponding $\mathcal{X}$-coordinate.  
Recall   the   $5-$periodic recurrence  for the $\mathcal{X}$-coordinates:
\[
X_{i+2}=\frac{X_{i+1}+1}{X_{i}}.
\]
Just as in the square case, the ``boundary of the boundary is  zero'' argument implies that   
\[
\delta({\Bbb L}[e_1]+ {\Bbb L}[e_2]+{\Bbb L}[e_3]+{\Bbb L}[e_4]+{\Bbb L}[e_5])=0.
\]
The following observation is crucial. Consider  the following element:
\[
W_{K_2}:= W_i-X_{i-1}\wedge X_{i}, ~~~~i\in \Z/5\Z.
\]
\bl \la{L6.7} The  element $W_{K_2}$ does not depend on $i\in \Z/5\Z$. 
\el

\begin{proof} 
Since $W_{i+1}-W_i=X_i\wedge (X_{i}+1)$ and $X_{i-1}X_{i+1}=1+X_{i}$, we have
\[
\begin{split}
&(W_{i+1}-X_{i}\wedge X_{i+1}) - (W_i-X_{i-1}\wedge X_{i}) = \\
& (W_{i+1} - W_i) -X_{i}\wedge X_{i+1} X_{i-1}  =  X_i\wedge (X_{i}+1) - X_i\wedge (X_{i}+1)=0.\\
\end{split}
\]
\end{proof}

Using the element $W_{K_2}$  we define 
\[
{\Bbb L}[f_{K_2}]=\frac{1}{2}\left (\sum_{i=1}^{5}\{-X_{i}\}_3 \otimes \frac{X_{i-1}}{X_{i+1}}\right )\wedge e^{W_{K_2}}.
\]
The element ${\Bbb L}[f_{K_2}]$ is  cyclically invariant, and changes the sign if one alters the orientation of the pentagon $f_{K_2}.$

\bl \la{L6.8} One has 
\[
\delta {\Bbb L}[f_{K_2}]={\Bbb L}[e_1]+{\Bbb L}[e_2]+{\Bbb L}[e_3]+{\Bbb L}[e_4]+{\Bbb L}[e_5].
\]
\el
\begin{proof}
\[
\begin{split}
&\sum_{i=1}^{5}{\Bbb L}[e_i]=-\sum_{i=1}^{5}\{-X_{i}\}\otimes \frac{e^{W_i}+e^{W_{i+1}}}{2}=\\
&-\left (\sum_{i=1}^{5}\{-X_i\}\otimes \frac{e^{W_i-W_{K_2}}+e^{W_{i+1}-W_{K_2}}}{2}\right )\wedge e^{W_{K_2}}=\\
&-\left (\sum_{i=1}^{5}\{-X_i\}\otimes \frac{e^{X_{i-1}\wedge X_{i}}+e^{X_{i}\wedge X_{i+1}}}{2}\right )\wedge e^{W_{K_2}}=\\
&-\left (\sum_{i=1}^{5}\{-X_i\}\otimes (1+\frac{1}{2}(X_{i-1}\wedge X_i+X_{i}\wedge X_{i+1}))\right )\wedge e^{W_{K_2}}.\\
\end{split}
\]
The five-term relation (\ref{PENR})  implies that 
$
\sum_{i=1}^{5}\{-X_i\}_2\otimes e^{W_{K_2}}=0.
$ 
So we get 
\[
\begin{split}
&\sum_{i=1}^{5}{\Bbb L}[e_i]=-\left (\sum_{i=1}^{5}\{-X_i\}_2\otimes \frac{1}{2}(X_{i-1}\wedge X_i+X_{i}\wedge X_{i+1})\right )\wedge e^{W_{K_2}} =\\
&\delta \frac{1}{2} \left (\sum_{i=1}^{5}\{-X_i\}_3\otimes \frac{X_{i-1}}{X_{i+1}}\right )\wedge e^{W_{K_2}}.\\
\end{split}
\]
\end{proof}

Using this  and Theorem \ref{11.2.03.1d}   we define  for any cluster $K_2-$variety ${\cal A}$ the motivic cluster polylogarithm map  in the weights $2$ and $4$.

 \paragraph{Cluster polylogarithms as derived motivic differential forms.} Although we do not need it in this paper, we want to outline the general setup underlying 
  cluster polylogarithms. 

 Let $X$ be a variety over a field $k$. Denote by ${\cal X}$ its generic point. 
There is a canonical map 
\[
\begin{split}
&d\log:  K^{ M}_m(k({\cal X})) \lra \Omega^m_{\log}({\cal X}),\\
&f_1 \wedge \ldots \wedge f_n \lms d\log (f_1) \wedge \ldots \wedge d\log (f_n).\\
\end{split}
\]
We say that a differential $m$-form $\omega$ on  $X$ is {\it motivic} if it lies in the image of the map $d\log$: 
$$
\omega = d\log (W), ~~~~W\in K^M_m(k({\cal X})).
$$  

We would like to have an equivariant derived variant of motivic differential forms. Recall   the weight $m$ motivic complexes of sheaves 
$\underline \Z_{\cal M}(m)$, such that for ${\cal X}= {\rm Spec}(\F)$ one has 
\be \la{CCC}
H^m({\rm Spec}(\F), \underline \Z_{\cal M}(m)) = K^M_m(\F).
\ee
It follows that one can lift   $W$ to a class ${\cal W}$ in  a  Chech complex   calculating (\ref{CCC}).

An interesting situation appears for a cluster variety ${\cal A}$, which comes   with    the following  data: 

\begin{itemize} 
\item A {\it cluster atlas  of ${\cal A}$}  - a distinguished Zariski atlas  provided by cluster coordinate systems.

\item A  polyhedral complex ${\rm M}_{\cal A}$, called the {\it cluster modular complex}.

The  vertices of ${\rm M}_{\cal A}$ are the cluster coordinate systems, the oriented edges are   
mutations, 
the  two-dimensional faces are pentagons and squares, encoding the basic relations between mutations.  
Higher-dimensional cells should be view as higher relations. 

\item A  group $\Gamma_{\cal A}$, called the {\it cluster modular group},
acting by automorphisms of ${\cal A}$ and ${\rm M}_{\cal A}$.

\end{itemize}

\paragraph{Example.} The space ${\rm Conf}_n(V_2)$ has a cluster variety structure. The cluster coordinate systems are parametrized by 
the triangulations of the polygon ${\rm P}_n$. The cluster modular complex $M$, in this case, is the Stasheff polytope ${\rm K}_{n-3}$. 
The cluster modular group $\Gamma$ is the group $\Z/n\Z$ acting by the twisted cyclic shift $(l_1, \ldots, l_n) \lms (-l_n, l_1, \ldots, l_{n-1})$ on the space ${\rm Conf}_n(V_2)$. 

\vskip 3mm
Using this data, we assign to a complex of sheaves ${\cal F} $ on ${\cal A}$  its {\it cluster resolution}   ${\cal C}^\bullet_{\rm cl}({\cal F})$, 
whose definition mimics the definition of the Cech 
complex. Each vertex $v$ of ${\rm M}_{\cal A}$ corresponds to an open subset $U_v \subset {\cal A}$. We assign to each face $F$ of ${\rm M}_{\cal A}$ an open subset 
$U_F:= \cap_{v \in F}U_v$. Let $j_F: U_F \subset {\cal A}$ be the natural inclusion. We define a complex of sheaves ${\cal C}^\bullet_{\rm cl}({\cal F})$ as follows:
$$
{\cal C}^\bullet_{\rm cl}({\cal F}):= ~~~~\bigoplus_{{\rm dim}(F) = 0} j_{F*}j_F^* {\cal F}   \lra \bigoplus_{{\rm dim}(F) = 1} j_{F*}j_F^* {\cal F}   \lra \bigoplus_{{\rm dim}(F) = 2} j_{F*}j_F^* {\cal F}  \lra \ldots .
$$
If the cluster modular complex ${\rm M}_{\cal A}$ is contractable, which is the case in our main example, then we get a resolution of the sheaf ${\cal F}$. 
It is equipped  with a natural action of the cluster modular group $\Gamma_{\cal A}$. So by an {\it $\Gamma_{\cal A}$-equivariant derived motivic differential form} on a cluster variety 
${\cal A}$ we mean a cocycle 
$$
{\cal W} \in H^m\left({\cal C}^\bullet_{\rm cl}(\underline\Z_{\cal M}(m))^{\Gamma_{\cal A}}\right).
$$
There is a canonical projection
$$
d\log: H^m\left({\cal C}^\bullet_{\rm cl}(\underline\Z_{\cal M}(m))^{\Gamma_{\cal A}}\right) \lra \Omega^m_{\rm log}({\cal A}).
$$
So ${\cal W}$ is a derived motivic avatar of a $d\log$-form   $\omega:= d\log({\cal W})$.

\vskip 3mm
The map of complexes from Theorem \ref{MC4} can be viewed as   a cocycle in the complex obtained by taking the weight four part of the 
standard cochain complex of the motivic Lie algebra as a variant of the 
motivic complex $\underline\Z_{\cal M}(4)$, and using the cluster resolution ${\cal C}^\bullet_{\rm cl}(-)$  related to the cluster variety 
${\rm Conf}_m(V_2)$ to calculate $H^4$. 

 \section{Cluster polylogarithm maps for the  space ${\rm Conf}_{n}(V_2)$} \la{sec8}
 
  \subsection{The cluster polylogarithm map for  ${\rm Conf}_{n}(V_2)$ - the setup} \la{sec8.1}
  
  Let $V_2$ be a two dimensional symplectic vector space over a field $\F$ with a 2-form $\omega$. Recall the space
${\rm Conf}_n(V_2)$ parametrizing generic configurations of vectors $(l_1, \ldots, l_n)$ in $(V_2, \omega)$. 

Given a convex $n-$gon ${\rm P}_n$, there is a cluster variety ${\cal A}_{{\rm P}_n}$ related to it. Let us relate it to the moduli space ${\rm Conf}_n(V_2) $. Pick a vertex $v$ of the polygon.
 We label  the  vertices of   ${\rm P}_n$ cyclically by the vectors $l_1, \ldots, l_n$.  Then there is a canonical birational isomorphism 
  ${\rm I}_v\colon {\rm Conf}_n(V_2) \lra {\cal A}_{{\rm P}_n}$. Changing the vertex $v$ to the next one $v'$ following the cyclic order amounts to precomposition with the   twisted cyclic 
  shift $S\colon (l_1, \ldots, l_n) \lms (-l_n, l_1, \ldots, l_{n-1})$. Precisely, ${\rm I}_{v'} = {\rm I}_v\circ S$.

Given a generic configurations of vectors $( l_1, l_2, l_3)$, let us set 
\[
\bs
&W_t(l_1, l_2, l_3):=  \omega(l_1, l_2) \wedge  \omega(l_2, l_3) + \omega(l_2, l_3) \wedge  \omega(l_1, l_3) +\omega(l_1, l_3) \wedge  \omega(l_1, l_2) \in \Lambda^2\F^\times.\\
\end{split}
\]
Note that we have 
\[
\bs
&e^{W_t(  l_1, l_2, l_3)} = 1+ W_t(  l_1, l_2, l_3).\\
\end{split}
\]

Recall the Stasheff polytope ${\rm K}_{\rm P_n}$ assigned to the convex $n-$gon ${\rm P_n}$. Its $p-$dimensional  faces are parametrized by collections of $n-p-3$ non-intersecting  diagonals of ${\rm K}_{\rm P_n}$. So each 
 decomposition ${\cal Q}$ of the polygon 
${\rm P}_n = {\rm Q_0} \cup \ldots \cup {\rm Q_d}$,  obtained by cutting the polygon by $d$ non-intersecting diagonals, give rise to a face, denoted by ${\rm K}_{\cal Q}$.   

 Faces of  Stasheff polytopes have a beautiful factorization structure. Namely, 
denote by ${\rm K}_{\rm Q}$  the Stasheff polytope   assigned to the convex polygon ${\rm Q}$. 
 Then the face ${\rm K}_{\cal Q}$ assigned to a decomposition $ {\rm Q_0}\cup  \ldots \cup  {\rm Q_d}$ is   
 a product of Stasheff polytopes
\be \la{DSP}
{\rm K}_{\cal Q} = {\rm K}_{\rm Q_0} \times \ldots \times  {\rm K}_{\rm Q_{d}}.
\ee
 
 We denote by $[{\rm K}_{\rm Q}, {\rm or}_{\rm Q}]$ a Stasheff polytope ${\rm K}_{\rm Q}$ equipped with an orientation.  If ${\rm K}_{\rm Q}$ is a point, the orientation is canonical, and thus omitted from the notation.

 Recall the chain complex ${\rm C}_\bullet({\rm K}_{\rm P_n})$ of the Stasheff polytope of ${\rm P}_n$.
$$
{\rm C}_\bullet({\rm K}_{\rm P_n}):= ~~~~~~{\rm C}_{n-3}({\rm K}_{\rm P_n}) \stackrel{d}{\lra} {\rm C}_{n-4}({\rm K}_{\rm P_n})\stackrel{d}{\lra}  \ldots \stackrel{d}{\lra} {\rm C}_0({\rm K}_{\rm P_n}).
 $$
   Denote by ${\Bbb F}_{n}$ the field of rational functions with $\Q$-coefficients on the 
cluster variety assigned to the polygon ${\rm P}_n$.

Our next goal is to state Conjecture \ref{MTTHH11}, which we use as a setup. In this paper, we do not prove
Conjecture \ref{MTTHH11}, but we derive  the pieces of the map $\alpha$, which we need to prove our main results explicitly.  Most importantly, we do not define  the coalgebra ${\cal L}^g_\bullet$, just indicating that it can be defined as 
the Lie coalgebra of the category of, say, mixed Hodge-Tate structures of geometric origin. 

We denote by $\ast$ the product in the commutative dg-algebra ${\rm CE}^{*}\left({\cal L}^g_\bullet({\Bbb F}_{n})\right)$.

\bcon \la{MTTHH11} For each $n\geq 3$ there exists   a map, called the cluster polylogarithm map,  
\be \la{ICON4}
\alpha: {\rm C}_\bullet({\rm K}_{\rm P_n}) \lra {\rm CE}^{*}\left({\cal L}^g_\bullet({\Bbb F}_{n})\right)
\ee
from the chain complex of the Stasheff polytope of ${\rm P}_n$ to the Cartan-Eilenberg commutative dg-algebra of the motivic Lie coalgebra of the  field ${\Bbb F}_{n}$ with the following properties: 

\begin{itemize}

 \item {\em The initial condition:} For a triangle $t = {\rm P}_3$ we have 
\be \la{ICON1}
\bs
& \alpha [{\rm K}_{\rm t}]= e^{W_t}.
\\
\end{split}
\ee

\item {\em Compatibility with the differentials:}  
$$
\delta \circ \alpha = \alpha \circ d.
$$

\item {\em Multiplicativity:} for each  decomposition  ${\cal Q}$ of a polygon  ${\rm P}_n   = {\rm Q}_0 \cup \ldots \cup {\rm Q}_{s}$,  equipped  with a   compatible set of orientations 
 ${\rm or}_{\cal Q} = {\rm or}_{{\rm Q}_0} \cdot \ldots \cdot {\rm or}_{{\rm Q}_{s}}$,  one has 
\be \la{ICON11}
\alpha [{\rm K}_{\cal Q}, {\rm or}_{\cal Q}] = \alpha [{\rm K}_{\rm Q_0},  {\rm or}_{{\rm Q}_0}]   \ast \ldots \ast \alpha [{\rm K}_{\rm Q_{s}}, {\rm or}_{{\rm Q}_{s}}].
\ee
\end{itemize}
\econ

For example, the map $\alpha$ assigns to a triangulation ${\cal T}$ of the polygon ${\rm P_n}$ the element 
$$
\alpha[{\rm K}_{\cal T}] = e^W = \prod_{t \in T} e^{W_t}.
$$
Observe that ${\rm K}_{\cal T}$ is the vertex of the Stasheff polytope ${\rm K}_{\rm P_n}$ corresponding to a triangulation ${\cal T}$ of ${\rm P}_n$. 

  \begin{figure}[ht]
\centerline{\epsfbox{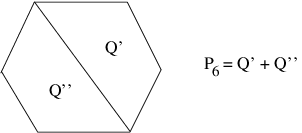}}
\caption{ Codimension one boundary faces ${\rm K}_{\rm Q}$ of the Stasheff polytope ${\rm K}_{\rm P_6}$ are parametrized by diagonals of the polygon ${\rm P_6}$, 
dividing it into  two polygons: ${\rm P_6} = {\rm Q}'\cup {\rm Q}''$.}
\label{zg12}
\end{figure}

Alternatively, the cluster polylogarithm map amounts to the following data:

\bi

\item For each integer $n \geq 3$, and element 
\be \la{ICON2}
\alpha [{\rm K}_{\rm P_n}, {\rm or}_{\rm P_n}] \in {\rm CE}^{*}\left({\cal L}^g_\bullet({\Bbb F}_{n})\right).
\ee
such that the  element $\alpha [{\rm K}_{\rm P_3}]$ is given by (\ref{ICON1}).

\item The coboundary of element (\ref{ICON2}) is given by a sum over  codimension one  faces ${\rm K}_{\rm Q}$  of the Stasheff polytope ${\rm K}_{\rm P_n}$, parametrized 
by  decompositions ${\rm P_n}= {\rm Q}' \cup {\rm Q}''$, induced by cutting the polygon by a diagonal:
\be \la{ICON3}
\bs
&\delta \circ \alpha [{\rm K}_{\rm P_n}, {\rm or}_{\rm P_n}] = \alpha \circ d  [{\rm K}_{\rm P_n}, {\rm or}_{\rm P_n}]:= \\
&\frac{{\rm or}_{\rm Q}}{{\rm or}_{\rm Q'}{\rm or}_{\rm Q''}}\sum_{{\rm Q}} \alpha [{\rm K}_{{\rm Q}'}, {\rm or}_{{\rm Q}'}] \ast \alpha [{\rm K}_{{\rm Q}''}, {\rm or}_{{\rm Q}''}].\\
\end{split}
\ee
Orientations of the faces ${\rm K}_{\rm Q}$ are induced by the orientation    ${\rm or}_{\rm P_n}$. The ratio of orientations is a sign 
$\frac{{\rm or}_{\rm Q}}{{\rm or}_{\rm Q'}{\rm or}_{\rm Q''}}\in \{\pm1\}$. 
\ei

The degrees of  components  of the element $\alpha [{\rm K}_{\rm P_n}] $ are of the same parity as  the dimension of the Stasheff polytope ${\rm K}_{\rm P_n}$. 
So the right-hand side of (\ref{ICON3}) does not depend on the order of factors. 
The above data determines uniquely  map (\ref{ICON4}) due to its  multiplicativity (\ref{ICON11}). 

\paragraph{Remark.} The map $\alpha$ from Conjecture \ref{MTTHH11} implies Conjecture \ref{MTTHH1} for the cluster variety ${\cal A}_{\rm P_n}$.  
It also implies Conjecture \ref{MTTHH1} for the cluster $K_2$-variety ${\cal A}_{SL_2, S}$ assigned  \cite{FG03} to any decorated surface $S$. 
It does not imply  however Conjecture \ref{MTTHH1} for cluster $K_2$-varieties ${\cal A}_{SL_m, S}$ where $m>2$. Indeed, although   the  polytopes in the cluster modular complex 
are always of finite type \cite[Lemma 2.24]{FG03a} for $m>2$ they 
 could be of   type different then $A_m$.

\paragraph{An example.} Recall that each triangulation ${\cal T}$ of ${\rm P}_n$ provides a $K_2$-seed
 and an element $W_{\cal T}$,   see Section \ref{sec8.2}. 
We assign to each triangulation ${\cal T}$ of the polygon ${\rm P}_n$ the $k$-th power of $W_{\cal T}$:
 $$
 {\cal T} \lms W_{\cal T}^{\wedge k}:= W_{\cal T} \wedge \ldots \wedge W_{\cal T} \in \Lambda^{2k} {\Bbb F}_{m+3}^{\times}.
 $$

Given a pair of positive integers $m,k$,   a cluster polylogarithm map  is  a  map  of complexes
$$
{\Bbb L}^*_{2k}: {\rm C}_{2k-*}({\rm K}_m) \lra {\rm CE}_{(2k)}^{*}\left({\cal L}^g_\bullet({\Bbb F}_{m+3})\right)
$$
such that the vertex $v_{\cal T}$ of the  polytope ${\rm K}_m$ assigned to a triangulation ${\cal T}$ and an integer $k$ is mapped to $W_{\cal T}^{k}$:
$$
{\Bbb L}^{2k}_{2k}: v_{\cal T} \in {\rm C}_0({\rm K}_m)\lms W_{\cal T} \wedge \ldots \wedge W_{\cal T} \in \Lambda^{2k}{\Bbb F}^\times_{m+3}.
$$
For example, in the most fundamental case  $m=2n-1, k=n$   we are looking for a map
 \be \la{MCDImmmx} 
\begin{gathered}
    \xymatrix{
        & {\rm C}_{2n-1}({\rm K}_{2n-1})  \ar[d]^{{\Bbb L}^1_{2n}}  \ar[r]^{~~~~~~d}&\ldots \ar[r]^{d~~~~} &   {\rm C}_{1}({\rm K}_{2n-1})  \ar[d]^{{\Bbb L}^{2n-1}_{2n}} \ar[r]^{d}& 
       {\rm C}_{0}({\rm K}_{2n-1})  \ar[d]^{{\Bbb L}^{2n}_{2n}}\\
               &  {\cal L}^g_{2n}({\Bbb F}_{2n+2})    \ar[r]^{~~~~\delta}&  \ldots   \ar[r]^{\delta~~~~~~~~~~~~}& \B_2  \otimes \Lambda^{2n -2} {\Bbb F}^\times_{2n+2}  \ar[r]^{~~~\delta}  
               &\Lambda^{2n}{\Bbb F}^\times_{2n+2}.}
                          \end{gathered}
 \ee 
In  Section \ref{SEC10.2}, we   define   cluster polylogarithm maps  in the    weights two and four. 





   \subsection{The cluster dilogarithm map of complexes and the Bloch complex} \la{SEC10.2}
   
   With the notable exception of the  crucial Proposition \ref{L7.4}, other results of Sections \ref{SEC10.2}-\ref{SECC7.4} already appeared in Section \ref{SEC7.4} in the  cluster form. 
 However, translating them into the concrete language of configurations  is a non-trivial task, which requires a careful setting of the background and 
   normalizations. It also  serves as a key example. Therefore we present all  details.

\paragraph{The $m=0, k=1$ case.} 
Then $ {\rm K}_{0}$ is a point, 
   ${\rm C}_{0}({\rm K}_{0})=\Z$, and we get a map, see Figure \ref{zc4}: 
\[
  \begin{split}
 &  {\Bbb L}_2: {\rm C}_{0}({\rm K}_{0}) \lra    \Lambda^{2}{\Bbb F}^\times_{3}, ~~~~1 \lms W_t.\\
 &  W_t:= \omega(l_1, l_2) \wedge  \omega(l_2, l_3) + \omega(l_2, l_3) \wedge  \omega(l_1, l_3) +\omega(l_1, l_3) \wedge  \omega(l_1, l_2).\\
   \end{split}
\]

 \paragraph{The $m=1, k=1$ case.} The polytope  $ {\rm K}_{1}$ is an interval, and  the map ${\Bbb L}^*_2$ looks   as follows:
  \be \la{MCDImmmy} 
\begin{gathered}
    \xymatrix{
        & {\rm C}_{1}({\rm K}_{1})  \ar[d]^{{\Bbb L}^1_2}  \ar[r]^{d} &        
       {\rm C}_{0}({\rm K}_{1})  \ar[d]^{{\Bbb L}^{2}_2}\\
               &  \B_2[{\Bbb F}_{4}]    \ar[r]^{\delta}    
               &\Lambda^{2}({\Bbb F}^\times_{4}).\\}
                          \end{gathered}
 \ee

The map   ${\Bbb L}_2$ is defined as follows. ${\Bbb L}_2^2$ sends the vertex corresponding to the triangulation ${\cal T}$ the form $W_{\mathcal{T}}.$ The map ${\Bbb L}_2^1$ sends the generator of the group   ${\rm C}_{0}({\rm K}_{1})$ to an element 
\[
{\rm L}_2^1(l_1, l_2, l_3, l_4)=\{[x_1,x_2,x_3,x_4]\}_2\in \B_2({\Bbb F}_{4}).
\]
The commutativity of the diagram (\ref{MCDImmmy}) follows now from (\ref{MCda}).

    \begin{figure}[ht]
\centerline{\epsfbox{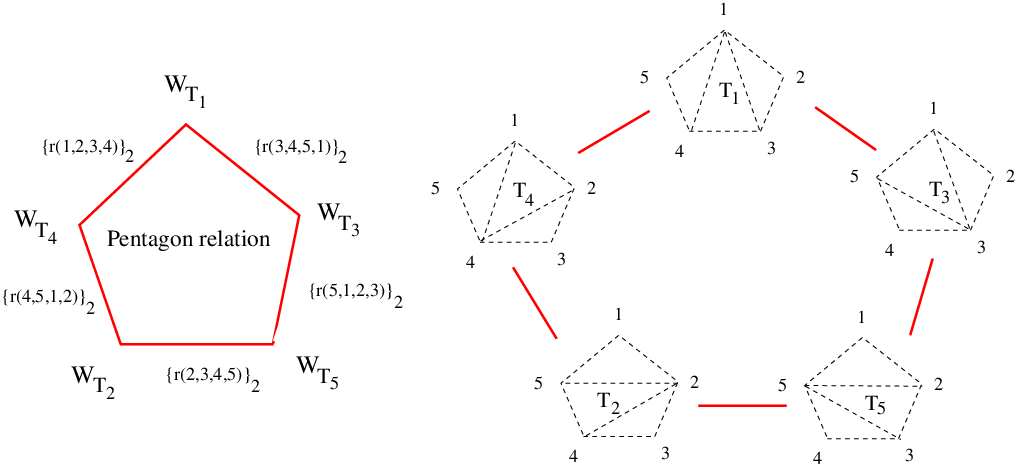}}
\caption{We label the  vertices of the  pentagon ${\rm K}_{2}$ on the left the elements $W_{\cal T}\in \Lambda^2{\Bbb F}_5^\times$ assigned to the triangulations ${\cal T}$ 
of the pentagon    on the right. We   denote  by ${\cal T}_{k}$ the triangulation by the diagonals sharing a vertex $k$. The edges of the pentagon ${\rm K}_{2}$ are labeled by the elements of the Bloch group 
${\rm B}_2({\Bbb F}_5)$. The oriented edge opposite to a vertex ${\cal T}_k$ carries an element $\{[x_{k+1}, x_{k+2}, x_{k+3}, x_{k+4}]\}_2$. The  pentagon ${\rm K}_{2}$ itself gives rise to the
 five-term relation in the Bloch group.}
\label{zc7}
\end{figure}

 \paragraph{The $m=2, k=1$ case and the Bloch complex.}  
 

For the pentagon ${\rm K}_2$ there is  a map, whose restriction to the right square is a map of complexes:
 \be \la{MCDImmmyt} 
  \begin{gathered}
    \xymatrix{
   &  {\rm C}_{2}({\rm K}_{2})  \ar[d]^{{\Bbb L}^0_2}   \ar[r]^{d}    & {\rm C}_{1}({\rm K}_{2})  \ar[d]^{{\Bbb L}^1_2}  \ar[r]^{d} &        
       {\rm C}_{0}({\rm K}_{2})  \ar[d]^{{\Bbb L}^{2}_2}\\
  & 0       \ar[r]^{~}      &  \B_2({\Bbb F}_{5})    \ar[r]^{\delta}    
               &\Lambda^{2}{\Bbb F}^\times_{5}.}
                          \end{gathered}
 \ee
 The map ${\Bbb L}^2_2$ assigns to a vertex of ${\rm K}_2$ the form $W_{{\cal T}}$ for the corresponding triangulation ${\cal T}.$ Next, the map ${\Bbb L}^1_2$ sends an edge $e$ of ${\rm K}_2$ the corresponding element $\{[x_{k+1}, x_{k+2}, x_{k+3}, x_{k+4}]\}\in \B_2(\mathbb{F}_5).$ Finally, the map ${\Bbb L}^0_2$ acts on the generator $[{\rm K}_2]$ of ${\rm C}_{2}({\rm K}_{2})=\Z$ as
 $$
 {\Bbb L}^0_2: [{\rm K}_2] \lms \sum_{k\in \Z/5\Z}    \{[x_{k+1}, x_{k+2}, x_{k+3}, x_{k+4}]\}_2=0.
 $$
This is the five-term relation. Commutativity of the right square, established in the previous section, implies that $\delta$ vanishes on the five-term relation.


\subsection{The weight four cluster polylogarithm map of complexes} \la{SECC7.4}
 
    The Stasheff polytope ${\rm K}_{3}$ is defined by using partial triangulations of the  hexagon ${\rm P_6}$,  provided by any collections of pairwise non-intersecting (except possibly in vertices) diagonals of the hexagon, 
whose vertices are cyclically labeled by $(1, \ldots,  6)$. 
Our first goal is to define   a map of complexes
   \be \la{MCDImmmys} 
\begin{gathered}
    \xymatrix{
      {\rm C}_{3}({\rm K}_{3})  \ar[d]^{{\Bbb L}^1_4}  \ar[r]^{d} &  {\rm C}_{2}({\rm K}_{3})  \ar[d]^{{\Bbb L}^2_4}  \ar[r]^{d} &    {\rm C}_{1}({\rm K}_{3})  \ar[d]^{{\Bbb L}^3_4}  \ar[r]^{d} &        
       {\rm C}_{0}({\rm K}_{3})  \ar[d]^{{\Bbb L}^{4}_4}\\
          \mathbb{L}_{4}({\Bbb F}_{6})    \ar[r]^{\delta~~~~~~~} &  {\rm B}_{3}  \otimes  {\Bbb F}^\times_{6} \bigoplus \Lambda^2 {\rm B}_{2}  \ar[r]^{~~~~\delta}  &      {\rm B}_{2}  \otimes  \Lambda^2{\Bbb F}^\times_{6} \ar[r]^{\delta}    
               &\Lambda^{4}{\Bbb F}^\times_{6}.}
                          \end{gathered}
 \ee  
 
We do that with the use of formulas of  Section \ref{SEC7.4}. To simplify computations, we work in this sections with the form $W^2,$ which is differ by a factor of $2$  from the weight four component of $e^W$ we used in Section \ref{SEC7.4}.

\paragraph{1. The map ${{\Bbb L}^4_4}$.} It assigns to a generic configuration $(l_1, \ldots , l_6)$ of $6$ vectors  in $V_2$, equipped with a symplectic form $\omega_2$,  and a triangulation ${\cal T}$ of the hexagon ${\rm P_6}$ an element 
$$
W_{\cal T}^2(l_1, \ldots , l_6) \in \Lambda^4{\Bbb F}_6^\times.
$$ 
 Since  triangulations   of the hexagon  correspond to the vertices   of  the polytope ${\rm K}_3$, we get a map 
$$
 {{\Bbb L}^4_4}: {\rm C}_{0}({\rm K}_{3})   \lra 
 \Lambda^4{\Bbb F}^\times_{6}.
$$ 

\paragraph{2. The map $ {{\Bbb L}^3_4}$.} Below we denote by $(x_1, \ldots , x_6)$ the configuration of $6$ points on ${\Bbb P}^1$ corresponding 
to  the configuration of $6$ vectors $(l_1, \ldots , l_6)$. An oriented edge $e$  of the Stasheff polytope ${\rm K}_3$ determines 
a cluster Poisson coordinate $X_e$ on the space of configurations $(x_1, \ldots , x_6)$ of six points in ${\Bbb P}^1$. Changing the orientation of the edge $e$   
amounts to inverting   the coordinate. 
 
 Given   a generic configuration $(l_1,  \ldots , l_6)$ of 6   vectors  in $V_2$, and an oriented edge   
$e$ of the Stasheff polytope ${\rm K}_3$ connecting  the vertices    corresponding to triangulations 
${\cal T}$ and ${\cal T}'$, we set
\be \la{MIII}
 {\rm e}(l_1,  \ldots , l_6):= \{-X_{{e}}\}_2\otimes (W_{{\cal T}} + W_{{\cal T}'}). 
\ee
Changing the orientation of the edge $e$ amounts to changing the sign of $\{-X_{{e}}\}_2$, and thus changes the sign of (\ref{MIII}). 
Therefore (\ref{MIII})     indeed defines a map 
$$
 {{\Bbb L}^3_4}: {\rm C}_{1}({\rm K}_{3})   \lra 
  {\rm B}_{2}({\Bbb F}_{6})  \otimes  \Lambda^2{\Bbb F}^\times_{6}.
$$
\paragraph{3. The map $ {{\Bbb L}^2_4}$.}  Let us define a map 
\be \label{MAPLL}
 {{\Bbb L}^2_4}: {\rm C}_{2}({\rm K}_{3})   \lra 
  ({\rm B}_{3}({\Bbb F}_{6})  \otimes  {\Bbb F}^\times_{6}) \oplus \Lambda^2 {\rm B}_{2}({\Bbb F}_{6}).
\ee
For that we need to define an element  on generic configurations of six points $(x_1, \ldots, x_6)$ on ${\Bbb P}^1(\F)$.
The $2$-faces of the Stasheff polytope ${\rm K}_3$ are   pentagons and  squares, accounted as follows:

\begin{itemize}

\item 
Six pentagons $P_i, i\in \Z/6\Z$  assigned to the 
short diagonals $(i-1, i+1)$ of the hexagon ${\rm P}_6$. 

\item Three squares $S_{j, j+3}$ assigned to the long diagonals 
$(j, j+3)$ of the hexagon ${\rm P}_6$. 
\end{itemize}

The ${\rm B}_{3}({\F} )  \otimes  {\F}^\times$-components of   (\ref{MAPLL}) are assigned to the pentagon faces of the polytope ${\rm K}_3$. The $\Lambda^2 {\rm B}_{2}({\F})$-components of   (\ref{MAPLL}) are assigned to the square faces of  the polytope ${\rm K}_3$.

\paragraph{3a. The map ${\rm P}$ for a pentagon.}
Given a configuration of 5 points $(x_1,  \ldots , x_5)$ on ${\Bbb P}^1(\F)$,   set:
\be \la{MI}
{\rm P}(x_1, \ldots, x_5):=    \sum_{k\in \Z/5\Z}  \{[x_{k+1}, x_{k+2}, x_{k+3}, x_{k+4}] \}_3 \otimes  \frac{[x_{k+3}, x_{k+4}, x_{k}, x_{k+1}]}{[x_{k-1}, x_{k}, x_{k+1}, x_{k+2}]} \in {\rm B}_{3}({\F} )  \otimes  {\F}^\times .
\ee
Observe that 
\be \la{SQSCR}
[x_{k+1}, x_{k+2}, x_{k+3}, x_{k+4}]= [x_{k+4}, x_{k+3}, x_{k+2}, x_{k+1}].
\ee
So reversing the cyclic order of the points $(x_1,  \ldots , x_5)$ we do not change the first factor, but inverse the second. Therefore reversing the cyclic order 
amounts to changing the sign of ${\rm P}$:
$$
{\rm P}(x_1, x_2, x_3, x_4, x_5) = - {\rm P}(x_5, x_4, x_3, x_2, x_1).
$$
Thus the map ${\rm P}$ maps the   subgroup of ${\rm C}_{2}({\rm K}_{3})$ generated by the  pentagons to 
${\rm B}_{3}({{\Bbb F}_5} )  \otimes  {\Bbb F}_5^\times$. 

\paragraph{3b. The map ${\rm S}$ for a square.} Given a cyclically ordered configuration of 6 points $(x_1,  \ldots , x_6)$ on ${\Bbb P}^1(\F)$, let us define the following  functions 
$S_j$ of this configuration, were $ j \in \Z/6\Z$:
\be \la{MII}
\begin{split}
&{\rm S}_{j}(x_1, \ldots, x_6):= \{[x_j, x_{j+1}, x_{j+2}, x_{j+3}]\}_2 \wedge \{[x_{j+3}, x_{j+4}, x_{j+5}, x_{j+6}]\}_2 \in \Lambda^2 {\rm B}_{2}({\F} ).\\
\end{split}
\ee
 The cyclic shift by $3$ alters the sign of this expression:    
$$
{\rm S}_{j} (x_1, \ldots, x_6)= - {\rm S}_{j+3}(x_1, \ldots, x_6), ~~~~ j\in \Z/6\Z.
$$
Thanks to (\ref{SQSCR}), reversing the cyclic order of the points  
amounts to changing the sign of ${\rm S}$:
$$
{\rm S}_{1}(x_1, x_2, \ldots, x_5, x_6) = - {\rm S}_{6}(x_6, x_5, \ldots, x_2 , x_1).
$$
Now we assign to the square $S_{j, j+3}$ the function (\ref{MII}).

\paragraph{4. The crucial map $ {{\Bbb L}^1_4}:  {\rm C}_{3}({\rm K}_{3})   \lra 
  \mathbb{L}_{4}({\Bbb F}_{6})$.} 
Equivalently, it as a map on generic configurations of six points $(x_1, \ldots , x_6)$ on ${\Bbb P}^1(\F)$:
\be 
 {{\rm L}^1_4}(x_1, \ldots , x_6)  \in
  \mathbb{L}_{4}({\F}).
\ee
Recall the signed cyclic summation (\ref{SIGNS}).  
Recall the  different shapes of the Casimir element:
\be
Z:=  [x_1, x_2, x_3, x_4, x_5, x_6] = \frac{[1,2,3,4]}{[4,5,6,1]} = -\frac{|12||34||56|}{|23||45||61|}.
\ee


 \bd
 For any   $6$ distinct points $(x_1, x_2,\ldots,x_6)$ on ${\Bbb P}^1(\F)$ we set:
 \be
 \begin{split}
 {{\rm L}^1_4}(x_1, \ldots , x_6) := 
&{\rm Cyc}^-_6\Bigl( \left \{[x_1, x_2,x_3,x_4],[x_4, x_5,x_6,x_1] \right\}_{3,1} - \left\{[x_1,x_3,x_4,x_5]\right \}_4   
 -\left\{Z\right \}_4\Bigr). \\
 \end{split}
 \ee
 
\end{definition}

 \bt \la{TH7.3} The defined above maps ${\Bbb L}^\ast_4$   describe a morphism of complexes (\ref{MCDImmmys}). 

\et
\begin{proof} 
Here is the key Proposition. The commutativity of the other two squares in (\ref{MCDImmmys}) has been proved in  Section \ref{SEC7.4}. 
 
\bp \la{L7.4}
The following diagramm  is commutative:
 \be
  \begin{gathered}
    \xymatrix{
         & {\rm C}_{3}({\rm K}_{3})  \ar[d]^{{\Bbb L}^1_4}  \ar[r]^{d} &  {\rm C}_{2}({\rm K}_{3})  \ar[d]^{{\Bbb L}^2_4}      \\
              & {\Bbb L}_{4}({\Bbb F}_{6})    \ar[r]^{\delta~~~~~~~} &  {\rm B}_{3}  \otimes  {\Bbb F}^\times_{6} \bigoplus \Lambda^2 {\rm B}_{2}     }
                          \end{gathered}
 \ee  
 \ep

\begin{proof} 
Recall the signed cyclic sum ${\rm Cyc}_6^-$ from (\ref{SIGNS}). We have shown that $(\mathbb{L}^2_4\circ d)(1)$ gives the following function on configurations $(x_1, \ldots , x_6)$:
\be \la{247}
\begin{split}
&{\rm Cyc}^-_6\left ({\rm P}(x_{1}, \ldots, x_{5}) +   
{\rm S}_{i}(x_1, \ldots, x_6)\right)=\\
&{\rm Cyc}^-_6\Bigl ( \{[1,2,3,4]\}_3\otimes \frac{[3,4,5,1]}{[4,5,1,2]}+\{[2,3,4,5]\}_3\otimes \frac{[4,5,1,2]}{[5,1,2,3]}+\\
&\{[3,4,5,1]\}_3\otimes \frac{[5,1,2,3]}{[1,2,3,4]}+\{[4,5,1,2]\}_3\otimes \frac{[1,2,3,4]}{[2,3,4,5]}+\\
&\{[5,1,2,3]\}_3\otimes \frac{[2,3,4,5]}{[3,4,5,1]}+\{[1,2,3,4]\}_2\wedge\{[4,5,6,1]\}_2 \Bigr ).\\
\end{split}
\ee
We can simplify this expression as follows:
\[
\begin{split}
(\ref{247}) = &{\rm Cyc}^-_6\Bigl (\{[1,2,3,4]\}_3\otimes\frac{[3,4,5,1][4,6,1,2]}{[4,5,1,2][3,4,6,1]}+\{[3,4,5,1]\}_3\otimes\frac{[5,1,2,3][4,5,6,1]}{[1,2,3,4][5,6,1,3]}+\\
&\frac{1}{2}\{[4,5,1,2]\}_3\otimes\frac{[1,2,3,4][5,6,1,2]}{[2,3,4,5][4,5,6,1]}+\{[1,2,3,4]\}_2\wedge\{[4,5,6,1]\}_2 \Bigr ). \\
\end{split}
\]
Elaborating this we get: 
\[
\begin{split}
(\ref{247}) =&{\rm Cyc}^-_6\Bigl (2 \{[1,2,3,4]\}_3 \otimes\frac{|15||46|}{|45||16|}+\{[3,4,5,1]\}_3\otimes\frac{|15||34|}{|13||45|}Z^{-2}+\\
&\frac{1}{2}\{[4,5,1,2]\}_3\otimes Z^2+\{[1,2,3,4]\}_2\wedge\{[4,5,6,1]\}_2\Bigr ).\\
\end{split}
\]
Combining the terms containing Casimir element, we finally arrive at
\[
\begin{split}
(\ref{247}) =&{\rm Cyc}^-_6\Bigl (\delta\{[3,4,5,1]\}_4+2\{[1,2,3,4]\}_3\otimes[1,5,4,6]+\\
&\{[1,2,3,4]\}_2\wedge\{[4,5,6,1]\}_2  +
  \left ( \{[4,5,1,2]\}_3-2\{[3,4,5,1]\}_3 \right )\otimes Z\Bigl).\\
\end{split}
\]
On the other hand, (\ref{26}) implies that
\be \la{251}
\begin{split}
&\delta \circ {\rm Cyc}^-_6 \{[1,2,3,4],[4,5,6,1]\}_{3,1}= \\
&{\rm Cyc}^-_6\Bigl ( \{[1,2,3,4],[4,5,6,1]\}_{2,1}\otimes Z+\left\{Z\right \}_3\otimes\frac{[1,3,2,4]}{[4,6,5,1]}+\\
&\{[1,2,3,4]\}_3\otimes[4,6,1,5]-\{[4,5,6,1]\}_3\otimes[1,3,4,2]+\{[1,2,3,4]\}_2\wedge\{[4,5,6,1]\}_2 \Bigr ).\\
\end{split}
\ee
Observe that ${\rm Cyc}^-_6 \left( [1,3,2,4] / [4,6,5,1]\right) = Z^2$. Therefore 
we get  
\[
\begin{split}
(\ref{251}) = 
&{\rm Cyc}^-_6 \Bigl ( \{[1,2,3,4],[4,5,6,1]\}_{2,1} \otimes Z+ 
   2\{[1,2,3,4]\}_3\otimes[4,6,1,5]\\ 
   &+\{[1,2,3,4]\}_2\wedge\{[4,5,6,1]\}_2 \Bigr )+2 \delta \left\{Z\right \}_4.
\end{split}
\]
We rewrite this expression using the crucial relation ${\bf Q}_3$ for the trilogarithm, whose proof is given in Section \ref{SEC9}, and does not use the results of Section \ref{SECC7.4}:
\[
\begin{split}
(\ref{251}) = 
&{\rm Cyc}^-_6 \Bigl ( \{[1,2,4,5]\}_3-2\{[1,3,4,5]\}_3 \otimes Z   + 
    2\{[1,2,3,4]\}_3\otimes[4,6,1,5] \\ & +\{[1,2,3,4]\}_2\wedge\{[4,5,6,1]\}_2 + \delta \left\{Z\right \}_4\Bigr ).\\
\end{split}
\]
From this it follows that
\[
\begin{split}
&\bigl((\mathbb{L}^2_4\circ d)(1)\bigr) (x_1, \ldots , x_6) = \delta \circ {\rm Cyc}^-_6\Bigl ( \{[1,2,3,4],[4,5,6,1]\}_{3,1}-\{[1,3,4,5]\}_4 -\{Z\}_4\Bigr ).\\
\end{split}
\]
Proposition \ref{L7.4} is proved. \end{proof}
 Theorem \ref{TH7.3} is proved.  \end{proof}

 


\section{Lie coalgebra structure on  ${\Bbb L}_{\bullet}(\F)$}\la{SEC9}
In this section, we show that cobracket maps on ${\Bbb L}_n(\F)$ are well-defined for $n\leq 4.$ It is sufficient to consider the cases $n=3$ and $n=4$ only.

\subsection{The map ${\Bbb L}_3(\F) \lra {\Bbb L}_2(\F)\otimes \F^\times$} 
Recall Definition \ref{DEFQ3} of the group $\mathbb{L}_3(\F)$. In particular, we recall   the shape of relation ${\bf Q}_3:$

 \begin{itemize}
  \item ${\bf Q}_3:$ For any   $6$ distinct points $(x_1, x_2,\ldots,x_6)$ on ${\Bbb P}^1(\F)$ the following cyclic sum is zero:
\[
 \begin{split}
&{\rm Cyc}_6\Bigl( \left \{[x_1, x_2,x_3,x_4],[x_4, x_5,x_6,x_1]\}_{2,1} - \{[x_1,x_2,x_4,x_5]\right \}_3  + 2\cdot \left\{[x_1,x_3,x_4,x_5] \right\}_{3}\Bigr)   \\
&  -4\cdot \left\{[x_1,x_2,x_3,x_4,x_5,x_6]\right \}_3 =0. \\
 \end{split}
\]
\end{itemize}
The relation ${\bf Q}_3$ is proved in Proposition \ref{PP1}. Let us discuss first   consequences of   relation    ${\bf Q}_3$.  We obtain these consequences by applying specialization, see (\ref{121}).

\begin{lemma} \la{LLMM}
The following relation hold:   
\be \la{KER}
\begin{split}
&\left \{x,y \right\}_{2,1}=
\left\{1-x^{-1}\right \}_3+\left\{1-y^{-1}\right \}_3+\left\{\frac{1-y}{1-x}\right \}_3-\left\{\frac{1-y^{-1}}{1-x^{-1}}\right \}_3+\left\{\frac{y}{x}\right \}_3- \{1\}_3.\\
\end{split}
\ee\end{lemma}
\begin{proof} 
Specializing  ${\bf Q}_3$ to the divisor $D_{26}\subset \overline {\mathcal{M}}_{0,6}$ we get, setting $x=[1,2,3,4], ~~y=[4,5,2,1]$:
\be \la{333}
\begin{split}
&\{x, y\}_{2,1} + \{y/x, 1\}_{2,1}  + \{x/y, 1\}_{2,1}  -\{1-y^{-1}\}_3  - \{1-x^{-1}\}_3 \\
& + \left \{\frac{1-x^{-1}}{1-y^{-1}}\right \}_3 - \left \{\frac{1-x}{1-y}\right \}_3 - 2 \left \{\frac{y}{x}\right \}_3   =-  2 \{1\}_3.\\ \end{split}
\ee

The following two equations are obtained by specialization.  Specializing  (\ref{333})  to $y=1$, we get the first one. 
Switching $x \leftrightarrow x^{-1}$  we get the second:
\be \la{333g}
\begin{split}
& 2\cdot  \{x, 1\}_{2,1}  + \{x^{-1}, 1\}_{2,1}  =\{1-x^{-1}\}_3    
  + 2\{x\}_3 -2\cdot \{1\}_3.\\ 
& 2\cdot  \{x^{-1}, 1\}_{2,1}  + \{x, 1\}_{2,1}  =\{1-x \}_3    
  + 2\{x\}_3 -2\cdot \{1\}_3.\\ 
    \end{split}
\ee

Solving this system of equations we get
\be \la{333s}
\begin{split}
&   \{x, 1\}_{2,1}  =    
  -   \{1-x \}_3.\\ 
    \end{split}
\ee 
Substituting to (\ref{333}), and using (\ref{EATRI}),  we get (\ref{KER}). \end{proof}

\bl
 i) The $22-$term functional equation for trilogarithm from \cite{Gon91} is just equivalent to the relation obtained by substituting relation (\ref{KER}) to  the  relation ${\bf Q}_3$.    
 
ii) The canonical map  $\B_3(\F) \lra \mathbb{L}_3(\F)$ is an isomorphism. 
\el

\begin{proof} i) This is straightforward.

 ii) The $22-$term relation for the trilogarithm  implies the one defined via the triple ratio, see the discussion on \S \ref{SectionFunctionalEquations}. On the other hand, the former can be obtained as the specialization of the latter. 
\end{proof}

\begin{definition}
Let $\delta \colon \mathbb{L}_3(\F) \longrightarrow \B_2(\F) \otimes \F^{\times}$ be a map given on the generators as follows: 
\be \la{145}
\begin{split}
&\delta\left \{x,y \right\}_{2,1}=
\left \{x\right\}_2 \otimes (1-y^{-1}) + \left \{y\right\}_2 \otimes (1-x^{-1})+\left \{\frac{1-y}{1-x}\right\}_2 \otimes \frac{y}{x}+\left \{\frac{y}{x}\right \}_2 \otimes \frac{1-y}{1-x},\\
&\delta\left \{x\right\}_3=\left \{x\right\}_2 \otimes x.\\
\end{split}
\ee
\end{definition}

It is easy to check that the composition $\mathbb{L}_3(\F) \stackrel{\delta}{\longrightarrow} \B_2(\F) \otimes \F^{\times} \stackrel{\delta}{\longrightarrow} \Lambda^3\F^\times$ is zero.
\bp \la{PP1}
The cobracket $\delta$ sends relation ${\bf Q}_3$ to zero.  
\ep

 \begin{proof} 
 Using $[1,3,4,2] =  1 - [1,2,3,4]^{-1}$ and formula (\ref{145}) for the cobracket we get:
 \be \la{210}
 \begin{split}
&\delta \circ {\rm Cyc}_6 \Bigl  (\left \{[1, 2,3,4],[4, 5,6,1] \right\}_{2,1}- \left\{[1,2,4,5]\right \}_3  +   2\left\{[1,3,4,5]\right \}_3  \Bigr ) 
 -4\delta \left\{Z\right \}_3 =\\
& {\rm Cyc}_6 \Bigl  (\left\{[1,2,3,4]\right \}_2 \otimes [4, 6,1,5]+\left\{[4,5,6,1]\right \}_2 \otimes [1, 3,4,2]+\\
&\left\{\frac{[4,6,5,1]}{[1,3,2,4]}\right \}_2 \otimes \frac{[4,5,6,1]}{[1,2,3,4]} +\left\{\frac{[4,5,6,1]}{[1,2,3,4]}\right \}_2 \otimes \frac{[4,6,5,1]}{[1,3,2,4]}\\
&- \delta\left\{[1,2,4,5]\right \}_3  + 2\delta \left\{[1,3,4,5]\right \}_3 \Bigr )  
 -4 \delta \left\{Z\right \}_3 \stackrel{(\ref{263})}{=}\\
 & {\rm Cyc}_6 \Bigl  (2\left\{[1,2,3,4]\right \}_2 \otimes [4, 6,1,5] +\left\{\frac{[4,6,5,1]}{[1,3,2,4]}\right \}_2 \otimes \frac{[4,5,6,1]}{[1,2,3,4]}\\
&- \delta\left\{[1,2,4,5]\right \}_3  + 2\delta \left\{[1,3,4,5]\right \}_3   \Bigr )
 -2 \delta \left\{Z\right \}_3.\\
 \end{split}
 \ee 
 
 
 The second equality follows from the following  observations. First, 
 \be \la{263}
\frac{[4,5,6,1]}{[1,2,3,4]} = Z^{-1}, 
~~~~{\rm Cyc}_6^-\frac{[4,6,5,1]}{[1,3,2,4]} = Z^{-2}.
\ee 
Using this we get 
\be
\begin{split}
&{\rm Cyc}_6\left\{[1,2,3,4]\right \}_2 \otimes [4, 6,1,5]={\rm Cyc}_6 \left\{[4,5,6,1]\right \}_2 \otimes [1, 3,4,2].\\
&{\rm Cyc}_6\left(\left\{\frac{[4,5,6,1]}{[1,2,3,4]}\right \}_2 \otimes \frac{[4,6,5,1]}{[1,3,2,4]}\right) \stackrel{(\ref{263})}{=}
2\delta \{Z\}_3.
\end{split}
\ee 
 
 Let us compute the $\B_2$-components of the cobracket (\ref{210}) with a given $\F^\times$-factor, after expanding each $\F^\times$-factor into irreducible components, up to a sign.
  Due to the $6-$fold cyclic symmetry it suffices to consider $|12|, |13|$ and $|14|$, 
and an easy   check shows that   the latter do not appear at all.
 
Let us calculate  the  element of $\B_2$  in front of the $|13|$-factor. From the last formula in (\ref{210}) we see that out of its five-terms, the two more complicated 
ones   do not contribute, and the rest give twice the  five-term relation  in the Bloch group $\B_2$ for  the points $(1,3,4,5,6)$:
\[
\begin{split}
& \{[4,5,6,1]\}_2+\{[3,4,5,6]\}_2+ \{[3,4,6,1]\}_2 +
\{[1,3,4,5]\}_2+ \{[1, 3,5,6]\}_2 =0.\\
\end{split}
\]
 
 Finally, computing the element of $\B_2$ in front of the $|12|$-factor we get twice the following:
 \be \la{220}
\begin{split}
& -\{[4,5,6,1]\}_2-\{[2,3,4,5]\}_2 -\{[1,2,4,5]\}_2 - \{[5,1,2,3]\}_2+\{[4,6,1,2]\}_2\\
&-\left\{\frac{[4,6,5,1]}{[1,3,2,4]}\right \}_2+\left\{\frac{[5,1,6,2]}{[2,4,3,5]}\right \}_2-\left\{\frac{[6,2,1,3]}{[3,5,4,6]}\right \}_2- \{Z\}_2.\\
\end{split}
\ee
The first three terms in the second line in this formula can be  written as  follows:
\[
\begin{split}
 &\frac{[4,6,5,1]}{[1,3,2,4]}  = \frac{[1,4,2,3]}{[1,4,6,5]}, ~~~~~~~~
 \frac{[5,1,6,2]}{[2,4,3,5]}= \frac{[1,4,2,6]}{[1,4,3,5]}~~~~~~~~
\frac{[6,2,1,3]}{[3,5,4,6]} =  \frac{[1,4,6,2]}{[1,4,5,3]}.\\
  \end{split}
\]
Using this, and the formula   $Z = \dfrac{[1,4,5,6]}{[1,4,3,2] }$, we     present (\ref{220}) as a sum of  three five-term relations:
\[
\begin{split}
&\{[1,2,4,6]\}_2-\{[1,3,4,5]\}_2+\left\{\frac{[1,3,4,5] }{[1,2,4,6]}\right \}_2-\left\{\frac{[1,4,6,2]}{[1,4,5,3]}\right \}_2  +\left\{\frac{[1,4,2,6]}{[1,4,3,5]}\right \}_2,\\
&\{[1,3,4,2]\}_2-\{[1,5,4,6]\}_2+\left\{\frac{[1,5,4,6]}{[1,3,4,2]}\right \}_2 -\left\{\frac{[1,4,2,3]}{[1,4,6,5]}\right \}_2 + \left\{\frac{ [1,4,3,2]}{[1,4,5,6]}\right \}_2, \\
&-\{[1,2,3,4]\}_2+\{[1,2,3,5]\}_2-\{[1,2,4,5]\}_2+\{[1,3,4,5]\}_2-\{[2,3,4,5]\}_2. \\
\end{split}
\]
Here we use  
\[
 \dfrac{[1,3,4,5] }{[1,2,4,6]} = \left(\dfrac{[1,5,4,6]}{[1,3,4,2]}\right)^{-1}
\]
to check that the sum of the middle terms in the first two lines vanishes. 
\end{proof}
\subsection{The map $\delta\colon \mathbb{L}_4(\F)\lra (\mathbb{L}_3(\F)\otimes \F^{\times})\oplus \Lambda^2 \B_2(\F)$} 
\begin{proposition}\la{FRWT4} 
The cobracket $\delta\colon \mathbb{L}_4(\F)\lra (\B_3(\F)\otimes \F^{\times})\oplus \Lambda^2 \B_2(\F)$ sends relation ${\bf Q}_4$ to zero. 
\end{proposition}
\begin{proof}
Consider the cluster polylogarithm map of complexes   for $m=4$. Its left part looks as follows: 
\[
\begin{gathered}
    \xymatrix{
       {\rm C}_{4}({\rm K}_{4})  \ar[d]   \ar[r]^{d} & {\rm C}_{3}({\rm K}_{4})  \ar[d]^{{\Bbb L}^1_4}  \ar[r]^{d} &  {\rm C}_{2}({\rm K}_{4})  \ar[d]^{{\Bbb L}^2_4}      \\
          0 \ar[r]    & {\Bbb L}_{4}({\Bbb F}_{7})    \ar[r]^{\delta~~~~~~~} &  {\rm B}_{3}  \otimes  {\Bbb F}^\times_{7} \bigoplus \Lambda^2 {\rm B}_{2}}
\end{gathered}.
\]
The identity $d \circ d=0$ and the commutativity of the right square imply that we have: 
$$
\delta  \sum_{i=1}^{7}(-1)^i {\rm L}^1_{4}(x_1, \ldots , \widehat x_i, \ldots , x_{7}) =0.
$$
More explicitly, this implies that $\delta$ vanishes on the $7$-term relation
 \be 
 \begin{split}
{\rm Cyc}_7& \Bigl( -\left \{[x_1,x_2,x_3,x_4],[x_4,x_6,x_7,x_1] \right\}_{3,1}\\
&+ \left \{[x_1, x_2,x_3,x_4],[x_4,x_5,x_7,x_1] \right\}_{3,1}\\
&- \left \{[x_1, x_2,x_3,x_4],[x_4,x_5,x_6,x_1 ] \right\}_{3,1}\\
&- \left \{[x_1, x_2,x_4,x_6]\right\}_{4}+ 3 \left \{[x_1,x_2,x_3, x_4,x_5, x_6]\right\}_{4}\Bigr).
\end{split}
\ee
Since $\delta$ commutes with specializations, it also vanishes on the specializations of the $7$-term relation. This finishes the proof of the proposition.
\end{proof}

\subsection{End of the proof of Theorem \ref{JMC}.} Since the map $\{x,y\}_{m-1,1} \to \{x,y\}^{\cal M}_{m-1,1} $ for $m=3,4$ respects  cobrackets by Proposition \ref{LL1}, and 
relations ${\bf Q}_m$ are killed by the cobracket thanks to Proposition \ref{PP1} and Theorem \ref{TH7.3},
 their images in ${\cal L}_m(\F)$ are  killed by the cobracket. Thus the image 
of  relation ${\bf Q}_m$ in the Hodge realization is constant. Since ${\rm Ext}_{{\cal M}_T(\Q)}^1(\Q(0), \Q(4))=0$ we complete  the proof when $m=4$ by specializing to a   point $(x,y)\in \Q^2$. 
When $m=3$ the Ext group is generated by $\{1\}_3$. The constant is computed by specializing $x=0$ at (\ref{KER}). 

It remains to show that  specialization formulas (\ref{DegenerationRules}) are preserved by the   motivic correlators map (\ref{MCOA1}).  
We already know that they hold up to a constant, since hold after the cobracket. Thus the
 constant is an element of ${\rm Ext}_{{\cal R}}^1(\Q(0), \Q(m))$, where ${\cal R}$ is the realization we consider. 
Since this Ext is of geometric origin, the constant is the realization of an element in ${\rm Ext}_{{\cal M}_T(\Q)}^1(\Q(0), \Q(m))$. The latter is zero if $m=4$. 
In the $m=3$ case   it is sufficient to check that it is annihilated by  the canonical real period map (\cite{Gon08}). 
 This boils down to proving that the appropriate 
Hodge correlator integrals behave continuously, which  follows from an argument given by N. Malkin  in \cite[Theorem 11]{M20}. 
Theorem \ref{JMC} is proved.

\section{Proof of Theorem \ref{THRN}} \la{SSEECC10}

For the convenience of the reader, we restate the main claim of Theorem \ref{THRN}. 

\begin{theorem} \la{L4TH}
The following sequence is exact:
\[
0 \longrightarrow {\rm B}_4(\F) \longrightarrow \mathbb{L}_4(\F)
\stackrel{p}{\longrightarrow}  \Lambda^2{\B}_2(\F)
\longrightarrow 0.
\]
\end{theorem}

\begin{proof} Since $p\{x,y\}_{3,1}= \{x\}_2 \wedge \{y\}_2,$ the second map is surjective. 

The first map is injective by the very definition.

To prove exactness in the middle term we  consider a map  
\be \la{MAPpr}
\begin{split}
&{\rm pr} \colon \Lambda^2 \mathbb{Q}(\F) \longrightarrow 
\dfrac{\mathbb{L}_4(\F)}{{\rm B}_4(\F)}.\\
&{\rm pr}(\{x\}\wedge \{y\}):=\{x,y\}_{3,1}.
\end{split}
\ee

Theorem \ref{L4TH} will follow if we show that the map 
(\ref{MAPpr})  factors through  $  {\B}_2(\F) \wedge  {\B}_2(\F).$

For any six distinct points 
$(x_1,x_2,x_3,x_4,x_5,x_6)$  on ${\Bbb P}^1(\F)$ let us define an element 
 $$
 F(x_1,x_2,x_3,x_4,x_5,x_6) \in \Lambda^2\mathbb{Q}[\F].
 $$
 It is  the cross-ratio of the first four points wedge the ``five-term relation'' for the last five:  
$$
F(x_1,x_2,x_3,x_4,x_5,x_6):=
$$
$$
[x_1,x_2,x_3,x_4]\wedge \Bigl([x_2,x_3,x_4,x_5]+[x_3,x_4,x_5,x_6]+[x_4,x_5,x_6,x_2]+[x_5,x_6,x_2,x_3]+[x_6,x_2,x_3,x_4]\Bigr).
$$
The claim  that the map (\ref{MAPpr}) factors through 
$ \Lambda^2{\B}_2(\F)$ is just equivalent to  the claim that  
\be \label{PROJECTION}
{\rm pr}(F(x_1,x_2,x_3,x_4,x_5,x_6))=0.
\ee

We  prove this by   specializing   relation ${\bf Q}_4$ to different strata of $\overline {\cal M}_{0,7}$. 

Recall that a genus zero stable curve with $n$ marked points is a pair 
$(C; y_1, \ldots, y_n)$, where $C$ is a simply-connected  curve  whose components are isomorphic to ${\Bbb P}^1$,  
the only singular points are simple double points distinct from the marked points, and on each component the total number of marked and singular points   is at least $3$. 
The moduli space 
$\overline {\cal M}_{0,n}$ parametrizes genus zero stable curves with $n$ marked points $(C; y_1, \ldots, y_n)$. 

The  space $\overline {\cal M}_{0,6}$ is identified with the divisor $D_{72}\subset \overline {\cal M}_{0,7}$  
as follows:\footnote{Here $C\cup_{x_2} {\Bbb P}^1$ is a new stable curve whose components $C$ and ${\Bbb P}^1$ have a single common point.}
$$
(C; x_1, \ldots , x_6) \lms (C \cup_{x_2} {\Bbb P}^1; \{x_1, \ldots , \widehat x_2, \ldots x_6\} \cup \{y_7, y_2\}).
$$
  Informally,   points $y_7, y_2$   collide into a  point  $x_2$, and  other points $y_j$ are identified with   $x_j$. 

Lemma \ref{COML} uses this identification. 
All identities below are   modulo ${\rm B}_4(\F)$. 

\bl \la{COML} The specialization  of  relation  ${\bf Q_4}$ on the divisor in $D_{72} \subset \overline{\mathcal{M}}_{0,n}$ is equal  
modulo  ${\rm B}_4(\F)$ to the following expression, where $[a,b,c,d]$ stands for $[x_a, x_b, x_c, x_d]$:
\be \la{200a}
\begin{split}
&-\left \{ [1, 2,3,4],[4, 6,2,1] \right \}_{3,1}
-\left \{ [3, 4,5,6],[6,1,2,3] \right \}_{3,1}
-\left \{ [5,6,2,1],[1, 3,4,5] \right \}_{3,1}\\
&+\left \{ [1, 2,3,4],[4,5,2,1] \right \}_{3,1}
+\left \{ [2,3,4,5],[5,6,1,2] \right \}_{3,1}
+\left \{ [4,5,6,2],[2,1,3,4] \right \}_{3,1}\\
&+\left \{ [5,6,2,1],[1,2,4,5] \right \}_{3,1}\\
&-\left \{ [1,2,3,4],[4,5,6,1] \right \}_{3,1}
-\left \{ [3,4,5,6],[6,2,1,3] \right \}_{3,1}
-\left \{ [5,6,2,1],[1,2,3,5] \right \}_{3,1}
=0.\\
 \end{split}
\ee
\el

\begin{proof} The first line in (\ref{200a}) is obtained by taking the  sum 
$-{\rm Cyc}_7 \left \{[x_1,x_2,x_3,x_4],[x_4,x_6,x_7,x_1] \right\}_{3,1}$ and imposing the relation 
$x_2=x_7$. Since we work modulo   ${\rm B}_4$,     only three out of  seven terms in the cyclic sum 
survive. Indeed, the other four terms are 
\be \la{201}
\begin{split}
 &\left \{[2, 3, 4, 5],[5, 7, 3, 2] \right\}_{3,1}, ~~ \left \{[4, 5, 6, 7],[7, 2, 5, 4] \right\}_{3,1}, \\
&  \left \{[6, 7, 1, 2],[2, 4, 7, 6] \right\}_{3,1}, ~~\left \{[7, 1, 2, 3],[3, 5, 1, 7] \right\}_{3,1} .\\
  \end{split}
\ee
Each of them has a cross-ratio containing both points $x_2$ and $x_7$, and hence is zero modulo ${\rm B}_4$ thanks to   specialization  relations (\ref{202}).

The second plus third  lines, as well as line four,  are obtained similarly. \end{proof} 

\begin{lemma} 
The following relations hold in the quotient $\dfrac{\mathbb{L}_4(\F)}{{\rm B}_4(\F)}$:
\[
\begin{split}
&a)\left \{x,y \right \}_{3,1} = -\left\{y,x \right \}_{3,1},\\
&b)\left \{x,y \right \}_{3,1} = \left\{x^{-1}, y^{-1} \right \}_{3,1},\\
&c)\left \{x,y \right \}_{3,1} = \left\{1-x,1-y \right \}_{3,1},\\
&d)\left \{x,y \right \}_{3,1} + \left\{x,1-y \right \}_{3,1}+\left \{\frac{x}{y},\frac{1-x}{1-y}  \right \}_{3,1} + \left\{\frac{x}{y},\frac{1-y}{1-x}  \right \}_{3,1}=0.\\
\end{split}
\]

\end{lemma}
\begin{proof}
 
a) Specializing   relation  ${\bf Q_4}$ to  divisor $D_{67}$  we get the following identity, equivalent to  a): 
\[
\left \{[3,4,5,6],[6,1,2,3] \right \}_{3,1} = -\left\{[6,1,2,3],[3,4,5,6] \right \}_{3,1}.
\]

b) 
Specializing   relation  (\ref{200a}) to the divisor $(651x,x432)$ given by the stable curve   ${\Bbb P}^1 \cup_x {\Bbb P}^1$  with the double    
point $x$ and marked points $\{6,5,1\}$ and $\{4,3,2\}$ respectively,  
we obtain 
\[
\left \{ [2,3,4,5],[5,6,1,2] \right \}_{3,1}=\left \{ [1,2,3,4],[4,5,6,1] \right \}_{3,1}. 
\]
This implies  $
\left \{ x,y \right \}_{3,1}=\left \{ x^{-1}, y^{-1}  \right \}_{3,1}.
$

c) Specializing   relation  (\ref{200a})   to the divisor $(13x,2x456)$  of $D_{72}$ given by the stable curve ${\Bbb P}^1 \cup_x {\Bbb P}^1$ with the  marked points $\{1,3\}$ and $\{2,4,5,6\}$ respectively, we obtain that
\[ 
\left \{ [2,1,4,5],[5,6,1,2] \right \}_{3,1}+\left \{ [5,6,2,1],[1,2,4,5] \right \}_{3,1}=0.
\]
Equivalently, 
$ 
\left \{ x,y \right \}_{3,1}+\left \{ (1-y^{-1})^{-1},(1-x^{-1})^{-1} \right \}_{3,1}=0.
$ 
Then c) follows by using a) and b).

d) Specializing   relation  (\ref{200a})      to the divisor  $(52x, x1346)$  of $D_{72}$ given by the stable curve ${\Bbb P}^1 \cup_x {\Bbb P}^1$ with the  marked points $\{5,2\}$ and $\{1,3,4,6\}$ respectively, we obtain that
\[
\begin{split}
&\left \{ [1,2,3,4],[4,6,2,1] \right \}_{3,1}+\left \{ [1,2,3,4],[4,2,6,1] \right \}_{3,1}+\\
&\left \{ [3,4,2,6],[6,1,2,3] \right \}_{3,1}+ \left \{ [3,4,2,6],[6,2,1,3] \right \}_{3,1}=0.\\
\end{split}
\]
One checks that this is just equivalent to
\[ 
\begin{split}
&\left \{ x,y \right \}_{3,1}+\left \{ x,1-y \right \}_{3,1}+
\left \{ \frac{x}{x-y},\frac{1-y}{x-y} \right \}_{3,1}+\left \{ \frac{x}{x-y},\frac{x-1}{x-y} \right \}_{3,1}=0.\\
\end{split}
\]
Using b) and c), we get that the following two identities, which imply d): 
\[
\begin{split}
&\left \{ \frac{x}{x-y},\frac{1-y}{x-y} \right \}_{3,1}=\left \{ \frac{1}{1-\frac{y}{x}},\frac{1}{1-\frac{1-x}{1-y}} \right \}_{3,1} = \left \{  \frac{x}{y}, \frac{1-y}{1-x}  \right \}_{3,1}, \\
&\left \{ \frac{x}{x-y},\frac{x-1}{x-y} \right \}_{3,1}=\left \{ \frac{1}{1-\frac{y}{x}},\frac{1}{1-\frac{1-y}{1-x}} \right \}_{3,1}= \left \{  \frac{x}{y}, \frac{1-x}{1-y}  \right \}_{3,1} . \\
\end{split}
\]
\end{proof}

\begin{lemma} \la{LEM1}
The following relations hold  the quotient $\dfrac{\mathbb{L}_4(\F)}{{\rm B}_4(\F)}$:
\[
\begin{split}
&i)\left \{x,y \right \}_{3,1}=-\left\{x,1-y \right \}_{3,1},\\
&ii)\left \{x,y \right \}_{3,1}=-\left\{x,  y^{-1}  \right \}_{3,1}.\\
\end{split}
\]

\end{lemma}
\begin{proof}
1. Let us specialize   relation   (\ref{200a})  to $(63y,y1245).$ We get that 
\[
\begin{split}
&-\left \{ [1,2,3,4],[4,3,2,1] \right \}_{3,1} -\left \{ [5,3,2,1],[1,2,3,5] \right \}_{3,1} -\left \{ [5,3,2,1],[1,3,4,5] \right \}_{3,1} \\
&+\left \{ [1,2,3,4],[4,5,2,1] \right \}_{3,1}+
\left \{ [2,3,4,5],[5,3,1,2] \right \}_{3,1}
+\left \{ [4,5,3,2],[2,1,3,4] \right \}_{3,1}\\
&+\left \{ [5,3,2,1],[1,2,4,5] \right \}_{3,1}-
\left \{ [1,2,3,4],[4,5,3,1] \right \}_{3,1}
=0.\\
\end{split}
\]
The first two terms vanish since  $\{x,x\}_{3,1}=0$. Using $[i,j,k,l]=[l,k,j,i]$ and Lemma \ref{LEM1} we get
\be \label{D1}
\begin{split}
&\left \{ [1,2,3,4],[4,5,2,1] \right \}_{3,1}
-\left \{ [1,2,3,4],[5,4,3,2] \right \}_{3,1}-
\left \{ [1,2,3,4],[4,5,3,1] \right \}_{3,1}=\\
&\left \{ [1,2,3,5],[1,3,4,5] \right \}_{3,1}
-\left \{ [1,2,3,5],[1,2,4,5] \right \}_{3,1}
+\left \{ [1,2,3,5],[2,3,5,4] \right \}_{3,1}.\\
\end{split}
\ee

2. 
 Relation   (\ref{200a})  and Lemma \ref{LEM1} imply that
\be \label{D3}
\begin{split}
&-\left \{ [1,2,3,4],[4,6,2,1] \right \}_{3,1}
+\left \{ [1,2,3,4],[4,5,2,1] \right \}_{3,1}\\
&-\left \{ [1,2,3,4],[5,4,6,2] \right \}_{3,1}
-\left \{ [1,2,3,4],[4,5,6,1] \right \}_{3,1}-\\
&-\left \{ [5,6,2,1],[1,3,4,5] \right \}_{3,1}
-\left \{ [5,6,2,1],[2,3,5,4] \right \}_{3,1}\\
&+\left \{ [5,6,2,1],[1,2,4,5] \right \}_{3,1}
-\left \{ [5,6,2,1],[1,2,3,5] \right \}_{3,1}=\\
&\\
&\left \{ [3,4,5,6],[6,1,2,3] \right \}_{3,1}+
\left \{ [3,4,5,6],[6,2,1,3] \right \}_{3,1}.\\
 \end{split}
\ee

Let us consider now a  configuration  $(x_1, x_2, x_3, x_4, x_5, x_6)$ of six distinct points on $\mathbb{P}^1,$  such that there exists a 
projective involution ${\rm I}$ exchanging $x_1\leftrightarrow x_2$, $x_3 \leftrightarrow x_6$, $x_4\leftrightarrow x_5$. 
Then  
\be \la{EQ1}
[1,2,3,4]=[{\rm I}(1), {\rm I}(2), {\rm I}(3), {\rm I}(4)]=[2,1,6,5].
\ee
Similarly, 
\be \la{EQ2} 
\begin{split}	
&[4,6,2,1]=[5,3,1,2], ~~~~
[4,5,2,1]=[5,4,1,2],\\
&[5,4,6,2]=[4,5,3,1],~~~~
[4,5,6,1]=[5,4,3,2].\\
\end{split}
\ee
Using this and  Lemma \ref{LEM1} we see that the sum of the first four lines in (\ref{D3}) is equal to 
\be \label{D4}
\begin{split}
&-2\left \{ [1,2,3,4],[5,3,1,2] \right \}_{3,1}
+2\left \{ [1,2,3,4],[5,4,1,2] \right \}_{3,1}\\
&-2\left \{ [1,2,3,4],[4,5,3,1] \right \}_{3,1}
-2\left \{ [1,2,3,4],[5,4,3,2] \right \}_{3,1}.\\
\end{split}
\ee
Indeed, the first two lines in (\ref{D3}) are identified with the  half of   (\ref{D4}) using (\ref{EQ2}). The last two lines in (\ref{D3}) are identified with the  half of  (\ref{D4}) using (\ref{EQ1}).     This and (\ref{D3}) implies
\be \la{D5}
(\ref{D4}) = \left \{ [3,4,5,6],[6,1,2,3] \right \}_{3,1}+
\left \{ [3,4,5,6],[6,2,1,3] \right \}_{3,1}.
\ee
By Lemma \ref{LEM1} the right-hand side of  (\ref{D5})  is invariant under the flip $4 \leftrightarrow  5$. On the other hand, using (\ref{D1}) we conclude that (\ref{D4})     is antisymmetric. 
This implies that  
\[
\left \{ [3,4,5,6],[6,1,2,3] \right \}_{3,1}+
\left \{ [3,4,5,6],[6,2,1,3] \right \}_{3,1}=0.
\]
So we proved i).
 Part ii) follows from i) using part d) of Lemma \ref{LEM1}.
\end{proof}

\bc \la{COR1}
The element ${\rm pr}(F(x_1,x_2,x_3,x_4,x_5,x_6))$ is symmetric in  $x_2,x_3,x_4$, and antisymmetric in $x_5,x_6.$
\ec

\begin{lemma} \la{LEM2}
The following relation holds:
$$
{\rm pr}(F(x_6,x_5,x_1,x_2,x_3,x_4))={\rm pr}(F(x_3,x_1,x_2,x_4,x_5,x_6)).
$$
\end{lemma}
\begin{proof}
Recall for the convenience of the reader the   relation  (\ref{200a}): 
\be \la{200}
\begin{split}
&-\left \{ [1, 2,3,4],[4, 6,2,1] \right \}_{3,1}
-\left \{ [3, 4,5,6],[6,1,2,3] \right \}_{3,1}
-\left \{ [5,6,2,1],[1, 3,4,5] \right \}_{3,1}\\
&+\left \{ [1, 2,3,4],[4,5,2,1] \right \}_{3,1}
+\left \{ [2,3,4,5],[5,6,1,2] \right \}_{3,1}
+\left \{ [4,5,6,2],[2,1,3,4] \right \}_{3,1}\\
&+\left \{ [5,6,2,1],[1,2,4,5] \right \}_{3,1}\\
&-\left \{ [1,2,3,4],[4,5,6,1] \right \}_{3,1}
-\left \{ [3,4,5,6],[6,2,1,3] \right \}_{3,1}
-\left \{ [5,6,2,1],[1,2,3,5] \right \}_{3,1}
=0.\\
 \end{split}
\ee

The sum  of  terms 2) and 9)  is zero by Lemma \ref{LEM1}.

The sum of the other 8 terms is accounted, using  Lemma \ref{LEM1}, as follows: 
\[
\begin{split}
&1)+4)+6)+8) ={\rm pr} (F(x_3, x_1, x_2,x_4, x_5, x_6)) - \{[1,2,3,4], [5, 6,2,1]\}_{3,1}; \\
& 3)+5)+7)+10) =  {\rm pr} (F(x_6, x_5, x_1, x_2, x_3, x_4)) +\{[1,2,3,4], [5,6,2,1]\}_{3,1}. \\
\end{split}
\]
Adding them up, we get the claim. 
\end{proof}

Now we finish the proof of (\ref{PROJECTION}). We have
\[
\begin{split}
&{\rm pr}(F(x_6,x_5,x_1,x_2,x_3,x_4)) =
{\rm pr}(F(x_3,x_1,x_2,x_4,x_5,x_6))=\\
&{\rm pr}(F(x_5,x_1,x_2,x_6,x_4,x_3))=
-{\rm pr}(F(x_5,x_6,x_1,x_2,x_3,x_4)).\\
\end{split}
\]
The first and second equalities follow from Lemma \ref{LEM2}. The last follows from Corollary \ref{COR1}.

Thus 
${\rm pr}(F(x_1,x_2,x_3,x_4,x_5,x_6))$ is antisymmetric in the first two arguments. Since by Corollary \ref{COR1} it is symmetric in the second and third arguments, it vanishes. 

 Therefore the map   ${\rm pr}$  induces an isomorphism  (\ref{MAPpr}). 
Theorem \ref{L4TH} is proved. 
\end{proof}

\begin{remark} The map (\ref{MAPpr})  is naturally lifted to a map 
\[
\begin{split}
&{\rm p} \colon \Lambda^2 \mathbb{Q}(\F) \longrightarrow  \mathbb{L}_4(\F),\\
&{\rm p}(\{x\}\wedge \{y\}):=\{x,y\}_{3,1}.\\
\end{split}
\] 
Theorem \ref{L4TH} tells that  
\be \la{F1a}
{\rm p}(F(x_1,x_2,x_3,x_4,x_5,x_6))  \in {\rm B}_4(\F).
\ee
Its proof  can be used to give
 an explicit  formula for the element (\ref{F1a}) in ${\rm B}_4(\F)$.   

\end{remark}

\section{Bigrassmannian    and   motivic complexes} \la{SEC2}

 We want to define a map from the Bigrassmannian complex to the weight $4$ part of the 
 cochain complex of the motivic Tate Lie coalgebra ${\cal L}_\bullet({\rm F})$ of a field ${\rm F}$. However, unless $\F$ is a number field, 
the latter is a conjectural object.  So we use a version of this complex, defined by setting 
$ {\cal L}_k({\rm F}) := {\B}_k({\rm F})$ for $k=1,2,3$, and introducing 
a substitute for ${\cal L}_4(\F)$, called the group $\G_4(\F)$.  

A key step in this direction was done in \cite{Gon00}, where a  formula for the 
Grassmannian 4-logarithm function was  
 found, and the nine-term and the dual nine-term functional equations for this function were proved. 
 The nine-term  equation just means that the  
Grassmannian 4-logarithm function gives  a measurable $7$-cocycle of ${\rm GL}_4(\C)$.
 Using the decorated flag construction and the dual nine-term   equation, we get 
a continuous $7$-cocycle of ${\rm GL}_N(\C)$,  $N\geq 4$. 

The Grassmannian 4-logarithm function from \cite{Gon00} 
coincides with the $n=4$ case 
of the Grassmannian $n$-logarithm function from \cite{Gon09}. 
So one can cast the   above result as  an explicit formula for the differential of the 
Grassmannian 4-logarithm function,  expressed 
via the classical trilogarithms, dilogarithms and logarithms. 

However, results of \cite{Gon00} are not sufficient to prove 
that this $7$-cocycle is non-trivial. 

What is even more important, they do not  give a cocycle for the motivic Chern class
\be \label{mschc4}
c_{4,{n} }^{\cal M} \in H^8(\textup{BGL}_{n} , \Z_{\cal M}(4)).
\ee
This is what we are going to do now.

\subsection{Decorated $p-$flag complex $~\lra~$   Bigrassmannian complex}

\paragraph{The Bigrassmannian  complex \cite{Gon93}.} Let $G$ be a group acting on a set $X$.
{\it Configurations} of $m$ elements in $X$ are the $G$-orbits  on $X^m$. 
Denote by $(x_1, \ldots, x_m)$ the configuration 
provided by an $m$-tuple  $\{x_1, \ldots, x_m\}$.   

A configuration of $m$ vectors in a $q$-dimensional vector space 
$V_q$   over a field $\F$ under the action of the group ${\rm GL}_q(\F)$ 
is {\it generic} if any $k\leq q$ of the vectors are linearly independent. Configuration spaces assigned to isomorphic vector spaces with a fixed number of vectors are {\it canonically} isomorphic. 
Denote by $C_m(q)$ the free abelian group generated by 
generic configurations of $m$ vectors in  $V_q$. 
There are two kinds of natural homomorphisms: 
\begin{enumerate}

\item Forgetting the $i$-th vector $l_i$:
$$
\partial_i: C_{m+1}(q) \lra C_{m}(q), ~~~~(l_0, \ldots, l_m)\longmapsto 
(l_0, \ldots, \widehat l_i, \ldots, l_m).
$$

\item Projecting the vectors $(l_0, \ldots, \widehat l_j, \ldots , l_m)$ to the quotient $V_q/(l_j)$:
$$
p_j: C_{m+1}(q) \lra C_{m}(q-1), ~~~~(l_0, \ldots, l_m)\longmapsto 
(l_j~|~ l_0, \ldots, \widehat l_j, \ldots  l_m). 
$$
\end{enumerate}
Using these maps, the groups $C_m(q)$ are organized into  the {\it Grassmannian bicomplex}:

\be  \label{dfltobigr}
\begin{gathered}
    \xymatrix{
       &   &  &     \ldots \ar[r]^{\partial}  \ar[d]^{p} &   C_5(4)   \ar[d]^{p} \\
      &         &   \ldots  \ar[d]^{p}\ar[r]^{ \partial}  \ar[d]^{p} & C_5(3) \ar[r]^{\partial}    \ar[d]^{p}   &C_4(3)\ar[d]^{p} \\
        &  \ldots \ar[r]^{\partial}  \ar[d]^{p}   & C_5(2)  \ar[r]^{\partial}\ar[d]^{p}& C_4(2) \ar[r]^{\partial}   \ar[d]^{p}    &C_3(2) \ar[d]^{p}   \\
    \ldots \ar[r]^{\partial}    &  C_5(1) \ar[r]^{\partial} &\ar[r]^{\partial} C_4(1)    & C_3(1)  \ar[r]^{ \partial} & C_2(1)       }
\end{gathered}
 \ee 
Here the maps $\partial$ and $p$ on a given group $C_m(n)$ are alternating sums of the maps $\partial_j$ and  $p_i$:
$$
\partial = \sum_{s=1}^m(-1)^{s-1}\partial_s, ~~~~ p  = \sum_{s=1}^m(-1)^{s-1}p_s.
$$
Each of the rows forms  a  
{\it weight $q$ Grassmannian complex}:
$$
C_{\bullet}^{(q)}:= \ \ \ \stackrel{\partial}{\lra} C_m(q) \stackrel{\partial}{\lra} C_{m-1}(q) \stackrel{\partial}{\lra}\ldots \stackrel{\partial}{\lra} C_{q+1}(q). 
$$
Let $BC_{m}:= \bigoplus_{q=1}^{m-1} C_{m}(q)$ be the sum of the groups on the $m$-th diagonal. 
Changing the signs of the differentials in the bicomplex, we get the {\it Bigrassmannian complex}. 

The bottom $m-1$ Grassmannian complexes in (\ref{dfltobigr}) 
form a sub-bicomplex of the Grassmannian bicomplex. The quotient  by this sub-bicomplex is called the weight 
$m$ Grassmannian bicomplex. The total complex associated with it is called {\it the weight $m$ Bigrassmannian complex}, and denoted by $BC_\bullet^{(m)}.$ 
The corner group $C_{m+1}(m)$ is in the degree $m+1$, and the differential has the degree $-1$. For example, the weight $4$ Bigrasmannian complex $BC_\bullet^{(4)}$ is the total complex associated with the following bicomplex:
\be  \label{dfltobigr1}
\begin{gathered}
    \xymatrix{
       &   &  &     \ldots \ar[r]^{\partial}  \ar[d]^{p} &   C_7(6)   \ar[d]^{p} \\
      &         &   \ldots  \ar[d]^{p}\ar[r]^{ \partial}  \ar[d]^{p} & C_7(6) \ar[r]^{\partial}    \ar[d]^{p}   &C_6(5)\ar[d]^{p} \\
        &  \ldots \ar[r]^{\partial}     & C_7(4)  \ar[r]^{\partial} & C_6(4) \ar[r]^{\partial}      &C_5(4)    }
\end{gathered}
 \ee 

\paragraph{Decorated $p-$flag complexes.} A {\it  $p-$flag} $F_\bullet$ in an $n$-dimensional vector 
space $V_n$ is  a nested collection of subspaces

\be \label{flag}    F_0 \subset F_1 \subset F_2 \subset\cdots 
\subset F_{p}\,, ~~~{\rm dim}F_i=i.
\ee   
A {\it decorated $p-$flag} $F_\bullet$ is a $p-$flag $F_\bullet$ plus a choice of a non-zero vector $f_i \in F_{i}/F_{i-1}$ for each $i=1, \ldots, p$. 
A collection of decorated $p-$flags $(F_{1, \bullet}, \ldots, F_{m, \bullet})$ 
is  {\it generic}, if 
for any integers $a_1, \ldots, a_m$ such that $a_1+...+a_m\leq n$   one has 
${\rm dim}(F_{1, a_1} + \ldots + F_{m, a_m})=a_1+...+a_m$, that is the sum  is a direct sum. 
Denote by $C_m({\cal A}^{(p)}_n)$ the free abelian group generated by 
 generic configurations  of $m$ decorated $p-$flags in $V_n$. 
There is  the complex $C_\bullet({\cal A}^{(p)}_n) $ of generic configurations of decorated $p-$flags in $V_n$ with the standard simplicial differential: 
\be \la{GDF}
\ldots \lra C_m({\cal A}^{(p)}_n) \lra \ldots \lra C_2({\cal A}^{(p)}_n)  \lra C_1({\cal A}^{(p)}_n).
\ee

The following crucial result was proved in  \cite{Gon93}.\footnote{See also  a version of this for complete flags in \cite[Section 4]{Gon15}.} 
\begin{theorem} \la{MTSG}
For any integers ${n} >0$ and $p \geq 0$, there is a canonical homomorphism of complexes $t_{\ast}: C_{n+*}({\cal A}^{(p+1)}_{n+p} ) \lra BC_{n+*}(n)$, that is 
\be  \label{dfltobigr1}
\begin{gathered}
    \xymatrix{
      \ldots \ar[r]^{} & C_{n+4}({\cal A}^{(p+1)}_{n+p} ) \ar[r]^{} \ar[d]^{t_4} & C_{n+3}({\cal A}^{(p+1)}_{n+p} ) \ar[r]^{} \ar[d]^{t_3} &C_{n+2}({\cal A}^{(p+1)}_{n+p} ) \ar[r]^{} \ar[d]^{t_2} &  C_{n+1}({\cal A}^{(p+1)}_{n+p} )   \ar[d]^{t_1} \\
       \ldots \ar[r]^{} & BC_{n+4}(n)  \ar[r]^{ }       & BC_{n+3}(n)  \ar[r]^{ }   & BC_{n+2}(n)  \ar[r]^{ }       &BC_{n+1}(n)}
\end{gathered}
 \ee 
\end{theorem}
 
The map $t_\ast$ was defined by  \cite[Formula 2.5]{Gon93}. Namely, take a generic  collection of decorated 
$(p+1)-$flags $(F_{1, \bullet}, \ldots, F_{m, \bullet})\in C_{m}({\cal A}^{(p+1)}_{n+p} )$. Pick any collection $\alpha$ of  integers $a_1, ..., a_m\geq 0$ such that $d:= a_1+...+a_m\leq p$. 
Consider the following quotient  $V_\alpha$ of the vector space $V_{n+p}$:
$$
V_\alpha := V_{n+p}/(F_{1, a_1}\oplus\ldots \oplus F_{m, a_m}), \ \ \ \ \ \ {\rm dim}V_\alpha = n+p-d.
$$
The decoration vector $f_{i, a_i+1}\in F_{i, a_i+1}/F_{i, a_i}$   projects to a vector $\overline f_{i, a_i+1}\in V_\alpha$.  These vectors form  a 
  generic configurations  of $m$ vectors  in $V_\alpha$, providing an element 
  \be \la{the}
  (\overline f_{1,  a_1+1}, \ldots, \overline f_{m, a_m+1})\in C_m(n+p-d).
\ee
Then $t_{m-p}(F_{1, \bullet}, \ldots, F_{m, \bullet})$ is defined as  the sum of  elements (\ref{the}) over all $\alpha$'s. By \cite[Lemma 2.1]{Gon93} we get a map of complexes.

\subsection{Bigrassmannian complex $~\lra~$ weight $\leq 3$ polylogarithmic complexes} \label{sec3.2.4n}

 \paragraph{The Bloch complex.} Below we use a \underline{different normalization} of the cross-ratio then  $[1,2,3,4]= \dfrac{|12||34|}{|14||32|}$ used above. This makes almost no difference for us, but 
 is consistent with the normalizations used in the literature before. 
 So the cross-ratio of four generic points on the projective line, defined via the corresponding configuration of four vectors in $V_2$:
\be \la{item24p1}
r(1,2,3,4):= \frac{|13||24|}{|14||23|}, ~~~~ r(1,2,3,4) = [1,3,2,4].
\ee

  
\paragraph{A map to the Bloch complex \cite{Gon95}.}  There  is a  map of complexes
\be \la{bctobloch2} 
\begin{gathered}
    \xymatrix{
        BC^{(2)}_4 \ar[r]^{} \ar[d] & BC^{(2)}_3 \ar[d]\\
         {\rm B}_2   \ar[r]^{\delta}       & \Lambda^2{\rm F}^\times}
\end{gathered}
 \ee
 
Its only non-zero component is the following
 map of complexes:
\be \la{rs22} 
\begin{gathered}
    \xymatrix{
        C_4(2) \ar[r]^{\partial} \ar[d]_{r_4(2)} & C_3(2) \ar[d]^{r_3(2)} \\
         {\rm B}_2   \ar[r]^{\delta}       & \Lambda^2{\rm F}^\times}
\end{gathered}
 \ee
 
 The map $r_3(2)$ is given by:
\be \la{123not}
r_3(2): (1,2,3) \lms [1,2,3]:= |12| \wedge |23| + |23| \wedge |31| +   |31| \wedge |12|. 
\ee
The map $r_4(2)$ is defined   using the  cross-ratio of four vectors (\ref{item24p1}) in a two dimensional  space: 
\be \la{3.7.16.3}
r_4(2)(1,2,3,4):= \{r(1,2,3,4)\}_2.
\ee
 Diagram  (\ref{rs22}) is commutativity  thanks to the  Plucker relation: $|12||34| - |13||24| + |14||23| =0$. 
 This implies that $\delta_2({\rm R}_2({\rm F}))=0$. 
To get a map of complexes, we change the sign of  $r_4(2)$. 

The   existence of   map of complexes (\ref{rs22}) gives rise to the definition of   cluster dilogarithm. 

\paragraph{A map to the weight $3$ polylogarithmic motivic complex \cite{Gon94}, \cite{Gon95}, \cite{Gon95a}.}

There is a  map of complexes
\be \la{bctobloch} 
\begin{gathered}
    \xymatrix{
        BC^{(3)}_6 \ar[r]  \ar[d]  &BC^{(3)}_5 \ar[d]  \ar[r]  \ar[d] & BC^{(3)}_4  \ar[d] \\
         {\B}_3   \ar[r]^{\delta~~~}       &{\B}_2 \otimes {\rm F}^\times\ar[r]^{~\delta}&   \Lambda^3{\rm F}^\times}
\end{gathered}
 \ee
 Its  only non-zero component is the following
 map of complexes:
\be \la{rs2} 
\begin{gathered}
    \xymatrix{
        C_6(3) \ar[r]^{\partial} \ar[d]^{r_6(3)} &C_5(3) \ar[d]^{r_5(3)} \ar[r]^{\partial} \ar[d]_{} & C_4(3) \ar[d]^{r_4(3)}\\
         {\B}_3   \ar[r]^{\delta~~}       &{\B}_2 \otimes {\rm F}^\times\ar[r]^{\delta}&   \Lambda^3{\rm F}^\times}
\end{gathered}
 \ee
This was done 
in  \cite[Section 3.2]{Gon94},  \cite[Section 5]{Gon95}. 
Namely,   set 
\be \la{3.7.16.2}
\begin{split}
&r_5(3)(1,2,3,4,5):=   \alt_5\Bigl(\{r(1|2,3,4,5)\}_2\otimes |345|\Bigr) \in {\B}_2\otimes {\rm F}^\times.\\
&r_4(3)(1,2,3,4):=  2\cdot  \alt_4\left (|123| \wedge |124| \wedge |134|   \right ) \in \Lambda^3{\rm F}^\times.\\
\end{split}
\ee
Then the right square of the diagram is commutative.
Next,  there exists a map $r_6(3): C_6(3) \lra {\B}_3$ making  the left square of   digram (\ref{rs2}), that is the following  diagram,  
commutative:
\be \la{COMMDIAx} 
\begin{gathered}
    \xymatrix{
       C_6(3) \ar[r]^{\partial} \ar[d]_{r_6(3)} & C_5(3) \ar[d]^{r_5(3)} \\
          {\B}_3   \ar[r]^{\delta~~}       &{\rm B}_2 \otimes {\rm F}^\times}
 \end{gathered}
 \ee
We will not use  an explicit formula for the map $r_6(3)$, although it 
looks nice:\footnote{The coefficient $1/15$ in \cite[Formula (16)]{Gon95a} is compatible with $1/5$ in (\ref{1/5}) since our map $r_5(3)$ is $3$ times the map used there.}
\be \la{1/5}
r_6(3)(1,2,3,4,5,6):= \frac{1}{5}\cdot {\rm Alt}_6\left\{\frac{|124||235||136|}{|125||236||134|}\right\}_3.
\ee
Let us set 
\be \la{FTRR}
\begin{split}
(1,2,3,4,5,6)_3 := r_6(3)(1,2,3,4,5,6) \in {\rm B}_3.
\end{split}
\ee
 The claim that we get a homomorphism of complexes, that is that the composition $$
 C_7(3) \stackrel{\partial}{\lra} C_6(3) \stackrel{r_6(3)}{\lra} \B_3
 $$
  is equal to zero, is just equivalent to Theorem \ref{R37term}. 
Let us now prove  Theorem \ref{R37term}.
 
\subsection{Intermezzo: a proof of Theorem \ref{R37term}.}  \la{intermezzo}

\begin{figure}[ht]
\centerline{\epsfbox{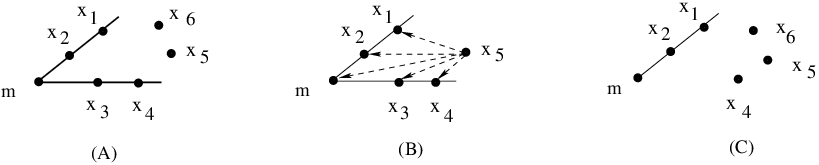}}
\caption{Reduction of the calculation of ${\rm M}_3(x_1, ..., x_6)$ to  degenerate configurations.}
\label{zac76}
\end{figure}

 The result follows from  Theorem A on page 293 in \cite{Gon95} by  the skew-symmetrization of the map ${\rm M}_3$ defined on page 286. 

Recall the group ${\cal G}_3(\F)$ from Definition 1.5 on p. 211  in loc. cit., generated by arbitrary configurations ${\bf x}=(x_1, ..., x_6)$ of 6 points in ${\Bbb P}^2(\F)$, subject to the following relations:
\vskip 1mm

\begin{enumerate} 

\item  One has ${\bf x}=0$ if two points coincide, or four points lie on the same line. 
 
\item  The 7-term relation: for any configuration of 7 points $(x_1, ..., x_7)$ in ${\Bbb P}^2(\F)$ we have 
 \be \la{7terms}
 \sum_{i=1}^7(-1)^i(x_1, \ldots , \widehat x_i, \ldots,  x_7) =0.
  \ee
  
\item {\it Relation R3}. Consider a  configuration $(a_1,a_2,a_3,b_1, b_2,  b_3)$ where $b_i$ is on the line $a_{i}a_{i+1}$, $i=1,2,3$, see  Figure \ref{zag78}. It is determined uniquely by the invariant 
$$
z:=r'(b_1|a_2, a_3, b_2, b_3), \ \ \mbox{where} \ \ r'(x_1, x_2, x_3, x_4):= \frac{(x_1-x_3)(x_2-x_4)}{(x_1-x_4)(x_2-x_3)}.
$$
Note  that the two versions  $r'$ and $r$ of the cross-ratio are related by
$$
r'(x_1, x_2, x_3, x_4)  \stackrel{(\ref{CR1})}{= } -r(x_1, x_3, x_2, x_4).
$$
Note also that
$$
1-r'(x_1, x_2, x_3, x_4) = r'(x_1, x_3, x_2, x_4)
$$
 Using the notation $(a_1, a_2, a_3, b_1, b_2, b_3)_z$ for the configuration with the invariant $z$, we define 
\be
{\rm T}'(z):=-  {\rm T}(z)-2   {\rm T}(1-z)+  {\rm T}(1); \qquad  {\rm T}(z):= (a_1, a_2, a_3, b_1, b_2, b_3)_z.
\ee
Let $(y_1, ..., y_6) = (x_1, x_2, m, x_3, x_4, x_5)$  be a configuration   of type (B) on  Figure \ref{zac76}. Then 
\be \la{R3}
({\rm  R3}) \ \ \ \ \ \ \ \ \ \ \ \ \ \ 3(y_1, ..., y_6) = \sum_{i=1}^5 (-1)^{i-1} {\rm T}'(r'(y_6|y_1, ..., \widehat y_i, ..., y_5)).
\ee
\end{enumerate}

  \begin{figure}[ht]
\centerline{\epsfbox{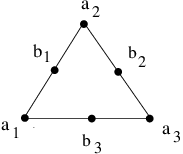}}
\caption{}
\label{zag78}
\end{figure} 

By Lemma 1.7 in loc. cit. the generators ${\bf x}$ are skew-symmetric in $x_i$'s. 

By Theorem A in loc. cit.  there is a canonical isomorphism  defined explicitly on  p. 286-287:
$$
{\rm M}_3: {\cal G}_3(\F)\lra \B_3(\F).
$$
The formula for the map ${\rm M}_3$    is canonical for degenerate configurations, but includes a choice for  generic ones. 
 Precisely,  the 7-term relation (\ref{7terms}) for a configuration (A) of 7 points $(m,x_1, ..., x_6)$ in ${\Bbb P}^2$ 
on  Figure \ref{zac76} expresses  ${\rm M}_3(x_1, ..., x_6)$ (which we  get forgetting the point $m$) as the map ${\rm M}_3$ applied to a sum of configurations of 6 points of type (B) (forget $x_5$ or $x_6$) and (C).  
Let us show  that  the skew-symmetrization in $(x_1, ..., x_6)$ of  the map ${\rm M}_3$ applied to the obtained sum of six degenerate configurations of 6 points delivers formula (\ref{7term}). 

 \bp   The skew-symmetrization in $(x_1, ..., x_6)$ of  the map ${\rm M}_3$ is given by, see (\ref{TR}),
\be
{\rm Alt}_{\{x_1, ..., x_6\}}{\rm M}_3(x_1, ..., x_6) = \frac{3}{2}\cdot r_3(x_1, x_2, x_3, x_4, x_5, x_6). \ee
\ep

 \begin{proof} The map ${\rm M}_3$ on  a configuration $(y_1, ..., y_6)$ of type B  is given by  recasting relation (\ref{R3}): 
\be 
\begin{split}
&3{\rm M}_3(y_1, ..., y_6) := \sum_{i=1}^5 (-1)^{i-1} {\cal T}\{r'(y_6|y_1, ..., \widehat y_i, ..., y_5)\}_3\\
& {\cal T}\{z\}_3:= -\{z\}_3-2\{1-z\}_3+ \{1\}_3; \\
\end{split}
\ee 

 \bl 
The skewsymmetrization of  the element ${\rm M}_3(x_5|x_1,x_2,m, x_3, x_4) \in \B_3(\F)$ in the variables $\{x_1, ..., x_6\}$ is equal to zero. 
  \el 
\begin{proof}  
Forgetting one of the points $x_1, x_2, x_3, x_4$ we get expressions of the following shape:
 \be \la{quq}
 \{r'(x_5|x_1, m, x_3, x_4)\}_3 + 2  \{r'(x_5|x_1, x_3, m, x_4)\}_3 + \{1\}_3.
 \ee
If we forget $m$, we get
\be \la{xxx}
-\{r'(x_5|x_1,x_2, x_3, x_4)\}_3 -2 \{r'(x_5|x_1,x_3, x_2, x_4)\}_3 +\{1\}_3.
\ee

The first terms in (\ref{quq}) and (\ref{xxx}) are equal to zero already after the skew-symmetrization in 
$\{x_3, x_4\}$ since $r'(x_1, x_2, x_3, x_4) = r'(x_1, x_2, x_4, x_3)^{-1}$ and $\{z\}_3 = \{z^{-1}\}_3$ in $\B_3(\F)$. 

The second term in (\ref{xxx})  vanishes after the skew-symmetrization in $\ell_2, \ell_4$ since $\{z\}_3 = \{z^{-1}\}_3$  and \be \la{quq0}
\begin{split} 
&\{r'(x_5|x_1, x_3, x_2, x_4) \}_3=
  \left\{-\frac{\omega(\ell_5, \ell_1, \ell_2)\omega(\ell_5, \ell_3, \ell_4)}{\omega(\ell_5, \ell_2, \ell_3)\omega(\ell_5, \ell_1, \ell_4)} \right\}_3.\\
 \end{split}
  \ee

For the second term in (\ref{quq}) we   use the following observation: the point $m = x_1x_2 \cap x_3x_4$ can be represented by the following vector $\widehat m$, where $\ell_i$ is a vector representing the point $x_i\in {\Bbb P}_2(\F)$:
\be \la{pop}
\widehat m := \omega(\ell_1, \ell_3, \ell_4) \ell_2- \omega(\ell_2, \ell_3, \ell_4) \ell_1 =  -\omega(\ell_1, \ell_2, \ell_3) \ell_4+ \omega(\ell_1, \ell_2, \ell_4) \ell_3.
 \ee
The second term in (\ref{quq})    is  invariant under the transposition $\ell_3 \leftrightarrow \ell_5$ since $\{z\}_3 = \{z^{-1}\}_3$ and \be
\begin{split} 
&\{r'(x_5|x_1, x_3, m, x_4) \}_3=
  \left\{\frac{\omega(\ell_5, \ell_1, \widehat m)\omega(\ell_5, \ell_3, \ell_4)}{\omega(\ell_5, \ell_3, \widehat m)\omega(\ell_5, \ell_1, \ell_4)} \right\}_3 \\
& =\left\{-\frac{\omega(\ell_5, \ell_1,\ell_2)\omega(\ell_1, \ell_3, \ell_4)\omega(\ell_5, \ell_3\, \ell_4)}{\omega(\ell_5, \ell_3,\ell_4)\omega(\ell_1, \ell_2, \ell_3)\omega(\ell_5, \ell_1, \ell_4)} \right\}_3
=  
 \left\{-\frac{\omega(\ell_5,\ell_1,\ell_2)\omega(\ell_1,\ell_3,\ell_4) }{ \omega(\ell_1,\ell_2,\ell_3)\omega(\ell_5,\ell_1,\ell_4)}\right\}_3.\\
 \end{split}
  \ee
 So it vanishes after the skew-symmetrization. The term $\{1\}_3$ obviously vanishes as well. 
  \end{proof}

According to p. 286 in loc. cit.,     to calculate  ${\rm M}_3(x_1, x_2, m, x_4, x_5, x_6)$ for a type (C) configuration on Figure \ref{zac76},   we  project 
from each of the  points $x_4, x_5, x_6$, getting a configuration of 5 points on the line, like $(x_4|x_5, x_6, x_1, x_2, m)$, and forget one of the three points  $(x_4|x_1, x_2, m)$.  Forgetting the point $m$ we get, after the skew-symmetrization, terms similar to (\ref{quq0}). So they vanish. 

The remaining expressions  after the skew-symmetrization in $\{x_1, ..., , x_6\}$ are the following
 \be \la{quq1}
\begin{split}
& \{r'(x_4|x_5, x_6, x_1, m)\}_3 + 2  \{r'(x_4|x_5, x_1, x_6, m)\}_3 = \\
 & \{r'(x_4|x_5, x_6, x_1, x_3)\}_3 + 2  \{r'(x_4|x_5, x_1, x_6, x_3)\}_3\\
 \end{split}
  \ee
and
  \be \la{quq2}
 \{r'(x_5| x_4, x_6, x_1, m)\}_3 + 2  \{r'(x_5|x_4, x_1, x_6, m)\}_3.
 \ee
 The second line in  (\ref{quq1}) is similar after the skew-symmetrization to  (\ref{xxx}), and so vanishes.


 
 \bl \la{L3.89} The skewsymmetrization of (\ref{quq2}) in the group $\B_3(\F)$ is given by
 \be \la{st}
 {\rm Alt}_{\{x_1, ..., x_6\}} \Bigl(  \{r'(x_5| x_4, x_6, x_1, m)\}_3 + 2  \{r'(x_5|x_4, x_1, x_6, m)\}_3\Bigr) = \frac{3}{2}\cdot r_3(x_1, x_2, x_3, x_4, x_5, x_6).
 \ee
 \el 
  
  \begin{proof} To calculate the first term in (\ref{st}) we start from $(x_1, x_2, x_3, x_4, x_5, x_6)$, set ${\rm m}:=x_1x_2\cap x_3x_4$, 
  forget $x_3$ and project from $x_5$, getting a configuration of four points $(x_5| x_4, x_6, x_1, m)$, shown on the  left  on Figure \ref{zag80}. Let us 
   apply  an even permutation $x_1, x_2, x_3, x_4, x_5, x_6 \longleftrightarrow x_1, x_2, x_6, x_5, x_4, x_3$. Then we
   start from $(x_1, x_2,  x_6, x_5, x_4, x_3)$, set ${\rm m}^*:=x_1x_2\cap x_6x_5$, 
  forget $x_6$ and project from $x_4$, getting a configuration  $(x_4| x_5, x_3, x_1, {\rm m}^*)$ shown in the middle  of Figure \ref{zag80}. 
  Comparing the three diagrams  on Figure \ref{zag80}  we get 
  $$
  r'(x_5| x_4, x_6, x_1, {\rm m}) =  r'({\rm q}, {\rm m}^*, x_1, {\rm m}); \ \ \ \ \ \ r'(x_4| x_5, x_3, x_1, {\rm m}^*)  =r'({\rm q}, {\rm m}, x_1, {\rm m}^*).
  $$
  Note that 
  $$
  r'({\rm q}, {\rm m}^*, x_1, {\rm m}) = (1- r'({\rm q}, {\rm m}, x_1, {\rm m}^*)^{-1})^{-1}.
  $$
   So, thanks to the relations (\ref{EATRI})  in the group $\B_3(\F)$, putting all together, 
   the first term in (\ref{st}) is 
  $$
  {\rm Alt}_{\{x_1, ..., x_6\}}  \{r'(x_5| x_4, x_6, x_1, m)\}_3  =    -\frac{1}{2} {\rm Alt}_{\{x_1, ..., x_6\}}  \{1- r'({\rm q}, {\rm m}, x_1, {\rm m}^*)\}_3.
  $$
  Let us look at the second term in (\ref{st}). Observe that 
  $$
  r'(x_5|x_4, x_1, x_6, m)  =  r'({\rm q}, {x_1, \rm m}^*, {\rm m}) 
   = (1-r'({\rm q}, {\rm m}, x_1, {\rm m}^*))^{-1}.
  $$ 
  So, using $\{z\}_3=\{z^{-1}\}_3$,   and formula  (\ref{pop}) for  the vector $\widehat m$, we see that (\ref{st}) is equal to 
    $$
     \frac{3}{2} {\rm Alt}_{\{x_1, ..., x_6\}}  \{ r'(x_5|x_4, x_1, x_6, m))\}_3 
   = \frac{3}{2} {\rm Alt}_{\{\ell_1, ..., \ell_6\}}  \left\{ \frac{\omega(\ell_5, \ell_4, \ell_6)\omega(\ell_5, \ell_1, \ell_2)\omega(\ell_1, \ell_3, \ell_4)}{\omega(\ell_5, \ell_1, \ell_6)\omega(\ell_5, \ell_4, \ell_3)\omega(\ell_1, \ell_2, \ell_4)}\right\}_3.
  $$
Since $(5,4,1,2,6,3)$ is an even permutation, the right hand side is $ \frac{3}{2} r_3(x_1, x_2, x_3, x_4, x_5, x_6)$.  \end{proof} 
    \begin{figure}[ht]
\centerline{\epsfbox{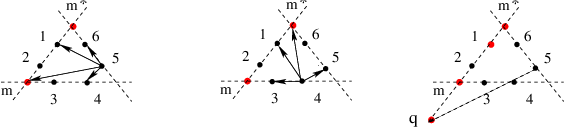}}
\caption{Projecting from the points $x_5$ or   $x_4$, we get the same configuration of four  points $({\rm q, m, x_1, m}^*)$ on a line on the right. }
\label{zag80}
\end{figure}

Theorem \ref{R37term} is proved.  \end{proof}

  \subsection{Dual Bigrassmannian complex $\lra$ a weight 4 motivic complex}

\paragraph{The Grassmannian duality.}

Let $V_m$ be an m-dimensional vector space with a given basis $\{e_1, \ldots, e_m\}$. Denote by 
${\rm Gr}^*_k(V_m)$ the Grassmannian of  $k$-dimensional subspaces in $V_m$ which are transversal to the coordinate hyperplanes. 
Denote by ${\rm Conf}_m(q)$ the set of configurations of $m$ generic vectors in a $q$-dimensional vector space. 
Then there is a canonical isomorphism
$$
{\rm Gr}^*_{m-q}(V_m)\lra {\rm Conf}_m(q).
$$ 
It assigns to an $(m-q)$-dimensional subspace $h \subset V_m$ the configuration of $m$ vectors in the $q$-dimensional 
quotient $V_m/h$ obtained by projecting the basis vectors $\{e_i\}$ to the quotient.

The dual space $V_m^*$ has a basis $\{f_i\}$ dual to the basis $\{e_i\}$. Assigning to a $(m-q)$-dimensional subspace $h$ its annihilator $h^\perp\subset V_m^*$, we get 
a canonical isomorphism 
$$
{\rm Gr}^*_{m-q}(V_m)\lra  {\rm Gr}^*_{q}(V^*_m).
$$
Combining these two isomorphisms, we arrive at   the duality isomorphism
$$
\ast: {\rm Conf}_{m}(q) \stackrel{\sim}{\lra}{\rm Conf}_{m}(m-q). 
$$
The duality interchanges the projection $p$ and the differential $\partial$ in the Grassmannian complex:
$$
p \circ \ast = \ast \circ \partial, ~~~~\partial\circ \ast = \ast \circ p.
$$

\paragraph{From the   Bigrassmannian complex to polylogarithmic complex.}  
Applying the Grassmannian duality map $\ast$ to the Bigrassmannian complex we get the {\it dual Bigrassmannian complex}. The weight four  dual Bigrassmannian complex looks as follows: 
 \be  \label{MCDI1zzkis}
\begin{gathered}
    \xymatrix{
&    & C_9(3)  \ar[r]^{p}  \ar[d]^{\partial} &    C_8(2)   \ar[r]^{p}  \ar[d]^{\partial} &   C_7(1)     \ar[d]^{\partial} \\
      &        C_9(4)  \ar[r]^{p}  \ar[d]^{\partial} & C_8(3)       \ar[r]^{p}  \ar[d]^{\partial} & C_7(2) \ar[r]^{}    \ar[d]^{\partial}   &C_6(1)\ar[d]^{\partial} \\
  C_9(5) \ar[r]^{p}      & C_8(4) \ar[r]^{p}      & C_7(3)  \ar[r]^{p}  & C_6(2) \ar[r]^{p}       &C_5(1)    \\      }
\end{gathered}
 \ee 
Let us compose the map of complexes (\ref{dfltobigr}) with the Grassmannian duality $\ast$. Then our  goal is to define a map from the dual 
Bigrassmannian complex (\ref{MCDI1zzkis}) to the weight four motivic complex. Let us now introduce the variant of the weight four motivic complex we use. 

Below we use the notation ${\rm B}_2$ for ${\rm B}_2(\F)$, etc. Consider the following subcomplex in the weight $4$ part of the standard cochain 
complex of the motivic Lie algebra of ${\rm F}$:
\[
  {\B}_{3}  \otimes {\rm F}^{\times} ~\oplus ~ \Lambda^2{\rm B}_{2}  
~~ \stackrel{{\delta}}{\longrightarrow}~ ~
{\B}_{2}   
\otimes \Lambda^{2} {\rm F}^{\times} ~~  \stackrel{{\delta}}{\longrightarrow}~ ~ \Lambda^{4} {\rm F}^{\times}.
\]  
The differential \(\delta\) acts as follows
\be \la{DIFFDEL}
\begin{split}
 & \{x\}_{3} \otimes y\in {\B}_{3}   \otimes {\rm F}^{\times}  
  ~ \stackrel{{\delta}}{\longrightarrow}~ \{ x\}_{2} \otimes x \wedge y \in
  {\B}_{2} \otimes \Lambda^{2} {\rm F}^{\times},\\
  & \lbrace x\rbrace_{2} \wedge  \{ y\}_{2} \in \Lambda^{2} {\B}_{2}  
   ~ \stackrel{{\delta}}{\longrightarrow}~  \lbrace x\rbrace_{2} \otimes (1-y)
  \wedge y - \{ y\}_{2} \otimes (1-x) \wedge x \in {\B}_{2} 
  \otimes \Lambda^{2} {\rm F}^{\times},\\
&   \{x\}_{2} \otimes y
  \wedge z \in {\B}_{2}  \otimes \Lambda^{2} {\rm F}^{\times} ~ 
\stackrel{{\delta}}{\longrightarrow}~   (1-x) \wedge x \wedge y \wedge z \in
  \Lambda^{4} {\rm F}^{\times}.
\end{split}
\ee

To define a substitute for ${\cal L}_4({\rm F})$, 
we consider first the $\Q$-vector space $\widetilde \G_4(\F)$ generated by   generic configurations of $8$ points in ${\Bbb P}^3$ over ${\rm F}$. 
There is a canonical map 
\be \la{R84}
r_8(4): C_8(4) \lra \widetilde \G_4(\F), ~~~~(v_1, \ldots, v_8) \lms (\overline v_1, \ldots, \overline v_8). 
\ee
It assigns to a configuration of $8$ vectors in a four dimensional vector space 
the configuration of $8$ points in its projectivization, given by the one dimensional subspaces generated by the vectors. 

\bd The $\Q$-vector space $\G_4(\F)$ is the following quotient of the space $\widetilde \G_4(\F)$:
\be \la{G4F}
\G_4(\F):= \frac{\widetilde \G_4(\F)}{r_8(4)\circ \partial(C_9(4)) + r_8(4)\circ p(C_9(5))}.
\ee
\ed

 \paragraph{Preliminary settings.} Pick  a volume form $\omega_m$ in an $m$-dimensional vector space $V_m$, and 
set 
\be \la{br}
|v_{1}, \dotsc, v_{m}|:= \langle \omega_m, v_{1} \wedge \dotsc \wedge v_{m}\rangle, ~~~~v_i\in V_m.
\ee
Given any function $f(v_1, \ldots, v_n)$, we define an operation of alternation by setting
  \begin{equation}
    \label{eq:alt-def}
    {\alt_n} f(v_1, \dotsc, v_n) = \sum_{\tau \in \mathcal{S}_{n}}
    (-1)^{|\tau|} f(v_{\tau(1)}, \dotsc, v_{\tau(n)}).
  \end{equation}
  We use the notation  ${\alt_n} f(\{1, \ldots, k\}, \{k+1, \ldots, n\})$ for averaging over the subgroup $S_k \times S_{n-k}$.

Below we use shorthands like $|1234|$ for $|v_1, v_2, v_3, v_4|$, 
skipping the vectors and keeping the indices only. 
Given a configuration of five generic vectors $(v_1, \ldots, v_5)$ in a three dimensional vector space $V_3$, we use the notation $(1|
    2345)$ for the configuration of four vectors in the two dimensional space $V_3/\langle v_1\rangle$ obtained by 
projecting of the vectors $v_2, v_3, v_4, v_5$ along the subspace generated by the vector $v_1$, etc. We write 
$(1,2,3,4,5)$ for the configuration of  vectors $(v_1, v_2, v_3, v_4, v_5)$.


\begin{theorem} \label{MapBtoP}
There exists a  map $\delta$ and a map from the dual Bigrassmannian complex to the weight $4$ motivic  complex, as depicted below:
\be  \label{MCDI1zz}
\begin{gathered}
    \xymatrix{
&    & C_9(3)  \ar[r]^{p}  \ar[d]^{\partial} &    C_8(2)   \ar[r]^{p}  \ar[d]^{\partial} &   C_7(1)     \ar[d]^{\partial} \\
      &        C_9(4)  \ar[r]^{p}  \ar[d]^{\partial} & C_8(3)       \ar[r]^{p}  \ar[d]^{\partial} & C_7(2) \ar[r]^{}    \ar[d]^{\partial}   \ar@{.>}[ldd] \ar[d]^{\partial}  &C_6(1)\ar[d]^{\partial} \\
  C_9(5) \ar[r]^{p}      & C_8(4) \ar[r]^{p}  \ar[d]^{{\bf r^*_8(4)}}     & C_7(3)  \ar[r]^{p}\ar[d]^{{\bf r^*_7(3)}}  & C_6(2) \ar[r]^{p}   \ar[d]^{{\bf r^*_6(2)}}     &C_5(1)   \ar[d]^{{\bf r^*_5(1)}}   \\ 
   0   \ar[r]      &  \G_4(\F) \ar[r]^{\delta~~~~~~} &\ar[r]^{~~\delta}  {\rm B}_3 \otimes {\rm F}^\times \oplus \Lambda^{2} {\rm B}_{2}     &  {\rm B}_2 \otimes \Lambda^2{\rm F}^\times \ar[r]^{~~\delta} & \Lambda^4{\rm F}^\times     }
\end{gathered}
 \ee 
 This map is given by the following  formulas:
\be \la{56}
\begin{split}
&{\bf r_5^*(1) } \colon (1,2,3,4,5) \lms 
-\frac{5}{2}\alt_5  \Bigl(|1|\wedge |2| \wedge |3|\wedge |4|\Bigr),\\
&{\bf r_6^*(2)} \colon (1,2,3,4,5,6) \lms \\
&\alt_6\Biggl(\Bigl(  \frac{1}{2}\cdot\{r(2,4,5, 6)\}_2
-  \{r(1,2,4,5)\}_2 \Bigr)\otimes |12| \wedge |23| +
 \{r(1,3,5,6)\}_2 \otimes |12|\wedge |34|\Biggr),\\
&{\bf r_7^*(3)} \colon (1,2,3,4,5,6,7) \lms \\
& \alt_{7}\Biggl(\Bigl(\frac{1}{108}\cdot(1,2,3,4,5,6)_3  - 
  \{r(6|2, 4,3,5)\}_3\Bigr) \otimes   |123| -\frac{3}{28}\cdot \{r(1|2,5,6,7)\}_2\wedge  \{r(2|1,3,4,5)\}_2\Biggr),\\
&{\bf r_7^*(2)} \colon (1,2,3,4,5,6,7) \lms -\frac{1}{84} \cdot {\rm Alt}_7\Bigl(\{r(1,2,3,4)\}_2\wedge \{r(1,5,6,7\}_2\Bigr),\\
&{\bf r_8^*(4)}:= \textup{pr}\circ r_8(4) \text{ \ where \ } \textup{pr}\colon \widetilde{\G}_4(\F)\lra \G_4(\F)  \text{ \ is the projection map. \ }
\end{split}
\ee
The remaining maps are zero.
\end{theorem}

Recall  that the triple ratio is normalized by (\ref{COMMDIAx}) and (\ref{FTRR}) so that one has 
\be \la{TRIPR}
\delta (1,2,3,4,5,6)_3 =\alt_6 (\{1|2,3,4,5\}_2\otimes |345|).
\ee

 Theorem \ref{MapBtoP} is proved in Section  \ref{Sec4}. Another way to get the map ${\bf r_7^*(3)}$ see at Section \ref{Sec8.3}.
The commutativity of the bottom right square in (\ref{MCDI1zz}) is a special case of \cite[Proposition 3.3]{Gon93}.

\paragraph{The  homomorphism $\delta$.}   
Let us consider  the following composition 
\be \la{COMPO}
  {\bf r}^*_7(3)\circ p:  ~ C_8(4) \lra C_7(3) \lra {\B}_3 \otimes {\rm F}^\times ~\oplus~ \Lambda^2{\B}_2.
\ee

\bl \la{LemmalS}
The composition (\ref{COMPO})  
 does not change if we rescale any of the vectors.  
\el

\begin{proof} This is evident for the $\Lambda^2{\B}_2$ component. Let us look at the 
 ${\B}_3 \otimes {\rm F}^\times$-component. 
Let us rescale $l_1 \lms al_1$. Then, for some $c_1, c_2 \in \Q$, we have 
$$
{\bf r_7^*(3)}_{3,1} \circ p: (1,2,3,4,5,6,7,8)\lms 
 \alt_8\Bigl(c_1 \cdot(8|1,2,3,4,5,6)_3  + c_2 \cdot \{r(8,6|2, 4,3,5)\}_3 \Bigr)\otimes   a  = 0.
 $$  
Indeed, the second term vanishes  under the alternation $\alt_8$ sign. The first vanishes thanks to the 7-term relation for $(8|1,2,3,4,5,6)_3$. 
 \end{proof} 
Therefore the map (\ref{COMPO}) descends to a well defined map 
\[
\widetilde \delta:   \widetilde \G_4\lra {\B}_3 \otimes {\rm F}^\times ~\oplus~ \Lambda^2{\B}_2.
\]
  It makes, by its very definition,  the bottom  left square in (\ref{MCDI1zz})  commutative.

\bp The map $ \widetilde \delta$ descends to a map  
\[
\delta:   \G_4\lra {\B}_3 \otimes {\rm F}^\times ~\oplus~ \Lambda^2{\B}_2.
\]
\ep

\begin{proof} The group $\G_4(\F)$  in (\ref{G4F}) is   the quotient of the group $\widetilde \G_4(\F)$ by  the images of the maps  ${r_8(4)}\circ \partial$ and   ${r_8(4)} \circ p$. 
We have to prove that  $ \widetilde \delta$  kills these   relations.

 \bl  \la{nbv} One has $ {\bf r_7^*(3)}\circ \partial =0$ on $C_8(3)$. 
 \el

 \begin{proof}
 The ${\B}_{3} \otimes {\rm F}^*$-part  
  of   ${\bf r_7^*(3)}(1,2,3,4,5,6,7)$  does not depend on $7$. So    by the antisymmetry in $(7,8)$, we have   ${\bf r_7^*(3)}\circ \partial (1,2,3,4,5,6,7,8)=0$. 
Let us prove that the ${\rm B}_2 \wedge {\rm B}_2$-part of the composition is  zero. 
Using the $5$-term relation for the configuration $(1|2,5,6,7,8)$ we write it as  
\be
\begin{split}
&\frac{3}{28}\cdot\alt_{8}\Bigl(\Bigl(\{r(1|2,6,7,8)\}_2 -\{r(1|5,6,7,8)\}_2 \Bigr)\wedge  \{r(2|1,3,4,5)\}_2\Bigr).\\
\end{split}
\ee
The first term vanishes since the involution $(12)(36)(47)(58)$ is even and interchanges the two factors in ${\rm B}_2 \wedge {\rm B}_2$. 
By the $5$-term relation for the configuration $(1| 4,5,6,7,8)$ the second term is  
$$
\frac{9}{56}\cdot\alt_{8} \Bigl( \{r(1|4,5,6,7)\}_2 \wedge  \{r(2|1,3,4,5)\}_2\Bigr).
$$
Since it is symmetric in $(4,5)$, it vanishes after the antisymmetrization. 
\end{proof}

 Since $p \circ p=0$, we have   $ \widetilde \delta \circ p=0$.    Lemma \ref{nbv} implies that   
 $
 \widetilde\delta\circ   \partial = {\bf r_7^*(3)}\circ \partial \circ p \stackrel{ }{=} 0.
 $
This just means that the map $\widetilde \delta$ kills the relations in (\ref{G4F}). \end{proof}

 \section{Weight four motivic complex  $\Gamma(X;4)$ and the regulator map} \la{Sec7}

In Section \ref{Sec7} we define, for any regular variety $X$ over an arbitrary field  $k$, a complex of $\Q-$vector spaces $\Gamma(X;4)$. This complex  is conjecturally  quasiisomorphic 
to the weight $4$ motivic complex of $X$, tensored by $\Q$.

 In the case when $X$ is a regular complex variety, we recall in Section \ref{SECT8.1}   the canonical real period map of complexes ${\rm P}$ defined in \cite{Gon08}, see Theorem \ref{PM1}.  Using 
 the map ${\rm P}$,   we define   the   regulator map  to the weight $4$ real Deligne complex of $X$:
\be \la{regd}
r_{\cal D}: \Gamma(X;4) \lra \R_{\cal D}(X; 4).
\ee

 \subsection{The weight four motivic complex $\Gamma(X;4)$}  \la{sec9.3}

\paragraph{1. The weight four motivic complex $\Gamma(X;4)$.}    
Let $X$ be a regular variety over a field $k$ and  $\F:=k(X)$. 
Let  $X^p$ be  the set of  codimension $p$ points  $x$ of $X$, given by closed irreducible codimension $p$ subvarieties $Y_x \subset X$. Let $\F_x:= k(Y_x)$.  
Let  ${{\rm Gal}_{\F_x}} $  be the Galois group $  {\rm Gal}(\overline \F_x/\F_x)$. 
 The next theorem and Theorem \ref{MTHRCCC} below  provide  the main constructions of Section \ref{Sec7}. 
 
 \bt \la{MTHRCC} Let $X$ be a regular  variety over a field $k$ and $\F=k(X)$. Then there exists     the following complex $\Gamma(X; 4)$:   
 \be \la{MCDIm} 
\begin{gathered}
    \xymatrix{
       {\cal B}_4(\F) \ar[r]^{\delta}   &{\cal B}_{3}(\F)\otimes \F^\times \ar[d]^{{\rm res} } \ar[r]^{\delta}  & {\cal B}_{2}(\F)\otimes \Lambda^2 \F^\times \ar[r]^{\delta}  \ar[d]^{{\rm res} }& 
        \Lambda^{4} \F^\times \ar[ddl]^{{\rm H}} \ar[d]^{{\rm res} }\\
               &\bigoplus_{x\in X^1}   \Bigl( {\cal B}_3(\F_x)    \ar[r]^{\delta}&   {\cal B}_2(\F_x) \otimes  \F_{x}^\times  \ar[d]^{{\rm res} }
               \ar[r]^{\delta }& \Lambda^3\F_{x}^\times \Bigr) \ar[d]^{{\rm res} }\\
                   \Gamma(X; 4) :& &\bigoplus_{x\in X^2}    \Bigl(   {\cal B}_2(\overline \F_x)   \ar[r]^{ \delta}& \Lambda^2\overline \F_{x}^\times\Bigl)^{{\rm Gal}_{\F_x}}\ar[d]^{{\rm res} }\\
                & & &\bigoplus_{x\in X^3}   \F_{x}^\times   \ar[d]^{{\rm res}}\\   & & &\bigoplus_{x\in X^4}   \Q}\\               \end{gathered}
 \ee
Here  ${\cal B}_4(\F) $ is in the degree $1$,    
  the differential is the sum of the arrows,  the horisontal arrows are   differentials in the polylogarithmic motivic complexes ${\cal B}^\bullet(\ast;p)$.   
 The   $p-$th horizontal line, counted from below   from $p=0$,   
 is the direct sum   over  codimension $4-p$  points.   
 \et
 
Note that one has 
$
{\rm res}= \sum_x{\rm res}_x
$, where the direct sum is over points $x$ in the target group. 

 \bcon  For any  regular variety $X$ over a field $k$, the complex  $\Gamma(X;4)$ is  quasiisomorphic to  the weight $4$ motivic complex 
 ${\rm RHom}_{{\cal MS}_X}(\Q_{\cal M}(0), \Q_{\cal M}(4))$, where   ${\cal MS}_X$ is the conjectural derived category of mixed motivic sheaves  on   $X$  \cite{Bei87}, and 
 $\Q_{\cal M}(n)$ are   Tate sheaves on $X$. 
 In particular,  
 $$
 H^i\Gamma(X;4) \stackrel{?}{=} {\rm gr}^\gamma_4K_{8-i}(X)_\Q.
 $$
 \econ

 \paragraph{2. Remarks.}   Few comments are in order.
 
 \begin{enumerate}  \item  The passage to $\overline \F_x$ in (\ref{MCDIm}) is necessary only to define the residue maps ${\rm res}$ to   singular points, and    ${\rm H}(\A)=0$ if the 
  divisor of  $\A\in \Lambda^4\F^\times$ is a normal crossings divisor. In the application to motivic 
  Chern classes on ${\rm BGL}_{n} $ in Section \ref{SEC8.4n}, only residues at smooth points appear, and all divisors are normal crossings. 

\item   The  Gersten type resolution $\Gamma(X, n)$ of  polylogarithmic motivic complexes ${\cal B}^\bullet(\F; n)$ for  $n \leq 3$ was defined in \cite{Gon95},  providing the weight $n \leq 3$ motivic complexes  for regular $X$.  
 
 For $n\geq 4$ there is the following  problem  \cite[Section 1.15]{Gon95}. 
To define the residue map ${\rm res}_{x/y}$ from a point $x\in X^{p}$ to   a   point $y\in X^{p+1}$ when $  y$ is singular in the closure $\overline x$ of $x$, 
one needs   transfer maps for  the motivic complexes ${\cal B}^\bullet(\F; n-p-1)$. 
This problem is absent if   $y$ is a non-singular point of $\overline x$. So for regular varieties $X$ over a field the residue maps were not   defined    when $n=4, p=1$.\footnote{If  $n=, p=0$ there is no issue due to the regularity of $X$.} 

To handle this problem, we   work  with the complex
${\cal B}^\bullet(\overline \F; 2)^{{\rm Gal}_\F}$ in the middle raw of (\ref{MCDIm}). 
It is expected to be quasiisomorphic to  the complex ${\cal B}^\bullet(\F; 2)$, 
due to Galois descent on the motivic cohomology modulo torsion. Using the complex ${\cal B}^\bullet(\overline \F; 2)^{{\rm Gal}_\F}$ 
it is easy to define  residue maps ${\rm res}$ in $\Gamma(X;4)$, see (\ref{RESM12}).\footnote{Note that , evidently, there is no need to use the Galois invariants  at the bottom two raws.} 
However the composition ${\rm res}\circ {\rm res}$ is   not necessarily  zero. Yet   it
 is canonically homotopic to zero by Lemma \ref{LLL1} below. We denote the homotopy by ${\rm H}$. 
 The map $\delta + {\rm res} + {\rm H}$  is the differential in the complex $\Gamma(X;4)$. 

\item  There is a version $ \Gamma'(X;4)$ of the complex $\Gamma(X;4)$ where we use  for the top line in (\ref{MCDIm}) the weight four part ${\rm CE}_{(4)}^*{\cal L}(\F)$ of the Chevalley-Eilenberg complexes of the conjectural motivic Tate Lie algebra ${\cal L}(\F)$, or its appropriate substitute. 
The rest is the same, and the residue map vanishes on the $\Lambda^2{\cal B}_2(\F)$ component we add. 
This is a natural explicit model for the weight $4$ motivic complex, and our motivic Chern classes land naturally in its cohomology.

   \end{enumerate}

 The  proof of Theorem \ref{MTHRCC} occupies the rest of Section \ref{Sec7}. 

We start with the definition of the vertical arrows ${\rm res}$, called the {residue maps}.

\paragraph{3. Residue maps.} 
 Recall the  {\it the residue map} on the  polylogarithmic   complexes 
\cite[Section 15.1]{Gon95}.  Let $K$ be a field with a discrete valuation $v$, the residue field $k_v$, 
and the group of units $U$. Let $u \lms \overline u$ be the projection $U \to k_v^\times$. 
Choose a uniformizer $\pi$. There is a homomorphism $\theta_v: \Lambda^mK^\times \lra  \Lambda^{m-1}k_v^\times$  
 uniquely defined by the following properties ($u_i \in U$):
\be \la{theta}
\theta_v (\pi \wedge u_1 \wedge \ldots \wedge u_{m-1}) = \overline u_1 \wedge \dots  \wedge \overline u_{m-1}; ~~~~\theta_v (u_1 \wedge \ldots  \wedge u_m) =0. 
\ee
It is   independent of $\pi$. We define a homomorphism $s_v: \Z[{\Bbb P}^1_K] \lra \Z[{\Bbb P}^1_{k_v}]$ by
setting  $s_v\{x\}:= \{\overline x\}$  if  $x$ is a unit,  and  $0$ otherwise. It induces a map  
$s_v: {\cal B}_n(K) \lra  {\cal B}_n(k_v) $, $n \geq 2$. We set 
\be \la{RESM}
{\rm res}_v:= s_v\otimes \theta_v: {\cal B}_n(K)\otimes \Lambda^mK^\times \lra {\cal B}_n(k_v)\otimes \Lambda^{m-1}k_v^\times.
\ee 
The maps $\theta_v$ and  ${\rm res}_v$ define  a morphism of complexes 
\be \la{RESM1}
{\rm res}_v : {\cal B}^\bullet(K, m) \lra{\cal B}^\bullet(k_v, m - 1)[-1].
\ee

Now let $X$ be any variety over a field $k$,  possibly singular.\footnote{Although we define complexes $\Gamma(X;4)$  
only regular varieties $X$, singular varieties will appear in the definition of the residue to the middle raw, that is when $m=3$.}
 Let $\F=k(X)$. Given a codimension one point $x \in X^1$, let us define the residue map 
\be \la{RESM12}
{\rm res}_x : {\cal B}^\bullet(\F, m) \lra{\cal B}^\bullet(\overline \F_x, m - 1)[-1]^{{\rm Gal}_{\F_x}}.
\ee
If $x$ is a regular point of $X$, it is given by the map (\ref{RESM1}). If the point $x$ is singular, we consider the points $y_1, ..., y_m$ on the normalization $\widetilde X$ of $X$ over the point $x$,  
Each of them gives a map 
\be \la{RESM12a}
{\rm res}_{y_i}: {\cal B}^\bullet(\F, m) \lra{\cal B}^\bullet(\overline \F_{y_i}, m - 1)[-1].
\ee
Consider all possible  different  embeddings to $j_{i, \alpha}: \F_{y_i}\hra  \overline {\F}_x$, and  the   induced  maps 
$$
j_{i, \alpha \ast}: {\cal B}^\bullet(\F_{y_i}, m-1) \lra{\cal B}^\bullet(\overline \F_x, m-1 ). 
$$
Then
 take the normalized sum of the residue maps (\ref{RESM12a}):
\be \la{RESM12b}
{\rm res}_x : = {\cal B}^\bullet(\F, m) \lra \sum_{i, \alpha}  [\F_{y_i}: \F_x]\cdot j_{i, \alpha *}\circ {\rm res}_{y_i}\subset  {\cal B}^\bullet(\overline \F_x, m - 1)[-1]^{{\rm Gal}_{\F_x}}.
\ee
The Galois invariance property is forced by the averaging over all embeddings $j_{i, \alpha}$.\\

The vertical arrows in (\ref{MCDIm}) are given by the residue maps  (\ref{RESM}) or $\theta_v$,  
 related to a valuation $v$ given by an appropriate  divisor.  The map   ${\rm H}$ is defined in Theorem \ref{PPP0} in Section \ref{HMTPM}.  
 To motivate its definition, we begin to  construct the regulator map $r_{\cal D}$ in (\ref{regd}).

 \subsection{The canonical real period morphism} \la{SECT8.1}
 
 In Section \ref{SECT8.1} $X$ is a regular complex variety. 

\paragraph{1. The real Beilinson-Deligne complexes.} Let $X$ be a regular complex variety, and $\underline{\R}_X(n)$ the constant variation of real Hodge-Tate structures $\R(n)$ on $X$. Let $({\cal D}^\bullet(X),  {\rm D})$  be the de Rham complex of $\C-$valued  currents on $X$. Let  $({\cal D}_\R^\bullet(X)(n-1),  {\rm D})$ be the subcomplex of $\R(n-1)-$valued  currents on $X$, and  $(\Omega_{\rm log}^\bullet, d)$  the de Rham complex of holomorphic forms with logarithmic singuilarities on $X$. 
The weight $n$ real Beilinson-Deligne cohomology of    $X$ is   the cohomology of the   complex of sheaves   associated with the  following  bicomplex,  where ${\cal D}_\R^0(X)$ is in the degree $1$:
 \be \la{RDC}
\begin{gathered}
    \xymatrix{
       \Bigl( {\cal D}_\R^0(X)  \ar[r]^{{\rm D}}&{\cal D}_\R^1(X)   \ar[r]^{{\rm D}}  & \ldots  \ar[r]^{{\rm D} \ \ \ } &  {\cal D}_\R^{n}(X) \ar[r]^{{\rm D} \ \ } & {\cal D}_\R^{n+1}(X) \ar[r]^{ {\rm D}} & \ldots  \Bigr)(n-1)\\
              &   &   
                & \ar[u]^{\pi_n}{\Omega}_{\rm log}^{n}(X)  \ar[r]^{d \ \ }&  \ar[u]^{-\pi_n}{\Omega}_{\rm log}^{n+1}(X) \ar[r]^{\ \ \ d}&\ldots \\}
                                             \end{gathered}
\ee
The map  
 $\pi_n$ is induced by the projection $\C \to \C/\R(n) = \R(n-1)$. Equivalently, 
\be \la{165}
(\ref{RDC}) = {\rm Cone}\Bigl({\Omega}_{\rm log}^{\geq n}(X)\stackrel{\pi_n}{ \lra} {\cal D}_\R^\ast(X)\Bigr)[-1].
\ee
Let us replace the complex of sheaves $\Omega_{\rm log}^{\geq n}(X)$ in (\ref{165}) by its  Dolbeault resolution: 
$$
{\cal D}_{\rm log}^{\geq n, \ast}(X):= \Omega_{\rm log}^{\geq n}(X)\otimes_{\Omega^{\geq n}(X)} {\cal D}^{\geq n, \ast}(X).
$$
Then there is a natural map of complexes
$$
\pi_n:{\cal D}_{\rm log}^{\geq n, \ast}(X) \stackrel{ }{\lra} {\cal D}_\R^\ast(X), ~~~~
 \pi_n: {\cal D}_{\rm log}^{p, q}(X) \stackrel{ }{\lms} {\cal D}_\R^{p+q}(X).
 $$ 
We get   a complex of abelian groups calculating the cohomology of the complex of sheaves (\ref{RDC}):
\be \la{RDC1}
 \R_{\cal D}(X; n):= {\rm Cone}\Bigl({\cal D}_{\rm log}^{\geq n, \ast}(X) \stackrel{\pi_n}{\lra} {\cal D}_\R^\ast(X)\Bigr)[-1].
 \ee
It calculates ${\rm RHom}$'s in the category of real 
mixed Hodge structures on $X$:
\be \la{RHMNS}
{\rm RHom}^*_{\rm MHS_\R}(\underline{\R}_X(0), \underline{\R}_X(n)) \stackrel{}{= } \R_{\cal D}(X; n).
\ee 

Given a closed embedding of a regular  complex variety   $i: Y \hra X$   of codimension $d$,  there is a canonical map of complexes, provided by the push-forward  of currents:
\be \la{PUSHD}
i_*: \R_{\cal D}(Y; n)[-2d] \lra \R_{\cal D}(X; n+d).
\ee

 \paragraph{2. The canonical real period map ${\rm P}$.}  Recall  defined in  \cite[Section 4]{Gon08} complexes ${\cal C}_{\cal H}^\ast(\underline{\R}_X(n))$,    quasiisomorphic to     (\ref{RHMNS}):
$$
 {\cal C}_{\cal H}^\ast(\underline{\R}_X(n)) \stackrel{\sim}{= } \R_{\cal D}(X; n).
$$
Their key feature  is  a commutative associative product $\ast$, 
providing  a  commutative dg-algebra
\be \la{CH}
{\cal C}_{\cal H}^\ast(X):= \bigoplus_{n=0}^\infty{\cal C}_{\cal H}^\ast(\underline{\R}_X(n)).
\ee
This dg-algebra  has two gradings: the cohomological grading $\ast$, and the weight grading $\bullet$. \\

The category ${\cal H}{\cal T}_X$ of good variations of rational Hodge-Tate structures on a regular complex variety $X$ gives rise to a Hopf 
algebra of framed variations, and hence to the related $\Z_{>0}$-graded Lie coalgebra  ${\cal L}^{\rm HT}_{\bullet}(X)$. One has: 
\be \la{uiuiu}
{\cal L}^{\rm HT}_1(X) ={\cal O}^\times(X)_\Q .
\ee
Applying  the Chevalley-Eilenberg  functor ${\rm CE}^*( -)$, see (\ref{CEil}),   to the Lie coalgebra   
   ${\cal L}^{\rm HT}_\bullet(X)$ we get  a  commutative dg-algebra with an extra weight grading:
 $$
{\rm CE}^\ast\left({\cal L}^{\rm HT}_\bullet(X)\right).
$$
Let $({\cal A}^\bullet(X), d)$ be the real  de Rham complex on $X$. Set  ${\cal A}^\bullet(X)(m)= {\cal A}^\bullet(X)\otimes\R(m)$.

\bt \la{PM1} Let $X$ be a regular complex variety. Then:

1. There is a canonical map of commutative dg-algebras -  the real period map:
\be \la{MP}
{\rm P}: {\rm CE}^\ast\left({\cal L}^{\rm HT}_\bullet(X)\right) \lra {\cal C}_{\cal H}^\ast(X),
\ee
such that one has 
 \be \nonumber  \begin{split}
  &{\cal L}^{\rm HT}_1(X)  \stackrel{(\ref{uiuiu})}{=} {\cal O}^\times(X)_\Q \lra {\cal A}^0(X),\\
&f  \lms \log|f|.\\
 \end{split}
 \ee
 2. It provides a  regulator map  
to the weight $n$ real Deligne complex (\ref{RDC}) of $X$:
\be \la{CEM}
{\rm P}: {\rm CE}_{(n)}^\ast\left({\cal L}^{\rm HT}_\bullet(X)\right) \lra  \R_{\cal D}(X; n).
\ee
\et

\begin{proof} 1. It is based on the construction of {\it canonical generators} of the Tannakian Lie coalgebra of the category of real Hodge  structures. Their key property is that 
the corresponding periods  satisfy a Maurer-Cartan system of differential equations. 
The canonical generators were introduced for the subcategory of Hodge-Tate structures in \cite{Lev01}, and in full generality in \cite[Section 5]{Gon08}. 
The Maurer-Cartan differential equations for the canonical period functions  imply that the canonical period functions extend to a homomorphism of complexes. 

Taking the weight $n$ part of the map (\ref{MP}), where $n>0$,  we get   a    map of complexes: 
\be \la{MOC}
\begin{gathered}
    \xymatrix{
        {\cal L}^{\rm HT}_n(X) \ar[r]^{}  \ar[d]^{ {\rm P}}    &\Lambda_{(n)}^2 {\cal L}^{\rm HT}_\bullet(X) \ar[d]^{ {\rm P} } \ar[r]^{}  & \ldots  \ar[r]^{} & 
       \Lambda^n  {\cal L}^{\rm HT}_1(X) \ar[d]^{{\rm P}}  \\
             {\cal A}^0(X)(n-1)  \ar[r]^{ d}  &   {\cal A}^1(X)(n-1)  \ar[r]^{~~~~~~d}&   \ldots  
               \ar[r]^{ d~~~~~~}& {\cal A}^{n-1}(X)(n-1)   \\}
                                             \end{gathered}
\ee
The explicit formula for the product $\ast$ on dg-algebra (\ref{CH})  \cite[Definition 4.10]{Gon08} implies that the   the right column map is given by 
\be \la{PP}
 \begin{split}
 &f_1 \wedge \ldots \wedge f_n \lms \omega_{n-1}(\log |f_1|, \ldots , \log |f_n|),\\  
 &d \omega_{n-1}(\log |f_1|, \ldots , \log |f_n|) = \pi_n(d\log f_1 \wedge \ldots \wedge d\log f_n).\end{split}
\ee
Here $\pi_n$ is induced by the projection $\C \to \R(n-1)$. 
See a  formula for  $\omega_n$   in \cite[Section 2.1]{Gon08}.

2. Thanks to   (\ref{uiuiu}) and   (\ref{PP}),  the map (\ref{MOC}) provides a map (\ref{CEM}). \end{proof}

 The real Hodge realization map on the conjectural motivic Lie coalgebra, whenever the latter is defined, provides a map of Lie coalgebras,  
where    $\F:=\C(X)$ and $\eta_X:= {\rm Spec}\ \C(X) $: 
$$
 {\cal L}_\bullet(\F)  \lra  {\cal L}^{\rm HT}_\bullet(\eta_X).
$$
Composing the induced map of the Chevalley-Eilenberg complexes with the canonical real period map, we get a motivic version of the real period map:
\[
{\rm P}: {\rm CE}^\ast\left({\cal L}_\bullet(\F)\right) \lra {\cal C}_{\cal H}^\ast(\eta_X).
\]

\paragraph{3. Examples.}  1. The classical polylogarithms give rise to good variations of framed Hodge-Tate structures. Therefore Theorem \ref{PM1} provides a map of complexes:
 \be \la{MP3} 
\begin{gathered}
    \xymatrix{
        {\cal B}_n(\F) \ar[r]^{}  \ar[d]^{ {\rm P} }    & {\cal B}_{n-1}(\F) \otimes \F^\times\ar[d]^{ {\rm P} } \ar[r]^{}  & \ldots  \ar[r]^{} & 
       \Lambda^n  \F^\times \ar[d]^{{\rm P}^{ }}  \\
             {\cal A}^0(\eta_X)(n-1)  \ar[r]^{ d}  &   {\cal A}^1(\eta_X)(n-1)  \ar[r]^{~~~~~~d}&   \ldots  
               \ar[r]^{ d~~~~~~}& {\cal A}^{n-1}(\eta_X)(n-1).   \\}
                                             \end{gathered}
 \ee
 Explicit formulas   can be written using the formula for the $\ast-$product   \cite[Definition 4.10]{Gon08}. 
 \vskip 2mm

 2.  The  crucial for us example is given by the weight four period map: 
\be \la{MP4} 
\begin{gathered}
    \xymatrix{
        {\cal B}_4(\F \ar[r]^{}  \ar[d]^{ {\rm P} })    & {\cal B}_{3}(\F) \otimes \F^\times \ar[d]^{ {\rm P} } \ar[r]^{}  & {\cal B}_{2}(\F) \otimes \Lambda^2\F^\times  \ar[d]^{ {\rm P} } \ar[r]^{} & 
       \Lambda^4  \F^\times \ar[d]^{{\rm P}^{ }}  \\
             {\cal A}^0(\eta_X)(3)  \ar[r]^{ d}  &   {\cal A}^1(\eta_X)(3)  \ar[r]^{~d}&   {\cal A}^2(\eta_X)(3)
               \ar[r]^{ d}& {\cal A}^{3}(\eta_X)(3).   \\}
                                             \end{gathered}
 \ee

 Here are   explicit formulas.     Let us set, following \cite[Section 11.1.3]{Gon08}
\be \la{xryu}
\bs
&{\Bbb L}^*_{4}(z):= 2^{-4}\Bigl(5 {\cal L}_4(z) + \frac{1}{3}{\cal L}_2(z){\rm log}^2 |z| \Bigr),\\
&{\Bbb L}^*_{3}(z) := 2^{-3}3 {\cal L}_3(z),\\
&{\Bbb L}^*_{2}(z):= 2^{-1} {\cal L}_2(z).\\
\end{split}
\ee
Then, using \cite[Definition 4.10]{Gon08},  the map (\ref{MP4}) is given explicitly by
 \be
\bs
 \{x\}_4 \lms &{\Bbb L}^*_{4}(x).\\
 \{x\}_3 \otimes y \lms &\frac{5}{6} {\Bbb L}^*_{3}(x)  d^\C \log |y| - \frac{1}{6}  \log |x| d^\C {\Bbb L}^*_{3}(y).\\
 \{x\}_2 \otimes y \wedge z \lms &\frac{3}{7}{\Bbb L}^*_{2}(x) d^\C  \log |y| \wedge d^\C\log |z| - \frac{1}{7}\log |y|  d^\C {\Bbb L}^*_{2}(x) \wedge d^\C\log |z| \\
 &+ \frac{1}{7}\log |z|    d^\C {\Bbb L}^*_{2}(x) \wedge d^\C\log |y|, \\
 x_1\wedge x_2\wedge x_3 \wedge x_4 \lms &\omega_4(x_1, x_2, x_3, x_4) = \\
 &\frac{1}{6}{\rm Alt}_4\Bigl(\log|x_1| d^\C \log |x_2| \wedge d^\C \log |x_3|\wedge d^\C \log |x_4|\Bigr).\\
\end{split}
\ee

Recall  the de Rham complex $({\cal D}^\bullet(X),  {\rm D})$  of $\C$-valued currents on $X.$

\bp \la{GCEz}  Let $X$ be a regular complex variety and $\F=\C(X)$.  Then:    
 
 i) The  differential forms in the image of the  maps
\be \la{111}
\begin{split}
&{\rm P}:  \Lambda^q {\F}^\times  \lra {\cal A}^{q-1}(\eta_X)(q-1), \ \ \ \ \ \ \ \ \ \  \ \forall q\geq 1;\\
&{\rm P}: {\cal B}_{p}(\F)\otimes \Lambda^q {\F}^\times  \lra {\cal A}^q(\eta_X)(p+q-1),\ \ \ \ \forall p\geq 2, q\geq 0.\\
\end{split}
\ee
  are integrable on $X$, providing  currents    on $X$.

ii) Let  $i: Y \subset X$ be a closed irreducible subvariety of   codimension $d$, and  $\F_Y:=\C(Y)$. 
Then the forms on $Y$ given by the maps    (\ref{111})    provide  $\delta-$currents on $X$, 
resulting in   maps
\be \la{187}
i_*\circ {\rm P}:  \Lambda^q {\F}_Y^\times  \lra {\cal D}^{q-1+2d}(X), \ \ \ \ \ \ i_*\circ {\rm P}: {\cal B}_{p}(\F_Y)\otimes \Lambda^q \F_Y^\times  \lra {\cal D}^{q+2d}(X).
\ee 
\ep
  
 \begin{proof} i) Given an element $\A=\{g\}_p\otimes f_1 \wedge \ldots \wedge f_q \in {\cal B}_{p}(\F)\otimes \Lambda^q \F^\times$, 
take a blow up   $\sigma: \widehat X \lra X$ such that the divisors of the functions $\sigma^*(f_i)$,  $\sigma^*(g)$,  $\sigma^*(1-g)$ form a normal crossing divisor. 
Then near each point $p\in \widehat X$ there are 
 local coordinates $\{z_i\}$  such that each of those functions   
 is either regular at $p$, or given by the equation $z_j=0$ for some $j$. An example: $\{x\}_2 \otimes x\wedge y$.

Note that for any smooth functions $\varphi_j$ on near $p$, the following form is integrable near $p$: 
 \be \la{000}
 {\rm P}(\{z_0\}_p \otimes z_1\wedge \ldots \wedge z_k \wedge \varphi_1 \wedge \ldots \wedge \varphi_l).
 \ee 
  Indeed,  functions  ${\cal L}_p(t)\log^a|t|$ and   1-forms  ${\cal L}_p(t)d^\C \log^a|t|$ and $\log^a|t| d^\C{\cal L}_p(t)$ are integrable on $\C$.   
The $\ast-$product   in   the definition of the period map ${\rm P}$ is given by the product of such functions/their $d^\C-$differentials.  Therefore the direct product of these  currents   is a current on $\C^{q+1}$. Multiplying it by a smooth form we get the current assigned to (\ref{000}).

Since the push forward $\sigma_*$ of currents is  well defined,  we define the current ${\rm P}_{\rm c}(A)$ on $X$ by  
$$
{\rm P}_{\rm c} (\A):= \sigma_*{\rm P}(\sigma^*(\A)).
$$
Below we   skip the subscript in ${\rm P}_{\rm c}$ when it is clear that we mean the current. 

ii)  Take an element $\B \in {\cal B}_{p}(\F_Y)\otimes \Lambda^q \F_Y^\times$. Take a blow up  $\widehat X \to X$ providing   a regular subvariety $i: \widehat Y\subset \widehat X$.  
Since $\Q(Y) = \Q(\widehat Y)$,  by the part i) we get a current ${\rm P}_{\rm c} (\B)$ on $\widehat Y$.  
   Then, 
  given a smooth  test form $\varphi$ on $X$, the current $i_*{\rm P}_{\rm c} (\B)$   is given by
\be \la{IC}
 \langle i_*{\rm P}_{\rm c} (\B), \varphi\rangle := \int_{\widehat Y} {\rm P}_{\rm c} (\B) \wedge i^*\varphi.
\ee
\end{proof}

\paragraph{4. The map $i_{x*}\pi_{ }\circ {\rm P}$.}  Let $x \in X^2$ be a point given by a closed irreducible codimension 2   
 subvariety $i_x: Y_x \subset X$.  Let us   define the following map of complexes generalizing (\ref{187}): 
\be \la{kl}
 i_{x*}\pi_{ }\circ {\rm P}: \Bigl(  {\cal B}_2(\overline \F_x)   \stackrel{{\delta}}{\lra}  \Lambda^2\overline \F_{x}^\times\Bigl)^{{\rm Gal}_{\F_x} } ~~~~~\lra ~~~~~
 \Bigl({\cal D}^4(X) \stackrel{\rm D}{\lra} {\cal D}^5(X)\Bigr).
\ee
 
   Take a finite field extension  $\F_y/\F_x$, which we can assume to be   Galois. It  gives rise to a finite cover $\pi_{y/x}: Y^\circ_y\to Y^\circ_x$ of the generic point $Y_x^\circ \subset  Y_x$. 
   Given a smooth form $\varphi$  on $X$, and an element $\A \in {\cal B}_2(\F_y)$, we have to define a number    $ {\rm deg}(\F_y/\F_x)^{-1}\langle i_{x*}\pi_{ }\circ {\rm P}(\A), \varphi\rangle$. We set 
    $$
     \langle i_{x*}\pi_{ }\circ {\rm P}(\A), \varphi\rangle:=      \langle  {\rm P}(\A), \pi_{y/x}^*{\rm res}_{Y_x^\circ}\varphi\rangle.      
   $$
Let us show that the integral over $Y_y^\circ$ given by the right hand side converges, providing a current on $X$.  Take any smooth irreducible projective variety $Y$, containing  $Y_y^\circ$ as a Zariski  open subset, and 
    equipped with a map $p:Y \to Y_x$ extending the map $\pi_{y/x}$. Since ${\cal L}_2(z)$ is  continuous   on $\C\P^1$,  for any rational function $f$ on $Y$, 
    the function ${\cal L}_2(f)$ is continuous on $Y$. Since $p^*\varphi$ is a continous function on $Y$,   the integral converges, and coincides with   its restriction to the open part $Y_y^\circ$.   
 
  The definition for an element $\B \in \Lambda^2\F_y^\times$ is similar, using the fact that the 1-form  $\log|f|di \arg g - \log|g| di \arg f$ is   integrable  on $Y$ for any two rational functions $f,g$ on $Y$. \\

Proposition \ref{GCEz} allows to upgrade a map of complexes (\ref{MP3}) to the  smooth de Rham complex at the generic point of $X$ to maps of each individual group into the space of currents. 
However these maps may not commute with the   differentials. This leads us   to   residual differentials.

  
\paragraph{5. The residual differential.} Let $\omega$ be a smooth  form   on   a Zariski open part $U$ of  a regular complex variety $X$. Let us assume that it   provides a current  $[\omega]$ on $X$, given by 
 $\langle [\omega], \varphi\rangle:= \int_U\omega \wedge \varphi$, where the integral over  $U$  converges for any smooth on $X$ test form $\varphi$. 
 
 Denote by 
  $d_U$ the de Rham differential  
 on forms   on $U$. Recall   the de Rham differential ${\rm D}$ on currents on $X$:  $\langle {\rm D}[\omega], \varphi \rangle:= (-1)^{|\omega|} \langle  [\omega], d_X\varphi \rangle$. 
 
 Let us assume that the form $d_U\omega$ on $U$ defines a current $[d_U\omega]$ on $X$.  Then in general $[d_U\omega] \not = {\cal D}[\omega]$. 
 We call the difference the {\it residual differential } ${\rm res}_{\cal D}(\omega)$ of the current $\omega$ on $X$:
 \be
 {\rm res}_{\cal D}(\omega):=  {\rm D}[\omega] - [d_U\omega].
 \ee
 For example, the 1-form $\omega = di \arg (z)$ on $\C^\times$ defines a current on $\C$. Furthermore,  $d(di \arg(z))=0$ on $\C^\times$, while ${\rm D}[di\arg(z)]= 2\pi i\delta(z)$. 
  So  
 the residual differential is defined, and  
 $$
 {\rm res}_{\cal D}(di\arg (z))  = 2\pi i \delta(z).
 $$
 
 Here is an example of the residue differential. 
  \bl \la{LMP5} If the divisor of the functions $f_1, ..., f_n$ is a normal crossing divisor, then:
 \be \la{MP5}
\bs
&{\cal D} \circ {\rm P} (f_1 \wedge ... \wedge f_n) - \pi_n(d\log f_1 \wedge
\ldots \wedge d\log f_n) = 
2 \pi i\cdot \sum_{Y \subset X^{1}}{\rm P}  \circ {\rm res}_{Y}(f_1 \wedge ... \wedge f_n) .\\
\end{split}
\ee
\el

\begin{proof}  Formula (\ref{MP5}) is  proved directly, see \cite{Gon00a}.
\end{proof}

Our  next goal is to investigate the composition ${\rm res}_{\cal D}\circ {\rm P}$ of the canonical real period morphism ${\rm P}$ 
with the {residual differential }   for the weight four period morphism (\ref{MP4}).

  \subsection{The homotopy map $\rm H$.} \la{HMTPM} Let $X$ be  a  regular variety over a  field $k$, and $\F = k(X)$. Denote by $\F_x$ the field of functions 
  for a point $x \in X^1$ provided by an irreducible divisor $Y_x \subset X$.   There is the discrete valuation 
${\rm val}_{x}$ of $\F_x$ given by the order of zero of a rational function  on   $Y_x$.   Recall the residue homomorphism
\be
{\rm res}: \Lambda^4\F^\times \lra \bigoplus_{x \in X^1}\Lambda^3\F_x^\times. 
\ee
Its component ${\rm res}_x$ for a point $x \in X^1$  is given by the map (\ref{theta}) for the discrete valuation 
${\rm val}_{x}$:  
\be \la{resncd}
{\rm res}_x(f_0\wedge f_1\wedge f_2 \wedge f_3) = \sum_{i=0}^3 (-1)^{i}{\rm val}_{x}(f_i) f_0 \wedge \ldots \widehat f_i   \ldots \wedge f_3 \in \Lambda^3\F_x^\times.
\ee
Although the field $\F$ depends only on the birational type of $X$, the residue     depends on its model. 
\vskip 2mm

In the case when  $k=\C$, and $X$ is a regular complex variety, recall the map $i_{\ast}\pi\circ {\rm P}$ in (\ref{kl}). We denote by $\pi\circ {\rm P}$ the sum of these maps over all $x \in X^2$. 

Recall the map 
\be
\begin{split}
&d\log: \Lambda^k\F^\times \lra \Omega_{\rm log}^k(X), \\
& d\log (f_1\wedge\ldots  \wedge f_k) = d\log (f_1) \wedge \ldots \wedge d\log (f_k).\\
\end{split}
\ee

\bt \la{PPP0}  i) Let $X$ be  a  regular variety over a  field $k$, and $\F = k(X)$.  Then there is  a \underline{canonical} map
\be \la{themapH}
{\rm H}:  \Lambda^{4} \F^\times  \lra \bigoplus_{x\in X^2}    {\cal B}_2(\overline \F_x) ^{{\rm Gal}_{\F_x}}
 \ee
 such that   ${\rm H}(k^* \wedge \Lambda^{3} \F^\times)= 0$  and, see the diagram (\ref{MCDImwww}) below, 
 \be
\begin{split}
& \delta \circ {\rm H} + {\rm res}\circ {\rm res} =0. \\
\end{split}
 \ee

   ii)   If   $k=\C$,   for any $\A \in \Lambda^{4} \F^\times $,  the differential of the 3-current ${\rm P}(\A)$ on $X$ is   
 \be
 {\rm D}\circ {\rm P}(\A) =     \pi_4(d\log (\A)) +   2\pi i \cdot {\rm P}\circ  {\rm res}(\A) +  (2\pi i)^2 \cdot  \pi \circ {\rm P}\circ   {\rm H}(\A). 
  \ee
 \et
 
So the part i) provides    the following diagram: 
\be \la{MCDImwww} 
\begin{gathered}
    \xymatrix{
        {\cal B}_{2}(\F)\otimes \Lambda^2 \F^\times \ar[r]^{\ \ \ \ \ \delta}  \ar[d]^{{\rm res} }& 
        \Lambda^{4} \F^\times \ar[ddl]^{{\rm H}} \ar[d]^{{\rm res} }\\
               \bigoplus_{x\in X^1}    \Bigl(   {\cal B}_2(\F_x) \otimes  \F_{x}^\times  \ar[d]^{{\rm res} }
               \ar[r]^{\ \ \ \ \ \ \ \  \delta }& \Lambda^3\F_{x}^\times \Bigr) \ar[d]^{{\rm res} }\\
                  \bigoplus_{x\in X^2}    \Bigl(  {\cal B}_2(\overline \F_x)   \ar[r]^{ \ \ \ \ \ \delta}& \Lambda^2\overline \F_{x}^\times\Bigl)^{{\rm Gal}_{\F_x} }}
                                \end{gathered}
 \ee
 
  \begin{proof}  i) 
  We use the following   result,  conjectured in   \cite[Conjecture 6.2]{Gon02a}, and proved in full generality by V. Bolbochan   in \cite{Bol21}, see also \cite{Rud15s} for partial results.  
  \bt  Given a  field $K$, and a curve $C$ over $K$, there exists a canonical functorial   map 
\be
\begin{split}
&h: \Lambda^3K(C)^\times \lra {\cal B}_2(\overline  K)^{{\rm Gal}_K}.\\
\end{split}
\ee
The   map $h$   enjoys the following three  properties:\footnote{We use  the  $1-$form  $d i \arg(z)$, and    polylogarithms $i^{n-1}{\cal L}_n(z)\in \R(n-1)$, forcing the factor $2\pi i \cdot i $ in  (\ref{hhh10}).} 
\be \la{hhh10}
 \begin{split}
 &\delta \circ {h} (f_1  \wedge f_2  \wedge f_3 ) = \sum_{p \in C(\overline K)} {\rm res}_p(f_1  \wedge f_2  \wedge f_3).\\
&h: K^\times \wedge \Lambda^2 K(C)^\times \lra 0.\\
& \mbox{If $K=\C$, then}   \ \ \ \ \ \  \int_{C(\C)}\omega_2(f_1, f_2, f_3) = 2\pi i \cdot i{\cal L}_2(h(f_1  \wedge f_2  \wedge f_3)).\\
\end{split}
\ee
 \et

If $C$ is  ${\Bbb P}^1$   or   an elliptic curve, there are explicit formulas for the map $h$ \cite[Theorems 6.5, 6.14]{Gon02a}. For example, if $C= {\Bbb P}^1$ we get, using the cross-ratio 
      ${\it r}'(x_0, x_1, x_2, x_3):= \frac{(x_0-x_2)(x_1-x_3)}{(x_1-x_2)(x_0-x_3)}$, 
$$
f_1  \wedge f_2  \wedge f_3  \lms \sum_{x_i \in \P^1(\overline K)} {\rm val}_{x_1}(f_1){\rm val}_{x_2}(f_2){\rm val}_{x_3}(f_3)\{{\it r}'(x_1, x_2, x_3,\infty)\}_2.\
$$
   Take an element $\A \in  \Lambda^{4} \F^\times$. Take a sequence of blowups  
\be \la{BBB0}
\sigma: \widehat X \lra X
\ee 
 such that $\widehat X$ is non-singular, and the divisor of the element $\sigma^*\A$ is a normal crossing divisor. 
Let $\{E_i\}$ be   irreducible components  of the exceptional divisor $E$.   Consider the residue map at the divisor $E_i$: 
\be
{\rm res}_{E_i}: \Lambda^{4} \F^\times  \lra \Lambda^{3} k(E_i)^\times.
\ee

Next, 
let $Y_i:= \sigma(E_i)$,  $K_i:=k(Y_i)$. Assuming that ${\rm codim}(Y_i)=2$, the   fiber of the map $\sigma_{|E_i}: E_i \to Y_i$ over the generic point of $Y_i$     provides a 
curve $ C$  over a finite extension   of the field $K_i$. 
Indeed, to get the divisor $E_i$ we   blow up $\widehat Y_i \to Y_i$,   then   blow up a divisor in $\widehat Y_i$ which is finite over $Y_i$, etc. The curve $C$ can be any curve, e.g. when $X$ is a singular surface given by the cone over a curve $C$. Therefore by Theorem \ref{PPP0} there is a     map 
 $$
h_{E_i}:  \Lambda^{3} k(E_i)^\times \lra {\cal B}_2(\overline K_i)^{{\rm Gal}_{K_i}}. 
  $$
  Composing the  two maps, we get    the  homomorphism 
\be \la{he}
 {\rm H}_{E_i}:= h_{E_i} \circ  {\rm res}_{E_i}:  \Lambda^{4} \F^\times  \lra {\cal B}_2(\overline K_{i})^{{\rm Gal}_{K_i}}.
\ee 
\bd 
The $x-$component ${\rm H}_x$ of the map ${\rm H}$ in (\ref{themapH}) is the sum of the maps (\ref{he}) over all  components $E_i$ of the exceptional divisor $E$ over $x$ such that $\sigma(E_i)$ has  codimension $2$:
\be
{\rm H}:= \sum_{x \in X^2}{\rm H}_x, \ \ \ \  {\rm H}_x:= \sum_{{\rm codim}\ \sigma(E_i)=2} {\rm H}_{E_i}: \Lambda^{4} \F^\times  \lra  {\cal B}_2(\overline \F_x)^{{\rm Gal}_{\F_x}}.
\ee
\ed

\bl \la{LLL0} The map ${\rm H}$ does not depend on the choice of blow ups in its definition. 
\el

 \begin{proof} For any two blow ups there is a third one dominating them. So it is sufficient to check that the map ${\rm H}$ does not change when we blow up $\widehat X$. 
 This follows  from the second property in (\ref{hhh10}).   Indeed,   blowing up a codimension $2$ point on a variety with an element $\A$ with a normal crossing divisor, and  restricting $\sigma^*\A $ to  the exceptional divisor $E$, we get an element of  $k^\times \wedge \Lambda^2\F_E^\times$, annihilated by $h_{E}$. \end{proof}

 \bl \la{LLL1} One has 
 \be \la{iikk}
(\delta  \circ {\rm H}  + {\rm res}\circ {\rm res} )(\A) =0.
\ee
 \el
 
 \begin{proof} On the blow up $\widehat  X$  we have  ${\rm res} \circ {\rm res}  \ \sigma^*\A  =0$ since the divisor  of  $\sigma^*\A$ 
has normal crossings.  
 Denote by    $\{Z_j\}$ the  non-exceptional irreducible divisors on $\widehat X$ with  
 ${\rm res}_{Z_j}\sigma^*\A \not = 0$. Then 
 $$
 {\rm res} \circ {\rm res} (\A)=\sigma_*\Bigl({\rm res} \circ \sum_j {\rm res}_{Z_j}  \sigma^*\A  \Bigr).
  $$
Indeed, clearly $  {\rm res} (\A)=\sigma_*( \sum_j {\rm res}_{Z_j}  \sigma^*\A )$. To define the residue on a   divisor $\sigma(Z_j)$ we  normalize it, calculate the residue, and push it down to $\sigma(Z_j)$. 
Since $Z_j$ is non-singular, it coincides with the normalization of $\sigma(Z_j)$ up to the   codimension $\geq  2$. 
 So it remains to prove that 
 $$
\delta \circ {\rm H} (\A)=\sigma_{2, *}\Bigl({\rm res} \circ \sum_i {\rm res}_{E_i}  \sigma^*\A  \Bigr).
$$
Here $\sigma_{2, *}$ denotes the contribution to $\sigma_*$ of the components whose projection has codimension $2$. Note that  ${\rm codim}\ \sigma(E_i)>1$.  
 If  ${\rm codim}\ \sigma(E_i)>2$, then $E_i$ does not contribute neither to ${\rm H} (\A)$, nor to the right hand side. 
If the codimension of $Y_i:=\sigma(E_i)$   is $2$, the claim 
 follows from the definition of ${\rm H}$ and the top formula in (\ref{hhh10}).    
 Indeed, the generic point $E^\circ_i$ of $E_i$ is obtained by  blowing up  $Y_i$, then blowing up  a 
finite over $Y_i$ divisor there, etc. So $E^\circ_i$ is a curve $C$  over a finite extension of $k(Y_i)$.   \end{proof} 
The part  i)   of Theorem \ref{PPP0} follows from Lemmas \ref{LLL0} and \ref{LLL1}. \vskip 2mm

ii)     From now on till the end of the proof, $X$ is a regular complex variety. By the very definition,  
$$
{\rm D}\circ {\rm P}(\A):= \sigma_* \circ  {\rm D} \circ {\rm P}(\sigma^* \A).
$$

Lemma \ref{LMP5}, applied to the normal crossing divisor of $\sigma^* \A$, calculates the current  ${\rm D} \circ {\rm P}(\sigma^* \A)$ on $\widehat X$. 
Its push forward   to $X$ is a sum of the two components: 

1. The sum of  push forwards of  contributions of   non-exceptional divisors $\{Z_j\}$.

2. The sum of  push forwards of  contributions of   the exceptional divisor $E = \{E_i\}$.

 The first one is equal to ${\rm P}\circ {\rm res}(\A)$.

The contribution of a component $E_i$ of the exceptional divisor  with ${\rm codim} ~\sigma_*(E_i)  >2$ is zero for the following   reason.  By the very definition, given  a smooth form $\varphi$  on $X$,  
\be \la{pushf}
\langle \sigma_* \circ  {\rm D} \circ {\rm P}(\sigma^* \A), \varphi \rangle:=   \langle  {\rm D} \circ {\rm P}(\sigma^* \A), \sigma^*\varphi \rangle.
\ee 
 The restriction of the form $\sigma^*\varphi$   to the exceptional divisor is zero. The component of  the 
 $4-$current ${\rm D} \circ {\rm P}(\sigma^* \A)$ concentrated on $E_i$ 
 is given by a 2-form at the generic point of $E_i$.  
 The evaluation of  the right hand side in (\ref{pushf}) is given by the   integration along the fibers of the fibration $E_i \to \sigma(E_i)$. 
 The  fibers have dimension  at least $k= {\rm codim}~\sigma(E_i)-1$.   Therefore  the push-forward is zero if $k>1$, since the real dimension  of the fiber, which is  at least $2k$,  
 is bigger than the degree of the form, which is $2$.

A similar reasoning  for a subvariety $Y_x:=  \sigma_*(E_i) $, ${\rm codim} ~ Y_x=2$, delivers  $g(y) \cdot \delta_{Y_x}$, where   $g(y)$ is a function on $Y_x$ given by the integral 
along the fibers of the projection $E_i \to Y_x$. These fibers are curves.   The generic fiber is a curve over a finite extension $\F_y$ of the field $\F_x$. Geometrically, it is a fibration on  curves  over a finite cover $\pi_{y/x}: Y_y\to Y_x$.  So  the map $\sigma: E_i \to Y_x$ is a composition $E_i \to Y_y \to Y_x$ where  the fibers of the first map are connected curves $C_y$. 
Denote by $\pi_{y/x*}$ the pushforward of currents. Then  \be \la{mappiyx}
g(y):=  \pi_{y/x*}\int_{C_y(\C)} \omega_2({\rm res}_{E_i}  \sigma^*\A ) =2 \pi i \cdot i {\cal L}_2({\rm H}_{E_i}(\sigma^*\A)).
\ee    
To prove the  second equality we calculate the middle integral  using    the last formula in (\ref{hhh10}):
$$
\int_{C_y(\C)} \omega_2({\rm res}_{E_i} \sigma^*\A) = 2 \pi i \cdot i{\cal L}_2({h}_{E_i}(\sigma^*\A)).
$$
  \end{proof}

The complex $\Gamma(X,4)$ has been defined, and the part i) of Theorem \ref{MTHRCC} is proved.  

\subsection{The real regulator map} 

Recall the canonical real period map $\rm P$ from  Section \ref{SECT8.1}.

   \bt \la{MTHRCCC} Let $X$ be a regular  complex variety and $\F=\C(X)$. Then there is a natural   map of complexes, called the real regulator map 
   \be \la{iii}
r_{\cal D}:   \Gamma(X;4) \lra  \R_{\cal D}(X; 4),  
 \ee
whose restriction to the horizontal complexes is given by the  canonical real period map ${\rm P}$. 
 \et
 
 \begin{proof}

 The regulator map (\ref{iii}) 
is given by its   two components:
\be \la{167a}
\begin{split}
&\alpha:   \Gamma(X;4) \lra  {\cal D}_\R^*(X)(n-1)[-1].\\
&\beta:   \Gamma(X;4) \lra  {\cal D}^{\geq n, \ast}(X).\\
\end{split}
\ee
The map $\alpha$ is given by applying the canonical period map ${\rm P}$ to the polylogarithmic motivic complexes sitting at generic points of subvarieties $Y_x \subset X$, see (\ref{MP4}), and considering the obtained forms as currents on $X$.  Here we use Proposition \ref{GCEz} which guarantee that we indeed get currents as well as Section \ref{SECT8.1}.3  where 
the map $i_{x*}\pi\circ {\rm P}$ was defined. The result is depicted  below: 
 \be \la{MCDImq} 
\begin{gathered}
    \xymatrix{
     {\rm P}\Bigl(  {\cal B}_4(\F) \ar[r]^{}   &{\cal B}_{3}(\F)\otimes \F^\times \ar[d]^{{\rm res}_{ \cal D}} \ar[r]^{}  & {\cal B}_{2}(\F)\otimes \Lambda^2 \F^\times \ar[r]^{}  \ar[d]^{{\rm res}_{ \cal D}}& 
        \Lambda^{4} \F^\times \Bigr)\ar[ddl]^{{\rm res}_{ \cal D}}\ar[d]^{{\rm res}_{ \cal D}}\\
               &\bigoplus_{x\in X^1} i_{x\ast}   {\rm P}\Bigl( {\cal B}_3(\F_x)    \ar[r]^{}&   {\cal B}_2(\F_x) \otimes  \F_{x}^\times  \ar[d]^{{\rm res}_{ \cal D}}
               \ar[r]^{ }& \Lambda^3\F_{x}^\times \Bigr) \ar[d]^{{\rm res}_{ \cal D}}\\
                  &  \alpha (\Gamma(X; 4)) :  &\bigoplus_{x\in X^2}   i_{x\ast}   \pi  \circ {\rm P}\Bigl(  {\cal B}_2(\overline \F_x)   \ar[r]^{ }& \Lambda^2\overline \F_{x}^\times\Bigl)^{{\rm Gal}_{\F_x}} \ar[d]^{{\rm res}_{ \cal D}}\\
                & & &\bigoplus_{x\in X^3}  i_{x\ast}   {\rm P}( \F_{x}^\times ) \ar[d]^{{\rm res}_{ \cal D}}\\   & & &\coprod_{x\in X^4} i_{x\ast}   \R }\\               \end{gathered}
 \ee

  The map $\beta$ is  zero everywhere except   the right column of the diagram    (\ref{MCDIm}).  It is defined on the right column as follows:
  $$
 \beta: (f_1 \wedge \ldots \wedge f_k)_{|Y_x} \lms i_{x*}(d\log f_1 \wedge ... \wedge d\log f_k)  \in {\cal D}_{\rm log}^{n, \ast}(X)[1], \qquad x   \in X^{4-k} ,~~ k>0.
  $$

 We have to  study the differentials of   {currents}  in (\ref{MCDImq}). Such a differential   
  is   a sum of two terms:  the current provided, via Proposition \ref{GCEz},  by the   differential of the underlying  form,  
and  the  {residuel differential}. 

The claim that the map $r_{\cal D}$ commutes with the differential on $\Lambda^4\F^\times$ follows from the part ii) of Theorem \ref{PPP0}. The latter calculates  the residual differential on ${\rm P}(  \Lambda^4  \F^\times )$,  and involves, besides the residue map,  a new feature -  the homotopy ${\rm H}$. 

The claim that the map $r_{\cal D}$ commutes with the differential on the rest of the right column   of $\Gamma(X;4)$ just means   that the residual differential on ${\rm P}(  \Lambda^k  \F_x^\times )$ for $k\leq 3$ 
is controlled  by the residue map. This follows using  the arguments in the proof of   part ii) of Theorem \ref{PPP0} and Lemma \ref{LMP5}.

 \bp \la{PPP0a}  Given a   regular complex variety $X$,  for any 
 $$
 \B \in ({\cal B}_2(\F) \otimes   \F^\times) \oplus ({\cal B}_2(\F) \otimes \Lambda^{2} \F^\times) \oplus ({\cal B}_3(\F) \otimes   \F^\times)
  $$
     the residual differentials of   
  ${\rm P}( \B)$ 
 is provided by the residue maps. 
Precisely, setting $i_* = \sum_x i_{x*}$, 
 \be
 \begin{split}
  &{\rm res}_{\cal D}\circ {\rm P}(\B) =  2\pi i \cdot  i_*{\rm P}\circ  {\rm res}(\B),\\
  \end{split}
   \ee
 \ep
 
 \begin{proof}  Let us handle an element $\B \in  {\cal B}_2(\F) \otimes \Lambda^{2} \F^\times$. For any blow up $\sigma: \widehat X\to X$ such that the divisor of $\sigma^*(\B)$ is a normal crossing divisor, 
 \be \la{pushf1}
 \begin{split}
&\langle  {\rm D} \circ {\rm P}(  \B), \varphi \rangle:= \langle \sigma_* \circ  {\rm D} \circ {\rm P}(\sigma^* \B), \varphi \rangle:=   \langle  {\rm D} \circ {\rm P}(\sigma^* \B), \sigma^*\varphi \rangle \\
&=\langle    {\rm P}\circ ({\rm res} + \delta)(\sigma^* \B), \sigma^*\varphi \rangle.\\
 \end{split}
 \ee  
Here the first two equalities are the definitions, and the last follows from  
\bl \la{LMP51} If divisors of the functions $g, 1-g, f_1, f_2$ form a normal crossing divisor, then:
\[
{\rm D} \circ {\rm P} (\{g\}_2 \otimes f_1 \wedge  f_2) =  {\rm P} ( (1-g)\wedge g \wedge f_1\wedge f_2) + 
2 \pi i\cdot \sum_{Y \subset X^{1}}{\rm P}  \circ {\rm res}_{Y} (\{g\}_2 \otimes f_1 \wedge  f_2).
\]
\el
The rest is very similar to the proof of the part ii) of  Theorem \ref{PPP0}, with the only proviso that the contribution of the exceptional divisor is zero. Indeed, 
${\rm P}  \circ {\rm res}_{Y} (\B)$ is a 1-form, and so it contributes zero for any push-forward $\sigma_*$, since its    fibers are of the dimension $>0$. 
The other two cases of are very similar. 
 \end{proof}

 Note that the assumption that $X$ is smooth is sufficient to handle singular divisors $Y_x$ for elements $\B \in  {\cal B}_2(\F_x) \otimes   \F_x^\times$ as well.
 
  It is easy to see that  for the groups on the diagonal in (\ref{MCDImq}) the residuel differential is zero. 
 
 Theorem \ref{MTHRCCC} is proved. \end{proof}.


  \section{The weight $4$ motivic Chern class $\&$ Beilinson's regulator} \la{SEC8.4n}
  
  \subsection{Construction of the  weight 4 Beilinson's regulator: the strategy} \la{sect9}
  
   \paragraph{1. Milnor's simplicial model of the classifying space ${\rm B}\G_\bullet$.}  Given an algebraic group $\G$, recall the simplicial realization ${\rm E}\G_\bullet$ of the space ${\rm E}\G$: \\
 
 \begin{tikzpicture}
\path  
 node (a) at (3.8,0) {$\cdots$}
  node (b) at (5,0) {$G^3$}
  node (c) at (8,0) {$G^2$}
  node (d) at (11,0) {$G$};

\begin{scope}[->,>=latex]

    \foreach \i in {0,...,2}{%
      \draw[->] ([yshift=(\i - 1) * 0.6 cm]b.east) -- ([yshift= (\i - 1) * 0.6 cm]c.west) ;}
    \foreach \i in {0,...,1}{%
      \draw[<-] ([yshift=(2 * \i -1) * 0.3 cm]b.east) -- ([yshift=(2 * \i -1) * 0.3 cm]c.west) ;}

    \foreach \i in {0,...,1}{%
      \draw[->] ([yshift=(2 *\i - 1) * 0.3 cm]c.east) -- ([yshift=(2 *\i - 1) * 0.3 cm]d.west) ;}
    \foreach \i in {0,...,0}{%
      \draw[<-] ([yshift=\i * 0.3 cm]c.east) -- ([yshift=\i * 0.3 cm]d.west) ;}
\end{scope}

\end{tikzpicture} \\
 In particular, there are the $n+1$ standard maps  
\be
s_{n, i}: \G^{n+1} \lra \G^n, ~(g_0, ..., g_n) \lms (g_0, ..., \widehat g_i, ..., g_n), ~~i=0, ..., n.
\ee
Then we set 
 ${\rm B}\G_\bullet:= \G\backslash {\rm E}\G_\bullet$, where the group $\G$ acts on $\G^n$ diagonally. \vskip 2mm

Let   $X \lms {\cal F}^\bullet(X)$ be an assignment to an algebraic variety $X$ a  complex of abelian groups ${\cal F}^\bullet(X)$, contravariant under surjective maps $X \to Y$. 
We define the complex   ${\cal F}^\bullet({\rm E}\G_\bullet)$  as the total complex associated with the  bicomplex
\be
  \begin{split}
  &  \cdots \stackrel{s^*}{\longleftarrow}  \mathcal{F}^\bullet({\G^4}) \stackrel{s^*}{\longleftarrow}    \mathcal{F}^\bullet({\G^3})  \stackrel{s^*}{\longleftarrow}  
  \mathcal{F}^\bullet({\G^2})\stackrel{s^*}{\longleftarrow}  \mathcal{F}^\bullet(\G)\stackrel{s^*}{\longleftarrow}  \mathcal{F}^\bullet(\ast).\\
&s^*: = \sum_{i=0}^n (-1)^is_{n, i}^*: \mathcal{F}^\bullet({\G^{n}}) \lra \mathcal{F}^\bullet({\G^{n+1}}).\\
\end{split}
 \ee
 Applying this construction  to    complexes $\Gamma (\ast; 4)$,  and taking the $\G-$invariants, we get the   complex  
$$
\Gamma ( \B\G_\bullet;4) :=  \Gamma ({\rm E}\G_\bullet  ; 4)^\G.
$$
Its cohomology $H^*\Gamma ( \B\G_\bullet  ;4)$ are the weight four motivic cohomology of the classifying space $\B\G$.

Let $\G = {\rm GL}_n$, and let ${\rm U} \subset {\rm GL}_n$ be the subgroup preserving an $(n-3)-$flag. Then we have the   decorated $(n-3)-$flag variety, which we denote by 
 \be \la{Aa}
  \mathscr {A}:=\G /{\rm U} = {\cal A}^{(n-3)}_{n}.
\ee
The canonical   projection $({\rm GL}_n)^k\lra {\mathscr A}^k$ induces a map of complexes, denoted $\varphi_{{\cal A}  \to {\rm G }}$: 
  \begin{displaymath}
    \xymatrix{
 \Bigl( \ldots&    \ar[l]_{s^*~~~~~}   \ar[d]  \Gamma^\bullet({\mathscr A}^4 ; 4)    &    \ar[d]  \ar[l]_{~~s^*}     \Gamma^\bullet({\mathscr A}^3; 4)   &   \ar[d]  \ar[l]_{~~s^*}  
  \Gamma^\bullet({\mathscr A}^2; 4)  &   \ar[d]  \ar[l]_{~~s^*} \Gamma^\bullet({\mathscr A} )\Bigr)^{{\rm GL}_n} \\
  \Bigl(\ldots& \ar[l]_{s^*~~~~~}     \Gamma^\bullet(({\rm GL}_n)^4 ; 4)    & \ar[l]_{~~s^*}       \Gamma^\bullet({({\rm GL}_n)^3 }; 4)   &     \ar[l]_{~~s^*}   
 \Gamma^\bullet({({\rm GL}_n)^2 }; 4)  &   \ar[l]_{~~s^*}   
   \Gamma^\bullet({ \rm GL}_n; 4) \Bigr)^{{\rm GL}_n} \\}
         \end{displaymath}

\paragraph{2. Universal motivic Chern classes. } Recall that one should have, for all $n \geq 1$,  
 the universal   motivic Chern class     of the classifying space ${\rm BGL}_{{n}}$ 
 of the algebraic group ${\rm GL_n}$    with values in the weight $p$ motivic cohomology:
$$
{c}_{p, {n}}^{\cal M}\in H^{2p}_{\rm Zar}({\rm BGL}_{n}, \underline{\Z}_{\cal M}(p)).
$$
Here 
$H_{\rm Zar}$ is Zariski cohomology, and $\underline{\Z}_{\cal M}(p)$  the weight $p$  motivic complex  of sheaves \cite{Voe00}.
These classes would provide   geometric constructions of  Chern classes in various realizations. \vskip 2mm

The construction of the complex   $\Gamma(X;4)$ is  immediately  generalized to simplicial varieties. We use   Milnor's simplicial model ${\rm BGL}_{{n}\bullet}$  of the classifying space ${\rm BGL}_{{n}}$, and  
define a cocycle  
\be \la{MCC}
{\rm C}_{4, {n}}^{\cal M}\in H^{8}\Gamma({\rm BGL}_{{n}\bullet} ;4) 
\ee
  representing a universal motivic Chern class.  This    generalizes the construction   of the universal motivic Chern classes in \cite{Gon93}, \cite[Section 5]{Gon94} for the weights 2 and 3.\footnote{See more details for the weight 3 in   \cite[Appendix]{Gon95a}. }

To define the class ${\rm C}_{4, n}^{\cal M}$ we define a class in $H^8\Gamma ( {\mathscr A}_{\bullet};4)$, and apply  to it  the map $\varphi_{{\cal A} \to \G}$. 
  \vskip 2mm

Applying   the regulator map $r_{\cal D}$ defined in Section \ref{Sec7} to the cocycle ${\rm C}_{4, {n}}^{\cal M}$, we  get a cocycle ${\rm C}_{4, {n} }^{\cal H}$ for the  weight $4$ Beilinson regulator class in the  
  the weight $4$ real Deligne cohomology:
\[
{C}_{4, {n}}^{\cal H}:= r_{\cal D}({\rm C}_{4, {n}}^{\cal M}) \in H^8\R_{\cal D}({\rm BGL}_{{n}\bullet}(\C);4) = H^8({\rm BGL}_{{n}\bullet}(\C), \underline{\R}_{\cal D}(4)).
\]
This is all we need to
finish the proof of Zagier's conjecture on $\zeta_\F(4)$.

 \paragraph{3. The strategy to define   the  class (\ref{MCC}).}   Recall the weight $4$   motivic complex  
 \be \la{G4}
{\rm G}_4(\F) \lra ({\B}_3(\F)\otimes \F^\times) \oplus  \Lambda^2 {\B}_2(\F) \lra {\B}_2(\F)\otimes \Lambda^2\F^\times \lra \Lambda^4\F^\times.
\ee

Recall   the decorated  $(n-3)-$flag complex $\Z[{\mathscr A}_{ \bullet}(\F)]$ for ${\rm GL}_{n}(\F)$:
  \begin{displaymath}
    \xymatrix{
 \Z[{\mathscr A}_{ \bullet}(\F)]:= \Bigl(  \ldots  ~~  \ar[r]^{~~~~~s_*}      &      \ar[r]^{ s_*}       \Z[{\mathscr A}^3(\F)]     &      \ar[r]^{ s_*~~~}    
  \Z[{\mathscr A}^2(\F) ]&     \Z[{\mathscr A}(\F)] \Bigr)^{{\rm GL}_n}.  }
         \end{displaymath} 
 Here $s_* = \sum_{i=0}^n(-1)^is_{n,i}$. 
 The   main result of Section \ref{SEC2} is  a   construction  of the map  
\be \la{COMDA}
\mbox{ decorated $(n-3)-$flag complex $\Z[{\mathscr A}_\bullet(\F)]$ } \lra \mbox{complex (\ref{G4})}. 
\ee
Combined with the map $\varphi_{{\cal A} \to \G}$, this is just    equivalent to a construction of a cocycle   $ \widetilde {C}_{4, {n} ; \circ}^{\cal M}$ for 
  the motivic Chern class   at the generic point  of    ${\rm BGL}_{{n} \bullet}$ with values in the complex (\ref{G4}).  
 Using Theorem \ref{JMCi} and   the basic relation between  motivic correlators and the motivic fundamental group of the punctured ${\Bbb P}^1$,   we deduce from the existence of  the  map (\ref{COMDA}) the existence of a similar map to the weight $4$ part of the Chevalley-Eilenberg complex the Lie coalgebra ${\Bbb L}_{\leq 4}(\F)$.    
To proceed further, we recall the weight $4$ polylogarithmic motivic complex:
  \be \la{GB4}
{\rm  B}_4(\F) \lra {\B}_3(\F)\otimes \F^\times  \lra {\B}_2(\F)\otimes \Lambda^2\F^\times \lra \Lambda^4\F^\times.
\ee 

Next,    
  Theorem \ref{THRN} implies the existence of a  
map 
\be \la{COMDA1}
\mbox{decorated $(n-3)-$flag complex   $\Z[{\mathscr A}_\bullet(\F)]$} \lra \mbox{weight $4$ polylogarithmic motivic complex}. 
\ee   
Combining with the map $\varphi_{{\cal A} \to \G}$, we interpret it as a   cocycle ${C}_{4, {n}; \circ}^{\cal M}$  
 for 
  the motivic Chern class   at the generic point  of   the simplicial scheme ${\rm BGL}_{{n}\bullet}$ with values in    the weight $4$   polylogarithmic  motivic complex.    
  In Section \ref{SEC8.4n} we   calculate  the  residues of  the cocycle  $ \widetilde {C}_{4, {n}; \circ}^{\cal M}$ on ${\rm BGL}_{{n}\bullet}$, 
  and using this extend  it   to a cocycle on ${\rm BGL}_{{n}\bullet}$.  This implies easily   an extension of the cocycle 
  ${C}_{4, {n}; \circ}^{\cal M}$ to ${\rm BGL}_{{n}\bullet}$. Finally, we deduce from its  construction that the real Hodge cocycle  ${C}_{4, {n}; \circ}^{\cal H}$   
  represents Beilinson's class.

\subsection{Calculating the residues} \la{SEC9.2}

  \bt \la{TH1.7}  Let ${n} \geq 4$. There is a cocycle  
  $$
 {\rm C}^{\cal M}_{4, {n} } \in  \Gamma({\rm BGL}_{{n}  \bullet}; 4).
  $$
  It  represents the universal motivic    Chern   class $c^{\cal M}_{4, {n} }$. 
  Its real Deligne realization  is a cocycle 
  $$
 {\rm C}^{\cal H}_{4, {n} } \in \Gamma_{\cal D}({\rm BGL}_{{n}   \bullet}(\C); 4).
  $$ 
  It provides the weight $4$ Beilinson regulator. 
\et

Denote by $\widehat \Gamma(X;4)$ the analog of the complex $\Gamma(X;4)$  in (\ref{MCDIm}), where we put instead of 
${\cal B}_4(\F) \lra{\cal B}_3(\F) \otimes \F^\times$ the map $\G_4(\F) \lra \Lambda^2\B_2(\F)\oplus \B_3(\F) \otimes \F^\times$, and set the residue map 
on $\B_2(\F) \wedge \B_2(\F)$ to be zero.   Our main goal is to produce a cocycle  
\be \la{1588}
 \widehat {\rm C}^{\cal M}_{4, {n}  } \in  \widehat \Gamma({\rm BGL}_{{n}   \bullet}; 4).
\ee
Then in the second major step  we will show that it can be corrected by a coboundary to land in $\Gamma({\rm BGL}_{{n}   \bullet}; 4)$. 
This will affect only the top row of the bicomplex due to   Lemma \ref{L9.5} below.

\vskip 2mm
Recall  the    cocycle $\widetilde {\rm C}^{\cal M}_{4, {n} , \circ }$  defined at the generic point of 
$\B{\rm GL}_{{n}  \bullet}$.  Let us calculate   its residues  on $\B{\rm GL}_{{n}  \bullet}$. 
Few comments are in order. \vskip 2mm

1. Cocycle $\widetilde {\rm C}^{\cal M}_{4, {n} , \circ }$ is obtained by a three-step construction:

\begin{itemize} 

\item  Map  the decorated $(n-3)-$flag complex $\Z[{\mathscr A}_\bullet(\F)]$  to the weight $4$ Bigrassmannian complex.

\item  Apply the Grassmannian duality.

\item   Apply  map (\ref{56}) from the  weight $4$ dual Bigrassmannian complex to  the complex (\ref{G4}).

 \end{itemize} 

The Grassmannian duality 
is a birational map. For the residue calculation we  work with an equivalent variant of the cocycle obtained in a two-step construction, followed by the map $\varphi_{{\cal A}\to \G}$:

\begin{itemize} 

\item
(a) Map  the  complex $\Z[{\mathscr A}_\bullet(\F)]$ to the weight $4$ Bigrassmannian complex.

\item (b) Apply the map from the Bigrassmannian complex  to  the motivic complex (\ref{G4}).

 \end{itemize}
 
 To get the map (b), we simply apply the Grassmannian duality to the map  (\ref{56}).  \vskip 2mm
 
2. The $\Lambda^*\F^\times$-components of the cocycle in (b) are given by  regular functions on $\G^m$, where $ \G:= {\rm GL_n}$. 
Indeed, each $\F^\times$-component is given  by a minor on the space of $m+1$ decorated flags in an $n-$dimensional vector space. Calculating the residues  
it is sufficient to do this just for the Birassmannian, and pull back the result first to the simplicial variety of the decorated flags, and then to $\B{\rm GL_N }$. 
This is precisely what we do below.\footnote{We could  avoid the   Grassmannian duality from the beginning, but this would result in longer formulas.}   \vskip 2mm

 3. The calculation of the first residue can be performed   using   simpler map (\ref{56}). This is the first thing we will do.   Construction (a)-(b)    
leads to the same result, as one  easily checks.  \vskip 2mm

4. To check that the cocycle $\widetilde {\rm C}^{\cal M}_{4, {n} , \circ }$  extends to codimension one strata it is sufficient to check the following. The cocycle has components 
on $\G^8, \G^7, \G^6, \G^5$. One needs just to check that their residues are lifted from $\G^7, \G^6, \G^5, \G^4$ respectively via the standard simplicial map 
\be \la{57}
\G^m \lra \G^{m-1}, ~~~~(g_1, \ldots , g_m) \lms \sum(-1)^i(g_1, \ldots , \widehat g_i, \ldots , g_m).
\ee
We use the fact that the residue map is a morphism of complexes. The residue of the component on $\G^8$ is zero by definition. 
 We will  prove that the residue 
of the component on $\G^7$ is zero. We  will calculate the residue on $\G^6$, and present it as the pull back of an element with values in $\B_2(\F_x) \otimes \F_x^\times$ from $\G^5$ using the maps (\ref{57}).  
We show that the residue of the latter is zero. 

The residue on the component $\G^4$ is not zero, and its successive residues   play the central role in proving that our  cocycle indeed represents the Chern class. Let us start from there.

\paragraph{1. The residue   ${\rm res}\circ {\bf r^*_5(1)}$.} All  work for the ``Milnor'' column 
  in any weight $n$, which is in our case is the right column of the complex $\Gamma(X;4)$ in   (\ref{MCDIm}),
 has been done in \cite[Section 4.4]{Gon93}, in the setting of a cocycle for the  Chern class  $c_n \in H^{2n}({\rm BGL}_{{n} \bullet}, \underline{K}^M_n)$.  
In fact the cocycles   there   provides on the nose the  cocycles with  values in $\Lambda^*\F_x^\times$.  See also \cite[Section 5.2]{Gon94}. 
 
 Let us recall schematically how it was done. 
The  divisor  of the $\G-$invariant element  in $\Lambda^4\Q(\G^5)^\times$  given by the map ${\bf r^*_5(1)}$  is a normal crossing divisor. 
 Thus    ${\rm H}  \circ {\bf r^*_5(1)}=0$, and   ${\rm res}\circ {\bf r^*_5(1)}$ is  calculated via formula (\ref{resncd}). 
The residue  ${\rm res}\circ {\bf r^*_5(1)}$   is on the nose   the alternating sum of the pull backs,  by the five maps $\G^5 \to \G^4$, see (\ref{57}), of the $\G-$invariant 
element   $X \in \Lambda^3\Q(Y_x)^\times$ on a divisor  $Y_x\subset \G^4$. 
The divisor of $X$ is a normal crossing divisor, and ${\rm res}(X)$  is on the nose given by 
 the alternating sum of the pull backs,  by the maps $\G^4 \to \G^3$, of the  $\G-$invariant element   $Y \in \Lambda^2\Q(Y_y)^\times$ on a codimension $2$ subvariety $Y_y \subset \G^3$. 
Finally, ${\rm res}(Y)$  is provided by  the $\G-$invariant element $Z \in \Q(Y_z)^\times$    at a codimension $3$ subvariety   $Y_z \subset \G^2$, given by an element of $Z'\in \Q(Y_z/\G)^\times$ on the codimension $3$ subvariety  $Y_z/\G\subset \G$. Then ${\rm res}(Z')=0$. So it defines a cocycle presenting a class in    $H^7(\G(\C), \Z(4))=\Z$.   The key point is that  its cohomology class  is the 
generator. This  proves that   
 the whole cocycle  
    gives the Beilinson regulator. 

\paragraph{2. The residue  ${\rm res}\circ  {\bf r^*_6(2)}$.}  
Recall the formula (\ref{56}) for the map ${\bf r_6^*(2)} $:
\be \la{56a}
\begin{split}
&{\bf r_6^*(2)} \colon (1,2,3,4,5,6) \lms \\
&\alt_6\Biggl(\Bigl(  \frac{1}{2}\cdot\{{\it r}(2,4,5, 6)\}_2
-  \{{\it r}(1,2,4,5)\}_2 \Bigr)\otimes |12| \wedge |23| +
 \{{\it r}(1,3,5,6)\}_2 \otimes |12|\wedge |34|\Biggr),\\
\end{split}
\ee



Using the $5-$term relation for $(1,2,4,5,6)$,   write the middle term in (\ref{56a}) as a multiple of 
$$
\alt_6 \Bigl(  
  \{{\it r}(1,4,5, 6)\}_2 - {\it r}(2,4,5, 6)\}_2 \Bigr)\otimes |12| \wedge |23|.
  $$
 Let us calculate   the residue of the element  ${\bf r^*_6(2)}$.  
The divisor $|12|=0$ carries  a   rational function ${\rm L}_{2/1}$ given by the ratio of collinear vectors $l_2$ and $l_1$, that is 
$l_2 = {\rm L}_{2/1}l_1$.
\bl
The  residue of the element ${\bf r^*_6(2)}$ in (\ref{56}) at the divisor $|12|=0$ is equal to 
\be \la{HALLOW1vvx}
\begin{split}
&{\rm Res}_{(12)}{\bf r_6^*(2)}(1,2,3,4,5,6)= 
  - \alt_{\{3,4,5,6\}}\Bigl(\{{\it r}(2,3,4,5)\}_2 \otimes   L_{2/1}\Bigr).\\
  \end{split}
\ee 
\el

\begin{proof}     Note that the permutation $(13)$ takes $|12|\wedge |23|$ to $|23|\wedge |12|$. Therefore, since  $\{{\it r}(1,2,4,5)\}_2=0$ at $|12|=0$,  we get:
\[
\begin{split}
&{\rm Res}_{(12)}{\bf r^*_6(2)}(1,2,3,4,5,6)=  \\
 &\alt_{\{1,2\}, \{3,4,5,6\}}\Bigl(  \frac{1}{2}\cdot(\{{\it r}(2,4,5, 6)\}_2  +  \{{\it r}(2, 3,4,5)\}_2)\otimes |23| + 
  \{{\it r}(1,3,5,6)\}_2 \otimes   |34|\Bigr).
\end{split}
\]
Since $\{r(1,3,5,6)\}_2 = \{r(2,3,5,6)\}_2$ on the divisor $|12|=0$, the third summand vanishes after the alternation 
by $\{1,2\}$.  Since $L_{2/1} = |23|/|13|$,  after the alternation $\{1,2\}$, and using the odd permutation $(6,5,4,3)$, we get
(\ref{HALLOW1vvx}).\end{proof}

\paragraph{3. The residue  ${\rm res}\circ {\bf r^*_7(3)}$.}    Recall that the residue map  on $ \Lambda^2 {\rm B}_2$ is zero anyway. 

Denote by ${\bf r^*_7(3)}_{3,1}$ the  ${\rm B}_3 \otimes \F^\times$-component  
 of the map ${\bf r^*_7(3)}$. It is given by  formula (\ref{56}):
   \be \la{56aa}
\begin{split}
&{\bf r_7^*(3)}_{3,1} \colon (1,2,3,4,5,6,7) \lms  
  \alt_{7}\Bigl(\Bigl(\frac{1}{108}\cdot(1,2,3,4,5,6)_3  - 
  \{r(6|2, 4,3,5)\}_3\Bigr) \otimes   |123| \Bigr).\\ 
\end{split}
\ee 
So the residue    
  on the divisor $|123|=0$ is  
\be \la{Map44vv}
\begin{split}
&{\rm Res}_{(123)}{\bf r^*_7(3)}_{3,1}(1,2,3,4,5,6,7) = \\
& \alt_{\{1,2,3\}, \{4,5,6,7\}} \Bigl(\frac{1}{108}\cdot(1,2,3,4,5,6)_3  -
  \{{\it r}(6|2, 4,3,5)\}_3\Bigr)  .
\end{split}
\ee

\bl \la{902} Given three distinct points $l_1, l_2, l_3$ on a line and three other generic points $m_1, m_2, m_3$ in ${\Bbb P}^2$, one has
\be \la{900}
\begin{split}
&\frac{1}{3}\cdot (l_1, l_2, l_3, m_1, m_2, m_3)_3=   {\rm Alt}_{\{l_1, l_2, l_3\}, \{m_1, m_2, m_3\}} \{r(l_1, m_1, l_2, m_2)\}_3.\\
\end{split}
\ee
\el

\begin{proof}
We compare  our element $(l_1, l_2, l_3, m_1, m_2, m_3)_3$ normalized by (\ref{TRIPR}), 
with  the element $r_3(l_1, l_2, l_3, m_1, m_2, m_3)$ from  \cite{Gon00}, normalized there  right before Section 3.3:
\be \la{901}
\frac{1}{3} \cdot(l_1, l_2, l_3, m_1, m_2, m_3)_3 =  r_3(l_1, l_2, l_3, m_1, m_2, m_3).
\ee
The factor $1/3$ comes   from the  following identity, which follows from (and in fact is equivalent to) the dual five-term relation, which is deduced from    the standard five-term relation: 
 $$
 \frac{1}{3} \cdot   \alt_6 \left  (\{{\it r}(1|2,3,4,5)\}_2\otimes |345|\right ) = -\frac{1}{2}\cdot {\rm Alt}_6 \left (\{{\it r}(1|2,3,4,5)\}_2 \otimes |123| \right ). 
  $$
  Indeed,  the expression on the right was used in the formula normalizing $r_3$ just before Section 3.3 in \cite{Gon00}, while 
  the expression on the left is $1/3$ times the one which appears in (\ref{TRIPR}). 
  
Now identity (\ref{900}) follows from \cite[Lemma 3.7]{Gon00} and (\ref{901}).
 \end{proof}

\bc The residue (\ref{Map44vv}) is equal to zero.  \ec

\begin{proof} Follows immediately from Lemma \ref{902}. 
  \end{proof}
  
  \paragraph{4. Extending the cocycle to $\B {\rm GL_n}$.} Recall   $\G = {\rm GL_n}$.  All we need to check 
is that each component of the residue on $\G^m$ is the pull back of an element on a divisor in  $\G^{m-1}$ via  map (\ref{57}). 
The  right-hand side of (\ref{HALLOW1vvx}) is lifted from the configuration space of $5$   vectors.\footnote{Note that the right-hand side of (\ref{HALLOW1vvx}) is antisymmetric in $(3,4,5,6)$ but not in $(1,2)$. This is fine, since taking the residue we broke the $(1,2)$ symmetry.} 
Therefore it provides an extension of cocycle (\ref{56}), originally defined at the generic part of $\B\G$, to divisors in $\G^5$,  
with values in $\B_2 \otimes \F_x^\times$. 
 As   discussed above, 
 a similar extension to divisors in $\G^4$, with values in $\Lambda^3(\F_x^\times)_{\Q}$ is provided by \cite{Gon93}.

This way, we  extend  cocycle (\ref{56}) to the second from the top line in the  complex (\ref{MCDIm}). Let us write it without using the Grassmannian duality. 
Then its component on the divisor $|123|=0$ in the configuration space $(l_1, \ldots, l_5)$ of five vectors in $V_3$ is given by 
\[
\begin{split}
&{\bf r_6^{**}(3)}(1,2,3,4,5)= 
  c\cdot\alt_{\{123\}}\Bigl(\{r(5|1,2,3,4)\}_2 \otimes   \frac{|234|}{|235|}\Bigr).\\
  \end{split}
\]
 
The only non-trivial residues are the ones on the divisors $|ijk|=0$, permuted by the action of the symmetric group $S_5$.  The residue on the divisor $|234|=0$ is given by 
\[
\begin{split}
&{\rm Res}_{|234|=0}{\bf r_6^{**}(3)}(1,2,3,4,5)= 
  c'\cdot\{r(1,2,3,4)\}_2.\\
  \end{split}
\]
It is evidently lifted from the configuration space of $4$ vectors in $V_3$.

\subsection{End of the proof of Theorem \ref{TH1.7} and other main theorems} \la{Sec8.3}

To prove our main results  in their strongest form, we  have to define a map from algebraic $K$-groups of $\F$ to the cohomology of the 
weight 4 polylogarithmic motivic complexes ${\cal B}^\bullet(\F;4)$. To prove that we can land indeed in the complex ${\cal B}^\bullet(\ast;4)$ of the field $\F$ itself rather than its finite extension, 
we need to use a different way to write  the map ${\bf r}^*_7(3)$, reflecting the construction of   the Grassmannian tetralogarithm using the scissor congruence groups ${\rm A}_3(\F)$. 

\paragraph{1. Relation with the Grassmannian tetralogarithm.} 
 Let us show   that the map (\ref{COMPO}) coincides with a similar map defined in \cite[Section 5.2]{Gon00}.  
The     latter  map is more transparent. However  it is unclear how to show  that it leads to a non-trivial cohomology class 
of ${\rm GL}_{n}(\C)$.

\vskip 3mm
For an infinite field $\F$ there are scissor congruence groups $\A_n(\F)$ reflecting the properties of Aomoto polylogarithms, defined in 
\cite{BMS87}, \cite{BGSV90}. We use slightly modified groups, 
adding  the dual additivity relation. For the convenience of the reader, we recall the definition  from \cite{Gon09}. 

 \vskip 2mm
The abelian group $\A_n(\F)$ is generated by the elements 
$$
\langle L;M\rangle _{\A_n}:= \langle l_0,\ldots,l_n;m_0,\ldots,m_n\rangle _{\A_n}
$$ corresponding to generic  
configurations of $2(n+1)$ points
$(l_0,\ldots,l_n;m_0,\ldots,m_n)$ in ${\Bbb P}^n(\F)$. Here 
  $L= (l_0,\ldots,l_n)$ and $M = (m_0,\ldots,m_n)$. 
The relations are 
the following:

{\it Degeneration}.  $\langle L;M\rangle _{\A_n} = 0$ if $(l_0,\ldots,l_n)$ or $(m_0,\ldots,m_n)$ belong to a
hyperplane.

{\it Projective invariance}. $\langle gL;gM\rangle_{\A_n}  = \langle L;M\rangle_{\A_n}$ for any $g \in {\rm PGL}_{n+1}(\F)$.

{\it Antisymmetry}. $\langle \sigma L;M\rangle _{\A_n} = \langle L;\sigma M\rangle _{\A_n} = (-1)^{|\sigma|} \langle L;M\rangle _{\A_n}$
for any $\sigma \in S_{n+1}$.

{\it Additivity and dual additivity}.  For any generic configuration $(l_0,\ldots,l_{n+1}; m_0, \ldots, m_n)$ in ${\Bbb P}^n(\F)$, and respectively in ${\Bbb P}^{n+1}(\F)$, one has 
\be \nonumber  
\bs
&\sum_{i=0}^{n+1}(-1)^i \langle l_0,\ldots,\widehat l_i,\ldots,l_{n+1};m_0,\ldots,m_n\rangle _{\A_n} =0,\\
&\sum_{i=0}^{n+1}(-1)^i \langle l_i\vert l_0,\ldots,\widehat l_i,\ldots,l_{n+1};m_0,\ldots,m_n\rangle _{\A_n} =0.\\
\end{split}
\ee
There is a similar  
condition  for $M$. 

The direct sum $\bigoplus_{n=0}^\infty\A_n(\F)$ has a graded coalgebra structure with the cobracket $\delta$.

 \vskip 3mm

Below we use notation $(l_1, \ldots , l_8)$ for a configuration of vectors in generic position in $V_4$. Using the group $\A_3(\F)$, we define a map \cite[Section 5.2]{Gon00}
\[
\bs
&\widetilde \Delta: \G_4(\F) \lra (\A_3(\F)\otimes \F^\times) \oplus (\F^\times \otimes \A_3(\F)) \oplus \Lambda^2\B_2(\F)\\
&( 1, \ldots ,  8) \lms {\rm Alt_8}\Bigl(\langle  1, \ldots,  4 ; 5, \ldots,  8 \rangle_{{\rm A}_3} \otimes | 5678| - | 1   2   3  4| \otimes \langle  1, \ldots,  4 ;  5, \ldots,  8 \rangle_{{\rm A}_3}   \Bigr)\\
&  
 -\frac{3}{28}\cdot \{r( 3,4| 2, 5, 6, 7)\}_2\wedge  \{r( 6,7| 1, 3, 4, 5)\}_2. \\
\end{split}
\]
Next, there is a homomorphism \cite[Section 3.3]{Gon00}:\footnote{Note that our $a_3'$ is equal to $1/3 a_3'$ from  loc. cit.. We found the former a better normalization.}
\[
\bs
& a_3:  \A_3(\F) \lra \B_3(\F), \\
&\langle l_0, \ldots, l_3 ; l_4, \ldots, l_7 \rangle_{{\rm A}_3} \lms \frac{1}{2}a_3'-\frac{1}{3}a''_3,\\
\end{split}
\]
where the map $a_3'$ and $a_3''$ are defined as follows:\footnote{Recall that our normalization of the map $(*, \ldots , *)_3$ differs from the one \cite[Section 5.2]{Gon00}, see (\ref{901}).} 
\be \nonumber  
\bs
&a_3'(l_0, \ldots, l_3 ; m_0, \ldots, m_3):= \sum^3_{i,j=0} (-1)^{i+j}  (l_i ~|~ l_0, \ldots , \widehat l_i, \ldots , l_3; m_0, \ldots , \widehat m_j, \ldots , m_3)_3,\\
&\mu_3(x_1, x_2, x_3; y_1, y_2, y_3):= \frac{1}{2} \cdot {\alt}_{\{x_1, x_2, x_3\}, \{y_1, y_2, y_3\}}\{r(x_1, y_2, x_2, y_1)\}_3,\\
&a_3''(l_0, \ldots, l_3 ; m_0, \ldots, m_3):= \sum^3_{i,j=0} (-1)^{i+j} \mu_3(l_i, m_j~|~ l_0, \ldots , \widehat l_i, \ldots , l_3; m_0, \ldots , \widehat m_j, \ldots , m_3).\\
\end{split}
\ee

There is a similar map $a_2: \A_2(\F) \lra \B_2(\F)$ \cite{BGSV90}. 
The key property of the map $a_3$ is the following commutative diagram  \cite[Theorem 3.2]{Gon00}:

\[
\begin{gathered}
    \xymatrix{
        \A_3(\F)  \ar[r]^{\delta~~~~~~~~~~~~~~} \ar[d]_{a_3} &  (\A_2(\F) \otimes \F^\times) \oplus  (\F^\times \otimes \A_2(\F))\ar[d]^{a_2\wedge a_1} \\
          {\B}_3(\F)   \ar[r]^{\delta~~}       &{\rm B}_2(\F)  \wedge  {\rm F}^\times}
 \end{gathered}
\]
 
 The composition $\Delta:= (a_3 \wedge {\rm Id} + {\rm Id}  \wedge a_3 + {\rm Id}  \wedge {\rm Id} )\circ \widetilde \Delta$ is a map 
\be \la{Map3}
\bs
& \Delta: \G_4(\F) \lra \B_3(\F)\otimes \F^\times  \oplus \Lambda^2 \B_2(\F).\\
\end{split}
\ee

\bp The map $\Delta$ in (\ref{Map3}), followed by the Grassmannian duality, coincides with the map ${\bf r^*_7}(3)$ in (\ref{56}).
\ep

\begin{proof}  Follows by a careful comparison of the definition and  normalizations involved. \end{proof}

\paragraph{Remark.} The map $\widetilde \Delta$  satisfies the key property that its composition ${\rm res}\circ \widetilde \Delta$ with the residue map  is zero. 
Indeed, the residue can be non-zero only at the divisors $|ijkl|=0$. Due to the antisymmetry it is sufficient to calculate the residue at the divisor $|1234|=0$, which is given by 
$$
  -   a_3\Bigl(   \langle  1, \ldots,  4 ;  5, \ldots,  8 \rangle_{{\rm A}_3} + \langle  5, \ldots,  8  ; 1, \ldots,  4 \rangle_{{\rm A}_3}\Bigr)_{|1234|=0}. 
$$
Each of the two summands here is zero due to the degeneration relation in the group ${\rm A}_3$.

In particular, the claim that the ratio of the first two coefficients in ${\bf r}^*_7(3)$ in (\ref{56}) is$-1/108$ is nailed by the fact that the residue 
${\rm res}\circ {\bf r}^*_7(3)$ is zero, as we check   in Section \ref{SEC9.2} above.  

The only place where the value of the coefficient $3/28$ is important is the commutativity of the middle square at the bottom of diagram (\ref{MCDI1zz}). 

\paragraph{2. End of the proof of Theorem \ref{TH1.7}.} Let us identify the space ${\rm Conf}_8(V_4)$ with the Grassmannian ${\rm Gr}(4;8)$ of $4$-dimensional subspaces  in a coordinate $8$-dimensional vector space, which are in generic position to the coordinate hyperplanes. 
We identify the latter with the quotient  by the action of ${\rm GL}_4$ of the space ${\rm M}^*_{4,8}$ of non-degenerate $4 \times 8$ matrices $(a_{ij})$, where $1 \leq i \leq 4$ and $1 \leq j \leq 8$. 
Note that ${\rm M}^*_{4,8}$ is a subspace of the space ${\rm M}_{4,8}$ of all $4 \times 8$ matrices $(a_{ij})$. 
Consider a path $p$ in  ${\rm M}_{4,8}$ starting at a  point $(a_{ij}) \in {\rm M}^*_{4,8}$ and given by the composition 
$ p = {\rm I}_8 \circ \ldots \circ  {\rm I}_1$ of the following eight segments. The segment 
${\rm I}_k$ is given parametrically by $a_{1, k-1}=t$,  where $ 0 \leq t \leq 1$, 
 while the rest of the coordinates stay the same, and starts at a point with $a_{11}=0, \ldots,  a_{1, k-1}=0$.

Recall  \cite{Gon09} that 
 the symbol of the Grassmannian tetralogarithm is a rational multiple of 
 $$
\alt_8\Bigl( |1234| \otimes |2345| \otimes |3456| \otimes |4567|\Bigr).
 $$
Restricting it to the path ${\rm I}_k$ we get a sum of terms where  each factor in $\otimes^4{\Bbb F}^\times$ is a {\it linear function in the path variable $t$}. 
This implies that we can express  the Grassmannian tetralogarithm  as a sum of weight four framed sub-quotients of $\pi_1^{\cal M}({\Bbb P}^1 - \{b_0, \ldots, b_N\}, v_a)$ where 
 the points $b_i$ and $a$ belong to the number subfield $\F \subset \C$. 
 
 Therefore thanks to Corollary \ref{WTRD43} the Grassmannian tetralogarithm can be expressed 
 as a linear combination of the motivic correlators $\{z\}_4$ and $\{x,y\}_{3,1}$ with the arguments in $\F$. 
 
So we get   for each ${n} $ a map from the complex of decorated flags for ${\rm GL}_{n} $ to the weight four part of the standard cochain complex of the Lie coalgebra ${\Bbb L}_{\leq 4}(\F)$. 

Corollary \ref{COR112} of Theorem \ref{THRN} implies that  for each $N$ there exists a map 
from the complex of decorated flags for ${\rm GL_N}$ to the weight four polylogarithmic complex   ${\cal B}^\bullet(\F;4)$. 
One can define it   by   subtracting from the original map  the cobracket of the following element:   
$$
  {\rm Alt}_8\{r( 3,4| 2, 5, 6, 7), r( 6,7| 1, 3, 4, 5)\}_{3,1}. 
$$
Then the  obtained map lands  in ${\Bbb L}_4(\F) \to \B_3(\F) \otimes \F^\times\to  \B_2(\F) \otimes \Lambda^2 \F^\times \to  \Lambda^4 \F^\times$. 
So by Theorem \ref{THRN}c) it   lands in the subcomplex ${\rm B}_4(\F) \to \B_3(\F) \otimes \F^\times\to  \B_2(\F) \otimes \Lambda^2 \F^\times \to  \Lambda^4 \F^\times$. 

Finally, this will not affect residue calculations from Section \ref{SEC9.2} due to the following. 
 
 \bl \la{L9.5} One has 
 \be \nonumber  {\rm res} (\delta \{x, y\}_{3,1})=0.
  \ee
  \el
  
  \begin{proof} Due to the antisymmetry of $\{x, y\}_{3,1}$ it is sufficient to check the residues at $x=0,1, \infty$.   Then we have using (\ref{26}) and (\ref{144a}):
  \be
 \begin{split}
 &{\rm res}_{x=0}\circ \delta \left \{x,y \right\}_{3,1}= 
~ \left \{0,y\right\}_{2,1} +
\left \{y\right\}_3 \stackrel{(\ref{144a})}{=}0.\\
& {\rm res}_{x=1}\circ \delta \left \{x,y \right\}_{3,1}= 
~0.\\  
&{\rm res}_{x=\infty}\circ \delta \left \{x,y \right\}_{3,1}=0.\\
\end{split}
 \ee  
Indeed, we have
 \be \la{144ab}
\begin{split}
&\left \{0,y\right\}_{2,1}:= {\rm Sp}_{x= 0}\left \{x,y \right\}_{2,1}=
 \left\{1-y^{-1}\right \}_3  +\left\{ {1-y} \right \}_3   - \{1\}_3 = -\{y\}_3.\\
&\delta \left \{1,y \right\}_{3,1}:= {\rm Sp}_{x= 1}\delta \left \{x,y \right\}_{3,1}= 
\left \{1/y\right\}_3 
-\left \{y\right\}_3=0.\\
&\left \{\infty,y\right\}_{2,1}:= {\rm Sp}_{x= \infty}\left \{x,y \right\}_{2,1}=  0.\\
\end{split}
\ee
    \end{proof}

 Theorem \ref{TH1.7}    follows   from this.

 \paragraph{3. Proof of Theorem \ref{ZCZ4t2}.} Follows essentially immediately from what we have done. Indeed, we constructed a map from 
 the decorated $(n-3)-$flag complex to the weight four polylogarithmic motivic complex, and proved using  
  Theorem \ref{TH1.7} that it gives rise to  
 Beilinson's regulator map.  
 
 The only part of Theorem \ref{ZCZ4t2} which still needs a  proof is the claim that the restrictions   to ${\cal F}_3^{\rm rk}K_{8-i}(\F)_\Q$ of the maps defined by the part i) are zero. 
 The argument is similar to the proof of  \cite[Lemma 3.18]{Gon94a}. Choose  linearly independent vectors $v_1, ..., v_{n-3}$ in an $n-$dimensional vector space $V$ of the field $\F$, 
 and a  complimentary subspace $V'$. Recall  from (\ref{GDF}) the complex $C_\bullet({\cal A}^{(p)}_n) $ of generic configurations of decorated $p-$flags in $V$.  
 Then the resolution of the trivial ${\rm GL}_n(\F)-$module $\Z$ given by the complex of generic decorated 
 $(n-3)-$flags $C_\bullet({\cal A}^{(n-3)}_n) $ has a ${\rm GL}_3(\F)-$invariant section $s$, where  ${\rm GL}_3(\F)$ is identified with ${\rm Aut}(V')$ by choosing a basis in $V'$:
  \be \la{MOCsec}
\begin{gathered}
    \xymatrix{
    &   &    & & 
       \Z\ar[d]^{s}  \\
       &    \ldots \ar[r]^{\partial \ \ \ \ \ \ }  & C_3({\cal A}^{(n-3)}_n)  \ar[r]^{\partial}&  C_2({\cal A}^{(n-3)}_n)  \ar[r]^{\partial}& C_1({\cal A}^{(n-3)}_n) \\ }
                                             \end{gathered}
\ee 
It is given by $s: m \lms m\cdot [ v_1, ...., v_{n-3}]$, where 
 $[ v_1, ...., v_{n-3}]$  is the decorated $(n-3)-$flag given  by $\F_1 \subset \F_2 \subset ... \subset \F_{n-3}$ where  
 $\F_k := \langle v_1, ..., v_k\rangle$,   decorated by the vector $v_k$.   Given a projective resolution ${\cal P}_\bullet$ of the trivial ${\rm GL}_n(\F)-$module $\Z$, 
 there is a unique up to homotopy map of complexes of ${\rm GL}_n(\F)-$modules $\gamma: {\cal P}_\bullet \lra C_\bullet({\cal A}^{(n-3)}_n)$.  Restricting to ${\rm GL}_3(\F)$,  
 we get a projective resolution ${{\cal P}_\bullet}_{|{\rm GL}_3(\F)}$ of the trivial  ${\rm GL}_3(\F)-$module $\Z$. There is a map $\alpha: {{\cal P}_\bullet}_{|{\rm GL}_3(\F)} \lra \Z$ lifting the identity map $\Z \to \Z$.  Then   
  $s\circ \alpha $ is homotopic to the map $\gamma_{|{\rm GL}_3(\F)}$.  The claim follows.

  \paragraph{4. Proof of Theorem \ref{ZCZ4t}.} Follows   from Theorem \ref{ZCZ4t2} and Borel's theorem \cite{Bor77}.

  \paragraph{5. Remark.} Proposition 5.6 in \cite{Gon00}   provides  a measurable function on the configuration space of $8$ points in ${\Bbb C}\P^3$ satisfying the nine-term relation and the dual nine-term relation, and hence by the construction in \cite{Gon93} 
leading to a measurable cohomology class of ${\rm GL_n}(\C)$ for any ${n}  \geq 4$. 
Theorem \ref{TH1.7} implies that it is a $\Q^\times-$multiple of the   Borel class  \cite{Bor77}. Indeed, Theorem \ref{TH1.7}  provides a construction of the weight four  Beilinson regulator class. It restricts to a cohomology class in the measurable cohomology  ${\rm H}_{\rm meas}^7({\rm GL}_n(\C), \R)$, $n \geq 4$, of the Lie group ${\rm GL}_n(\C)$, called the Beilinson class. By  \cite[Corollary A5.2]{Bei84} it  is a $\Q^\times-$multiple of the Borel class. 
  
\section{Proof of  Theorem \ref{MapBtoP}} \la{Sec4}

In Section \ref{Sec4} we prove Theorem $\ref{MapBtoP}.$  One can understand the proof better by 
using   the representation theory of the symmetric group.  
We  provide   representation-theoretic explanations along with   computations.  

We introduce symbolic versions of  higher Bloch groups ${\rm B}_n$,   $n=1,2,3$. We denote them by ${\bf b}_n(q,m)$, where $q,m \in   \mathbb{N}$.     They are finite dimensional $\mathbb{Q}$-vector spaces  described   by  generators and relations. This groups describe some relations between polylogarithms {\it modulo constants}.

For example, ${\bf b}_2(2,6)$ is a $10-$dimensional  space generated by symbols $\{i_1,i_2,i_3,i_4\}_2$ satisfying two types of relations: the antisymmetry and the five-term relation. 
The definition reflects  properties of the  dilogarithm   evaluated on  cross-ratios  $r(v_{i_1}, v_{i_2}, v_{i_3}, v_{i_4})$ obtained by picking four     
 out of    six generic vectors   in  a $2$-dimensional 
  space.

Next, for every injection  of finite sets  ${\bf m_1} \hookrightarrow {\bf m_2}$ one has  a morphism 
\be \nonumber  {\bf b}_n(q, |{\bf m_1}|)  \lra {\bf b}_n(q, |{\bf m_2}|). 
\ee
Under the evaluation map it corresponds to   picking $|{\bf m_1}|$ vectors out of $|{\bf m_2}|.$ It means that ${\bf b}_n(q,\bullet)$ is a covariant 
functor from the category ${\bf FI}$ of finite sets with the morphisms given by 
injective maps to the category of finite dimensional $\mathbb{Q}$-vector spaces. In particular, ${\bf b}_n(q, |{\bf m}|)$ is a representation of the symmetric group ${\Bbb S}_{|{\bf m}|}$. Such objects are called FI-modules.

We do not use any results of the theory of FI-modules in the proof of Theorem \ref{MapBtoP}. Nevertheless, the main property of finitely generated FI-modules -  the representation stability -  partially explains why the alternating components of  representations  appearing in the proof  are very small. 

\subsection{Symbolic versions of  Bloch groups} 

\paragraph{FI-modules and representation stability \cite{CEF15}.} 

The category  ${\bf FI}$ of finite sets with injections   is equivalent to the category with the objects ${\bf m}:=\{1,2,\ldots,m\}$, and    with   morphisms from ${\bf m_1}$ to ${\bf m_2}$ given by   injections ${\bf m_1} \hookrightarrow {\bf m_2}.$ 
 An {\it ${\rm FI}$-module  $\B$} is a covariant functor from category ${\bf FI}$ to the category of finite-dimensional $\Q$-vector spaces. 
FI-modules form an abelian category.  Tensor products and Schur functors for FI-modules are defined termwise.

Given an  ${\rm FI}$-module ${\rm B}$, for every $m\in \mathbb{N}$ we denote by $\B_m$   the vector space $\B({\bf m}).$ It is a representation of the symmetric group ${\rm S}_m$. 
The latter   are in bijection with partitions $\lambda_1\geq \lambda_2\geq \ldots \geq \lambda_l$ of   $m.$ We denote these
 representations by 
\be \nonumber  V(\lambda)=[\lambda_1,\lambda_2,\ldots,\lambda_l].
\ee
So $[m]$ is the trivial representation, and $[1,1,\ldots,1]$ the alternating one. 
 
An FI-module $\B$ is called {\it finitely generated} if there exists a finite set of vectors $S$ in $\bigsqcup_{m \geq 1} \B_m$ such 
that $\B$ has no proper FI-submodules containing all vectors in $S.$
 
Finitely generated FI-modules are  preserved by    tensor products and Schur functors. One of key results  is that every  submodule of a finitely generated FI-module is finitely generated. 

Many applications of the theory of FI-modules come from the phenomena of representation stability of a sequence of representations. To state it, we  introduce an  axillary  notation.
Given a partition $\lambda=(\lambda_1,\lambda_2,\ldots,\lambda_l)$, denote by $V(\lambda)_m$ the irreducible ${\rm S}_m$-module   given by the partition
\be \nonumber 
\lambda[m]=(m-|\lambda|,\lambda_1,\lambda_2,\ldots,\lambda_l).
\ee

\begin{theorem} 
\cite{CEF15} Let $\B$ be a finitely generated FI-module. Then for sufficiently large $m$ we have 
\be \nonumber  \B_m \cong \bigoplus c_{\lambda} V(\lambda)_m,
\ee
where the multiplicities $c_{\lambda}$ do  not dependf on $m.$
\end{theorem}

\paragraph{Symbolic higher Bloch groups.} Let us introduce FI-modules ${\bf b}_n(q,\bullet)$ of {\it weights}  $n=1,2,3$.

\paragraph{The weight one.}

\begin{definition}
The $\Q$-vctor space ${\bf b}_1(q,m)$ is generated by the symbols $|i_1 i_2\ldots i_q|,$ where the indices $1 \leq i_k \leq m$ are distinct, and satisfy the symmetry  relation:  
\be \nonumber 
|i_1 i_2 \ldots i_q|=|i_{\sigma(1)} i_{\sigma(2)}\ldots i_{\sigma(q)}| ~~~~\forall \sigma \in {\rm S}_q.
\ee
\end{definition}

For a given  configuration of $m$ vectors $v_1,\ldots,v_m$ in general position in a $q-$dimensional vector space  over $\F$ with a chosen volume element $\omega_q$ there exists a  map of $\mathbb{Q}-$vector spaces
\be \nonumber 
ev \colon {\bf b}_1(q,m) \lra \F^{\times},
\ee
sending symbol $|i_1 i_2 \ldots i_q|$ to the determinant ${\rm det}(v_{i_1},\ldots,v_{i_q}):= \langle \omega_m, v_{i_1} \wedge\ldots \wedge v_{i_q}\rangle$ of the corresponding vectors. This map is well defined  since $-1$ is vanishing in    $\F^{\times}$. 

Vector spaces ${\bf b}_1(q,\bullet)$   form a finitely generated FI-module. The following lemma gives an example of representation stability phenomena. We omit the proof  since it is not used later. 
\bl
As representation of  ${\rm S}_m,$ 
\be \nonumber 
{\bf b}_1(q,m)={\rm Ind}_{{\rm S}_q \times {\rm S}_{m-q}}^{{\rm S}_m} \left ( [q]\times [m-q] \right ).
\ee
In particular,  
\be \nonumber 
\begin{split}
&{\bf b}_1(2,m)=[m-2,2]\oplus[m-1,1] \oplus[m], ~~~~\forall m \geq 5,\\
&
{\bf b}_1(3,m)=[m-3,3]\oplus[m-2,2]\oplus[m-1,1] \oplus[m], ~~~~\forall m \geq 6.\\
\end{split}
\ee
\el

\paragraph{The weight two.}

\begin{definition}
The vector space ${\bf b}_2(q,m)$ is generated by the symbols 
\be \nonumber 
\{i_1, \ldots, i_{q-2}|i_{q-1},i_{q},i_{q+1},i_{q+2}\}_2, ~~~~ 1 \leq i_k \leq m, ~~i_a\not = i_b.
\ee
 These symbols  are symmetric in the first $q-2$ variables, and antisymmetric in the last four. They satisfy the following relations:

{\bf the 5-term  relation:} 
\be \nonumber 
\{\bullet| i_1, i_2, i_3, i_4\}_2- \{\bullet| i_1, i_2, i_3, i_5\}_2+ \{\bullet| i_1, i_2, i_4, i_5\}_2- \{\bullet| i_1, i_3, i_4, i_5\}_2+ \{\bullet|  i_2, i_3, i_4, i_5\}_2=0,
\ee

{\bf the dual 5-term  relation:} 
\be \nonumber 
\{\bullet i_1| i_2, i_3, i_4, i_5\}_2- \{\bullet i_2 |i_1, i_3, i_4, i_5 \}_2+ \{\bullet i_3 |i_1, i_2, i_4, i_5 \}_2- \{\bullet i_4|  i_1, i_2, i_3, i_5\}_2+ \{\bullet i_5| i_1, i_2, i_3, i_4\}_2=0,
\ee
where $\bullet$ stands for a collection of $q-2$ indices distinct from $i_1,\dots, i_5.$
\end{definition}

For a given configuration of $m$ vectors $v_1,\ldots,v_m$ in generic position in a  $q$-dimensional vector space $V$ over $\F$ there is a  map of $\mathbb{Q}-$vector spaces
\be \nonumber 
{ ev}: {\bf b}_2(q,m) \lra {\rm B}_2(\F).
\ee
It is defined as follows. Let $W$ be the subspace of $V$ generated by the vectors $v_{i_{1}},\ldots,v_{i_{q-2}}.$
Then 
$$
{ ev}: \{i_1, \ldots, i_{q-2}|i_{q-1},\ldots, i_{q+2}\}_2 \lms \{r(v_{i_1},  \ldots, v_{i_{q-2}} | v_{i_{q-1}},\ldots, v_{i_{q+2}})\}_2 \in \B_2(\F),
$$ 
 where on the right stands the cross-ratio of the   four points  in the projective line $\mathbb{P}[V/W]$ provided by the projections of the vectors 
$ {v_{i_{q-1}}},  {v_{i_{q}}},  {v_{i_{q+1}}},  {v_{i_{q+2}}}$ to $V/W$. 

There is a unique symbolic analog of the cobracket compatible with the evaluation maps:
\be
\delta \colon {\bf b}_2(q,m) \lra \Lambda^2 [{\bf b}_1(q,m)].
\ee

\begin{example}
Consider the  symbolic version of the weight two polylogarithmic complex:
\be \nonumber {\bf b}_2(2,4) \lra \Lambda^2 {\bf b}_1(2,4).
\ee
Then ${\bf b}_2(4)$ is a sign representation of ${\rm S}_4$, and $\Lambda^2 {\bf b}_1(2,4)$ is computed as follows:
\be \nonumber
 [1,1,1,1]\oplus 2[2,1,1]\oplus [2,2] \oplus 2[3,1].
\ee
Therefore the cobracket map is   unique up to a constant.
\end{example}

\begin{example}
1. The vector space ${\bf b}_2(2,6)$ is a 10-dimensional representation of ${\rm S}_6$:
\be \nonumber 
{\bf b}_2(2,6)=[3,1,1,1].
\ee

2. ${\bf b}_2(3,7)=[3, 2, 1, 1] + [4, 1, 1, 1]$ and has dimension $55.$  
\end{example}

\paragraph{The weight three.}

\begin{definition}
The vector space ${\bf b}_3(3,m)$ is generated by the symbols  of two types:
\be \nonumber  
\begin{split}
&\{i_1|i_{2},i_{3},i_{4},i_{5}\}_3, 
~~~~\{i_1,i_2,i_3,i_4,i_5,i_6\}_3.\\
\end{split}
\ee
 The symbols  of the first type satisfy the following relations:
 \be \nonumber  
 \begin{split}
&\{i_1|i_{2},i_{3},i_{4},i_{5}\}_3=\{i_1|i_{3},i_{4},i_{5},i_{2}\}_3=\{i_1|i_{4},i_{3},i_{2},i_{5}\}_3,\\
&\{i_1|i_{2},i_{3},i_{4},i_{5}\}_3+\{i_1|i_{2},i_{4},i_{5},i_{3}\}_3+\{i_1|i_{2},i_{5},i_{3},i_{4}\}_3=0.\\
\end{split}\ee

The symbols of the second type are anti-symmetric and satisfy the  
{\rm 7-term  relation:}  
\be \nonumber  
\sum_{k=1}^7(-1)^k \{j_{1},  \ldots,  \widehat j_{k}, \ldots,  j_{7}\}_3=0.
\ee
\end{definition}
Notice that since we omit constants, the second equation does not include the term $\zeta(3).$

\begin{example}
The vector space ${\bf b}_3(3,5)$ is a 10-dimensional representation of ${\rm S}_5$:
\be \nonumber  
{\bf b}_3(3,5)=[3,2]+[2,2,1].
\ee
\end{example}
For a given configuration of $m$ vectors $v_1,\ldots,v_m$ in a $3-$dimensional vector space  over $\F$ there exist a natural evaluation map of $\mathbb{Q}$-vector spaces
\be \nonumber  
\begin{split}
&{\bf b}_3(3,m) \lra \B_3(\F),\\
&\{i_1|i_{2},i_{3},i_{4},i_{5}\}_3 \lms \{r(v_{i_1}| v_{i_{2}}, v_{i_{3}}, v_{i_{4}}, v_{i_{5}})\}_3.\\
&\{1,2,3,4,5,6\}_3 \lra \frac{1}{5} {\rm Alt}_6\left\{\frac{|v_{i_1},v_{i_2},v_{i_4}||v_{i_2},v_{i_3},v_{i_5}||v_{i_1},v_{i_3},v_{i_6}|}{|v_{i_1},v_{i_2},v_{i_5}||v_{i_2},v_{i_3},v_{i_6}||v_{i_1},v_{i_3},v_{i_4}|}\right\}_3.\\
\end{split}
\ee

There is a unique symbolic analog of the cobracket compatible with the evaluation maps:
\be \nonumber  
\delta \colon {\bf b}_3(3,m) \lra {\bf b}_2(3,m)\otimes{\bf b}_1(3,m).
\ee

\subsection{Commutativity of the first square}
\begin{theorem} \label{Sq1}
a) The map 
\be \nonumber  
\delta \colon {\bf b}_2(2,6)\otimes \Lambda^2 [{\bf b}_1(2,6)] \lra \Lambda^4 [{\bf b}_1(2,6)]
\ee
induces an isomorphism of alternating components.
The image of the element  
\be \nonumber 
\alt_6\Biggl(\Bigl(  \frac{1}{2}\{2, 4, 5 , 6\}_2
-  \{1,2,4,5\}_2 \Bigr)\otimes |12|\wedge |23|+
 \{1,3,5,6\}_2 \otimes |12|\wedge |34|\Biggr)
\ee
is equal to 
$
-\frac{5}{2}\alt_6(|12|\wedge |13| \wedge |14|\wedge |15|).
$

b)The following diagram is commutative:
\be   \nonumber 
\begin{gathered}
    \xymatrix{
 C_6(2) \ar[r]^{p}   \ar[d]^{{\bf r^*_6(2)}}     &C_5(1)   \ar[d]^{{\bf r^*_5(1)}}   \\ 
   {\rm B}_2 \otimes \Lambda^2{\rm F}^\times \ar[r]^{~~\delta} & \Lambda^4{\rm F}^\times      }
\end{gathered}
 \ee 
\end{theorem}
\begin{proof}
 a) Multiplicities of the alternating components  of the spaces 
\be \nonumber  {\bf b}_2(2,6) \otimes \Lambda^2 [{\bf b}_1(2,6)] ~~
\mbox{\rm and}~~
\Lambda^4 [{\bf b}_1(2,6)]
\ee
can be computed directly: both of them are equal to 3. 

The following  three elements in $\Lambda^4 [{\bf b}_1(2,6)]$ form a basis of its alternating component:
\be \nonumber 
\begin{split}
& a_1=\alt_6(|12|\wedge |13| \wedge |14| \wedge |15|),\\
& a_2=\alt_6(|12|\wedge |23| \wedge |34| \wedge |45|),\\
& a_3=\alt_6(|12|\wedge |23| \wedge |45|\wedge |24| ).
\end{split}
\ee

Next, consider the following  three elements in ${\bf b}_2(2,6) \otimes \Lambda^2 [{\bf b}_1(2,6)]$:
\be \nonumber 
\begin{split}
&b_1= \alt_6 \left( \{2,4,5, 6\}_2 \otimes |12|\wedge |23| \right),\\
&b_2= \alt_6 \left(\{1,2,4,5\}_2 \otimes |12|\wedge |23| \right),\\
&b_3=\alt_6 \left(\{1,3,5,6\}_2 \otimes |12|\wedge |34| \right).
\end{split}
\ee

\bl \la{L4.5} The map $\delta$ can be computed explicitely on these vectors:
\be \nonumber 
\begin{split}
&\delta(b_1)= -3a_1+6a_3,\\
&\delta(b_2)=a_1-2a_2-a_3,\\
&\delta(b_3)=-2a_2-4a_3.
\end{split}
\ee
\el
\begin{proof}
We will compute $\delta(b_1)$, the other computations are similar.
\be \nonumber 
\begin{split}
& \delta(b_1)=\alt_6 \left(\dfrac{|24||56|}{|26||45|} \wedge \dfrac{|25||46|}{|26||45|} \wedge |12|\wedge |23| \right)=\\
& \alt_6 \left( |24|\wedge |25|\wedge |12|\wedge |23| \right)+
\alt_6 \left( |24|\wedge |46|\wedge |12|\wedge |23| \right)+\\
&\alt_6 \left( |56|\wedge |25|\wedge |12|\wedge |23| \right)+
\alt_6 \left( |56|\wedge |46|\wedge |12|\wedge |23| \right)-\\
&-\alt_6 \left( |26|\wedge |25|\wedge |12|\wedge |23| \right)
-\alt_6 \left( |26|\wedge |46|\wedge |12|\wedge |23| \right)\\
&-\alt_6 \left( |45|\wedge |25|\wedge |12|\wedge |23| \right)
-\alt_6 \left( |45|\wedge |46|\wedge |12|\wedge |23| \right)\\
& -\alt_6 \left( |24|\wedge |26|\wedge |12|\wedge |23| \right)
-\alt_6 \left( |24|\wedge |45|\wedge |12|\wedge |23| \right)\\
&-\alt_6 \left( |56|\wedge |26|\wedge |12|\wedge |23| \right)
-\alt_6 \left( |56|\wedge |45|\wedge |12|\wedge |23| \right).
\end{split}
\ee
Then,
\be \nonumber 
\begin{split}
&\alt_6 \left( |24|\wedge |25|\wedge |12|\wedge |23| \right)=
-\alt_6 \left( |24|\wedge |26|\wedge |12|\wedge |23| \right)=\\
&-\alt_6 \left( |26|\wedge |25|\wedge |12|\wedge |23| \right)=-a_1.
\end{split}
\ee

{Indeed, for instance 
\be \nonumber 
\alt_6 \left( |24|\wedge |25|\wedge |12|\wedge |23| \right)=-a_1=-\alt_6(|12|\wedge |13| \wedge |14| \wedge [15]),
\ee
since the permutation 
$ 
\sigma=
\begin{pmatrix} 
1 & 2 & 3 & 4 & 5 & 6\\ 
4 & 1 & 5 & 2 & 3 & 6 
\end{pmatrix} $ is odd and sends  $|24|\wedge |25|\wedge |12|\wedge |23|$ to 
$|12|\wedge |13| \wedge |14|\wedge |15|$. The reader should keep in mind that symbols $|ij|$  are symmetric. 
}

Next,
\be \nonumber 
\begin{split}
&\alt_6 \left( |56|\wedge |46|\wedge |12|\wedge |23| \right)\stackrel{(4\leftrightarrow 6)}{=}
-\alt_6 \left( |45|\wedge |46|\wedge |12|\wedge |23| \right)\stackrel{(4\to 5\to 6)}{=}\\
&-\alt_6 \left( |56|\wedge |45|\wedge |12|\wedge |23| \right).
\end{split}
\ee
Each of the equalities is proved by applying a permutation in ${\rm S}_6$. 
But then
\[
-\alt_6 \left( |56|\wedge |45|\wedge |12|\wedge |23| \right)=0
\]
because $|56|\wedge |45|\wedge |12|\wedge |23|$ is fixed by an odd permutation $(16)(25)(34).$

Finally, all the remaining six terms are equal to $a_3.$  
\end{proof}

Lemma \ref{L4.5} implies that $\delta$ induces an isomorphism of the alternating components. So 
\be \nonumber 
\begin{split}
&\delta   \alt_6\Biggl(\Bigl(  \frac{1}{2}\{2,4,5, 6\}_2
-  \{1,2,4,5\}_2 \Bigr)\otimes |12|\wedge |23|+
 \{1,3,5,6\}_2 \otimes |12|\wedge |34|\Biggr) = \\
 &\delta \left(\frac{1}{2}b_1-b_2+b_3 \right)=\frac{1}{2}(-3a_1+6a_3)-(a_1-2a_2-a_3)+(-2a_2-4a_3)=-\frac{5}{2}a_1.
\end{split}
\ee
This finishes the proof of part a).

 b)   Let's fix a configuration 
 $\mathcal{C}$ of six vectors $v_1, \ldots ,v_6$ in two-dimensional vector space. Then 
\be \nonumber
\begin{split}
&{\bf r_5^*(1)}\circ p (\mathcal{C})=
ev\left( -\frac{5}{2}{\rm Alt}_6 (|12|\wedge |13|\wedge |14|\wedge |15|)\right)=ev \left (-\frac{5}{2}a_1 \right)\\
&{\bf r_6^*(2)}(\mathcal{C})=ev \left(\frac{1}{2}b_1-b_2+b_3 \right).\\
\end{split}\ee
Therefore 
$
\delta \circ {\bf r_6^*(2)}(\mathcal{C})={\bf r_5^*(1)}\circ p (\mathcal{C}).
$
\end{proof}
\subsection{Commutativity of the second square.}
\begin{theorem} \label{Sq2}
a) The map 
\be \nonumber  
\delta \colon \left ( {\bf b}_3(3,7)\otimes{\bf b}_1(3,7) \right ) \oplus
\Lambda^2 [{\bf b}_2(3,7)]  \lra {\bf b}_2(3,7) \otimes \Lambda^2 [{\bf b}_1(3,7)]
\ee
sends vector
\be \nonumber
\alt_{7}\Biggl(\frac{1}{108} \{1,2,3,4,5,6\}_3 \otimes   |123| -
  \{6|2, 4,3,5\}_3 \otimes   |123| - \frac{3}{28} \{1|2,5,6,7\}_2\wedge  \{2|1,5,3,4\}_2\Biggr)
\ee
to 
\be \nonumber  \alt_7\Biggl(\Bigl(  \frac{1}{2}\{1|3,5,6, 7\}_2
-\{1|2,3,5,6\}_2 \Bigr)\otimes |123|\wedge |134| +
 \{1|2,4,6,7\}_2 \otimes |123|\wedge |145|\Biggr).
 \ee
b) The diagram 
\be \nonumber 
\begin{gathered}
    \xymatrix{
 C_7(3) \ar[r]^{p}   \ar[d]^{{\bf r^*_7(3)}}     &C_6(2)   \ar[d]^{{\bf r^*_6(2)}}   \\ 
   {\rm B}_3 \otimes {\rm F}^\times \ar[r]^{\delta~~} & {\rm B}_2 \otimes \Lambda^2{\rm F}^\times    }
\end{gathered}
 \ee 
is commutative.
\end{theorem}

Consider the following vectors in  
${\bf b}_2(3,7) \otimes \Lambda^2 [{\bf b}_1(3,7)]:$\footnote{{Note that the alternating component of the representation   has dimension 7.}} 
\be \nonumber
\begin{split}
&t_1= \alt_7 \left( \{2|3,4,5, 6\} \otimes |123|\wedge |345|  \right),\\
&t_2=\alt_7 \left( \{3|2,4,6,7\} \otimes |123|\wedge |345| \right),\\
&t_3=\alt_7 \left( \{6|2,3,4,5\} \otimes |123|\wedge |345| \right),\\
&t_4=\alt_7 \left( \{5|2,3,4,6\} \otimes |123|\wedge |345| \right),\\
\end{split}
\ee
\be \nonumber 
\begin{split}
&t_5=\alt_7 \left( \{5|2,3,4, 6\} \otimes |123|\wedge |234| \right),\\
&t_6=\alt_7 \left( \{4|2,3,5,6\} \otimes |123|\wedge |234| \right),\\
&t_7=\alt_7 \left( \{2|3,5,6,7\} \otimes |123|\wedge |234| \right),\\
&t_8=\alt_7 \left( \{2|1,3,5,6\} \otimes |123|\wedge |234| \right).\\
\end{split}
\ee

\bl 
a)The following equalities hold:
\be \la{92}
\begin{split}
&t_1+t_3-2t_4+t_2=0,\\
&2t_5+2t_8-t_6=0.\\
\end{split}
\ee
b)
For a configuration $\mathcal{C}$ of seven vectors  in a 3-dimensional space, the evaluation map sends
\be \nonumber
-\frac{1}{2}t_7+t_8-t_2 ~~\lms ~~
r_6^*(2)\circ p(\mathcal{C}).
\ee
\el
\begin{proof}
a) We start  from the five-term relation for the configuration $(3|2,4,5,6,7)$: 
\be \nonumber
\{3|4,5,6,7\}_2-\{3|2,5,6,7\}_2+\{3|2,4,6,7\}_2-\{3|2,4,5,7\}_2+\{3|2,4,5,6\}_2=0.
\ee
Multiplying it by $|123|\wedge |345|$ and alternating we get 
\be \nonumber t_2=\alt_7 \left ( \{3|2,4,6,7\}_2 \otimes |123|\wedge |345| \right )=-\alt_7 \left ( \{3|2,4,5,6\}_2 \otimes |123|\wedge |345| \right ).
\ee
{ Here is a   detailed argument. The five-term relation above leads to the following identity:
\be \nonumber \begin{split}
&\alt_7 (\{3|4,5,6,7\}_2 \otimes |123|\wedge |345|)-\alt_7 (\{3|2,5,6,7\}_2\otimes |123|\wedge |345|)+\\
&\alt_7 (\{3|2,4,6,7\}_2\otimes |123|\wedge |345|)-\alt_7 (\{3|2,4,5,7\}_2\otimes |123|\wedge |345|)+\\
&\alt_7 (\{3|2,4,5,6\}_2\otimes |123|\wedge |345|)=0.
\end{split}
\ee
The expression $\{3|4,5,6,7\}_2 \otimes |123|\wedge |345|$ is symmetric in  $(1, 2)$ and thus   vanishes after the alternation. The remaining four terms group into two  pairs:
\be \nonumber \begin{split}
&\alt_7 \left ( \{3|2,4,6,7\}_2 \otimes |123|\wedge |345| \right )=-\alt_7 \left ( \{3|2,5,6,7\}_2 \otimes |123|\wedge |345| \right ),\\
&\alt_7 \left ( \{3|2,4,5,6\}_2 \otimes |123|\wedge |345| \right )=-\alt_7 \left ( \{3|2,4,5,7\}_2 \otimes |123|\wedge |345| \right ).
\end{split}
\ee
This gives the equality above.
}

Next,   the dual  five-term relation tells that
\be \la{93}
\{2|3,4,5,6\}-\{3|2,4,5,6\}+\{4|2,3,5,6\}-\{5|2,3,4,6\}+\{6|2,3,4,5\}=0.
\ee
Multiplying it by $|123|\wedge |345|$,  alternating, and using the previous computation we get 
\be \nonumber 
t_1+t_2-t_4-t_4+t_3=0.
\ee

To prove the second formula in (\ref{92})  we multiply   (\ref{93}) 
  by $|123|\wedge |234|$ and alternate, getting
\be \nonumber -t_8-t_8+t_6-t_5-t_5=0.
\ee

b) This   follows  from Theorem \ref{Sq1}. Indeed,
from it we see that
\be \nonumber \begin{split}
&{\bf r_6^*(2)}\circ p(\mathcal{C})=\\
&\alt_7\Biggl(\Bigl(  \frac{1}{2}\{1|3,5,6, 7\}_2
-\{1|2,3,5,6\}_2 \Bigr)\otimes |123|\wedge |134|+
 \{1|2,4,6,7\}_2 \otimes |123|\wedge |145|\Biggr)=\\
 &\frac{1}{2}\alt_7(\{1|3,5,6, 7\}_2 \otimes |123|\wedge |134|)-\alt_7 (\{1|2,3,5,6\}_2\otimes |123|\wedge |134|)+\\
&\alt_7 ( \{1|2,4,6,7\}_2 \otimes |123|\wedge |145|)=\\
&-\frac{1}{2}\alt_7(\{2|3,5,6, 7\}_2 \otimes |123|\wedge |234|)+\alt_7 (\{2|1,3,5,6\}_2\otimes |123|\wedge |234|)-\\
&\alt_7 ( \{3|2,4,6,7\}_2 \otimes |123|\wedge |345|)=-\frac{1}{2}t_7+t_8-t_2.\\
\end{split}
\ee
\end{proof}

We introduce the following elements in  $\Bigl ( {\bf b}_3(3,7)\otimes{\bf b}_1(3,7) \Bigr) \oplus \Lambda^2 [{\bf b}_2(3,7)]$:
\be \nonumber
\begin{split}
&r_1= \alt_7 \left(\frac{1}{108} \{1,2,3,4,5,6\}_3\otimes |123| \right),\\
&r_2=\alt_7 \left( \{6|2,4,3, 5\}_3 \otimes  |123| \right),\\
&r_3=\alt_7 \left( \{1|2,5,6,7\}_2 \wedge \{2|1,5,3,4\}_2
\right).\\
\end{split}
\ee

Part c) of Lemma \ref{r} below is the hardest part of the proof. In the next section, we will give a representation-theoretic argument clarifying it.

\bl \label{r}
\be  \nonumber 
\begin{split}
&\delta(r_1)=t_1+t_3-t_5,\\
&\delta(r_2)=t_6+2t_4,\\
&\delta(r_3)=14\left (\frac{1}{3}t_7-t_6 \right ).\\
\end{split}
\ee
\el
\begin{proof}
a) We start from computing the cobracket of $\delta(r_1)$ directly:
\be \nonumber 
\begin{split}
&\delta(r_1)=\delta \alt_7 \left ( \frac{1}{108} \{1,2,3,4,5,6\}_3\otimes |123| \right )=\\
&  \frac{1}{108} \alt_7 \left ( \alt_6 \left (\{1|2345\}_2\otimes |345| \right )\wedge |123| \right )=\\
&\frac{1}{108}\sum_{\sigma \in {\rm S}_7}\sum_{\tau \in {\rm S}_6}
(-1)^{\sigma+\tau} \bigl ( \{\sigma\tau(1)|\sigma\tau(2)\sigma\tau(3)\sigma\tau(4)\sigma\tau(5)\}_2\bigr )\otimes |\sigma\tau(3)\sigma\tau(4)\sigma\tau(5)|\wedge |\sigma(1)\sigma(2)\sigma(3)|=\\
& \frac{1}{108\cdot 6!} \alt_7 \left ( \alt_6 \left (\{1|2,3,4,5\}_2\otimes |345| \wedge Sym_6(|123|) \right )\right )=\\
& \frac{1}{108} \alt_7 \left ( \{1|2345\}_2\otimes |345| \wedge Sym_6(|123|) \right ).\\
\end{split}
\ee
For any indices $i, j, k \neq 6$ the following expression is symmetric under the transposition $6 \leftrightarrow 7$:
\be \nonumber 
\{1|2,3,4,5\}_2\otimes |345| \wedge |ijk|. 
\ee
Therefore it vanishes after the alternation.  
From this, we get
\be \nonumber 
\begin{split}
\frac{108}{36}\delta(r_1)=
&  \alt_7 \left ( \{1|2,3,4,5\}_2\otimes |345| \wedge |612| \right )+\alt_7 \left ( \{1|2,3,4,5\}_2\otimes |345| \wedge |613| \right )+\\
&\alt_7 \left ( \{1|2,3,4,5\}_2\otimes |345| \wedge |614| \right )+\alt_7 \left ( \{1|2,3,4,5\}_2\otimes |345| \wedge |615| \right )+\\
&\alt_7 \left ( \{1|2,3,4,5\}_2\otimes |345| \wedge |623| \right )+\alt_7 \left ( \{1|2,3,4,5\}_2\otimes |345| \wedge |624| \right )+\\
&\alt_7 \left ( \{1|2,3,4,5\}_2\otimes |345| \wedge |625| \right )+\alt_7 \left ( \{1|2,3,4,5\}_2\otimes |345| \wedge |634| \right )+\\
&\alt_7 \left ( \{1|2,3,4,5\}_2\otimes |345| \wedge |635| \right )+\alt_7 \left ( \{1|2,3,4,5\}_2\otimes |345| \wedge |645| \right ).\\
\end{split}
\ee
Observe that
\be \nonumber 
\alt_7\left ( \{1|2,3,4,5\}_2\otimes |345| \wedge |612| \right )=0.
\ee
One can easily deduce this from the five-term relation. Here is a representation-theoretic argument. Consider the sub-representation in 
$\Lambda^2 [{\bf b}_1(3,7)]$ generated by the vector $|345|\wedge|612|$ and denote it by $L$. It decomposes into irreducibles in the following way:
\be \nonumber 
L=[3,3,1]\oplus [4,3] \oplus [5,1,1] \oplus [5,2] \oplus [6,1].
\ee
A direct computation with characters shows that the representation ${\bf b}_2(3,7) \otimes L$ does not have an alternating component. From this, the statement follows.

Next,
\be \nonumber 
\begin{split}
&\alt_7\left ( \{1|2,3,4,5\}_2\otimes |345|\wedge |613|  \right )=\alt_7\left ( \{1|2,3,4,5\}_2\otimes |345|\wedge |614|  \right )=\\
&\alt_7\left ( \{1|2,3,4,5\}_2\otimes |345|\wedge |615|  \right )=t_1,\\
\end{split}
\ee
\be \nonumber 
\begin{split}
&\alt_7\left ( \{1|2,3,4,5\}_2\otimes |345|\wedge |623|  \right )=\alt_7\left ( \{1|2,3,4,5\}_2\otimes |345|\wedge |624|  \right )=\\
&\alt_7\left ( \{1|2,3,4,5\}_2\otimes |345|\wedge |625|  \right )=t_3,\\
\end{split}
\ee
\be \nonumber \begin{split}
&\alt_7\left ( \{1|2,3,4,5\}_2\otimes |345|\wedge |634|  \right )=\alt_7\left ( \{1|2,3,4,5\}_2\otimes |345|\wedge |635|  \right )=\\
&\alt_7\left ( \{1|2,3,4,5\}_2\otimes |345|\wedge |645|  \right )=-t_5,\\
\end{split}
\ee
It follows that
\be \nonumber 
\delta(r_1)=t_1+t_3-t_5.
\ee

b) We have
\be \nonumber 
\begin{split}
&\delta(r_2)=\alt_7 \left ( \{6|2,4,3,5\}_2\otimes \frac{|623||645|}{|643||625|} \wedge |123|  \right )=\\
&\alt_7\left ( \{6|2,4,3,5\}_2\otimes |623|\wedge |123|  \right )+\alt_7\left ( \{6|2435\}_2\otimes |645| \wedge |123| \right )-\\
&-\alt_7\left ( \{6|2,4,3,5\}_2\otimes |643|\wedge |123|  \right )-\alt_7\left ( \{6|2435\}_2\otimes |625|\wedge |123|  \right ).
\\
\end{split}
\ee
It is easy to see that 
\be \nonumber 
\begin{split}
&\alt_7\left ( \{6|2,4,3,5\}_2\otimes |623|\wedge |123|  \right )=t_6,\\
&\alt_7\left ( \{6|2,4,3,5\}_2\otimes |643|\wedge |123|  \right )=-t_4,\\
&\alt_7\left ( \{6|2,4,3,5\}_2\otimes |625|\wedge |123|  \right )=-t_4.
\\
\end{split}
\ee
It remains to show that
\be \nonumber 
\alt_7\left ( \{6|2,4,3,5\}_2\otimes |645| \wedge |123| \right )=0.
\ee
This follows from a similar representation-theoretic argument  as in part a), or can be checked directly.

c)
We start by expanding the cobracket $\delta (r_3)$:
\be \nonumber 
\begin{split}
&\frac{1}{2}\delta \left ( \alt_7\{1|2,5,6,7\}_2\wedge\{2|1,5,3,4\}_2\right )=\\
&\alt_7\left ( \{1|2,5,6,7\}_2 \otimes \frac{|215||234|}{|214||253|}\wedge \frac{|213||254|}{|214||253|} \right )=\\
&\alt_7 ( \{1|2,5,6,7\}_2\otimes |215| \wedge |213|)+\alt_7 (\{1|2,5,6,7\}_2\otimes |215| \wedge |254|)+ \\
&\alt_7 (\{1|2,5,6,7\}_2\otimes |234| \wedge |213|)+\alt_7 (\{1|2,5,6,7\}_2\otimes |234| \wedge |254|)- \\
&-\alt_7 (\{1|2,5,6,7\}_2\otimes |214| \wedge |213|)-\alt_7 (\{1|2,5,6,7\}_2\otimes |214| \wedge |254|)- \\
&-\alt_7 (\{1|2,5,6,7\}_2\otimes |253| \wedge |213|)-\alt_7 (\{1|2,5,6,7\}_2\otimes |253| \wedge |254|)- \\
&-\alt_7 (\{1|2,5,6,7\}_2\otimes |215| \wedge |214|)-\alt_7 (\{1|2,5,6,7\}_2\otimes |215| \wedge |253|)- \\
&-\alt_7 (\{1|2,5,6,7\}_2\otimes |234| \wedge |214|)-\alt_7 (\{1|2,5,6,7\}_2\otimes |234| \wedge |253|). \\
\end{split}
\ee
Next, we transform the right-hand side part of each of the terms to the form $|123| \wedge |234|$: 
\be \nonumber 
\begin{split}
&\frac{1}{2}\delta \left ( \alt_7\{1|2,5,6,7\}_2\wedge\{2|1,5,3,4\}_2\right )=\\
&-\alt_7 ( \{3|2,1,6,7\}_2\otimes |123| \wedge |234|)-\alt_7 (\{1|2,3,6,7\}_2\otimes |123| \wedge |234|)- \\
&-\alt_7 (\{1|2,5,6,7\}_2\otimes |123| \wedge |234|)-\alt_7 (\{5|2,1,6,7\}_2\otimes |123| \wedge |234|)- \\
&-\alt_7 (\{3|2,5,6,7\}_2\otimes |123| \wedge |234|)-\alt_7 (\{1|2,4,6,7\}_2\otimes |123| \wedge |234|)- \\
&-\alt_7 (\{1|2,4,6,7\}_2\otimes |123| \wedge |234|)-\alt_7 (\{5|2,3,6,7\}_2\otimes |123| \wedge |234|)- \\
&-\alt_7 (\{3|2,1,6,7\}_2\otimes |123| \wedge |234|)-\alt_7 (\{1|2,3,6,7\}_2\otimes |123| \wedge |234|)- \\
&-\alt_7 (\{1|2,5,6,7\}_2\otimes |123| \wedge |234|)-\alt_7 (\{5|2,1,6,7\}_2\otimes |123| \wedge |234|).\\
\end{split}
\ee
Finally we sum up the repeating terms:

\be \nonumber 
\begin{split}
&\frac{1}{2}\delta \left ( \alt_7\{1|2567\}_2\wedge\{2|1534\}_2\right )=\\
&-2\alt_7 ( \{3|2,1,6,7\}_2\otimes |123| \wedge |234|)-2\alt_7 (\{1|2,3,6,7\}_2\otimes |123| \wedge |234|)- \\
&-2\alt_7 (\{1|2,5,6,7\}_2\otimes |123| \wedge |234|)-2\alt_7 (\{5|2,1,6,7\}_2\otimes |123| \wedge |234|)- \\
&-2\alt_7 (\{1|2,4,6,7\}_2\otimes |123| \wedge |234|)-\alt_7 (\{3|2,5,6,7\}_2\otimes |123| \wedge |234|)- \\
&-\alt_7 (\{5|2,3,6,7\}_2\otimes |123| \wedge |234|). \\
\end{split}
\ee
Our goal is to express it in terms of the vectors $t_5,t_6,t_7,t_8.$

From the dual five-term relation for $(5,2,3,6,7)$ we conclude that
\be \nonumber 
3\alt_7(\{5|2,3,6,7\}_2 \otimes |123|\wedge|234|)=-2\alt_7 (\{2|3,5,6,7\}_2\otimes |123|\wedge|234|)=-2t_7.
\ee
Next,
\be \nonumber 
\begin{split}
&\alt_7 (\{3|2,1,6,7\}_2\otimes |123| \wedge |234|)=\alt_7 (\{2|1,3,5,6\}_2\otimes |123| \wedge |234|)=t_8,\\
&\alt_7 (\{1|2,3,6,7\}_2\otimes |123| \wedge |234|)=\alt_7(\{4|2,3,5,6\}_2\otimes |123| \wedge |234|)=t_6.\\
\end{split}
\ee
From the five-term relation for $(1|2,3,5,6,7)$ we have
\be \nonumber 
2\alt_7 (\{1|2,5,6,7\}_2\otimes |123| \wedge |234|)=3\alt_7 (\{1|2,3,5,6\}_2\otimes |123| \wedge |234|)=3t_6.
\ee
From  the dual five-term relation for $(1,2,5,6,7)$ we see that
\be \nonumber 
\begin{split}
&\alt_7 (\{1|2,5,6,7\}_2\otimes |123| \wedge |234|)-\alt_7 (\{2|1,5,6,7\}_2\otimes |123| \wedge |234|)+\\
&3\alt_7 (\{5|1,2,6,7\}_2\otimes |123| \wedge |234|)=0.
\end{split}
\ee
From the five-term relation for $(2|1,3,5,6,7)$ we obtain
\be \nonumber 
\alt_7 (\{2|1,5,6,7\}_2\otimes |123| \wedge |234|)=3t_8+t_7. 
\ee
Therefore we have
\be \nonumber 
\alt_7 (\{5|1,2,6,7\}_2\otimes |123| \wedge |234|)=\frac{1}{3}(-\frac{3}{2}t_6+3t_8+t_7).
\ee
Finally, from the five-term relation for $(1|2,3,4,6,7)$ we see that
\be \nonumber 
\begin{split}
&-2\alt_7 (\{1|2,4,6,7\}_2\otimes |123| \wedge |234|)+\alt_7 (\{1|2,3,6,7\}_2\otimes |123| \wedge |234|)+\\
&2\alt_7 (\{1|2,3,4,6\}_2\otimes |123| \wedge |234|)=0.
\end{split}
\ee
The last term vanishes since it is symmetric in (5,7). So
\be \nonumber 
\alt_7 (\{1|2,4,6,7\}_2\otimes |123| \wedge |234|)=
\frac{1}{2}\alt_7 (\{1|2,3,6,7\}_2\otimes |123| \wedge |234|)=\frac{1}{2}t_6.
\ee
Combining all the above, we get
\be \nonumber 
\begin{split}
&\frac{1}{2}\delta \left ( \alt_7\{1|2,5,6,7\}_2\wedge\{2|1,3,4,5\}_2\right )=\\
&-2t_8-2t_6-3t_6-\frac{2}{3}(\frac{3}{2}t_6-3t_8-t_7)-t_6
+t_7+\frac{2}{3}t_7=\\
&\frac{7}{3}t_7-7t_6.\\
\end{split}
\ee
\end{proof}
 
Now it is not hard to finish the proof of Theorem \ref{Sq2}.
 \be \nonumber 
  \begin{split} 
&\alt_7\Biggl(\Bigl(  \frac{1}{2}\{1|3,5,6, 7\}_2
-\{1|2,3,5,6\}_2 \Bigr)\otimes |123|\wedge |134|+
 \{1|2,4,6,7\}_2 \otimes |123|\wedge |145|\Biggr)=\\
&-\frac{1}{2}t_7+t_8-t_2=-\frac{1}{2}t_7+t_8-(-t_1-t_3+2t_4)=\\
&-\frac{1}{2}t_7+t_8+t_5+t_6+\delta(r_1)-\delta(r_2)=\\
&\frac{3}{2}t_6-\frac{1}{2}t_7+\delta(r_1)-\delta(r_2)=-\frac{3}{2}(\frac{1}{3}t_7-t_6)+\delta(r_1)-\delta(r_2)=\\
&-\frac{3}{28}\delta(r_3)+\delta(r_1)-\delta(r_2).\\
\end{split} 
\ee

\subsection{Comments}
\paragraph{The cobracket to $\Lambda^4 [{\bf b}_1(3,7)]$.}
Let us compute    the   map $\delta$ on elements $t_i.$ For this consider the following elements 
in $\Lambda^4 [{\bf b}_1(3,7)]$:
\be  \nonumber 
\begin{split}
& s_1=\alt(|256|\wedge |246|\wedge |345| \wedge |123|),\\
& s_2=\alt([235]\wedge |236| \wedge |345| \wedge |123|),\\
& s_3=\alt(|256|\wedge |236| \wedge |345|\wedge |123| ),\\
& s_4=\alt(|256|\wedge |254| \wedge |234|\wedge |123| ).\\
\end{split}
\ee
Elements $s_i$ are linearly independent, and span a hyperplane in the $5$-dimensional space $\left (\Lambda^4 [{\bf b}_1(3,7)] \right )^{\rm alt}.$
 
\bl
The following equalities hold:
\be  \nonumber 
\begin{split}
&\delta(t_1)= -s_1-2s_2-4s_3, ~~~~\delta(t_2)=4s_2-2s_4,\\
&\delta(t_3)=s_1+4s_3+2s_4, ~~~~~~\delta(t_4)=s_2,\\
&\delta(t_5)=-2s_2+2s_4, ~~~~~~~~~~\delta(t_6)=-2s_2,\\
&\delta(t_7)=-6s_2,~~~~~~~~~~~~~~~~~\delta(t_8)=s_2-2s_4.\\
\end{split}
\ee
\el
\begin{proof}
The proof is a direct computation.
\end{proof}
These equalities can be used to check the correctness of the computations in the previous section using the fact that $\delta^2=0.$

\paragraph{Representation theory approach to Lemma \ref{r}.} 
 Representation $\Lambda^2 b_1(3,7)$ decomposes into three pieces:

1)[3, 3, 1]+[4,3]+[5,1,1]+[5,2]+[6,1], generated by $|123|\wedge |456|$,

2)[3, 2, 1, 1]+[3, 3, 1]+[4, 1, 1, 1]+2[4, 2, 1]+[4, 3]+2[5, 1, 1]+[5, 2]+[6,1], generated by $|123|\wedge |234|$,

3)[3, 2, 1, 1]+[3, 2, 2]+2[3, 3, 1]+[4, 1, 1, 1]+3[4, 2, 1]+2[4, 3]+2[5, 1, 1]+2[5, 2]+[6, 1], generated by $|123|\wedge |345|$. 

We want   to prove that there exists a rational number $q$ such that
\be \nonumber 
q\delta(r_3)=\frac{1}{3}t_7-t_6. 
\ee 
Denote by $W$ the subspace of $\left ( {\bf b}_2(3,7) \otimes \Lambda^2 [{\bf b}_1(3,7)] \right )^{\rm alt}$ generated by the elements
\be  \nonumber 
\alt_7 \left ( \{i_1|i_2,i_3,i_4,i_5 \}_2 \otimes |123|\wedge |234|\right ).
\ee
A  computation with characters shows that it has dimension $3.$ Obviously, elements $t_5,t_6, t_7$ lie in $W$. The formula for the cobracket tells that the element $\delta(r_3)$  lies in $W$ as well. 

Since $\frac{1}{3}t_7-t_6$ lies in the kernel, the statement would follow if we show that $\delta(r_3)$  does not vanish, and spans the kernel of the cobracket map
\be \nonumber 
\delta \colon W \lra \Lambda^4 [{\bf b}_1(3,7)].
\ee

First, the vector $\delta(r_3)$ does not vanish. This follows from the fact that  the map 
\be \nonumber 
\delta \colon {\bf b}_2(3,7) \lra \Lambda^2 [{\bf b}_1(3,7)]
\ee
is an inclusion and that   the multiplicities of the alternating components in
\be \nonumber 
\Lambda^2 {\bf b}_2(3,7)~~\mbox{\rm and}~~ 
 {\bf b}_2(3,7) \otimes  {\bf b}_2(3,7)
\ee
are equal to 1. 
Next, on $W$ the rank of the map $\delta$ is at least $2$: $\delta(t_5)$ and $\delta(t_6)$ are linearly independent. So the dimension of the kernel is not greater than one. The statement follows.

\subsection{The third commutative diagram}
The composition 
\be  \nonumber 
p \circ {\bf r_5^*(1)} \colon C_6(1) \lra \Lambda^4{\rm F}^\times
\ee
vanishes. So this reduces the problem to proving the following Theorem \ref{Sq3}.

\begin{theorem} \label{Sq3}
a)The following diagram  is commutative:
\be   \nonumber
\begin{gathered}
    \xymatrix{
 C_7(2) \ar[r]^{\partial}   \ar[d]^{{\bf r^*_7(2)}}     &C_6(2)   \ar[d]^{{\bf r^*_6(2)}}   \\ 
   \Lambda^2 {\rm B}_2 \ar[r]^{\delta~~} & {\rm B}_2 \otimes \Lambda^2{\rm F}^\times   }
\end{gathered}
 \ee 

b) The symbolic variant of the map $\delta$ is given by 
\be \la{Iden}
\begin{split}
&\delta \colon \Lambda^2 [{\bf b}_2(2,7)]  \lra {\bf b}_2(2,7) \otimes \Lambda^2 [{\bf b}_1(2,7)];\\
&\\
&\delta: \alt_{7}(\{1,2,3,4\}_2 \wedge \{1,5,6,7\}_2)\lms \\
&\alt_7\Biggl(\Bigl(  \frac{1}{2}\{2,4,5, 6\}_2
-  \{1,2,4,5\}_2 \Bigr)\otimes |12|\wedge |23|+
 \{1,3,5,6\}_2 \otimes |12|\wedge |34|\Biggr).\\
 \end{split}
\ee\end{theorem}

\begin{proof}  Part a) is a direct corollary of part b).

The multiplicity of the alternating component in the representation 
\be \nonumber
{\bf b}_2(2,7) \otimes \Lambda^2 [{\bf b}_1(2,7)]
\ee
is equal to 1. So  identity (\ref{Iden}) holds up to a rational factor, which we compute below. 
\be \nonumber
\begin{split}
&\delta \left ( \alt_{7}(\{1,2,3,4\}_2 \wedge \{1,5,6,7\}_2)\right )= 2\left (\alt_{7}(\{1,2,3,4\}_2 \otimes \frac{|15||67|}{|17||56|} \wedge  \frac{|16||57|}{|17||56|} \right )=\\
&6\alt_{7}\left ((-\{2,4,5,6\}_2 +2\{1,5,7,4\}_2-\{6,7,5,4\}_2)\otimes |12|\wedge |23|\right ).\\
\end{split}
\ee
From the 5-term relation for $(1,5,7,4,6)$ we get that 
\be \nonumber 
\alt_{7}(\{6,7,5,4\}_2\otimes |12|\wedge |23| )+4\alt_{7}(\{1,5,7,4\}_2\otimes |12|\wedge |23|)=0.
\ee
From the 5-term relation for $(5,2,6,4,7)$ we get that 
\be \nonumber
\alt_{7}(\{6,7,5,4\}_2\otimes |12|\wedge |23| )-4\alt_{7}(\{2,4,5,6\}_2\otimes |12|\wedge |23|)=0.
\ee
So
\be \nonumber
\begin{split}
&\delta \left ( \alt_{7}(\{1,2,3,4\}_2 \wedge \{1,5,6,7\}_2)\right )= -42\alt_{7}(\{2,4,5,6\}_2\otimes |12|\wedge |23| ).
\end{split}
\ee
Now let us simplify the expression
\be \nonumber
\alt_7\Biggl(\Bigl(  \frac{1}{2}\{2,4,5, 6\}_2
-  \{1,2,4,5\}_2 \Bigr)\otimes |12|\wedge |23|+
 \{1,3,5,6\}_2 \otimes |12|\wedge |34|\Biggr).
\ee
First, the expression 
$
\alt_7\left (\{1,2,4,5\}_2 \otimes |12|\wedge |23|\right)
$ 
does  not depend on $(6,7)$, and thus  vanishes. 
Next, the expression 
\be \nonumber
\alt_7(\{1,3,5,6\}_2 \otimes |12|\wedge |34|)
\ee
also vanishes. This follows from the representation theory: the tensor product of 
${\bf b}_2(2,7)$ 
and the representation generated by  $|12|\wedge|34|$ does not have an alternating component.
So
\be \nonumber
\begin{split}
&\alt_7\Biggl(\Bigl(  \frac{1}{2}\{2,4,5, 6\}_2
-  \{1,2,4,5\}_2 \Bigr)\otimes |12|\wedge |23|+
 \{1,3,5,6\}_2 \otimes |12|\wedge |34|\Biggr)=\\
 &\frac{1}{2}\alt_7\{2,4,5, 6\}_2 \otimes |12|\wedge |23|=-\frac{1}{84} \delta \alt_{7}(\{1,2,3,4\}_2 \wedge \{1,5,6,7\}_2).
 \end{split}
\ee
\end{proof}



\end{document}